\documentclass[12pt,reqno]{amsart}
\usepackage{amsmath,amssymb,graphicx,mathrsfs,amsopn,stmaryrd,amsthm,color,bm,extarrows,mathdots,mathtools}
\usepackage[dvipsnames]{xcolor}
\usepackage[colorlinks=true,citecolor=blue,linkcolor=blue%,pagebackref
]{hyperref}
\usepackage{ytableau}
\usepackage[final]{showlabels}
\usepackage[english]{babel}
\usepackage[mathscr]{euscript}
\usepackage[left=2.5cm,
            right=2.5cm,
            bottom=2.75cm
            ]{geometry}
\usepackage{collectbox}
\makeatletter
\newcommand{\mybox}{%
    \collectbox{%
        \setlength{\fboxsep}{1pt}%
        \fbox{\BOXCONTENT}%
    }%
}
\makeatother

\let\emptyset\varnothing
\usepackage{dsfont}
\usepackage{tikz}
\usepackage{tikz-cd}
\usetikzlibrary{arrows,matrix,decorations.markings,shapes.geometric}   
\usetikzlibrary{decorations.pathreplacing}

\usepackage{xy}
\allowdisplaybreaks
\xyoption{all}
\numberwithin{equation}{section}
\newtheorem{thm}{Theorem}[section]
\newtheorem{prop}[thm]{Proposition}
\newtheorem{lem}[thm]{Lemma}
\newtheorem{cor}[thm]{Corollary}
\newtheorem{conj}[thm]{Conjecture}

\newtheorem{alphatheorem}{Theorem}

\theoremstyle{definition} 
\newtheorem{example}[thm]{Example}
\newtheorem{dfn}[thm]{Definition}
\theoremstyle{remark}
\newtheorem{rem}[thm]{Remark}

\newcommand{\beq}{\begin{equation}}
\newcommand{\eeq}{\end{equation}}
\newcommand{\be}{\begin{equation*}}
\newcommand{\ee}{\end{equation*}}

\DeclareFontFamily{OT1}{pzc}{}
\DeclareFontShape{OT1}{pzc}{m}{it}{ <-> s*[1.0] pzcmi7t }{}
\DeclareMathAlphabet{\mathpzc}{OT1}{pzc}{m}{it}
\newcommand{\pf}{\mathpzc{Pf}}

\newcommand{\bC}{\mathbb{C}}

\newcommand{\bZ}{\mathbb{Z}}

\newcommand{\bN}{\mathbb{N}}

\newcommand{\mc}{\mathcal}

\newcommand{\cF}{\mathcal{F}}

\newcommand{\cN}{\mathcal{N}}
\newcommand{\cS}{\mathcal{S}}

\newcommand{\sfh}{\mathsf{h}}
\newcommand{\sfb}{\mathsf{b}}
\newcommand{\sfd}{\mathsf{d}}
\newcommand{\sfe}{\mathsf{e}}
\newcommand{\sff}{\mathsf{f}}

\newcommand{\sfB}{{\overline{\mathsf{B}}}}

\newcommand{\bY}{\mathbf{Y}}

\newcommand{\cb}{\mathsf{b}}
\newcommand{\ch}{\mathsf{h}}

\newcommand{\g}{\mathfrak{g}}
\newcommand{\gl}{\mathfrak{gl}}

\newcommand{\fksp}{\mathfrak{sp}}
\newcommand{\fkso}{\mathfrak{so}}
\newcommand{\fkm}{\mathfrak{m}}
\newcommand{\fkl}{\mathfrak{l}}

\newcommand{\fksl}{\mathfrak{sl}}
\newcommand{\fkS}{\mathfrak{S}}

\newcommand{\fko}{\mathfrak{o}}

\newcommand{\fkc}{\mathfrak{c}}

\newcommand{\fks}{\mathfrak{s}}
\newcommand{\tq}{\tilde{q}}
\newcommand{\tp}{\tilde{p}}

\newcommand{\rU}{\mathrm{U}}
\newcommand{\rY}{\mathrm{Y}}

\newcommand{\rPD}{\operatorname{PD}}
\newcommand{\rQ}{\operatorname{Q}}
\newcommand{\End}{\mathrm{End}}
\newcommand{\Hom}{\mathrm{Hom}}
\newcommand{\id}{{\mathrm{id}}}   
\newcommand{\sdet}{{\mathrm{sdet}}}
\newcommand{\qdet}{{\mathrm{qdet}}}
\newcommand{\gr}{\operatorname{gr}}   
\newcommand{\reg}{{\mathrm{reg}}}
\newcommand{\Aut}{\mathrm{Aut}}
\newcommand{\PAut}{\operatorname{PAut}}
\newcommand{\Out}{\mathrm{Out}}
\newcommand{\Lie}{\operatorname{Lie}}
\newcommand{\ad}{\operatorname{ad}}
\newcommand{\Ad}{\operatorname{Ad}}
\newcommand{\sgn}{\operatorname{sgn}}
\newcommand{\GL}{\operatorname{GL}}
\newcommand{\SL}{\operatorname{SL}}
\newcommand{\SO}{\operatorname{SO}}
\newcommand{\Sp}{\operatorname{Sp}}

\newcommand{\tl}{\tilde}
\newcommand{\wtl}{\widetilde}
\newcommand{\gge}{\geqslant}
\newcommand{\lle}{\leqslant}
\newcommand{\la}{\lambda}

\newcommand{\Y}{{\mathscr{Y}}}
\newcommand{\X}{{\mathscr{X}}}

\newcommand{\arxiv}[1]{\href{http://arxiv.org/abs/#1}{\tt arXiv:\nolinkurl{#1}}}

\begin{document}
\pagestyle{myheadings}
\setcounter{page}{1}
\title[Shifted twisted Yangians and finite $W$-algebras of classical type]{Shifted twisted Yangians and finite $W$-algebras\\ of classical type}

\author[K. Lu]{Kang Lu}
\author[Y.-N. Peng]{Yung-Ning Peng}
\author[L. Tappeiner]{Lukas Tappeiner}
\author[L. Topley]{Lewis Topley}
\author[W. Wang]{Weiqiang Wang}
\address{Department of Mathematics, University of Virginia, Charlottesville, VA 22903, USA}\email{(Lu) kang.lu@virginia.edu, (Wang) ww9c@virginia.edu}
\address{Department of Mathematics, National Cheng Kung University, Tainan City 70101, Taiwan}\email{(Peng) ynp@gs.ncku.edu.tw}
\address{Department of Mathematical Sciences, University of Bath, Claverton Down, Bath BA2 7AY, UK}\email{(Tappeiner) lt862@bath.ac.uk, (Topley) lt803@bath.ac.uk}


\begin{abstract}
    We introduce parabolic presentations of twisted Yangians of types AI and AII, interpolating between the R-matrix presentation and the Drinfeld presentation. Then we formulate and provide parabolic presentations for the shifted twisted Yangians. We define quotient algebras known as truncated shifted twisted Yangians and equip them with baby comultiplications, generalizing the work of Brundan and Kleshchev. PBW bases for all (truncated) shifted twisted Yangians of type AI and AII are established along the way. Applying the theory of universal equivariant quantizations of conic symplectic singularities we show that the truncated twisted shifted Yangian is isomorphic to the finite $W$-algebra which quantizes a suitable Slodowy slice. This provides a presentation of the finite $W$-algebra associated with every even nilpotent element in type {\sf B} and {\sf C}, as well as every nilpotent element with two Jordan blocks in type {\sf D}. Finally we make a conjecture which would supply presentations in the remaining even cases in type {\sf D}. 
\end{abstract}

	\maketitle
	\tableofcontents

\thispagestyle{empty}
\section{Introduction}

\subsection{History} 

Drinfeld-Jimbo quantum groups of affine type and their degenerate cousin, Yangians, are one of the most ubiquitous themes in contemporary representation theory and mathematical physics. The importance of these algebras is partly due to the existence of several (Serre type, Drinfeld type, or R-matrix type) presentations. For example, for $\g = \gl_N$ or for $\g$ a complex simple Lie algebra, the Yangian $\rY(\g)$ admits an R-matrix presentation as well as a Drinfeld presentation deforming the current algebra $\rU(\g[z])$ \cite{Dr87}. Yangian symmetries appear in various contexts in mathematical physics, most notably the study of the 2-dimensional scattering amplitudes, and quantum integrable spin chain models. 

Twisted Yangians, associated with symmetric pairs of Satake type AI and AII, were introduced by Olshanski \cite{Ol92,MNO96} as coideal subalgebras of the Yangian $\rY(\gl_N)$, and they are intimately related to classical Lie algebras; Molev's book \cite{Mol07} is a standard reference in the subject. A general twisted Yangian associated to (classical) symmetric pairs $(\g, \g^\theta)$ has been introduced in \cite{Ma02} and then in R-matrix form in \cite{GR16}, deforming the twisted current algebras $\rU(\g[z]^\theta)$, where the involution $\theta$ is extended to $\g[z]$ by  $xz^r \mapsto (-1)^r\theta(x)z^r$. In recent years, the $\imath$quantum groups arising from quantum symmetric pairs introduced by G.~Letzter and generalized by S.~Kolb have emerged as a vast generalization of Drinfeld-Jimbo quantum groups. Many fundamental results on quantum groups have been generalized to the $\imath$quantum group setting; especially, the (split) affine $\imath$quantum groups admit Drinfeld presentations \cite{LW21, Z22}. Through degeneration, this has led to a Drinfeld presentation of split twisted Yangians \cite{LWZ25}, and for twisted Yangians of type AI, the same Drinfeld presentation was also obtained via Gauss decomposition in \cite{LWZ23}. Thus, it is natural to view twisted Yangians as a close cousin to affine $\imath$quantum groups and label them in terms of Satake diagrams as well. 

Finite $W$-algebras, on the other hand, come from a rather different lineage. They are constructed by quantum Hamiltonian reduction from nilpotent elements in the Lie algebras of reductive groups \cite{Ko78, Ly79, BT93}. More precisely, if $e \in \g$ is such an element then the Slodowy slice $\cS$ to the adjoint orbit of $e$ is equipped with a Poisson structure via Hamiltonian reduction, and the finite $W$-algebra $\rU(\g,e)$ is a filtered algebra with associated graded Poisson algebra $\bC[\cS]$. For the regular nilpotent elements, Kostant showed that the finite $W$-algebra is isomorphic to the center of the enveloping algebra, whilst for $e = 0$ we have $\rU(\g,e) = \rU(\g)$.

The finite $W$-algebras associated to even nilpotent elements were first studied by Lynch, who noticed that the $W$-algebra could be naturally embedded inside the enveloping algebra of the parabolic subalgebra associated to the nilpotent element \cite{Ly79}.  Finally, Premet introduced finite $W$-algebras in full generality in \cite{Pr02} into mathematics (cf. \cite{BT93}), whilst studying the modular representations of Lie algebras of reductive groups. In the same year, Gan--Ginzburg reinterpreted his work in the language of Hamiltonian reductions \cite{GG02}. The subsequent work of Premet, Losev, Ginzburg and others over the following decade highlights their numerous applications to the classification of primitive ideals in enveloping algebras, as well as the classification of quantizations of nilpotent orbits (see \cite{Gi09, Lo10, Lo22, Pr11, To23} and the references therein). More recent applications include the classification of small modular representations of Lie algebras, as well as the positive energy representations of affine $W$-algebras \cite{GT19, AM21, AvE23}. 

Despite the numerous uses of finite $W$-algebras, explicit presentations have not been forthcoming. The first substantial progress was the work of Ragoucy and Sorba \cite{RS99}, who showed that the finite $W$-algebra associated to a rectangular nilpotent element with $N$ Jordan blocks of size $\ell$ in the general linear Lie algebra, is isomorphic to the level $\ell$ truncation of the Yangian $\rY_N := \rY(\gl_N)$. %This truncation Yangian had been considered by Cherednik.
This work was generalized by Brundan and Kleshchev \cite{BK06}, who introduced the shifted Yangians $\rY_N(\sigma)$ and showed that every finite $W$-algebra for $\gl_M$ associated to an element with $N$ Jordan blocks, arises as a truncation of some $\rY_N(\sigma)$. The main technical tools introduced in their proof were parabolic presentations \cite{BK05} and baby comultiplication. We refer to \cite{Pe16, Pe21} for generalizations of these constructions to finite $W$-superalgebras of type A. 

Ragoucy discovered that the level $\ell$ truncation of the semiclassical twisted Yangian of type AI and  AII is isomorphic to a semiclassical finite $W$-algebra for a rectangular nilpotent element in a classical Lie algebra \cite{Ra01}. The quantum version of this result was proven by Brown \cite{Br09} using different methods. In \cite{TT24} two of the authors extended Ragoucy's semiclassical result to many other even nilpotent elements in type {\sf B}{\sf C}{\sf D}, combining their results on Dirac reduction of the semiclassical limit of $\rY_N(\sigma)$ with the new Drinfeld presentation for twisted Yangians of type AI introduced in \cite{LWZ23}.

There has been a long-standing hope that more general finite $W$-algebras in types {\sf B}{\sf C}{\sf D} will appear as quotients of the twisted Yangians, which was rekindled by the recent construction of Drinfeld presentations of twisted Yangians of type AI \cite{LWZ23}. Interesting partial results have been obtained by De Sole--Kac--Valeri \cite{DSKV18} and by Brown \cite{Br09,Br16}. 

In this paper, we make comprehensive progress on this problem by proving isomorphisms between truncated shifted twisted Yangians in general and even finite $W$-algebras in types {\sf B} and {\sf C}, as well as those associated to two-row partitions in type {\sf D}. Modulo a conjecture about the center of certain truncated shifted twisted Yangians, the same result holds for all even nilpotent elements in type {\sf D}.

%\subsection{Results}

\subsection{Parabolic presentations} 
Recall $\rY_N = \rY(\gl_N)$. The twisted Yangian $\Y_N^+ \subseteq \rY_N$ of type AI is associated with the symmetric pair $(\gl_N, \fkso_N)$ whilst the twisted Yangian $\Y_N^- \subseteq \rY_N$ of type AII is associated with the symmetric pair $(\gl_N, \fksp_N)$. More specifically, $\Y_N^\pm$ is generated by the coefficients of the $S$-matrix $S(u) = (s_{ij}(u))$ in $\operatorname{Mat}_N(\Y_N^\pm[[u^{-1}]])$, where $s_{ij}(u) = \sum_{r\gge 0} s_{ij}^{(r)} u^{-r}$, subject to the \textit{quaternary} and \textit{symmetry} relations \eqref{qua}--\eqref{sym}. We refer to Molev's book \cite{Mol07} and \S\ref{ss:TwistedYangians} for more detail.

%\beq
%\begin{split}
%(u^2-v^2)[s_{ij}(u),s_{kl}(v)]=&\, (u+v)(s_{kj}(u)s_{il}(v)-s_{kj}(v)s_{il}(u))\\
%- &\, (u-v)(\theta_k\theta_{j'}s_{ik'}(u)s_{j'l}(v)-\theta_i\theta_{l'}s_{ki'}(v)s_{l'j}(u))\\
%&\, \hskip 1.64cm  +  \theta_{i}\theta_{j'}(s_{ki'}(u)s_{j'l}(v)-s_{ki'}(v)s_{j'l}(u)),
%\end{split}
%\eeq
%and the \textit{symmetry} relations
%\beq\label{sym}
%\theta_i\theta_js_{j'i'}(-u)=  s_{ij}(u)\pm\frac{s_{ij}(u)-s_{ij}(-u)}{2u}.
%\eeq

%We say that $\mu = (\mu_1, \mu_2, ..., \mu_n)$ is a strict composition of length $n$ if all $\mu_i$ are positive integers and $N = \sum_i \mu_i$. 
For each (strict) composition $\mu= (\mu_1, \mu_2, ..., \mu_n)$ of $N$, applying a $\mu\times \mu$-block version of Gauss decomposition to $S(u)$ in the spirit of \cite{LWZ23} (also see \cite{BK05}) leads to a parabolic presentation of the twisted Yangian $\Y_N^+$ below; a parabolic presentation for $\Y_N^-$ requires all $\mu_i$ to be even. 
%This constitutes our first main theorem.

\begin{alphatheorem} [Parabolic presentations, Theorem~\ref{mainthm}, Theorem~\ref{redthm}] 
\label{thm:A}
\setlength{\itemsep}{6pt}
Let $\mu = (\mu_1,...,\mu_n)$ be a strict composition of $N$, with additional assumption that all $\mu_i$ are even in type AII. Then $\Y_N^\pm$ is generated by the elements
    \begin{eqnarray}
    \label{e:parabolicgenerators}
    \begin{array}{rcl}
        & & \{H_{a;i,j}^{(r)},\wtl H_{a;i,j}^{(r)}  \mid 1\lle a\lle n, \ 1\lle i,j,\lle \mu_a, \ r\gge 0\} \ \cup \\ & & \hspace{40pt} \{B_{a;i,j}^{(r)} \mid 1\lle a<n, \ 1\lle i\lle \mu_{a+1}, \ 1\lle j\lle \mu_a, \ r>0\},
        \end{array}
    \end{eqnarray}
    subject to either the relations \eqref{pr1}--\eqref{pr-2} or the relations \eqref{pr1}--\eqref{pr7} and  \eqref{pr-1-2-better}.
\end{alphatheorem}

The formulation and the proof of this theorem are significantly more involved than for $\rY_N$ in \cite{BK05}. Already the current Serre relations which involve lower terms for the Drinfeld presentation of $\Y^+_N$ (which is the special case for $\mu=(1^N)$ here) are highly nontrivial \cite{LWZ23, LWZ25}, and the Serre relations for $\Y_N^\pm$ in Theorem \ref{thm:A} even in this case are formulated quite differently. A complete proof of this theorem requires a careful study of the cases $n = 2,3$, which we carry out in Sections~\ref{sec:lower2} and \ref{sec:lower3}. 

Let us comment on the evenness on the composition $\mu$ for type AII. The symmetric pair $(\gl_N, \gl_N^\theta)$ (with $N=2r$ even) of type AII is not quasi-split, and the corresponding Satake diagram contains a subdiagram $I_\bullet$ of black nodes. The subalgebra corresponding to $I_\bullet$ is a direct sum of $r$ copies of $\gl_2$. Associated to the composition $\mu=(2^r)$, the minimal parabolic presentation for $\Y_N^-$ in Theorem \ref{thm:A} contains $r$ copies of $\Y_2^-$ in R-matrix form. It remains unclear how to construct a Drinfeld presentation for $\Y_N^-$ corresponding to $\mu=(1^N)$ already when $N=2$ (and one can question its existence). Theorem \ref{thm:A} in the generality with $\mu$ even for type AII turns out to be all we need for our application to finite $W$-algebras. 

When we consider the algebra $\Y_N^\pm$ with its parabolic presentation associated with $\mu$ we often use the notation $\Y_\mu^\pm$, although we note that this algebra does not depend on $\mu$.

\subsection{Shifted twisted Yangians}

Now fix $0 < n \lle N$. Following \cite{BK06}, a {\it shift matrix of size $N$} is a matrix $\sigma=(\fks_{i,j})_{1\lle i,j\lle N}$ of nonnegative integers (shifts) such that 
\beq\label{shift1}
\fks_{i,j}+\fks_{j,k}=\fks_{i,k}
\eeq
whenever $|i-k| = |i-j| + |j-k|$. In this paper, all shift matrices are symmetric $\fks_{i,j} = \fks_{j,i}$ for $1\lle i,j \lle N$. We note that $\fks_{i,i} = 0$ for $i=1,...,N$, and that a symmetric shift matrix is determined by the integers $\fks_{1,2}, \fks_{2,3},..., \fks_{N-1,N}$, thanks to \eqref{shift1}.

Let $\mu = (\mu_1,...,\mu_n)$ be a strict composition of length $n$. We say that $\mu$ is an {\it admissible shape for the shift matrix $\sigma$} if
$\fks_{i,j} = 0$ for all $\sum_{l = 1}^{a-1} \mu_l < i,j \lle \sum_{l=1}^a \mu_l$ for all $a = 1,...,n$ (by convention $\sum_{l = 1}^{0} \mu_l = 0$).

Let $\sigma$ be an $N\times N$ symmetric shift matrix and let $\mu$ be an admissible shape for $\sigma$. When we consider a twisted Yangian of type AII, we always assume that all parts of $\mu$ are even. We define the {\it shifted twisted Yangian} $\Y_\mu^\pm(\sigma)$ to be the algebra generated by
    \begin{eqnarray}
       \label{e:shiftedparabolicgenerators}
       \begin{array}{rcl}
           & & \{H_{a;i,j}^{(r)},\wtl H_{a;i,j}^{(r)}  \mid 1\lle a\lle n, \ 1\lle i,j,\lle \mu_a, \ r\gge 0\} \ \cup \\ & & \hspace{40pt} \{B_{a;i,j}^{(r)} \mid 1\lle a<n, \ 1\lle i\lle \mu_{a+1}, \ 1\lle j\lle \mu_a, \ r>\fks_{i,j}\}
           \end{array}
    \end{eqnarray}
subject to relations \eqref{pr1}--\eqref{pr-2} in every type, along with: additional relation \eqref{Zshifted} for $\Y_\mu^+$ and additional relation \eqref{Zshifteda2} for $\Y_\mu^-$, respectively. These relations should be imposed for all admissible indexes. The definition of $\Y_\mu^\pm(\sigma)$ is a little subtle (in contrast to \cite{BK06}) as it includes the relations \eqref{Zshifted} and \eqref{Zshifteda2} (which are derived relations in the twisted Yangians $\Y^\pm_N$ by Lemma \ref{lem:Za2}). 

Recall that the twisted Yangian $\Y_N^\pm$ is associated with a symmetric pair $(\gl_N, \g^\theta)$ where $\g^\theta = \fkso_N$ (for $+$) or $\fksp_N$ (for $-$). The involution we choose is described in \eqref{e:thetainvdefn}, and we extend the involution $\theta$ to $\gl_N[z]$ by the twist $z\mapsto -z$. The {\it twisted current algebra} $\gl_N[z]^\theta$ is spanned by symmetrized matrix units $f_{i,j}z^r := \frac{1}{2}(e_{i,j}z^r + (-1)^r\theta(e_{i,j})z^r)$. This allows us to define a {\it shifted twisted current algebra}, the subalgebra $\gl_N[z]^\theta(\sigma)$ spanned by
$\{ f_{i,j}z^r \mid r\gge \fks_{i,j}  \}.$

We recall that $\Y_N^\pm$ is equipped with the {\it loop filtration} such that $\gr \Y_N^\pm \cong \rU(\gl_N[z]^\theta)$ on the level of associated graded algebras. 

By definition, there is a natural homomorphism
\begin{eqnarray*}
    \Y_\mu^\pm(\sigma) \longrightarrow \Y^\pm_\mu. 
\end{eqnarray*}
Our next main theorem is a Poincar{\'e}--Birkhoff--Witt (PBW) for the shifted twisted Yangians.

\begin{alphatheorem} [PBW basis for shifted twisted Yangians, Theorem~\ref{pbwstw}, Theorem~\ref{a2shPBWgauss}]
\label{thm:B}
     The natural homomorphism $\Y_\mu^\pm(\sigma) \rightarrow \Y^\pm_\mu$ is injective, deforming the inclusion map $\rU(\gl_N[z]^\theta(\sigma))\subseteq \rU(\gl_N[z]^\theta)$ under the identification $\gr \Y_\mu^\pm = \rU(\gl_N[z]^\theta)$.  
\end{alphatheorem}
The PBW bases for $\Y_\mu^\pm(\sigma)$ in Theorem~\ref{pbwstw} and Theorem~\ref{a2shPBWgauss} are expressed in terms of parabolic (Drinfeld-type) generators. For $\Y_\mu^\pm$ (i.e., $\sigma=0$), a PWB basis in terms of $s_{ij}^{(r)}$ was known; cf. \cite{Mol07}. 

For each admissible shape $\mu$ for $\sigma$ we view the shifted twisted Yangian as a subalgebra, $\Y^\pm_\mu(\sigma) \subseteq \Y^\pm_N$. In Proposition~\ref{ind} and Corollary~\ref{a2shapeindcor} we show that this subalgebra does not depend on the choice of admissible shape. For the rest of the introduction, we will use the notation $\Y_N^\pm(\sigma)$ for the shifted twisted Yangian.

%Each shift matrix admits an admissible shape of minimal length, which we call the minimal admissible shape. If $\nu$ is another admissible shape for $\sigma$ then 
%Let $\mu = (\mu_1,...,\mu_n)$ be an admissible shape for $\sigma$, let $1\le a \le n$ and $\alpha, \beta > 0$ such that $\alpha + \beta = \mu_a$. If $\nu = (\mu_1,...,\mu_{a-1}, \alpha, \beta, \mu_{a+1}, ..., \mu_n)$ is an admissible shape for $\sigma$ then we say that $\nu$ {\it refines} $\mu$. 

\subsection{Truncation, baby comultiplication and the center}

Now we fix an integer $\ell > 2\mathfrak{s}_{1,N}$, which we call the {\it level}. For each level, we define a truncation $\Y^\pm_{N,\ell}(\sigma) := \Y_{N}^\pm(\sigma)/I_{\ell}$ where $I_\ell$ is the two-sided ideal generated by elements in \eqref{trdef} in type AI, and \eqref{trdefA2} in type AII; note that the generators in the ideal can be more sophisticated than those used in defining truncated shifted Yangians in \cite{BK06}. A priori, this truncation depends on the admissible shape $\mu$, and its parabolic presentation. However, we show that these truncations are canonically isomorphic (see Corollaries~\ref{trtwywd} and \ref{trtwywda2}), which justifies the notation.

With our chosen conventions, $\Y_N^\pm \subseteq \rY_N$ is a right coideal subalgebra, and so we have comultiplication
$$\Delta_R : \Y_N^\pm \to \Y_N^{\pm} \otimes \rY_N \twoheadrightarrow \Y_N^{\pm} \otimes \rU(\gl_N)$$
where the second map is the evaluation homomorphism. This structure does not descend to the shifted subalgebras or the truncations, however some residual structure is left behind after shifting and truncating. This residual structure is called the {\it baby comultiplication}, and it generalizes that for $\rY_N$, first appearing in the seminal work of Brundan and Kleshchev \cite{BK06}. 

Assume $\sigma\neq0$ and choose $\mu=(\mu_1,...,\mu_n)$ to be minimal admissible to $\sigma$; here ``minimal" means that $\mu$ has minimal length. Define $\dot\sigma:=(\dot\fks_{i,j})_{1\lle i,j\lle N}$ according to the following rule
\beq\label{1dotsij}
\dot\fks_{i,j}:=\begin{cases}
\fks_{i,j}-1, & \text{ if }i\lle N-\mu_n<j  \text{ or } j\lle N-\mu_n<i,\\
\fks_{i,j}, &\text{ otherwise. }\\
\end{cases}
\eeq

\begin{alphatheorem} [Baby comultiplication, Theorem~\ref{delR}]
\label{thm:C}
For each $N\times N$ shift matrix $\sigma$ there exists an injective homomorphism $\Delta_R : \Y_N^{\pm}(\sigma) \to \Y_N^{\pm}(\dot\sigma) \otimes \rU(\gl_{\mu_n})$ where $\mu = (\mu_1,...,\mu_n)$ is the minimal admissible shape. Moreover, for each $\ell > 2\fks_{1,N}$, this descends to an injective algebra homomorphism
$$\Delta_R : \Y_{N, \ell}^\pm(\sigma) \longrightarrow \Y_{N,\ell-2}^\pm(\dot \sigma) \otimes \rU(\gl_{\mu_n}).$$
\end{alphatheorem}
One finds explicit formulas for $\Delta_R$ acting on generators of $\Y_{N, \ell}^\pm(\sigma)$ in Theorem~\ref{delR}. Showing the baby comultiplication equipped with these formulas is an algebra homomorphism in our twisted Yangian setting requires significantly more effort than the counterpart in \cite{BK06}.
The baby comultiplication allows us to study the truncation $\Y_{N,\ell}^\pm(\sigma)$ inductively, by reducing the level. This allows us to obtain a PBW theorem, and describe the center of $\Y_{N,\ell}^\pm(\sigma)$. These results are pivotal when we relate twisted Yangians to finite $W$-algebras. In order to state these results, we now introduce the combinatorics involved in relating twisted Yangians and finite $W$-algebras.

Let $(\phi, \sigma, \ell)$ be a triple consisting of a sign $\phi \in \{\pm\}$, an $N \times N$ shift matrix $\sigma$ for $\Y_N^\phi$ and a level $\ell > 2\fks_{1,N}$. From this data we construct a pair $(\g, e)$ consisting of a classical Lie algebra $\g$ and a nilpotent element $e \in \g$, via the following recipe. We let $M := N\ell - \sum_{i=1}^N 2 \fks_{i,N}$, and then set
\begin{equation}
\label{liealgdef}
 \g=\begin{cases}
 \mathfrak{so}_M, & \textnormal{ for } \phi = +, \ \ell \textnormal{ odd}, \textnormal{ or } \phi = -, \ \ell \textnormal{ even}. \\
  \mathfrak{sp}_M, & \textnormal{ for } \phi = +, \ \ell \textnormal{ even}, \textnormal{ or } \phi = -, \ \ell \textnormal{ odd}.
 \end{cases}
\end{equation}
Consider the partition $p = (p_1,...,p_N)$ of $M$ given by
\beq\label{partn}
p_i:=\ell-2\fks_{i,N}. 
\eeq
By construction, there exist nilpotent elements in $\g$ such that the sizes of the Jordan blocks are given by $p$. We let $e \in \g$ be any such nilpotent element.

The Sklyanin determinant is the formal power series $\sdet S(u) = \sum_{r \gge 0} c_r u^{-r}$ with coefficients in $\Y_N^\pm$, such that $\{ c_r \mid r \in 2\mathbb Z_{>0}\}$ is a collection of algebraically independent generators for the center $Z \Y_N^\pm$, see \cite{MNO96}.

\begin{alphatheorem} [Truncated PBW theorem, Corollaries~\ref{loopgradedcentralizerAI}, \ref{loopgradedcentralizerAII} and Theorem~\ref{truncatedcenter}]
\label{thm:D}
$\qquad$
\begin{enumerate}
    \item[(i)]  There is a natural surjection $\gl_N[z]^\theta(\sigma) \twoheadrightarrow \g^e$, and the identification $\gr \Y_N^\pm(\sigma) = \rU(\gl_N[z]^\theta(\sigma))$ induces an identification $\gr \Y_{N,\ell}^\pm(\sigma) = \rU(\g^e)$.
    \item[(ii)] If $\g$ in \eqref{liealgdef} has type {\sf B} or {\sf C} then the map $Z \Y_N^\pm(\sigma) \twoheadrightarrow Z \Y_{N,\ell}^\pm(\sigma)$ is surjective and $Z \Y_{N,\ell}^\pm(\sigma)$ is a polynomial ring in rank $\g$ variables. 
\end{enumerate}
\end{alphatheorem}

We also show that Theorem \ref{thm:D}(ii) holds also if $\g$ has type $\sf{D}$ and $e$ is a nilpotent element with 2 Jordan blocks in Theorem \ref{pfr2thm}.

\subsection{Dirac reduction of shifted Yangians and Slodowy slices}   
    Let $(\phi, \sigma, \ell)$ and $(\g,e)$ be related as per \eqref{liealgdef} and \eqref{partn}. The Slodowy slice associated to $e$ will be denoted $\cS$. It admits a natural Poisson structure which is quantized by the finite $W$-algebra \cite{GG02}. In recent work \cite{To23, TT24} two of the authors used Dirac reduction to describe the Poisson structure on the Slodowy slice. If $(\g,e)$ arises from a triple with $\phi = +$, they showed that this Poisson structure actually arises as a quotient of the semiclassical twisted shifted Yangian of type AI. Our next results extend that work to type AII.

    Twisted Yangians are equipped with two filtrations. The loop filtration has been used above to deform twisted current algebras, whilst the canonical filtration has an associated graded algebra which is commutative. We write $y_N^\pm := \gr' \Y_N^\pm$ for the associated canonically graded algebra. As a commutative algebra, it is isomorphic to $S(\gl_N[z]^\theta)$;  however it comes equipped with an interesting Poisson structure, deforming the Lie--Poisson bracket. We introduce similar notation $y_{N,\ell}^\pm(\sigma)$ in the shifted and truncated setting.

    Write $\bC[\cS]$ for the algebra of regular functions on the Slodowy slice associated to $(\g,e)$, together with its graded Poisson structure.

\begin{alphatheorem}[Corollaries~\ref{gryangsliceAI} and \ref{gryangsliceAII}]
\label{thm:E}
    We have a graded Poisson algebra isomorphism
    $$y_{N,\ell}^\phi(\sigma) \overset{\sim}{\longrightarrow} \bC[\cS]$$
    where the doubled canonical grading coincides with the Kazhdan grading.
\end{alphatheorem}

\subsection{Equivariant quantizations and finite $W$-algebras}

The isomorphism between truncated shifted Yangians and finite W-algebras of type A in \cite{BK06} was constructive via miraculously explicit formulas, then leveraging the powerful baby comultiplications to prove that these formulas satisfy all of the desired properties. We are unable to generalize such a constructive approach in our twisted setting in spite of serious effort, partly due to the use of different bilinear forms for $\g$ in $\rU(\g,e)$ and twisted Yangians $\Y_N^\pm$. This is where our methods differ the most from \cite{BK06}.

Conic symplectic singularities are a class of singular Poisson varieties, introduced by Beauville \cite{Be00}. 
%In a precise sense, they are well approximated by smooth symplectic varieties. 
In recent years there has been a huge amount of research activity surrounding the deep relationship between their geometry and the representation theory of their quantizations. 
%This is a new field, which has been dubbed {\it symplectic representation theory}. 
The key example of a conic symplectic singularity in this paper is the nilpotent Slodowy variety, i.e., the intersection of a nilpotent cone with a Slodowy slice.

Let $X$ be a conic symplectic singularity. Namikawa has shown in \cite{Na10, Na11} that the functor $\rPD_X$ of Poisson deformations is representable, and Losev has shown in \cite{Lo22} that the universal Poisson deformation can be quantized. What is more, Losev's quantization admits a very nice universal property. In \cite{ACET23} the functor $\rQ_X$ of quantizations of flat Poisson deformations of $X$ was introduced, and it was explained that the universal property is equivalent to the representability of the functor of quantizations. The results of Namikawa and Losev were upgraded {\em loc. cit.} to the equivariant setting, with respect to a group of graded Poisson automorphisms of the underlying singularity.

Let $(\phi,\sigma, \ell)$ and $(\g,e)$ be as above.  In most cases, the Slodowy slice $\cS$ to $\g$ at $e$ can be identified with the universal Poisson deformation of the nilpotent part $\cS \cap \cN(\g)$ \cite{LNS12}. Consequently, the finite $W$-algebra $\rU(\g,e)$ is isomorphic to a universal quantization of this symplectic singularity. In the cases where the Slodowy slice is not universal, it turns out that the Slodowy slice is the universal equivariant deformation with respect to a certain group of Poisson automorphisms, unless $(\g,e)$ falls into two special cases: regular orbits or nilpotent elements with two odd blocks in type {\sf C} (see \cite{AT25} and \S\ref{ss:twistedKleinianSubregular} for more detail). 

If the Slodowy slice is the universal equivariant deformation, we prove that $\Y_{N,\ell}^\pm(\sigma)$ and $\rU(\g,e)$ are universal equivariant quantizations of the nilpotent Slodowy variety, and thus we prove our final theorem. Note that one of the key conditions for a filtered algebra to be a filtered quantization of the flat Poisson deformation is that the Poisson center of the deformations lifts to the center of the filtered algebra, and so our proof depends crucially on Theorem~\ref{thm:D}(ii). In the special cases where the Slodowy slice is not known to be the universal equivariant deformation we prove the following theorem directly in one case and use a result of \cite{Br09} for the other case. We can now formulate the $\sf B\sf C\sf D$ analogue of \cite{BK06} and a vast generalization of \cite{Br09}.

\begin{alphatheorem}[Theorem~\ref{t:finalisothm}]
    \label{thm:F}
    If $\g$ has type {\sf B} or {\sf C}, or if $\g$ has type {\sf D} and $e$ has only two Jordan blocks, then there exists an isomorphism of filtered algebras
    $$\Y_{N,\ell}^\pm(\sigma) \overset{\sim}{\longrightarrow} \rU(\g,e)$$
    where the doubled canonical filtration lines up with the Kazhdan filtration.
\end{alphatheorem}

We expect that Theorem~\ref{thm:F} holds without restrictions in type {\sf D}. To prove this it would suffice to show that the Poisson center of the Slodowy slice lifts to the center of $\Y_{N,\ell}^\pm(\sigma)$ through Theorem~\ref{thm:E}. This would be implied by an affirmative answer to our Conjecture~\ref{pfconj} or Conjecture~\ref{centerconj}.

The results of this article may shed light on several interesting old and new questions. One hopes that our work helps to provide a two-way interaction between the representation theory of twisted Yangians to the representation theory of finite $W$-algebras of type {\sf B}{\sf C}{\sf D} (compare, e.g., \cite{BK08, BG13,Mol07, GRW17}). It remains an open question to give a presentation for the remaining finite $W$-algebras of type {\sf B}{\sf C}{\sf D}. It will be interesting to understand the precise relationship to a wide class of truncated shifted twisted Yangians which are constructed very differently in \cite{LWW25}.

\subsection{Organization}

This paper is organized into two parts. 

Part 1, which consists of Sections \ref{sec:prelim}--\ref{sec:pfAII}, provides the parabolic presentations of twisted Yangians of type AI and AII; this can be viewed as a highly nontrivial twisted counterpart of \cite{BK05}. Part~ 2, which consists of Sections \ref{sec:shiftedI}--\ref{sec:singularities}, formulates the shifted twisted Yangians and establishes the connection of the truncated shifted twisted Yangians to finite $W$-algebras. This part achieves a similar goal as in \cite{BK06} for finite $W$-algebras of types $\sf B\sf C\sf D$; however, we require new techniques coming from universal equivariant quantization of nilpotent Slodowy varieties. 

In Section~\ref{sec:prelim}, we recall the definitions and basic properties of Yangians of type A and twisted Yangians of type AI and AII, including coproduct structures, PBW theorem, associated graded, various homomorphisms, and quantum/Sklyanin determinant and minors.

In Section~\ref{sec:Gauss_decomp}, we obtain new parabolic generators for twisted Yangians by performing the block-wise Gauss decomposition, and establish some simple relations and basic properties for these parabolic generators. Moreover, we describe the
Sklyanin determinant in terms of parabolic generators.

In Section~\ref{sec:parabolicI},
we state our first main result (Theorem \ref{thm:A}), the parabolic presentations of type AI and AII in a uniform way.
Assuming some results in low rank cases in subsequent Sections \ref{sec:lower2}--\ref{sec:lower3}, we provide a detailed proof for type AI. The defining relations in the special case when the composition $\mu$ has only two and three parts are derived in Section~\ref{sec:lower2} and Section~\ref{sec:lower3}, respectively; in particular the Serre relations are deduced.

In Section~\ref{sec:pfAII}, 
we complete the proof of parabolic presentations (Theorem \ref{thm:A}) for type AII, providing the necessary modifications in this type.  

In Section~\ref{sec:shiftedI},
we introduce the shifted twisted Yangians $\Y^{\pm}_{\mu}(\sigma)$ associated to a symmetric shift matrix $\sigma$ and an admissible shape $\mu$. We establish their PBW bases (Theorem \ref{thm:B}) and show that the definition does not depend on $\mu$. We also show that in type AI our definition coincides with a previous definition of the shifted Yangians given in \cite{TT24} in terms of the Drinfeld presentation from \cite{LWZ23}.

In Section~\ref{sec:baby},
we define the baby comultiplications from Theorem \ref{thm:C} and show that they give well defined algebra homomorphisms.

In Section~\ref{sec:DiracII},
we relate the semiclassical shifted twisted Yangian of type AII to a Dirac reduction of its untwisted counterpart. This is a type AII analogue of \cite[\textsection 3]{TT24} where this was done for type AI. This is one of our main technical tools used in the proof of Theorem~ \ref{thm:E}.

In Section~\ref{sec:truncationI}, 
we recall the notion of pyramids and relate a pyramid to a unique pair $(\sigma,\ell)$, where $\sigma$ is a symmetric shift matrix and $\ell$ is an integer called level. With $(\sigma,\ell)$, together with a composition $\mu$ which is admissible to $\sigma$, we then define the truncated shifted twisted Yangian $\Y_{N,\ell}^\pm(\sigma)$ as the quotient of $\Y_N^\pm(\sigma)$ over a certain ideal $I_\ell$. We show that the definition is independent of the choice of admissible shape $\mu$. The PBW bases of $\Y_{N,\ell}^\pm(\sigma)$ are also obtained. 

In Section~\ref{sec:center},
we describe the center of $\Y^{\pm}_{N,\ell}(\sigma)$ explicitly in terms of the Sklyanin determinant whenever the corresponding finite $W$-algebra is of type \textsf{B} or \textsf{C}. We introduce the notion of a Pfaffian generator in $\Y^{\pm}_{N,\ell}(\sigma)$ and conjecture the existence of such an element whenever the corresponding finite $W$-algebra is of type \textsf{D}. We verify this conjecture when $N = 2$.

In Section~\ref{sec:singularities}, we describe the Poisson structures on the Slodowy slices and describe the finite $W$-algebras. We then review the theory of equivariant flat graded Poisson deformations for a conic symplectic singularity $X$, as well as the classification of equivariant quantizations. The key fact we explain is the uniqueness (up to isomorphism) of the universal equivariant quantization (Corollary~\ref{C:uniqueuniversalquantization}). Automorphisms of the truncated shifted twisted Yangians related to subregular finite $W$-algebras in type {\sf B} are also studied. Finally we prove Theorem~ \ref{thm:F}.

\vspace{2mm}

\noindent {\bf Acknowledgements.} KL and WW are partially supported by WW's NSF grant DMS--2401351. YP is partially supported by the NSTC grant 111-2628-M-006-006-MY3 and the National Center of Theoretical Sciences, Taipei, Taiwan. LT is funded by the UKRI Future Leaders Fellowship grant number MR/Z000394/1.

\newpage

\part{Parabolic presentations of twisted Yangians}
\section{Preliminaries}
\label{sec:prelim}

In this section, we review the basics about Yangians and twisted Yangians of type AI and AII; cf. \cite{Mol07}.

\subsection{Yangians $\mathrm{Y}(\gl_N)$}
The Lie algebra $\gl_N$ is generated by elements $e_{ij}$, $1\lle i,j\lle N$, with the commutator relations
\[
[e_{ij},e_{kl}]=\delta_{jk}e_{il}-\delta_{il}e_{kj}.
\]

Introduce the generating series in an indeterminate $u$, for $1\lle i, j \lle N$, 
\[
t_{ij}(u)=  \delta_{ij}+t_{ij}^{(1)}u^{-1}+t_{ij}^{(2)}u^{-2}+\cdots.
\]
The  \textit{Yangian} $\rY(\gl_N)$ is a unital associative algebra with generators $t_{ij}^{(r)}$, where $1\lle i,j\lle N$ and $r\in\bZ_{>0}$, subject to relations 
\beq\label{A-rel}
(u-v)[t_{ij}(u),t_{kl}(v)]=t_{kj}(u)t_{il}(v)-t_{kj}(v)t_{il}(u).
\eeq
By convention, we also set $t_{ij}^{(0)}=\delta_{ij}$.
%\end{dfn}

Let $E_{ij}$, $1\lle i,j\lle N$, be the standard matrix units in $\End(\bC^N)$, and
\[
R(u)=1-u^{-1}P,\qquad P=\sum_{i,j=1}^N E_{ij}\otimes E_{ji},\qquad T(u)=\big(t_{ij}(u)\big)_{i,j=1}^N=\sum_{i,j=1}^NE_{ij}\otimes t_{ij}(u).
\]
Then the relations \eqref{A-rel} can be compactly described by the RTT relations
\[
R(u-v)T_1(u)T_2(v)=T_2(v)T_1(u)R(u-v),
\]
considered as an equality in $\End(\bC^N)\otimes \End(\bC^N)\otimes \rY(\gl_N)[[u^{-1},v^{-1}]]$.

There exists an evaluation homomorphism
\beq\label{eva}
\pi_N:\rY(\gl_N)\longrightarrow \mathrm{U}(\gl_N),\qquad t_{ij}^{(r)}\mapsto \delta_{1r}e_{ij}.
\eeq
The Yangian $\rY(\gl_N)$ is a Hopf algebra with the coproduct given by
\beq\label{comy}
\Delta(t_{ij}(u))=\sum_{k=1}^N t_{ik}(u)\otimes t_{kj}(u).
\eeq

\begin{prop}[{\cite[Cor.~3.2]{BK05}, \cite[Cor.~1.23]{MNO96}}]
The monomials in the generators $\{t_{ij}^{(r)} \mid 1\lle i,j\lle N, r\gge 1\}$ taken in any fixed linear order form a basis of $\rY(\gl_N)$.
\end{prop}

Define the \emph{loop} filtration on $\rY(\gl_N)$ by setting
\begin{align}  \label{filter:Y}
\deg t_{ij}^{(r)}=r-1
\end{align}
for every $r\gge 1$. Let $\mathcal F_s\rY(\gl_N)$ be the subspace of $\rY(\gl_N)$ spanned by elements of degree $\lle s$,
\[
\mathcal F_0\rY(\gl_N)\subset \mathcal F_1\rY(\gl_N)\subset \mathcal F_2\rY(\gl_N)\subset \cdots,\qquad \rY(\gl_N)=\bigcup_{s\gge 1}\mathcal F_s\rY(\gl_N).
\]
Denote by $\mathrm{gr}\rY(\gl_N)$ the associated graded algebra. We write $\bar t_{ij}^{(r)}$ for the image of $t_{ij}^{(r)}$ in $\mathrm{gr}\rY(\gl_N)$. 

Let $\gl_N[z]$ denote the \textit{current algebra} $\gl_N\otimes \bC[z]$ with the standard basis
\[
\{e_{ij}z^r~|~1\lle i,j\lle N,\ r\gge 0\}.
\]
Let $\mathrm{U}(\gl_N[z])$ be the universal enveloping algebra of $\gl_N[z]$. Then the map
\[
\mathrm{U}(\gl_N[z])\to \mathrm{gr}\rY(\gl_N), \qquad e_{ij}\otimes z^{r}\mapsto \bar t_{ij}^{(r+1)},
\]
induces a Hopf algebra isomorphism.  

Let $\mathfrak S_N$ be the symmetric group on the letters $\{1,2,\dots,N\}$. Define the \textit{quantum determinant} of $T(u)$ by
\beq\label{qdet}
\qdet T(u) = \sum_{\sigma\in\fkS_N}(-1)^\sigma t_{\sigma(1),1}(u)t_{\sigma(2),2}(u-1)\cdots t_{\sigma(N),N}(u-N+1).
\eeq
The coefficients of $\qdet T(u)$ are algebraically independent generators of the center of $\rY(\gl_N)$; see \cite[Thm.~1.7.5]{Mol07}. The quantum determinant is group-like under the coproduct \cite[Cor.~1.6.10]{Mol07}:
\beq\label{qdetcp}
\Delta(\qdet T(u))=\qdet T(u)\otimes \qdet T(u).
\eeq

\subsection{Twisted Yangians}
\label{ss:TwistedYangians}

We shall consider twisted Yangians of type AI and AII in some $N\times N$ matrix forms. We denote a classical Lie algebra by
\begin{align}\label{liealgtype}
    \g_N =
    \begin{cases}
        \fko_N, & \text{ in type AI},
        \\
        \mathfrak{sp}_{N}, & \text{ in type AII},
    \end{cases}
\end{align}
where in the latter case $N$ is assumed to be even. Throughout the article, whenever the signs $\pm$ or $\mp$ show up, the top one corresponds to type AI while the bottom one corresponds to AII. To describe the two cases in a uniform way, we introduce the following convention.

In type AI, we set
\beq\label{a1thetadef}
i'=i \quad \text{ and } \quad  \theta_{i}=1, \quad \forall 1\lle i\lle N.
\eeq
In type AII, we set
\beq\label{thetadef}
i'=\begin{cases}
i-1, &\text{ if }i=2k,\\
i+1, & \text{ if }i=2k-1,
\end{cases}
\qquad\qquad \theta_{i}=(-1)^i.
\eeq
In both cases, we take 
\beq\label{Gdef}
G =(\delta_{ij'}\theta_i)_{1\lle i,j\lle N}.
\eeq
Note that $G =\mathbf I_N$ is the identity matrix in type AI.

Define the generating series, for $1\lle i,j\lle N$,
\beq\label{siju}
s_{ij}(u)=  \delta_{ij}+s_{ij}^{(1)}u^{-1}+s_{ij}^{(2)}u^{-2}+\cdots .
\eeq

\begin{dfn}\label{tydef}
The  \textit{twisted Yangian} $\Y^\pm_N$ of type AI/AII is a unital associative algebra with generators $s_{ij}^{(r)}$, where $1\lle i,j\lle N$ and $r\in\bZ_{>0}$, subject to the \textit{quaternary} relations
\beq\label{qua}
\begin{split}
(u^2-v^2)[s_{ij}(u),s_{kl}(v)]=&\, (u+v)(s_{kj}(u)s_{il}(v)-s_{kj}(v)s_{il}(u))\\
- &\, (u-v)(\theta_k\theta_{j'}s_{ik'}(u)s_{j'l}(v)-\theta_i\theta_{l'}s_{ki'}(v)s_{l'j}(u))\\
&\, \hskip 1.64cm  +  \theta_{i}\theta_{j'}(s_{ki'}(u)s_{j'l}(v)-s_{ki'}(v)s_{j'l}(u)),
\end{split}
\eeq
and the \textit{symmetry} relations
\beq\label{sym}
\theta_i\theta_js_{j'i'}(-u)=  s_{ij}(u)\pm\frac{s_{ij}(u)-s_{ij}(-u)}{2u}.
\eeq
\end{dfn}

Note that our definition corresponds to the one in \cite[\S 2.15]{Mol07} with the matrix $G$ in \eqref{Gdef}. In the orthogonal case, it is the same as the one used in \cite{LWZ23} while in the symplectic case, it is the one used in \cite[\S 4.1]{Mol07} after rearranging the indices.

Define the modified transpose for the R-matrix $R(u)$ and the matrix $S(u)$, respectively, by
\[
R^{\sf T}(u)=1-u^{-1}\sum_{i,j=1}^N \theta_{i}\theta_jE_{ij}\otimes E_{i'j'},\quad S(u)=\big(s_{ij}(u)\big)_{i,j=1}^N=\sum_{i,j=1}^N E_{ij}\otimes s_{ij}(u).
\]
Note that the modified transpose coincides with the ordinary transpose for orthogonal case. The quaternary relations \eqref{qua} can be compactly described by the reflection equations
\[
R(u-v)S_1(u)R^{\sf T}(-u-v)S_2(v)=S_2(v)R^{\sf T}(-u-v)S_1(u)R(u-v).
\]

It follows from \eqref{qua} that
\beq\label{qua0}
[s_{ij}(u),s_{kl}^{(1)}]=\delta_{jk}s_{il}(u)-\delta_{il}s_{kj}(u)+\theta_i\theta_{l'}\delta_{i'k}s_{l'j}(u)-\theta_k\theta_{j'}\delta_{j'l}s_{ik'}(u).
\eeq

From \eqref{sym}, it is immediate that
\[
s_{i'i'}(-u)=s_{ii}(u)\pm \frac{s_{ii}(u)-s_{ii}(-u)}{2u}
\]
for all $1\lle i\lle N$. Indeed, given the quaternary relations, the symmetry relations can be reduced to a single relation as follows.

\begin{lem} [cf. {\cite[Lem. 2.2]{LWZ23}}]
\label{lem:s11}
Suppose $s_{ij}(u)$, $1\lle i,j\lle N$, satisfy the quaternary relations \eqref{qua}. If for any single $1\lle k\lle N$ we have
\beq\label{simplesym}
s_{k'k'}(-u)=s_{kk}(u)\pm \frac{s_{kk}(u)-s_{kk}(-u)}{2u},
\eeq
then $s_{ij}(u)$ also satisfy the symmetry relations \eqref{sym} for arbitrary $1\lle i,j\lle N$.
\end{lem}

\begin{proof}
For type AI, it is clear that \eqref{simplesym} is equivalent to $s_{kk}(u)=s_{kk}(-u)$ and hence the statement follows from \cite[Lem. 2.2]{LWZ23}. 

For type AII, it is similar to the proof of \cite[Lem. 2.2]{LWZ23}. Setting $i=k'$ and $u=-v$ in \eqref{qua}, we have
\begin{align*}
2v\theta_k\theta_{j'}s_{k'k'}(-v)s_{j'l}(v)-&2v\theta_{k'}\theta_{l'}s_{kk}(v)s_{l'j}(-v)\\
+&\theta_{k'}\theta_j(s_{kk}(-v)s_{j'l}(v)-s_{kk}(v)s_{j'l}(-v))=0.
\end{align*}
Substituting $s_{k'k'}(-v)$ using \eqref{simplesym} and simplifying it, we find a common factor $s_{kk}(v)$ multiplying from the left. Canceling $s_{kk}(v)$, we obtain the symmetry relation \eqref{sym}.
\end{proof}

\begin{rem}\label{varsym}
The following slightly stronger statement holds. If $s_{ij}(u)$, $1\lle i,j\lle N$, satisfy the quaternary relations \eqref{qua} and the equation $s_{kk}^{(2r-1)}+s_{k'k'}^{(2r-1)}=0$ holds for any single $1\lle k\lle N$, then $s_{ij}(u)$ satisfy the symmetry relations \eqref{sym} for any $1\lle i,j\lle N$.
\end{rem}

\begin{prop}[{\cite{MNO96,Mol07}}]\label{prop:PBW}
\quad
\begin{enumerate}
    \item The monomials in the elements
    $$\big\{s_{ij}^{(r)}\big\}_{r\in\bZ_{>0},1\lle j<i\lle N} \cup \big\{s_{kk}^{(2r)} \big\}_{r\in\bZ_{>0},1\lle k\lle N}$$ 
    taken in any fixed linear order form a PBW basis of $\Y_N^+$. 
    \item The monomials in the elements
\begin{align*}
\big\{ s_{2i-1,2i-1}^{(r)} \big\}_{r\in\bZ_{>0}, \, 1\lle i\lle N/2}
&\cup \big\{ s_{2i-1,2i}^{(2r-1)}, \, s_{2i,2i-1}^{(2r-1)} \big\}_{r\in\bZ_{>0}, \, 1\lle i\lle N/2}
\\
& \cup \big\{s_{ij}^{(r)} \big\}_{r\in\bZ_{>0}, \, 1\lle j<i\lle N, \, \lfloor \tfrac{i+1}{2}\rfloor\ne \lfloor \tfrac{j+1}{2}\rfloor}
\end{align*} 
taken in any fixed linear order form a PBW basis of $\Y_N^-$. 
\end{enumerate}
\end{prop}

It is well known that $\Y^\pm_N$ can be identified as a subalgebra of $\rY(\gl_N)$. Specifically, the map
\beq\label{embed2}
S(u)\mapsto T^t(-u)G T(u)G^{-1} %using this one might be better, agreeing with later part
\eeq
defines an embedding $\Y^\pm_N\hookrightarrow \rY(\gl_N)$. Here $t$ stands for the regular transpose. Then there is a filtration, which we call the \textit{loop filtration} again, on $\Y^\pm_N$ inherited from the one \eqref{filter:Y} on $\rY(\gl_N)$ such that $\deg s_{ij}^{(r)}=r-1$. Let $\mathcal F_s\Y^\pm_N$ be the subspace of $\Y^\pm_N$ spanned by elements of degree $\lle s$ such that
\begin{align}
\label{filter:tY}
    \mathcal F_0\Y^\pm_N \subset \mathcal F_1\Y^\pm_N \subset \mathcal F_2\Y^\pm_N \subset \ldots, 
    \qquad\qquad \Y^\pm_N =\bigcup_{s\gge 0} \mathcal F_s\Y^\pm_N.
\end{align}
Denote by $\gr \Y^\pm_N$ the associated graded algebra. Let $\bar s_{ij}^{(r)}$ be the image of $s_{ij}^{(r)}$ in the $(r-1)$-st component of $\gr\Y^\pm_N$.

Let $\theta$ be the involution of $\gl_N$ defined by
\begin{eqnarray}
\label{e:thetainvdefn}
  \theta: \gl_N\longrightarrow \gl_N,\quad e_{ij}\mapsto -\theta_i\theta_je_{j'i'}.  
\end{eqnarray}
Let $\gl_N^\theta$ denote the fixed point subalgebra of $\gl_N$ under $\theta$. For $1\lle i,j\lle N$, set 
$$f_{ij}:=e_{ij}-\theta_i\theta_je_{j'i'}.$$
Then $f_{j'i'}=-\theta_i\theta_j f_{ij}$ and $f_{ij}$ are elements in $\gl_N^\theta$. Moreover,  $\gl_N^\theta$ is spanned by $f_{ij}$ and $\g_N=\gl_N^\theta$. It is well known that the twisted Yangians $\Y_N^\pm$ have the evaluation homomorphism
\beq\label{tYev}
\xi^{\pm}: \Y_N^\pm \to \rU(\g_N),\qquad  s_{ij}(u)\mapsto \delta_{ij}+(u+\tfrac12)^{-1}f_{ij}.
\eeq
In addition, the map  $f_{ij}\mapsto s_{ij}^{(1)}$ defines an embedding of the algebra $\rU(\g_N)$ into $\Y_N^\pm$.

Extend $\theta$ to an involution on $\gl_N[z]$ (again denoted $\theta$) by sending $g\otimes z^r$ to $\theta(g)\otimes(-z)^r$ for $g\in \gl_N$ and $r\in\bN$. Let $\gl_N[z]^\theta$ be the fixed point subalgebra of $\gl_N[z]$ under the involution $\theta$. Then it is well known that the map
\[
\mathrm{U}(\gl_N[z]^\theta)\longrightarrow \gr\Y_N^\pm,\qquad e_{ij}\otimes z^{r}-(-1)^r\theta_i\theta_je_{j'i'}\otimes z^{r} \mapsto \bar s_{ij}^{(r+1)}
\]
induces an algebra isomorphism. 

The following result follows from \eqref{comy} and \eqref{embed2}.
\begin{lem}
The twisted Yangian $\Y_N^\pm$ is a right coideal subalgebra of $\rY(\gl_N)$. In particular, we have
\begin{equation}\label{coms}
\Delta \big( s_{ij}(u) \big) = \sum_{a,b=1}^N s_{ab}(u) \otimes \theta_a\theta_it_{a'i'}(-u)t_{bj}(u) \in \Y_N^{\pm}\otimes\rY(\gl_N).
\end{equation}
\end{lem}

\begin{rem}
There exists another embedding $\iota:\Y_N^\pm\hookrightarrow \rY(\gl_N)$ defined by 
\[
S(u)\mapsto T(u)GT^t(-u)G^{-1},
\] 
see \cite[(2.108)]{Mol07}. With respect to $\iota$, $\Y_N^\pm$ becomes a {\em left} coideal subalgebra of $\rY(\gl_N)$ instead, 
and the restriction of the comultiplication to $\Y_N^\pm$ is given by
\[\Delta \big( s_{ij}(u) \big) = \sum_{a,b=1}^N \theta_j\theta_bt_{ia}(u)t_{j'b'}(-u)\otimes s_{ab}(u).\]
\end{rem}

It is convenient to work with the extended twisted Yangians defined below instead of twisted Yangians. By abuse of notations, we shall keep using the same notations for the twisted Yangian and extended twisted Yangians 
of various elements such as $s_{ij}(u)$ and $S(u)$, etc.

\begin{dfn}\label{etydef}
The \textit{extended twisted Yangian} $\X^\pm_N$ of type AI/AII is the unital associative algebra  with generators $s_{ij}^{(r)}$, where $1\lle i,j\lle N$ and $r\in\bZ_{>0}$, satisfying the quaternary relations \eqref{qua}, where $s_{ij}(u)$ is again given by \eqref{siju}.
\end{dfn}

\subsection{Homomorphisms of (extended) twisted Yangians}\label{sec:hom}

In this subsection, we discuss various (anti)homomorphisms between (extended) twisted Yangians. 

Let $\wtl S(u)$ be the inverse of the generating matrix $S(u)$:
\begin{align} \label{tildeS}
\wtl S(u)=\big(\tl s_{ij}(u)\big):= S(u)^{-1}.
\end{align}

For $1\lle i\lle N$, set 
\[
\bar i=N+1-i. 
\]
Then we have the following (anti)homomorphisms, see e.g. \cite[\S 2]{Mol07}.
\begin{enumerate}
    \item For any $m\in\bN$, we have the natural homomorphism $\jmath:\X_N^\pm\to \X_{m+N}^\pm$ defined by
    \[
    \jmath_m(s_{ij}(u))=s_{ij}(u).
    \]
    Here $m$ is even if $\g_N=\fksp_N$.
    \item The map $\tau:\X_N^\pm\to \X_N^\pm$ defined by
    \[
    \tau(s_{ij}(u))=\theta_i\theta_js_{j'i'}(u)
    \]
    is an anti-automorphism.
    \item The map $\omega_{N}:\X_N^\pm\to \X_N^\pm$ defined by
    $$\omega_{N}(S(u))=\wtl S(-u-\tfrac{N}{2})$$
    is an involution.
    \item The map $\rho_N:\X_N^\pm\to \X_N^\pm$ defined by 
    \[
    \rho_N(s_{ij}(u))=s_{\bar i\, \bar j}(u)
    \]
    is an isomorphism.
    \item The map $\zeta_N:\X_N^\pm\to \X_N^\pm$ defined by 
    \[
    \zeta_N=\rho_N\circ \omega_N,  \qquad 
    s_{ij}(u)\mapsto \tl s_{\bar i\,\bar j}(-u-\tfrac{N}{2})
    \]
    is an isomorphism.
    \item For $m>0$, there is a homomorphism $\varphi_m:\X_N^\pm\to \X_{m+N}^\pm$ given by shifting indices:
    \[
    \varphi_m(s_{ij}(u))=s_{m+i,m+j}(u).
    \]
    Then the map $\psi_m:\X_N^\pm\to \X_{m+N}^\pm$ defined by
    \[
    \psi_m=\omega_{m+N}\circ \varphi_m\circ \omega_N
    \]
    is a homomorphism. Here $m$ is even if $\g_N=\fksp_N$.
\end{enumerate}

\subsection{Sklyanin determinants and minors}\label{sec:skldet}

We recall the definition of Sklyanin determinant and minors from \cite[Chap.~2.5--2.6]{Mol07}.

Let $m\in \bN$ be such that $m\lle N$. We adopt the following abbreviation:
\[
u_i=u-i+1,\qquad S_i=S_i(u_i) ,\qquad i=1,2,\ldots,m.
\]
Let
\[
R_{ij}^{\sf T}=R_{ji}^{\sf T}=R_{ij}^{\sf T}(-u_i-u_j),\qquad 1\lle i<j\lle m.
\]
Here we use the standard convention of subscripts to indicate the copies of $\bC^N$ in $(\bC^N)^{\otimes m}$ it acts on, cf. \cite[Chap.~1.2]{Mol07}. Denote by $A_m$ the image of the anti-symmetrizer $\sum_{\sigma\in \mathfrak{S}_m} \text{sgn }\sigma  \cdot \sigma$ under the natural action of $\mathfrak{S}_m$ on $(\bC^N)^{\otimes m}$.

For an arbitrary permutation $(p_1,p_2,\ldots,p_m)$ of the letters $1,2,\ldots,m$, we denote by
\[
\langle S_{p_1},\ldots,S_{p_m}\rangle =S_{p_1}(R_{p_1p_2}^{\sf T}\ldots R_{p_1p_m}^{\sf T})S_{p_2}(R_{p_2p_3}^{\sf T}\ldots R_{p_2p_m}^{\sf T})\ldots S_{p_m}.
\]

Let $m=N$. It is known that
\beq\label{eq:sky}
A_N\langle S_1,\ldots,S_N\rangle =\langle S_1,\ldots,S_N\rangle A_N.
\eeq
Since the image of the antisymmetrizer on $(\bC^N)^{\otimes N}$ is one-dimensional, the element \eqref{eq:sky} equals $A_N$ times a scalar series with coefficients in $\Y_N^\pm$.

\begin{dfn}
The \emph{Sklyanin determinant} of the matrix $S(u)$ with coefficients in $\Y_N^\pm$ is the formal series
\[
\sdet\,S(u)=1+c_1u^{-1}+c_2u^{-2}+\cdots
\]
such that the element \eqref{eq:sky} equals $A_N\sdet\, S(u)$.
\end{dfn}

If $m\lle N$, then we have 
\[
A_m\langle S_1,\ldots,S_m\rangle =\langle S_1,\ldots,S_m\rangle A_m.
\]
This element of the tensor product $\End(\bC^N)^{\otimes m}\otimes\Y_N^\pm[[u^{-1}]]$ can be written as 
\[
\sum e_{a_1b_1}\otimes\cdots\otimes e_{a_mb_m}\otimes \mathcal S^{a_1\cdots a_m}_{b_1\cdots b_m}(u),
\]
summed over the indices $a_i,b_i\in\{1,\ldots,N\}$, where $\mathcal S^{a_1\cdots a_m}_{b_1\cdots b_m}(u)\in \Y_N^\pm[[u^{-1}]]$ are the \emph{Sklyanin minors} of the matrix $S(u)$. Note that the Sklyanin minors are skew-symmetric with respect to permutations of the upper indices and of the lower indices:
\beq\label{eq:sign-sdet}
\mathcal S^{a_{\sigma(1)}\cdots a_{\sigma(m)}}_{b_{1}\cdots b_{m}}(u)=\mathcal S_{b_{\sigma(1)}\cdots b_{\sigma(m)}}^{a_{1}\cdots a_{m}}(u)=\mathrm{sgn}(\sigma)\mathcal S^{a_1\cdots a_m}_{b_1\cdots b_m}(u)
\eeq
for any $\sigma\in\mathfrak S_m$. Moreover,
\[
\mathcal S_b^a(u)=s_{ab}(u),\quad \mathcal S_{1\cdots N}^{1\cdots N}(u)=\sdet\,S(u).
\]
We shall also use the shorthand notation
\[
\mathcal C_k(u)=\mathcal S_{1\cdots k}^{1\cdots k}(u).
\]

% Let $$
% \sdet\,S(u)=1+c_1u^{-1}+c_2u^{-2}+\cdots.
% $$
The following is well known.
\begin{thm}[{\cite{MNO96}}]\label{freegen-center}
The coefficients  of the Sklyanin determinant $\sdet\,S(u)$ are central elements in $\Y_N^\pm$ (resp. $\X_N^\pm$). Moreover, $\{c_{2k}\}_{k\gge 1}$ is a set of algebraically independent generators of the center of $\Y_N^\pm$.
\end{thm}

Define $\gamma_{c}^\pm (u)$, for $c\in \mathbb Z$, by
\beq\label{gammau}
\gamma_{c}^\pm (u)=\begin{cases}
    1, & \text{ if }\g_N=\fko_N,\\
    \frac{2u+1}{2u-c+1}, &\text{ if }\g_N=\mathfrak{sp}_{N}.
\end{cases}
\eeq
Identifying $\Y_N^\pm$ as a subalgebra of $\rY(\gl_N)$ by \eqref{embed2}, one has by \cite[Thm 2.5.3]{Mol07} that
\beq\label{sdet=qdet}
\sdet\,S(u) =\gamma^\pm_N(u)\,\qdet\,T(u)\,\qdet\,T(-u+N-1).
\eeq
It is immediate from \eqref{qdetcp} and \eqref{sdet=qdet} that
\beq\label{sdetcp}
\Delta\big(\sdet\,S(u)\big) =\sdet\,S(u)\otimes \qdet\,T(u)\,\qdet\,T(-u+N-1).
\eeq

\section{Gauss decomposition and parabolic generators for type AI}
\label{sec:Gauss_decomp}

In this section we use Gauss decomposition to construct new parabolic generators and establish a PBW basis result for the twisted Yangians $\Y_N^+$ of type AI associated to any composition of $N$. 
Parallel results for twisted Yangians $\Y_N^-$ of type AII will be addressed in Section~\ref{sec:pfAII}.

\subsection{Gauss decomposition and quasideterminants}
We say that $\mu=(\mu_1,\ldots,\mu_n)$ is a \textit{(strict) composition} of $N$ of length $n$ if $\mu_i\in\bZ_{>0}$ and $\mu_1+\ldots+\mu_n=N$. Given a composition $\mu$ of $N$ with length $n$, define 
\beq\label{mua}
\mu_{(i)}:=\mu_1+\ldots+\mu_{i}
\eeq
for $1\lle i\lle n$. Also denote $\mu_{(0)}=0$.

Recall the $N\times N$ matrix $S(u)=\big( s_{ij}(u) \big)_{1\lle i,j\lle N}$ with entries in $\X_N^+[[u^{-1}]]$. Since the leading minors of the matrix $S(u)$ are invertible, for any given composition $\mu$ of $N$, it possesses a Gauss decomposition with respect to $\mu$; that is,
\begin{equation}\label{S=FDE}
S(u) = F(u) D(u) E(u)
\end{equation}
for unique $(\mu \times \mu)$-block matrices $D(u)$, $E(u)$ and $F(u)$ of the form
$$
D(u) = \left(
\begin{array}{cccc}
D_{1}(u) & 0&\cdots&0\\
0 & D_{2}(u) &\cdots&0\\
\vdots&\vdots&\ddots&\vdots\\
0&0 &\cdots&D_{n}(u)
\end{array}
\right),
$$

$$
E(u) =
\left(
\begin{array}{cccc}
I_{\mu_1} & E_{1,2}(u) &\cdots&E_{1,n}(u)\\
0 & I_{\mu_2} &\cdots&E_{2,n}(u)\\
\vdots&\vdots&\ddots&\vdots\\
0&0 &\cdots&I_{\mu_{n}}
\end{array}
\right),\:
$$

$$
F(u) = \left(
\begin{array}{cccc}
I_{\mu_1} & 0 &\cdots&0\\
F_{2,1}(u) & I_{\mu_2} &\cdots&0\\
\vdots&\vdots&\ddots&\vdots\\
F_{n,1}(u)&F_{n,2}(u) &\cdots&I_{\mu_{n}}
\end{array}
\right),
$$
where
\begin{align}
D_a(u) &=\big(D_{a;i,j}(u)\big)_{1 \lle i,j \lle \mu_a},\label{Da}\\
E_{a,b}(u)&=\big(E_{a,b;i,j}(u)\big)_{1 \lle i \lle\mu_a, 1 \lle j \lle\mu_b},\label{Eab}\\
F_{b,a}(u)&=\big(F_{b,a;j,i}(u)\big)_{1 \lle i \lle \mu_a, 1 \lle j \lle \mu_b},\label{Fba}
\end{align}
are $\mu_a \times \mu_a$,
$\mu_a \times \mu_b$
and  $\mu_b \times\mu_a$ matrices, respectively, for all $1\lle a\lle n$ in \eqref{Da}
and all $1\lle a<b\lle n$ in \eqref{Eab}--\eqref{Fba}. For each $a$, the submatrix $D_{a}(u)$ is invertible, and we define the $\mu_a\times\mu_a$ matrix
$\wtl D_a(u)=\big(\wtl D_{a;i,j}(u)\big)_{1\lle i,j\lle \mu_a}$ by
\begin{equation*}
\wtl D_a(u):=\big(D_a(u)\big)^{-1}.
\end{equation*}
The entries of these matrices are  power series in $u^{-1}$,
\begin{align*}
D_{a;i,j}(u) &= \sum_{r \gge 0} D_{a;i,j}^{(r)} u^{-r},\ \ \qquad 
\wtl D_{a;i,j}(u) = \sum_{r \gge 0}\wtl D^{(r)}_{a;i,j} u^{-r},\\
E_{a,b;i,j}(u) &= \sum_{r \gge 1} E_{a,b;i,j}^{(r)} u^{-r},\qquad 
F_{b,a;j,i}(u) = \sum_{r \gge 1} F_{b,a;j,i}^{(r)} u^{-r}.
\end{align*}
Note that $D_{a;i,j}^{(0)}=\wtl D_{a;i,j}^{(0)}=\delta_{i,j}$. In addition, for $1\lle a<n$, we set
\begin{align*}
E_{a;i,j}(u) :=&\, E_{a,a+1;i,j}(u)=\sum_{r \gge 1}  E_{a;i,j}^{(r)} u^{-r},\\
F_{a;j,i}(u) :=&\, F_{a+1,a;j,i}(u)=\sum_{r \gge 1} F_{a;j,i}^{(r)} u^{-r}.
\end{align*}

We describe all these series in terms of the generators $s_{ij}(u)$ explicitly by \textit{quasideterminants} \cite{GGRW:2005}. Here we follow the notation in \cite[(4.3)]{BK05}. Suppose that $A, B, C$ and $D$ are $a \times a$, $a \times b$, $b \times a$ and $b \times b$ matrices respectively with entries in some ring.
Assuming that the matrix $A$ is invertible, we define
\begin{equation*}
\left|
\begin{array}{cc}
A&B\\
C&\mybox{$D$}
\end{array}
\right| := D - C A^{-1} B.
\end{equation*}
Write the matrix $S(u)$ in block form according to $\mu$ as
\beq\label{blockS}
S(u) = \left(
\begin{array}{ccc}
{^\mu}S_{1,1}(u)&\cdots&{^\mu}S_{1,n}(u)\\
\vdots&\ddots&\vdots\\
{^\mu}S_{n,1}(u)&\cdots&{^\mu}S_{n,n}(u)\\
\end{array}
\right),
\eeq
where ${^\mu}S_{a,b}(u)$ is a $\mu_a \times \mu_b$ matrix.

\begin{prop}\cite{GGRW:2005}\label{quasi} We have
\begin{align}\label{quasid}
&D_a(u) =
\left|
\begin{array}{cccc}
{^\mu}S_{1,1}(u) & \cdots & {^\mu}S_{1,a-1}(u)&{^\mu}S_{1,a}(u)\\
\vdots & \ddots &\vdots&\vdots\\
{^\mu}S_{a-1,1}(u)&\cdots&{^\mu}S_{a-1,a-1}(u)&{^\mu}S_{a-1,a}(u)\\
{^\mu}S_{a,1}(u) & \cdots & {^\mu}S_{a,a-1}(u)&
\hbox{\begin{tabular}{|c|}\hline${^\mu}S_{a,a}(u)$\\\hline\end{tabular}}
\end{array}
\right|,\\[4mm]
&E_{a,b}(u) =\label{quasie}
\wtl D_a(u)
\left|\begin{array}{cccc}
{^\mu}S_{1,1}(u) & \cdots &{^\mu}S_{1,a-1}(u)& {^\mu}S_{1,b}(u)\\
\vdots & \ddots &\vdots&\vdots\\
{^\mu}S_{a-1,1}(u) & \cdots & {^\mu}S_{a-1,a-1}(u)&{^\mu}S_{a-1,b}(u)\\
{^\mu}S_{a,1}(u) & \cdots & {^\mu}S_{a,a-1}(u)&
\hbox{\begin{tabular}{|c|}\hline${^\mu}S_{a,b}(u)$\\\hline\end{tabular}}
\end{array}
\right|,\\[4mm]
&F_{b,a}(u) =\label{quasif}
\left|
\begin{array}{cccc}
{^\mu}S_{1,1}(u) & \cdots &{^\mu}S_{1,a-1}(u)& {^\mu}S_{1,a}(u)\\
\vdots & \ddots &\vdots&\vdots\\
{^\mu}S_{a-1,1}(u) & \cdots & {^\mu}S_{a-1,a-1}(u)&{^\mu}S_{a-1,a}(u)\\
{^\mu}S_{b,1}(u) & \cdots & {^\mu}S_{b,a-1}(u)&
\hbox{\begin{tabular}{|c|}\hline${^\mu}S_{b,a}(u)$\\\hline\end{tabular}}
\end{array}
\right|\wtl D_a(u),
\end{align}
for all $1\lle a\lle n$ in \eqref{quasid} and $1\lle a<b\lle n$ in \eqref{quasie}--\eqref{quasif}.
\end{prop}

Denote $S_{a,b;i,j}(u)$ the $(i,j)$-th entry of the $\mu_a\times \mu_b$ matrix ${}^\mu S_{a,b}(u)$. Let $S_{a,b;i,j}^{(r)}$ be the coefficient of $u^{-r}$ in $S_{a,b;i,j}(u)$. It is clear from Proposition \ref{quasi} that
\beq
E_{b-1;i,j}^{(1)}=S_{b-1,b;i,j}^{(1)},\qquad F_{b-1;j,i}^{(1)}=S_{b,b-1;j,i}^{(1)},
\eeq
for $1<b\lle n$, $1\lle i\lle \mu_{b-1}$, $1\lle j\lle \mu_b$. Moreover, we have
\beq
D_{1;i,j}^{(r)}=S_{1,1;i,j}^{(r)}=s_{ij}^{(r)},\quad 1\lle i,j\lle \mu_1,\ r\gge 1.
\eeq

\subsection{Parabolic generators}

Recall $\wtl S(u)=\big(\tl s_{ij}(u)\big)= S(u)^{-1}$ from \eqref{tildeS}.

\begin{lem}[{\cite[Prop. 2.12.2]{Mol07}}]
For $1\lle i,j,k,l\lle N$, we have
\beq\label{sts}
\begin{split}
(u^2-v^2)[s_{ij}(u),\tl s_{kl}(v)]= &\, (u+v)\Big(\delta_{jk}\sum_{a=1}^N s_{ia}(u)\tl s_{al}(v)-\delta_{il}\sum_{a=1}^N \tl s_{ka}(v) s_{aj}(u)\Big)\\
 +&\, (u-v)\Big(\delta_{jl}\sum_{a=1}^N \tl s_{ka}(v) s_{ia}(u)-\delta_{ik}\sum_{a=1}^N s_{aj}(u) \tl s_{al}(v)\Big)\\
 &\hskip 0.8cm+\Big(\delta_{ik}\sum_{a=1}^N s_{ja}(u)\tl s_{al}(v)-\delta_{jl}\sum_{a=1}^N \tl s_{ka}(v) s_{ai}(u)\Big).
\end{split}
\eeq
\end{lem}

\begin{cor}\label{stslem}
If $\{i,j\}\cap \{k,l\}=\emptyset$, then $[s_{ij}(u),\tl s_{kl}(v)]=0$.
\end{cor}

\begin{lem}[{cf. \cite[Lem.~3.7]{LWZ23}}]\label{genlem}
For $1<a+1<b<n$ and $1\lle i\lle\mu_a$, $1\lle j\lle \mu_b$, we have
\beq\label{efgen}
E_{a,b;i,j}^{(r)}=[E_{a,b-1;i,k}^{(r)},E_{b-1;k,j}^{(1)}],\qquad 
F_{b,a;j,i}^{(r)}=[F_{b-1;j,k}^{(1)},F_{b-1,a;k,i}^{(r)}],
\eeq
for any $1\lle k\lle \mu_{b-1}$.
\end{lem}
\begin{proof}
The statement can be proved by induction on $b-a>1$. We only prove the case $a=1$ and $b=3$ for the first equality. The general case can be proved in a similar way while the second one follows from the first one by applying the anti-automorphism $\tau$; see Lemma \ref{tauimg} below.

It follows from \eqref{quasie} that
\[
\big[E_{1,2;i,k}^{(r)},E_{2,3;l,j}^{(1)}\big]=\Big[\sum_{p=1}^{\mu_1}\sum_{q=0}^r \wtl D_{1;i,p}^{(q)}S_{1,2;p,k}^{(r-q)},S_{2,3;l,j}^{(1)}\Big].
\]
Since $\wtl D_{1;i,p}^{(q)}$ are expressed in terms of $s_{\alpha,\beta}^{(r)}$ for $1\lle \alpha,\beta\lle \mu_1$ and $$S_{2,3;l,j}^{(1)}= s_{\mu_1+l,\mu_1+\mu_2+j}^{(1)}=-\tl s_{\mu_1+l,\mu_1+\mu_2+j}^{(1)},
$$ 
it follows from \eqref{sts} that $\wtl D_{1;i,p}^{(q)}$ and $S_{2,3;l,j}^{(1)}$ commute. Thus
\begin{align*}
\big[E_{1,2;i,k}^{(r)},E_{2,3;l,j}^{(1)}\big]&=\sum_{p=1}^{\mu_1}\sum_{q=0}^r \wtl D_{1;i,p}^{(q)}\big[S_{1,2;p,k}^{(r-q)},S_{2,3;l,j}^{(1)}\big]\\&=\sum_{p=1}^{\mu_1}\sum_{q=0}^r \wtl D_{1;i,p}^{(q)}\big[s_{p,\mu_1+k}^{(r-q)},s_{\mu_1+l,\mu_1+\mu_2+j}^{(1)}\big]
\\&=\sum_{p=1}^{\mu_1}\sum_{q=0}^r \wtl D_{1;i,p}^{(q)}\delta_{kl} s_{p,\mu_1+\mu_2+j}^{(r-q)}=\delta_{kl} E_{1,3;i,j}^{(r)},
\end{align*}
where  we applied \eqref{qua0} and \eqref{quasie} in the third and last equalities, respectively. The general case is similar, except that the expression of $E_{a,b-1;i,k}^{(r)}$ by \eqref{quasie} is more complicated.
\end{proof}

\begin{prop}\label{gen}
The algebra $\X_N^+$ (or $\Y_N^+$) is generated by the following elements:
\begin{align*}
    D_{a;i,j}^{(r)},\wtl D_{a;i,j}^{(r)},\qquad\qquad &  1\lle a\lle n,\ 1\lle i,j\lle \mu_a,\ r>0,\\
    E_{a;i,j}^{(r)},F_{a;j,i}^{(r)}, \qquad \qquad & 1\lle a< n,\  1\lle i\lle \mu_a,\ 1\lle j\lle \mu_{a+1},\ r>0.
\end{align*}
\end{prop}
\begin{proof}
    By Lemma \ref{genlem}, all elements $E_{a,b;i,j}^{(r)}$ and $F_{b,a;j,i}^{(r)}$, for $r>0$, $1\lle a<b\lle n$, $1\lle i\lle \mu_a$, and $1\lle j\lle \mu_b$, can be generated by these elements. Thus it follows from \eqref{S=FDE} that all $s_{ij}^{(r)}$ for $1\lle i,j\lle N$ and $r>0$ can be generated by them.
\end{proof}

By translation on the spectral parameter, we introduce the following formal series.
Let $\mu=(\mu_1,\ldots,\mu_n)$ be given. For $1\lle a\lle n, 1\lle b<n$, define
\begin{align} 
\label{Ha} H_{a;i,j}(u)&=\sum_{r\gge 0} H_{a;i,j}^{(r)}u^{-r}:= D_{a;i,j}(u-\tfrac{\mu_{(a-1)}}{2}),\\ 
\label{tlHa} \wtl H_{a;i,j}(u)&=\sum_{r\gge 0} \wtl H_{a;i,j}^{(r)}u^{-r}:= \wtl D_{a;i,j}(u-\tfrac{\mu_{(a-1)}}{2}),\\ 
\label{Bb} B_{b;k,l}(u)&=\sum_{r\gge 1} B_{b;k,l}^{(r)}u^{-r}:= F_{b;k,l}(u-\tfrac{\mu_{(b)}}{2}),\\
\label{Cb} C_{b;l,k}(u)&=\sum_{r\gge 1} C_{b;l,k}^{(r)}u^{-r}:= E_{b;l,k}(u-\tfrac{\mu_{(b)}}{2}),\\
% \label{Hn} H_{n;f,g}(u)&=\sum_{r\gge 0} H_{n;f,g}^{(r)}u^{-r}:= D_{n;f,g}(u-\frac{\mu_{(n-1)}}{2}),\\
% \label{tlHn} \wtl H_{n;f,g}(u)&=\sum_{r\gge 0} \wtl H_{n;f,g}^{(r)}u^{-r}:= \wtl D_{n;f,g}(u-\frac{\mu_{(n-1)}}{2}),\\
\label{tlHH}
Z_{b;i,j,k,l}(u)&=\sum_{r\gge 0}Z_{b;i,j,k,l}^{(r)}u^{-r}=\wtl H_{b;i,j}(u-\tfrac{\mu_{b}}{2})H_{b+1;k,l}(u).
\end{align}
In particular, we have
\beq\label{Zdef}
Z_{b;i,j,k,l}^{(r)}=\sum_{t=0}^r\sum_{m=0}^{r-t}{t+m-1 \choose m }\Big(\frac{\mu_b}{2}\Big)^m \wtl H_{b;i,j}^{(t)} H_{b+1;k,l}^{(r-t-m)}.
\eeq

For $1\lle a<b\lle n$, set 
\begin{align} 
\label{Bab} B_{b,a;k,l}(u)&=\sum_{r\gge 1} B_{b,a;k,l}^{(r)}u^{-r}:= F_{b,a;k,l}(u-\tfrac{\mu_{(a)}}{2}),\\
\label{Cba} C_{a,b;l,k}(u)&=\sum_{r\gge 1} C_{a,b;l,k}^{(r)}u^{-r}:= E_{a,b;l,k}(u-\tfrac{\mu_{(a)}}{2}).
\end{align}
In case $\mu=(1^N)$, we further denote
\begin{align}
&\sfd_a(u)=1+\sum_{r\gge 0}\sfd_{a,r}u^{-r-1}:=D_{a;1,1}(u),\quad ~~~1\lle a\lle N,\label{dr0}\\
&\sfh_0(u)=1+\sum_{r\gge 0}\sfh_{0,r}u^{-r-1}:=H_{1;1,1}(u),\label{dr1}
\end{align}
and
\begin{align}
&\sfh_a(u)=1+\sum_{r\gge 0}\sfh_{a,r}u^{-r-1}:=Z_{a;1,1,1,1}(u),
\label{dr2}\\
&\sfb_a(u)=\sum_{r\gge 0}\sfb_{a,r}u^{-r-1}:=B_{a;1,1}(u),\qquad\quad  ~~\,1\lle a<N.
\label{dr3}
\end{align}

We call the elements listed in Proposition \ref{gen} and \eqref{Ha}--\eqref{Cb} the \emph{parabolic generators}. In case we would like to emphasize the choice of $\mu$, we write $\X_N^+,\Y_N^+$ as $\X_{\mu}^+,\Y_\mu^+$, respectively. The first main goal of this article is to explicitly write down the defining relations of $\Y_{\mu}^+$ in terms of the parabolic generators for any given $\mu$.

\subsection{Properties of parabolic generators}
We shall need the following properties of parabolic generators, and in particular the properties under various (anti)homomorphisms $\jmath$, $\psi_m$, $\tau$, $\zeta_N$, $\varphi_m$, $\omega_N$ introduced in \S\ref{sec:hom}.

\begin{lem}[{\cite[Lemma 2.14.1]{Mol07}}]\label{psilem}
For any $1\lle i,j\lle N$, we have
\[
\psi_m\big(s_{ij}(u+\tfrac{m}{2})\big)=\left\vert  \begin{array}{cccc} s_{11}(u) &\cdots &s_{1m}(u) &s_{1, m+j}(u)\\
\vdots &\ddots &\vdots &\vdots \\
s_{m1}(u) &\cdots &s_{mm}(u) &s_{m, m+j}(u)\\
s_{m+i, 1}(u) &\cdots &s_{m+i,m}(u) &\mybox{$s_{m+i,m+j}(u)$}
\end{array} \right\vert.
\]
\end{lem}

\begin{lem}\label{psi lem}
We have
\begin{align*}
D_{a;i,j}(u)&=\psi_{\mu_{(a-1)}}(D_{1;i,j}(u+\tfrac{\mu_{(a-1)}}{2})),\\
E_{a;i,j}(u)&=\psi_{\mu_{(a-1)}}(E_{1;i,j}(u+\tfrac{\mu_{(a-1)}}{2})),\\
F_{a;j,i}(u)&=\psi_{\mu_{(a-1)}}(F_{1;j,i}(u+\tfrac{\mu_{(a-1)}}{2})).
\end{align*}
\end{lem}
\begin{proof}
The proof is similar to that of \cite[Cor.~3.2]{LWZ23}.
\end{proof}

\begin{lem}[{\cite[Lem.~3.3]{LWZ23}}]
The subalgebras $\jmath(\X_m^+)$ and $\psi_m(\X_N^+)$ of $\X_{m+N}^+$ commute with each other.
\end{lem}

\begin{cor}[{cf. \cite[Cor.~3.4]{LWZ23}}]\label{comcor}
We have $[D_{a;i,j}(u),D_{b;k,l}(v)]=0$ if $a\ne b$ and
\begin{align}
&[E_{a;i,j}(u),E_{b;k,l}(v)]=[E_{a;i,j}(u),F_{b;k,l}(v)]=[F_{a;i,j}(u),F_{b;k,l}(v)]=0, &\text{if }|a-b|>1,\notag\\
&[D_{a;i,j}(u),E_{b;k,l}(v)]=[D_{a;i,j}(u),F_{b;k,l}(v)]=0,&\text{if }a\ne b,b+1.\label{defcom}
\end{align}
\end{cor}

\begin{prop}\label{prop:B=C}
We have $B_{b,a;l,k}(u)=C_{a,b;k,l}(-u)$ for $1\lle a<b\lle n$, $1\lle k\lle \mu_a$, and $1\lle l\lle \mu_b$.
\end{prop}
\begin{proof}
It is proved by induction on $b-a$. By applying \eqref{efgen} we reduce the general case to the cases for $b-a=1,2$. These two initial cases are based on \eqref{b=c n2} and Lemma \ref{e=f gen}, which requires a careful study in low rank cases (we skip the details here to avoid repetition). Then we apply the shift homomorphism and use Lemma \ref{psi lem}. 
\end{proof}

\begin{cor}\label{gen:cor}
The algebra $\X_N^+$ (or $\Y_N^+$) is generated by the following elements:
\begin{align*}
    H_{a;i,j}^{(r)},\qquad\qquad & 1\lle a\lle n,\ 1\lle i,j\lle \mu_a,\ r>0,\\
    B_{a;j,i}^{(r)},\qquad \qquad & 1\lle a< n,\  1\lle i\lle \mu_a,\ 1\lle j\lle \mu_{a+1},\ r>0.
\end{align*}
\end{cor}
\begin{proof}
This follows from  Propositions \ref{gen} and \ref{prop:B=C}.
\end{proof}

\begin{rem}\label{a2gen}
With an additional assumption that each $\mu_i$ is even, the Gaussian decomposition \eqref{S=FDE} exists and Propositions \ref{gen} holds for type AII as well.
It then follows from Proposition~\ref{a2prop:B=C} below that Corollary~\ref{gen:cor} also holds for $\X_N^-$ (or $\Y_N^-$).
\end{rem}

\begin{thm}\label{PBWgauss}
The monomials in the elements
\begin{align*}
\{H_{a;i,i}^{(2r)} \}_{1\lle a\lle  n, 1\lle i\lle \mu_a,r\gge 1},
\quad
\{H_{a;i,j}^{(r)}\}_{1\lle a\lle  n, 1\lle j<i\lle \mu_a, r\gge 1},
\quad
\{B_{b,a;i,j}^{(r)} \}_{1\lle a<b\lle n, 1\lle i\lle \mu_b,1\lle j\lle \mu_a, r\gge 1},
\end{align*}
taken in any fixed linear order form a PBW basis of $\Y_N^+$.
\end{thm}

\begin{proof}
Recall the filtration $\{\mathcal F_s\Y_N^+\}_{s\gge 0}$ from \eqref{filter:tY}. By Gauss decomposition and the definitions of these elements from \eqref{Ha}--\eqref{Bb}, we have
\begin{align*}
& H_{a;i,i}^{(2r)}\equiv s_{\mu_{(a-1)}+i,\mu_{(a-1)}+i}^{(2r)} \pmod{\mathcal F_{2r-2}\Y_N^+},\\
& H_{a;i,j}^{(r)}\equiv s_{\mu_{(a-1)}+i,\mu_{(a-1)}+j}^{(r)} \pmod{\mathcal F_{r-2}\Y_N^+},\\
& B_{b,a;i,j}^{(r)}\equiv s_{\mu_{(b-1)}+i,\mu_{(a-1)}+j}^{(r)} \pmod{\mathcal F_{r-2}\Y_N^+}.
\end{align*}
Hence the statement follows from Proposition \ref{prop:PBW}.
\end{proof}

\begin{lem}\label{tauimg}
We have 
\begin{align*}
\tau(D_{a;i,j}(u))&=D_{a;j,i}(u), \qquad\qquad 1\lle   a\lle n ,1\lle i,j\lle \mu_{a},\\
\tau(E_{a,b;i,j}(u))&=F_{b,a;j,i}(u), \qquad\qquad 1\lle a<b\lle n,1\lle i\lle \mu_a,1\lle j\lle \mu_{b},\\
\tau(F_{b,a;j,i}(u))&=E_{a,b;i,j}(u), \qquad\qquad 1\lle a<b\lle n,1\lle i\lle \mu_a,1\lle j\lle \mu_{b}.
\end{align*}
\end{lem}

Finally, it is also convenient to have the following property of the automorphism $\zeta_N$ of $\X_N^+$. For a strict composition $\mu=(\mu_1,\mu_2,\ldots,\mu_n)$ of $N$, define
\[
\bar\mu=(\mu_n, \ldots,\mu_2,\mu_1).
\]
\begin{prop}\label{zeta-pro}
For all $1\lle i,j,k\lle \mu_a$, $1\lle l\lle \mu_{a+1}$, we have $\zeta_N:\X_\mu^+\to \X_{\bar\mu}^+$ with
\begin{align*}
\zeta_N\big(D_{a;i,j}(u)\big)&= \wtl D_{n+1-a;\mu_a+1-i,\mu_a+1-j}(-u-\tfrac{N}2),\hskip 0.97cm 1\lle a\lle n,\\
\zeta_N\big(E_{a;k,l}(u)\big)&= -F_{n-a;\mu_a+1-k,\mu_{a+1}+1-l}(-u-\tfrac{N}2),\qquad 1\lle a< n,\\
\zeta_N\big(F_{a;l,k}(u)\big)&= -E_{n-a;\mu_{a+1}+1-l,\mu_a+1-k}(-u-\tfrac{N}2),\qquad 1\lle a< n.
\end{align*}
In particular, we have
\begin{align*}
\zeta_N\big(B_{a;l,k}(u)\big)&= -B_{n-a;\mu_a+1-k,\mu_{a+1}+1-l}(u),\\
\zeta_N(H_{a;i,j}(u))&=\wtl H_{n+1-a;\mu_a+1-i,\mu_a+1-j}(-u),\\
\zeta_N(\wtl H_{a;i,j}(u))&= H_{n+1-a;\mu_a+1-i,\mu_a+1-j}(-u).
\end{align*}
\end{prop}
\begin{proof}
It is proved similarly to \cite[Prop. 1]{Go07} and \cite[Lem. 3.6]{LWZ23}.
\end{proof}

\subsection{Sklyanin determinants in parabolic generators} 
Recall the notations from \S\ref{sec:skldet}. For $k\in\bZ_{>0}$ and tuples $\bm i=(i_1,\ldots,i_k)$ and $\bm j=(j_1,\ldots,j_k)$ of integers from $\{1,\ldots,N\}$, we denote that  
\[
\mathcal S_{\bm i,\bm j}(u):=\mathcal S^{i_1\cdots i_k}_{j_1\cdots j_k}(u).
\]
Denote by $m+\bm i$ the $k$-tuple $(m+i_1,\ldots,m+i_k)$; similar for $m+\bm j$.
\begin{lem}\label{lem:shift}
The homomorphism $\varphi_m:\X_N^+\to \X_{m+N}^+$ sends
$\mathcal S_{\bm i,\bm j}(u)\mapsto \mathcal S_{m+\bm i,m+\bm j}(u)$.
\end{lem}

\begin{proof}
This follows from the recurrence relations for Sklyanin minors established in \cite[proof of Thm.~2.7.2]{Mol07}, cf. \cite[Prop.~2.13.10]{Mol07}.
\end{proof}

\begin{prop}\label{prop:cind}
If $\g_N=\mathfrak{o}_N$, then we have $\mathcal C_k(u)=\sfd_1(u)\sfd_2(u-1)\cdots \sfd_k(u-k+1)$ in $\Y_N^+$.
\end{prop}
\begin{proof}
For $k=N$, the claim follows from \cite[Thm.~2.12.1]{Mol07}. If $k<N$, then it reduces to the previous case by applying the homomorphism $\iota:\X_k^+\to \X_N^+$, $s_{ij}(u)\mapsto s_{ij}(u)$. 
\end{proof}

\begin{prop}\label{prop:anti}
Let $\bm i=(i_1,\ldots,i_k)$ and $\bm j=(j_1,\ldots,j_k)$ be $k$-tuples of distinct integers from $\{1,\ldots,N\}$. Let $\bm i'=(i_{k+1},\ldots,i_N)$ and $\bm j'=(j_{k+1},\ldots,j_N)$ be such that
$
\{i_1,\dots,i_N\}=\{j_1,\dots,j_N\}=\{1,\dots,N\}.
$
Let $\xi$ denote the sign of the permutation $(i_1,\dots,i_N)\mapsto (j_1,\dots,j_N)$. Then we have
\[
\omega_N\big(\mathcal S_{\bm i,\bm j}(u)\big) =\xi\big(\sdet\, S(-u+\tfrac{N}2-1)\big)^{-1}\mathcal S_{\bm j',\bm i'}(-u+\tfrac{N}2-1).
\]
\end{prop}

\begin{proof}
By \cite[Proof of Thm. 2.13.9]{Mol07}, we have
\begin{align*}
A_N\langle S_1,\ldots,S_k\rangle &\mathop{\overrightarrow\prod}\limits_{i=1,\ldots,k}(R^{\sf T}_{i,k+1}\cdots R_{iN}^{\sf T})\\
=&A_N\,\sdet\,S(u)\, S_N^\circ R_{N-1,N}^{\circ}S_{N-1}^\circ\cdots S_{k+2}^\circ R_{k+1,N}^\circ \cdots R_{k+1,k+2}^\circ S_{k+1}^\circ,
\end{align*}
where $S^\circ(u)=\omega_N(S(u))$, $u_i=u-i+1$, and
\[
 u_i^\circ=-u_i-\tfrac{N}2,\qquad S_i^\circ=S_i^\circ(u_i^\circ),\qquad R_{ij}^\circ =R^{\sf T}_{ij}(-u_i^\circ-u_j^\circ).
\]
Equating the coefficients of the element
\[
e_{i_1j_1}\otimes \cdots\otimes e_{i_k,j_k}\otimes e_{i_Ni_N}\otimes \cdots\otimes e_{i_{k+1}i_{k+1}}
\]
in the above identity and using \eqref{eq:sign-sdet}, we deduce that
\[
\mathcal S_{\bm i,\bm j}(u)=\xi\ \sdet\,S(u)\ \omega_N\big(\mathcal S_{\bm j',\bm i'}(-u+\tfrac{N}2-1)\big).
\]
The claim follows by suitable substitutions.
\end{proof}

Now we are able to describe the images of the Sklyanin minors under the homomorphism $\psi_m$.
\begin{lem}\label{lem:shiftt}
Let $\bm i,\bm j$ be $k$-tuples of distinct integers from $\{1,\ldots,N\}$. Then
\[
\psi_m\big(\mathcal S_{\bm i,\bm j}(u)\big) =\mathcal C_m(u+\tfrac{m}2)^{-1}\mathcal S_{m\#\bm i,m\#\bm j}(u+\tfrac{m}2),
\]
where $m\#\bm i$ denotes the $(m+k)$-tuple $(1,\ldots,m,m+i_1,\ldots,m+i_k)$.
%and $m\#\bm j$ is defined similarly.
\end{lem}

\begin{proof}
Note that $\psi_m=\omega_{m+N}\circ \varphi_m\circ \omega_N$. The lemma follows from applying Proposition \ref{prop:anti} along with Lemma \ref{lem:shift}.
\end{proof}

The next result is a generalization of \cite[Thm. 2.12.1]{Mol07} and \cite[Thm. 6.1]{Br16}; also cf. \cite[Thm. 8.6]{BK05}, \cite[Prop. 3.2]{CH23}.

\begin{prop} \label{sdetdecomp}
Let $\mu=(\mu_1,\ldots,\mu_n)$ be a composition of $N$ with length $n$; recall $\mu_{(i)}$ from \eqref{mua}. Then, we have
\[
\sdet\,S(u)=\sdet\, D_1(u-\mu_{(0)})\, \sdet\, D_2(u-\mu_{(1)})\cdots \, \sdet\, D_n(u-\mu_{(n-1)})
\]
and, for $1\lle i\lle n$,
\[
\sdet \, {}^\mu D_i(u- \mu_{(i-1)} )=\prod_{k=\mu_{(i-1)}+1}^{\mu_{(i)}}\sfd_k(u-k+1).
\]
\end{prop}
\begin{proof}
Fix $1\lle i\lle n$ and let $\bar\mu:=(\mu_{i},\mu_{i+1},\ldots,\mu_n)$. It follows from Lemma \ref{psi lem}, we have
\[
{}^\mu D_i(u-\mu_{(i-1)})=\psi_{\mu_{(i-1)}}({}^{\bar \mu}D_1(u-\tfrac{\mu_{(i-1)}}{2})),
\]
where ${}^{\bar \mu}D_1(u-\mu_{(i-1)})\in \X^+_{N-\mu_{(i-1)}}$. Note that ${}^{\bar \mu}D_1(u)={}^{\bar \mu}S_{1,1}(u)$. Setting $\bm i=\bm j=(1,\dots,\mu_i)$ and applying Lemma \ref{lem:shiftt}, we have
\[
\sdet\,{}^\mu D_i(u-\mu_{(i-1)})= \psi_{\mu_{(i-1)}}(\sdet\,{}^{\bar \mu}S_{1,1}(u-\tfrac{\mu_{(i-1)}}{2}))=\mathcal C_{\mu_{(i-1)}}(u)^{-1}\mathcal C_{\mu_{(i)}}(u).
\]
Now the first statement easily follows from the above as $\sdet S(u)=\mathcal C_N(u)$.

It is known from \cite{LWZ23} that $[d_i(u),d_j(v)]=0$ for all $1\lle i,j\lle N$. Thus, it follows from Proposition \ref{prop:cind} that
\[
\sdet \, {}^\mu D_i(u-\mu_{(i-1)})=\prod_{k=\mu_{(i-1)}+1}^{\mu_{(i)}}\sfd_k(u-k+1).\qedhere
\]
\end{proof}
\begin{rem}
Note that the first equality also holds for type AII with the same proof provided each part $\mu_i$ in the composition $\mu$ is even; cf. \cite[\S7.6]{MNO96}. An analogue of the second equality for type AII can be deduced from \cite[Thm.~4.1.7]{Mol07}.
\end{rem}

\section{Parabolic presentations of twisted Yangians}
\label{sec:parabolicI}

In this section, we formulate the parabolic presentation of $\Y_N^\pm$ associated to an arbitrary (assuming even for AII) composition of $N$. We outline the proof of the main result for type AI, and postpone some lengthy 2-block and 3-block verification to the subsequent Sections \ref{sec:lower2} and \ref{sec:lower3}. The proof for type AII is similar, but some modifications are required which are discussed in Section \ref{sec:pfAII}.

\subsection{The formulation}

Let $\mu=(\mu_1,\ldots,\mu_n)$ be any composition of $N$, for type AII we assume that each $\mu_i$ is even. Recall from Corollary~\ref{gen:cor} and Remark~\ref{a2gen} that we have a generating set $\{H_{a;i,j}^{(r)},\wtl H_{a;i,j}^{(r)}, B_{a;i,j}^{(r)}\}$ for the twisted Yangian $\Y_N^\pm$.

\begin{thm}\label{mainthm}
Let $\mu=(\mu_1,\ldots,\mu_n)$ be a composition of $N$, where for type AII we assume that each $\mu_i$ is even. The twisted Yangian $\Y_N^\pm$ is generated by elements 
$$
\{H_{a;i,j}^{(r)},\wtl H_{a;i,j}^{(r)}\}_{1\lle a\lle n,1\lle i,j,\lle \mu_a,r\gge 0},
\qquad 
\{B_{a;i,j}^{(r)}\}_{1\lle a<n,1\lle i\lle \mu_{a+1},1\lle j\lle \mu_a,r>0}
$$  
subject to only the following relations,
\begin{align}
H_{a;i,j}^{(0)}&=\delta_{ij},\qquad H_{1;1,1}^{(2r-1)}+H_{1;1',1'}^{(2r-1)}=0,\label{pr1}\\
\sum_{p=1}^{\mu_a}\sum_{t=0}^r H_{a;i,p}^{(t)}\wtl H_{a;p,j}^{(r-t)}&=\delta_{r0}\delta_{ij},\label{pr2}\\
[H_{a;i,j}^{(r)},H_{b;k,l}^{(s)}] &= \delta_{ab}\Big(\sum_{t=0}^{r-1}\big(H_{a;k,j}^{(r-1-t)}H_{a;i,l}^{(s+t)}-H_{a;k,j}^{(s+t)}H_{a;i,l}^{(r-1-t)}\big) \label{pr3}\\
&\quad -\sum_{t=0}^{r-1}(-1)^t\big(\theta_{k}\theta_{j'} H_{a;i,k'}^{(r-1-t)}H_{a;j',l}^{(s+t)}-\theta_i\theta_{l'}H_{a;k,i'}^{(s+t)}H_{a;l',j}^{(r-1-t)} \big) \notag\\
&\quad +\sum_{t=0}^{\lfloor r/2\rfloor-1}\theta_i\theta_{j'}\big(H_{a;k,i'}^{(r-2-2t)}H_{a;j',l}^{(s+2t)}- H_{a;k,i'}^{(s+2t)}H_{a;j',l}^{(r-2-2t)}\big)\Big), \notag\\
[H_{a;i,j}^{(r)}, B_{b;k,l}^{(s)}] & = \delta_{ab}\Big(\sum_{p=1}^{\mu_a}\sum_{m=0}^{r-1}\sum_{t=0}^m (-1)^m\delta_{jl'}\theta_{l'}\theta_p{m \choose t}\Big(\frac{\mu_a}{2}\Big)^{m-t} H_{a;i,p}^{(r-m-1)}B_{a;k,p'}^{(s+t)}  \label{pr4}\\
&\qquad ~~ -\sum_{p=1}^{\mu_a}\sum_{m=0}^{r-1}\sum_{t=0}^m (-1)^{m-t}\delta_{il}{m \choose t}\Big(\frac{\mu_a}{2}\Big)^{m-t}B_{a;k,p}^{(s+t)}H_{a;p,j}^{(r-m-1)}\Big)\notag\\
&\quad  +\delta_{a,b+1}\Big(\sum_{m=0}^{r-1} (-1)^{m+1}\theta_{j'}\theta_k H_{a;i,k'}^{(r-1-m)}B_{b;j',l}^{(s+m)} + \sum_{m=0}^{r-1} B_{b;i,l}^{(s+m)}H_{a;k,j}^{(r-1-m)}\Big),\notag\\
 [B_{a;i,j}^{(r)},B_{b;k,l}^{(s)}]&=0 ,\qquad \text{ if } |a-b|>1 \text{ or if }b=a+1, \text{ and } i\ne l,\label{pr5}\\
  [B_{a;i,j}^{(r)},B_{a;k,l}^{(s)}]
 &=\sum_{t=1}^{r-1} B_{a;k,j}^{(r+s-1-t)} B_{a;i,l}^{(t)}  -  \sum_{t=1}^{s-1} B_{a;k,j}^{(r+s-1-t)} B_{a;i,l}^{(t)}  
-(-1)^{r}\theta_{i}\theta_j Z_{a;j',l,k,i'}^{(r+s-1)}, \label{pr6}\\
[B_{a;i,j}^{(r+1)},B_{a+1;k,l}^{(s)}]&= [B_{a;i,j}^{(r)},B_{a+1;k,l}^{(s+1)}] + \frac{\mu_{a+1}}{2}[B_{a;i,j}^{(r)},B_{a+1;k,l}^{(s)}]-\delta_{il}\sum_{q=1}^{\mu_{a+1}} B_{a;q,j}^{(r)}B_{a+1;k,q}^{(s)},\label{pr7}
\end{align}
and the Serre relations
\begin{align}
&\big[[B_{a;i,j}^{(s)},B_{a+1;k,l}^{(r_1)}],B_{a+1;f,g}^{(r_2)}\big]+\big[[B_{a;i,j}^{(s)},B_{a+1;k,l}^{(r_2)}],B_{a+1;f,g}^{(r_1)}\big]\label{pr-1}\\
=&\begin{cases}
\delta_{il}((-1)^{r_1}+(-1)^{r_2})\theta_{k}\theta_l[Z_{a+1;l',g,f,k'}^{(r_1+r_2-1)},B_{a;i,j}^{(s)}], \quad \text{ if }\mu_{a+1}>1,\\[2mm]
((-1)^{r_1}+(-1)^{r_2})\sum\limits_{p=1}^{r_1+r_2-1}\sum\limits_{q=0}^{r_1+r_2-1-p}(-1)^{p-1} 2^{-q}{p+q-1\choose q} B_{a;i,j}^{(p+s-1)}Z_{a+1;l,l,f,k}^{(r_1+r_2-1-p-q)},~\text{otherwise},
\end{cases}\notag\\[2mm]
&\big[[B_{a+1;i,j}^{(s)},B_{a;k,l}^{(r_1)}],B_{a;f,g}^{(r_2)}\big]+\big[[B_{a+1;i,j}^{(s)},B_{a;k,l}^{(r_2)}],B_{a;f,g}^{(r_1)}\big]\label{pr-2}\\
=&\begin{cases}
\delta_{jk}((-1)^{r_1}+(-1)^{r_2})\theta_{k}\theta_l[Z_{a;l',g,f,k'}^{(r_1+r_2-1)},B_{a+1;i,j}^{(s)}], \quad \text{ if }\mu_{a+1}>1,\\[2mm]
((-1)^{r_1}+(-1)^{r_2})\sum\limits_{p=1}^{r_1+r_2-1}\sum\limits_{q=0}^{r_1+r_2-1-p}(-1)^{q} 2^{-q}{p+q-1\choose q} B_{a+1;i,j}^{(p+s-1)}Z_{a;g,l,j,j}^{(r_1+r_2-1-p-q)},~\text{ otherwise, }
\end{cases}\notag
\end{align}
for all possible admissible indices. 
For $\Y_N^+$ we follow the convention \eqref{a1thetadef}, while for $\Y_N^-$ we follow \eqref{thetadef}.
Here the elements $Z_{a;i,j,k,l}^{(r)}$ are defined in \eqref{Zdef}. Moreover, in \eqref{pr-1} for $\mu_{a+1}>1$, the left hand side does not depend on $i$ and $l$ but on $\delta_{il}$. Thus we assume further that $i\ne g$; see the proof of the relation in \S \ref{ssec:serre} for further clarification. Similarly, we assume that $j\ne f$ in \eqref{pr-2}. 
\end{thm}
\begin{rem}\label{rem-Z}
In the case of $\Y_N^+$, it follows by setting $i=k$, $j=l$, and $r=s$ in  \eqref{pr6} %, cf. \eqref{n2pf6},
that 
\beq\label{Z's}
Z_{a;i,i,j,j}^{(2r-1)}=0,\qquad 1\lle a<n,~1\lle i\lle \mu_a,~1\lle j\lle \mu_{a+1},~r\gge 1.
\eeq
\end{rem}
\begin{rem}\label{remext}
When $\mu=(1^N)$, then the Serre relations \eqref{pr-1} and \eqref{pr-2} correspond to an analogue of \cite[(6.18)]{LWZ25} instead of \cite[(5.5)]{LWZ23}. A detailed proof of equivalence between the Serre relations here and in \cite{LWZ23} can be found in \cite[\S2.4]{Lu25}. Specifically, the component-wise formula \eqref{pr-1} corresponds to the expansion of the RHS of \cite[(2.25)]{Lu25} while \cite[(6.18)]{LWZ25} corresponds to the RHS of \cite[(2.26)]{Lu25}.
\end{rem}

The following reduced presentation is helpful, since it can be used to avoid the verification of the sophisticated Serre relations, cf. \cite[Thm. 4.14]{LWZ25}.

\begin{thm}\label{redthm}
The twisted Yangian $\Y_N^\pm$ is generated by elements 
$$\{H_{a;i,j}^{(r)},\wtl H_{a;i,j}^{(r)}\}_{1\lle a\lle n,1\lle i,j,\lle \mu_a,r\gge 0},\qquad \{B_{a;i,j}^{(r)}\}_{1\lle a<n,1\lle i\lle \mu_{a+1},1\lle j\lle \mu_a,r>0}$$ subject to the relations \eqref{pr1}--\eqref{pr7} plus the finite type Serre relations:
\beq
\label{pr-1-2-better}
\begin{split}
\big[[B_{a+1;i,j}^{(1)},B_{a;k,l}^{(1)}],B_{a;f,g}^{(1)}\big]=-\delta_{jk}\delta_{g'l}\theta_f\theta_gB_{a+1;i,f'}^{(1)},\\
\big[[B_{a;i,j}^{(1)},B_{a+1;k,l}^{(1)}],B_{a+1;f,g}^{(1)}\big]=-\delta_{il}\delta_{f'k}\theta_f\theta_gB_{a;g',j}^{(1)}.
\end{split}
\eeq
\end{thm}

Theorem \ref{mainthm} and Theorem \ref{redthm} will be proved in the next subsection (modulo some lengthy 2-block and 3-block verification which can be found in Sections \ref{sec:lower2}--\ref{sec:lower3}). 

\subsection{Proof of main theorems for type AI}

In this subsection, we give detailed proofs of Theorem \ref{mainthm} and Theorem~\ref{redthm} for type AI. Additional details for the proofs of these theorems for type AII will appear in 
Section \ref{sec:pfAII}. 

\begin{proof}[Proof of Theorem \ref{mainthm} for type AI]

The most complicated situation in type AI is when some $\mu_i=1$, which does not occur in type AII.

Following the convention \eqref{a1thetadef}. We first show that the twisted Yangian $\Y_N^+$ has generators $H_{a;i,j}^{(r)}$, $B_{b;k,l}^{(r)}$ and these generators satisfy the relations stated in the theorem. Define $H_{a;i,j}^{(r)}$, $B_{b;k,l}^{(r)}$, and $C_{a;l,k}^{(r)}$ (also the generating series) via Gauss decomposition as in \eqref{Ha}--\eqref{Cb} for $1\lle a\lle n$, $1\lle b<n$, and $r\in\bZ_{>0}$. By Corollary \ref{gen:cor}, the elements $H_{a;i,j}^{(r)}$, $\wtl H_{a;i,j}^{(r)}$, $B_{b;k,l}^{(r)}$ for $1\lle i\lle n$, $1\lle b<n$, $r\in\bZ_{>0}$, generate the twisted Yangian $\Y_N^+$.

We prove that $H_{a;i,j}^{(r)}$, $\wtl H_{a;i,j}^{(r)}$, and $B_{b;k,l}^{(r)}$ satisfy the listed relations by induction on $n$. The proofs for the base cases for $n=2,3$ are nontrivial and long; they will be presented in Sections \ref{sec:lower2}--\ref{sec:lower3}.  

The condition $H_{1;1,1}^{(2r-1)}=0$ follows from Lemma~\ref{lem:s11} as $H_{1;1,1}(u)=D_{1;1,1}(u)=s_{11}(u)$ is even. Suppose now the relations hold for the twisted Yangians $\Y_{\nu}^+$ for all $\nu$ of length $n$. We would like to prove it for the twisted Yangians $\Y_{\mu}^+$ with $\mu =(\mu_1, \ldots, \mu_{n+1})$ of length $n+1$ case. Set 
\[
\mu^+=(\mu_1,\dots,\mu_{n}),\qquad \mu^-=(\mu_2,\dots,\mu_{n+1}).
\]
Recall that we have the natural embedding $\Y_{\mu^+}^+\to \Y_{\mu}^+$, $s_{ij}(u)\to s_{ij}(u)$ which sends $H_{a;i,j}(u)$, $\wtl H_{a;i,j}(u)$, $B_j(u)$ of $\Y_{\mu^+}^+$ to the series in $ \Y_{\mu}^+$ of the same name. Note that we also have the homomorphism $\vartheta_{\mu_1}:\X_{\mu^-}^+\to \X_{\mu}^+$ with the properties in Lemma \ref{psi lem}. Hence the above relations hold for the case when ($a\lle n$, $b<n$) or $a,b\gge 2$. It remains to show the relations for the case when ($a=1$ and $b=n$) or ($a=n+1$ and $b=1$). This follows immediately from Corollary \ref{comcor} when $n\gge 3$.

Denote by $\bY_\mu$ the algebra generated by the above generators and relations as in the statement of the theorem. The above argument implies that there is a surjective homomorphism
\beq\label{surjm}
\bY_\mu \twoheadrightarrow \Y_N^+,
\eeq
which sends the generators $H_{a;i,j}^{(r)}$, $B_{b;k,l}^{(r)}$ of $\bY_\mu$ to the elements of $\Y_N^+$ denoted by the same symbols. Hence to complete the proof, we need to prove the homomorphism is injective. By Theorem \ref{PBWgauss}, the set of monomials in
\begin{align*}
H_{a;i,i}^{(2r)},\qquad & 1\lle a\lle  n, 1\lle i\lle \mu_a,\quad r\gge 1,\\
H_{a;i,j}^{(r)},\qquad & 1\lle a\lle  n, 1\lle j<i\lle \mu_a,\quad r\gge 1,\\
B_{a,b;i,j}^{(r)} \qquad & 1\lle b<a\lle n, 1\lle i\lle \mu_a,1\lle j\lle \mu_b,\quad r\gge 1,
\end{align*}
taken in some fixed linear order is linearly independent in the twisted Yangian $\Y_N^+$. Theorem~ \ref{mainthm} will follow if we can prove that the algebra $\bY_\mu$ is spanned by the monomials in the elements above with the same symbols, taken in some fixed linear order.

For $1\lle b<a\lle n$, $1\lle i\lle \mu_a$, and $1\lle j\lle \mu_b$, define elements $B_{a,b;i,j}^{(r)}$ of $\bY_\mu$ inductively by 
\beq\label{pfN0}
B_{b+1,b;i,j}^{(r)}:=B_{b;i,j}^{(r)},\quad B_{a,b;i,j}^{(r)}:=[B_{a-1;i,k}^{(1)},B_{a-1,b;k,j}^{(r)}],
\eeq
cf. Lemma \ref{genlem}. It follows by induction that the definition is  independent of the choice of $k$. Indeed, by \eqref{pr5} we have $[B_{a-1;i,k}^{(1)},B_{a-1,b;k',j}^{(r)}]=0$ if $k\ne k'$. Bracketing with $H_{a-1;k,k'}^{(1)}$, it follows from \eqref{pr4} that 
$[B_{a-1;i,k}^{(1)},B_{a-1,b;k,j}^{(r)}]=[B_{a-1;i,k'}^{(1)},B_{a-1,b;k',j}^{(r)}].$

Define a filtration on $\bY_\mu$ by setting $\deg H_{a;i,j}^{(r)}=\deg B_{b;k,l}^{(r)}=r-1$. Let $\gr\bY_\mu$ be the associated graded algebra. Let $\sfB_{a,a;i,j}^{(r)}$ and $\sfB_{a,b;k,l}^{(r)}$ be the images of $H_{a;i,j}^{(r+1)}$ and $B_{a,b;,k,l}^{(r+1)}$, respectively, in the $r$-th component of the graded algebra $\gr\bY_\mu$. 

We prepare several useful relations in the associated graded algebra $\gr\bY_\mu$. The relation \eqref{pr5} implies that for $|a-c|>1$, we have
\beq\label{pfN4}
[\sfB_{a+1,a;i,j}^{(r)},\sfB_{c+1,c;k,l}^{(s)}]=0.
\eeq
Due to \eqref{pr6} and \eqref{Zdef}, we have
\beq\label{pfN5}
[\sfB_{a+1,a;i,j}^{(r)},\sfB_{a+1,a;k,l}^{(s)}]= (-1)^r(\delta_{jl}\sfB_{a+1,a+1;k,i}^{(r+s)}-\delta_{ik}\sfB_{a,a;j,l}^{(r+s)}).
\eeq

If $|a-c|=1$, then using \eqref{pr7} we obtain that
\beq\label{pfN6}
[\sfB_{a+1,a;i,j}^{(r+1)},\sfB_{c+1,c;k,l}^{(s)}]=[\sfB_{a+1,a;i,j}^{(r)},\sfB_{c+1,c;k,l}^{(s+1)}].
\eeq

If $a-b>1$, then we prove that
\beq\label{pfN8}
\sfB_{a,b;i,j}^{(r)}=[\sfB_{a,a-1;i,k}^{(0)},\sfB_{a-1,b;k,j}^{(r)}]=[\sfB_{a,b+1;i,l}^{(r)},\sfB_{b+1,b;l,j}^{(0)}],
\eeq
where the first equality holds by definition \eqref{pfN0}. Indeed, for the second equality, we argue by induction on $a-b$ with the following computation for the inductive step:
\begin{align*}
\sfB_{a+1,b;i,j}^{(r)}&\stackrel{\eqref{pfN0}}{=}[\sfB_{a+1,a;i,k}^{(0)},\sfB_{a,b;k,j}^{(r)}]=[\sfB_{a+1,a;i,k}^{(0)},[\sfB_{a,b+1;k,l}^{(r)},\sfB_{b+1,b;l,j}^{(0)}]]\\&\stackrel{\eqref{pfN4}}{=}[[\sfB_{a+1,a;i,k}^{(0)},\sfB_{a,b+1;k,l}^{(r)}],\sfB_{b+1,b;l,j}^{(0)}]\stackrel{\eqref{pfN0}}{=}[\sfB_{a+1,b+1;i,l}^{(r)},\sfB_{b+1,b;l,j}^{(0)}].
\end{align*}
Repeating this argument with the help of \eqref{pfN8}, we prove that for $a-b>1$ and $r,s\in\bN$, we have
\beq\label{1pfn}
\sfB_{a,b;i,j}^{(r+s)}=[\sfB_{a,a-1;i,k}^{(r)},\sfB_{a-1,b;k,j}^{(s)}]
\eeq
with any $1\lle k\lle \mu_{a-1}$. 

We first show that $\sfB_{a,a;j,l}^{(r)}=(-1)^{r+1}\sfB_{a,a;l,j}^{(r)}$. We prove it by induction on $a$. For the base case $a=1$, setting $i=k=1$ and $s=0$ in \eqref{pr3}, we find that
\beq\label{pfn1}
[\sfB_{a,a;1,j}^{(r)},\sfB_{a,a;1,l}^{(0)}]=\delta_{1j}\sfB_{a,a;1,l}^{(r)}-\delta_{1l}\sfB_{a,a;1,j}^{(r)}-(-1)^r\sfB_{a,a;j,l}^{(r)}+(-1)^r\delta_{jl}\sfB_{a,a;1,1}^{(r)}.
\eeq
Switching $r\leftrightarrow 0$ and $j\leftrightarrow l$, we obtain
\beq\label{pfn2}
[\sfB_{a,a;1,l}^{(0)},\sfB_{a,a;1,j}^{(r)}]=\delta_{1l}\sfB_{a,a;1,j}^{(r)}-\delta_{1j}\sfB_{a,a;1,l}^{(r)}-\sfB_{a,a;l,j}^{(r)}+ \delta_{jl}\sfB_{a,a;1,1}^{(r)}.
\eeq
Note that the relation \eqref{pr1} is equivalent to $\sfB_{a,a;1,1}^{(r)}=(-1)^{r+1}\sfB_{a,a;1,1}^{(r)}$. Comparing \eqref{pfn1} and \eqref{pfn2}, we conclude that $\sfB_{a,a;j,l}^{(r)}=(-1)^{r+1}\sfB_{a,a;l,j}^{(r)}$. Thus, we have proved it for the case $a=1$. For the general case, note that $Z_{a;i,i,j,j}^{(2r-1)}=0$ implies that $\sfB_{a+1,a+1;1,1}^{(2r)}=\sfB_{a,a;1,1}^{(2r)}$. Hence for induction step, one only needs to repeat the above calculations.

For $1\lle b<a\lle n$, $1\lle k\lle \mu_a$, and $1\lle l\lle \mu_b$, define $\sfB_{b,a;l,k}^{(r)}:=(-1)^{r+1}\sfB_{a,b;k,l}^{(r)}$. Note that this relation is consistent with the case $a=b$, i.e. $\sfB_{a,a;i,j}^{(r)}=(-1)^{r+1}\sfB_{a,a;j,i}^{(r)}$. 

We claim that the following relations hold:
\begin{align}
% &\sfH_{a;i,j}^{(r)}+(-1)^r\sfH_{a;j,i}^{(r)}=0,\label{tod1}\\
% &[\sfH_{a;i,j}^{r},\sfH_{b;k,l}^{(s)}]=\delta_{ab}\big(\delta_{kj}\sfH_{a;i,l}^{(r+s)}-\delta_{il}\sfH_{a;k,j}^{(r+s)}-(-1)^r\delta_{i,k}\sfH_{a;j,l}^{(r+s)}+\delta_{jl}(-1)^r\sfH_{a;k,i}^{(r+s)}\big),\label{tod2}\\
% &[\sfH_{a;i,j}^{(r)},\sfB_{b,c;k,l}^{(s)}]=\delta_{ab}\delta_{jk}\sfB_{b,c;i,l}^{(r+s)}
% -\delta_{ac}\delta_{il}\sfB_{b,c;k,j}^{(r+s)}
% -(-1)^r\delta_{ab}\delta_{ik}\sfB_{b,c;j,l}^{(r+s)}+(-1)^r
% \delta_{ac}\delta_{jl}\sfB_{b,c;k,i}^{(r+s)},\label{tod3}\\
&[\sfB_{a,b;i,j}^{(r)},\sfB_{c,d;k,l}^{(s)}]=\delta_{bc}\delta_{jk}\sfB_{a,d;i,l}^{(r+s)}-\delta_{ad}\delta_{il}\sfB_{c,b;k,j}^{(r+s)}-(-1)^r\delta_{ac}\delta_{ik}\sfB_{b,d;j,l}^{(r+s)}+(-1)^r\delta_{bd}\delta_{jl}\sfB_{c,a;k,i}^{(r+s)}.\label{todo}
\end{align}
By the convention $\sfB_{b,a;l,k}^{(r)}:=(-1)^{r+1}\sfB_{a,b;k,l}^{(r)}$, it suffices to establish \eqref{todo} for the case when $a\gge b$ and $c\gge d$; the remainder of this proof is devoted to this task.

We first consider the case $a=b$ or $c=d$. Assume that $a=b$. If $c=d$, then \eqref{todo} follows immediately from \eqref{pr3}. If $c>d$, then one first proves it for the particular case $c-d=1$ using \eqref{pr4}. For the general case, it suffices to use an induction on $c-d$, \eqref{pfN0}, and \eqref{1pfn} along with Jacobi identity; see e.g. the proof of \cite[Thm. 5.1]{LWZ23}. Now it remains to prove \eqref{todo} for the cases when $a>b$ and $c>d$.

If $c=a+1$, then we get from \eqref{pr-2}, \eqref{Zdef}, and \eqref{todo} for the case $a=b$ that
\beq\label{pfN7}
\begin{split}
[\sfB_{a+1,a;i,j}^{(r)},&[\sfB_{a+1,a;k,l}^{(s)},\sfB_{c+1,c;f,g}^{(p)}]]\\
=-&[\sfB_{a+1,a;i,j}^{(s)},[\sfB_{a+1,a;k,l}^{(r)},\sfB_{c+1,c;f,g}^{(p)}]]-((-1)^r+(-1)^s)\delta_{kg}\delta_{jl}\sfB_{c+1,c;f,i}^{(r+s+p)}.
\end{split}
\eeq
Similarly, if $c=a-1$, then we get 
\beq\label{pfN7.5}
\begin{split}
[\sfB_{a+1,a;i,j}^{(r)},&[\sfB_{a+1,a;k,l}^{(s)},\sfB_{c+1,c;f,g}^{(p)}]]\\
=-&[\sfB_{a+1,a;i,j}^{(s)},[\sfB_{a+1,a;k,l}^{(r)},\sfB_{c+1,c;f,g}^{(p)}]]-((-1)^r+(-1)^s)\delta_{ik}\delta_{lf}\sfB_{c+1,c;j,g}^{(r+s+p)}.
\end{split}
\eeq

We shall verify \eqref{todo} case by case, considering all relations between pairs of indices $a>b$ and  $c>d$. There will be seven cases.

(1) If $b>c$, then \eqref{todo} is straightforward by \eqref{pfN4} and \eqref{pfN8}.

(2) If $b=c$, then using \eqref{pfN6} and \eqref{pfN8} we find that
\[
[\sfB_{b+1,b;i_1,j}^{(r)},\sfB_{b,b-1;k,l_1}^{(s)}]=\delta_{jk}\sfB_{b+1,b-1;i_1,l_1}^{(r+s)}.
\]
Taking the commutator with $\sfB_{b+2,b+1;i_2,i_1}^{(0)},\ldots,\sfB_{a,a-1;i,i_{a-b-1}}^{(0)}$ and using \eqref{pfN4}, we obtain
\[
[\sfB_{a,b;i,j}^{(r)},\sfB_{b,b-1;k,l_1}^{(s)}]=\delta_{jk}\sfB_{a,b-1;i,l_1}^{(r+s)}
\]
Similarly, apply the commutator with $\sfB_{b-1,b-2;l_1,l_2}^{(0)},\ldots,\sfB_{d+1,d;l_{b-d-1},l}^{(0)}$ to obtain \eqref{todo} for $b=c$.

(3) Now we prove \eqref{todo} with $a=c$ and $b>d$. Let us first consider the commutator $[\sfB_{a+3,a+1;i,j}^{(r)},\sfB_{a+2,a;k,l}^{(s)}]$. It suffices to do the calculation for the case $a=1$ as the same argument works for all $a$. Clearly, it follows from \eqref{pfN4}, \eqref{pfN7}, and \eqref{pfN7.5} that
\beq\label{pfN9}
[\sfB_{43;i,k}^{(r)},[\sfB_{32;k,j}^{(0)},[\sfB_{32;k,j}^{(0)},\sfB_{21;j,l}^{(s)}]]]=[[[\sfB_{43;i,k}^{(r)},\sfB_{32;k,j}^{(0)}],\sfB_{32;k,j}^{(0)}],\sfB_{21;j,l}^{(s)}]=0.
\eeq
We have
\begin{align*}
 [\sfB_{42;i,j}^{(r)},\sfB_{31;k,l}^{(s)}] &\stackrel{\eqref{pfN8}}{=} [[\sfB_{43;i,k}^{(r)},\sfB_{32;k,j}^{(0)}],[\sfB_{32;k,j}^{(0)},\sfB_{21;j,l}^{(s)}]]\\
 &\stackrel{\eqref{pfN9}}{=} [[\sfB_{43;i,k}^{(r)},[\sfB_{32;k,j}^{(0)},\sfB_{21;j,l}^{(s)}]],\sfB_{32;k,j}^{(0)}]\\
&\stackrel{\eqref{pfN4}}{=}  [[[\sfB_{43;i,k}^{(r)},\sfB_{32;k,j}^{(0)}],\sfB_{21;j,l}^{(s)}],\sfB_{32;k,j}^{(0)}]\\
&\stackrel{\eqref{pfN9}}{=} [[\sfB_{43;i,k}^{(r)},\sfB_{32;k,j}^{(0)}],[\sfB_{21;j,l}^{(s)},\sfB_{32;k,j}^{(0)}]]\stackrel{\eqref{pfN8}}{=}  - [\sfB_{42;i,j}^{(r)},\sfB_{31;k,l}^{(s)}].
\end{align*}
Therefore,
\beq\label{pfN10}
[\sfB_{a+3,a+1;\tl i ,\tl j }^{(r)},\sfB_{a+2,a;\tl k ,\tl l}^{(s)}]=[\sfB_{42:i,j}^{(r)},\sfB_{31;k,l}^{(s)}]=0.
\eeq

Then we look at a particular case,
\begin{align*}
[\sfB_{32;i,j}^{(r)},\sfB_{31;k,l}^{(s)}]&\stackrel{\eqref{pfN0}}{=}[\sfB_{32;i,j}^{(r)},[\sfB_{32;k,j}^{(0)},\sfB_{21;j,l}^{(s)}]] \\
& \stackrel{\eqref{pfN7.5}}{=}-[\sfB_{32;i,j}^{(0)},[\sfB_{32;k,j}^{(r)},\sfB_{21;j,l}^{(s)}]]-(1+(-1)^r)\delta_{ik} \sfB_{21;j,l}^{(r+s)}\\
&\stackrel{\eqref{pfN6}}{=}-[\sfB_{32;i,j}^{(0)},[\sfB_{32;k,j}^{(0)},\sfB_{21;j,l}^{(r+s)}]]-(1+(-1)^r)\delta_{ik}\sfB_{21;j,l}^{(r+s)}
\\
&\stackrel{\eqref{pfN7.5}}{=}\delta_{ik}\sfB_{21;j,l}^{(r+s)}-(1+(-1)^r)\delta_{ik}\sfB_{21;j,l}^{(r+s)}=(-1)^{r+1}\delta_{ik}\sfB_{21;j,l}^{(r+s)}.
\end{align*}
This implies further that
\begin{align*}
[\sfB_{32;i,j}^{(r)},\sfB_{41;g,l}^{(s)}]&\stackrel{\eqref{pfN0}}{=}[\sfB_{32;i,j}^{(r)},[\sfB_{43;g,k}^{(0)},\sfB_{31;k,l}^{(s)}]]\\
&~~=~~[[\sfB_{32;i,j}^{(r)},\sfB_{43;g,k}^{(0)}],\sfB_{31;k,l}^{(s)}]-[[\sfB_{32;i,j}^{(r)},\sfB_{31;k,l}^{(s)}],\sfB_{43;g,k}^{(0)}]\\
& \stackrel{\eqref{pfN0}}{=}-\delta_{ik}[\sfB_{42;g,j}^{(r)},\sfB_{31;k,l}^{(s)}]-[(-1)^{r+1}\delta_{ik}\sfB_{21;j,l}^{(r+s)},\sfB_{43;g,k}^{(0)}]
\overset{\eqref{pfN4}}{\underset{\eqref{pfN10}}{=}}
0-0=0.
\end{align*} 

Similarly, we have $[\sfB_{d+2,d-1;i,j}^{(r)},\sfB_{d+1,d;k,l}^{(s)}]=0$. Applying \eqref{pfN8} to it, we get
\beq\label{pfN11}
[\sfB_{a,b;i,j}^{(r)},\sfB_{d+1,d;k,l}^{(s)}]=0
\eeq
for any $a>d+1$ and $b<d$. 

Finally, to show \eqref{todo} for the case $a=c$ and $b>d$, namely 
\beq\label{pfN12}
[\sfB_{a,b;i,j}^{(r)},\sfB_{a,d;k,l}^{(s)}]=(-1)^{r+1}\delta_{ik}\sfB_{b,d;j,l}^{(r+s)},
\eeq
it essentially reduces to verifying for the case $[\sfB_{a,a-1;i,j}^{(r)},\sfB_{a,a-2;k,l}^{(s)}]$ by using \eqref{pfN8} and \eqref{pfN11}. This case is the same for all $a$, and for $a=3$ it has been done above.

(4) The case $a>c$ and $b=d$ is similar the case (3). Again, the application of \eqref{pfN8} and \eqref{pfN11} reduces the calculation to prove that $[\sfB_{a,a-2;i,j}^{(r)},\sfB_{a-1,a-2;k,l}^{(s)}]=(-1)^{s+1}\delta_{jl}\sfB_{a,a-1;i,k}^{(r+s)}$.

(5) For the case $a=c$ and $b=d$, we show by induction on $a-b$ that
\[
[\sfB_{a,b;i,j}^{(r)},\sfB_{a,b;k,l}^{(s)}]=(-1)^r(\delta_{jl}\sfB_{a,a;k,i}^{(r+s)}-\delta_{ik}\sfB_{b,b;j,l}^{(r+s)}).
\]
The base case $a-b=1$ is known as \eqref{pfN5}. By the induction hypothesis, we have
\begin{align*}
[\sfB_{a,b;i,j}^{(r)},\sfB_{a,b;k,l}^{(s)}]&\stackrel{\eqref{pfN0}}{=}[\sfB_{a,b;i,j}^{(r)},[\sfB_{a,a-1;k,f}^{(0)},\sfB_{a-1,b;f,l}^{(s)}]]\\
&~~=~~[[\sfB_{a,b;i,j}^{(r)},\sfB_{a,a-1;k,f}^{(0)}],\sfB_{a-1,b;f,l}^{(s)}]+[\sfB_{a,a-1;k,f}^{(0)},[\sfB_{a,b;i,j}^{(r)},\sfB_{a-1,b;f,l}^{(s)}]]\\&
\stackrel{\eqref{pfN12}}{=}[\delta_{ik}\sfB_{a-1,b;f,j}^{(r)},\sfB_{a-1,b;f,l}^{(s)}]+[\sfB_{a,a-1;k,f}^{(0)},(-1)^{(s+1)}\delta_{jl}\sfB_{a,a-1;i,f}^{(s)}]\\
&~~=~~(-1)^{r}\delta_{ik}(\delta_{jl}\sfB_{a-1,a-1;f,f}^{(r+s)}-\sfB_{b,b;j,l}^{(r+s)})+(-1)^{s+1}\delta_{jl}(\sfB_{a,a;i,k}^{(r+s)}-\delta_{ik}\sfB_{a-1,a-1;f,f}^{(r+s)})\\
&~~=~~(-1)^{r}(\delta_{jl}\sfB_{a,a;k,i}^{(r+s)}-\delta_{i,k}\sfB_{b,b;j,l}^{(r+s)}).
\end{align*}

(6) If $a>c>b>d$, then due to \eqref{pfN8} and \eqref{pfN11}, it suffices to verify that $$[\sfB_{a+3,a+1;i,j}^{(r)},\sfB_{a+2,a;k,l}^{(s)}]=0$$ which is done in \eqref{pfN10}.

(7) If $a>c>d>b$, then similarly it is reduced to \eqref{pfN11}.

Thus, we have completed the verification of the relations \eqref{todo}. 

Therefore we have shown that the associated graded algebra $\gr\bY_\mu$ is spanned by the set of monomials in the elements $\sfB_{a,b;i,j}^{(r)}$, $1\lle b<a\lle n$, $1\lle i\lle \mu_a$, $1\lle j\lle \mu_b$, $\sfB_{a,a;i,j}^{(r)}$, $1\lle a\lle n$, $1\lle j<i\lle \mu_a$, and $\sfB_{a,a;i,i}^{(2r+1)}$, $1\lle a\lle n$, $1\lle i\lle \mu_a$, $r\in\bN$ taken in some fixed linear order. This proves the injectivity of the homomorphism \eqref{surjm} and completes the proof of the theorem.
\end{proof}

\begin{proof}[Proof of Theorem \ref{redthm}]
The proof for type AI is almost the same as that of Theorem \ref{mainthm}. The main difference is that one needs to show that \eqref{pfN7} and \eqref{pfN7.5} hold for general $r,s,p$ assuming that \eqref{pfN7} and \eqref{pfN7.5} hold for $r=s=p=0$ (corresponding to \eqref{pr-1-2-better}) along with other relations whose proof does not involve Serre relations. This can be proved by bracketing $\sfB_{a,a;i,i}^{(2q-1)}$ to \eqref{pfN7} and \eqref{pfN7.5} for $r=s=p=0$. The strategy is very similar to the proof of Lemma \ref{serre-lem1}; see also \cite[Prop. 2.9]{LWZ25}. We omit the details.

The proof for type AII is similar with suitable modifications by bracketing $\sfB_{a,a;i,j}^{(q)}$ to the image of \eqref{pr-1-2-better} in the associated graded, see \S\ref{sec:pf-AII-parabolic} for more detail.
\end{proof}

\section{Relations in $2$-block parabolic presentations}
\label{sec:lower2}

In this section, we verify the $2$-block relations among the parabolic generators for twisted Yangians of type AI. This is one of the main steps in the proof of the parabolic presentation of $\Y_N^+$ formulated in Theorem~\ref{mainthm}, following the convention \eqref{a1thetadef}.
In this section, we 
%follow the convention \eqref{a1thetadef}, 
set $n=2$ and $\mu=(\mu_1,\mu_2)$. 

\subsection{A recap for $n=2$}

Let $\mu=(\mu_1,\mu_2)$ be a strict composition of $N$. The parabolic generators are given as follows: 
\begin{align*}
    D_{a;i,j}^{(r)},\wtl D_{a;i,j}^{(r)},\qquad & a=1,2,\ 1\lle i,j\lle \mu_a,\ r\in\bN,\\
    E_{1;i,j}^{(r)}, F_{1;j,i}^{(r)}, \qquad & 1\lle i\lle \mu_1,\ 1\lle j\lle \mu_2,\ r>0.
\end{align*}
Comparing both sides of the matrix equation (\ref{S=FDE}) and its inverse, we have
\begin{align}
&s_{ij}(u)=D_{1;i,j}(u), &1\lle i,j\lle \mu_1,\\
&s_{i,\mu_1+j}(u)=\sum_{p=1}^{\mu_1}D_{1;i,p}(u)E_{1;p,j}(u),&1\lle i\lle \mu_1, 1\lle j\lle \mu_2,\\
&s_{\mu_1+i,j}(u)=\sum_{p=1}^{\mu_1}F_{1;i,p}(u)D_{1;p,j}(u),&1\lle i\lle \mu_2, 1\lle j\lle \mu_1, \label{sparabolic} \\
&s_{\mu_1+i,\mu_1+j}(u)=D_{2;i,j}(u)+\sum_{p,r=1}^{\mu_1}F_{1;i,p}(u)D_{1;p,r}(u)E_{1;r,j}(u),&1\lle i,j\lle \mu_2,\\
&\tl s_{ij}(u)=\wtl D_{1;i,j}(u)+\sum_{q,t=1}^{\mu_2}E_{1;i,q}(u)\wtl D_{2;q,t}(u)F_{1;t,j}(u), &1\lle i,j\lle \mu_1,\\
&\tl s_{i,\mu_1+j}(u)=-\sum_{q=1}^{\mu_2}E_{1;i,q}(u)\wtl D_{2;q,j}(u),&1\lle i\lle \mu_1, 1\lle j\lle \mu_2,\\
&\tl s_{\mu_1+i,j}(u)=-\sum_{q=1}^{\mu_2}\wtl D_{2;i,q}(u)F_{1;q,j}(u),&1\lle i\lle \mu_2, 1\lle j\lle \mu_1,\\
&\tl s_{\mu_1+i,\mu_1+j}(u)=\wtl D_{2;i,j}(u),&1\lle i,j\lle \mu_2.
\end{align}

\subsection{Relations among $D, E, F$ generators}

\begin{prop} \label{prop2block}
Let $\mu=(\mu_1,\mu_2)$ be a strict composition of $N$. Then the following relations hold in $\Y_\mu^+$: 
\begin{align}
[D_{a;i,j}(u),D_{b;k,l}(v)]=&\,\delta_{ab}\Big(\frac{1}{u-v}\big(D_{a;k,j}(u)D_{a;i,l}(v)-D_{a;k,j}(v)D_{a;i,l}(u)\big)\notag\\&-\frac{1}{u+v+\mu_{(a-1)}}\big(D_{a;i,k}(u)D_{a;j,l}(v)-D_{a;k,i}(v)D_{a;l,j}(u)\big)\label{dd n2}\\&+\frac{1}{(u-v)(u+v+\mu_{(a-1)})}\big(D_{a;k,i}(u)D_{a;j,l}(v)-D_{a;k,i}(v)D_{a;j,l}(u)\big)\Big),\notag
\\
[D_{1;i,j}(u),E_{1;k,l}(v)]=&\,\frac{\delta_{jk}}{u-v}\sum_{p=1}^{\mu_1}D_{1;i,p}(u)\big(E_{1;p,l}(v)-E_{1;p,l}(u)\big)\label{d1e n2}
\\&\, +\frac{\delta_{ik}}{u+v+\mu_1}\sum_{p=1}^{\mu_1}\big(F_{1;l,p}(u)-E_{1;p,l}(v)\big)D_{1;p,j}(u),\notag
\\
[D_{1;i,j}(u),F_{1;k,l}(v)]=&\,\frac{\delta_{il}}{u-v}\sum_{p=1}^{\mu_1}\big(F_{1;k,p}(u)-F_{1;k,p}(v)\big)D_{1;p,j}(u)\label{d1f n2}\\&\, +\frac{\delta_{jl}}{u+v+\mu_1}\sum_{p=1}^{\mu_1}D_{1;i,p}(u)\big(F_{1;k,p}(v)-E_{1;p,k}(u)\big),\notag
\\
[D_{2;i,j}(u),E_{1;k,l}(v)]=&\,\frac{1}{u-v}D_{2;i,l}(u)\big(E_{1;k,j}(u)-E_{1;k,j}(v)\big)\label{d2e n2}\\&\qquad +\frac{1}{u+v+\mu_1} \big(E_{1;k,i}(v)-F_{1;i,k}(u)\big)D_{2;l,j}(u),\notag
\\
[E_{1;i,j}(u),F_{1;k,l}(v)]=&\, \frac{1}{u-v}\big(\wtl D_{1;i,l}(u)D_{2;k,j}(u)-\wtl D_{1;i,l}(v)D_{2;k,j}(v)\big)\label{ef n2}\\
&\qquad +\frac{1}{u+v+\mu_1}\big(E_{1;i,k}(u)-F_{1;k,i}(v)\big)\big(E_{1;l,j}(u)-F_{1;j,l}(v)\big),\notag
\\
[E_{1;i,j}(u),E_{1;k,l}(v)]=&\,\frac{1}{u-v}\big(E_{1;i,l}(u)-E_{1;i,l}(v)\big)\big(E_{1;k,j}(u)-E_{1;k,j}(v)\big)\label{ee n2}\\
&+\frac{1}{u+v+\mu_1}\big(\wtl D_{1;i,k}(u) D_{2;l,j}(u)-\wtl D_{1;i,k}(-v-\mu_1) D_{2;l,j}(-v-\mu_1)\big),\notag
\\
\label{e=f n2}
F_{1;j,i}(u)=&\, E_{1;i,j}(-u-\mu_1).\hskip 7cm
\end{align}
\end{prop}
\begin{proof}
\mybox{Equations \eqref{dd n2}}. If $a=b=1$, it is directly from \eqref{qua}. If $a=b=2$, then it follows from the case $a=b=1$ by applying the shift homomorphism $\psi_{\mu_1}$. If $a=1$ and $b=2$, then $D_{1;i,j}(u)=s_{i,j}(u)$ and $\wtl D_{2;k,l}(v)=\tl s_{\mu_1+k,\mu_1+l}(v)$. By \eqref{sts}, we have $[D_{1;i,j}(u),\wtl D_{2;k,l}(v)]=0$ for all $1\lle i,j\lle \mu_1$ and $1\lle k,l\lle \mu_2$. Thus we also have $[D_{1;i,j}(u), D_{2;k,l}(v)]=0$ for all $1\lle i,j\lle \mu_1$ and $1\lle k,l\lle \mu_2$.

\mybox{Equations \eqref{d1e n2} and \eqref{d1f n2}}. 
By \eqref{sts}, we have
\begin{align*}
(u^2-v^2)[s_{ij}(u),&\,\tl s_{k,\mu_1+l}(v)]=\delta_{jk}(u+v)\sum_{a=1}^Ns_{ia}(u)\tl s_{a,\mu_1+l}(v)\\-&\,\delta_{ik}(u-v)\sum_{a=1}^Ns_{aj}(u)\tl s_{a,\mu_1+l}(v)+\delta_{ik}\sum_{a=1}^Ns_{ja}(u)\tl s_{a,\mu_1+l}(v).
\end{align*}
Transforming it in terms of parabolic generators, we have
\begin{align*}
(u^2-v^2)&\big[D_{1;i,j}(u),\sum_{q=1}^{\mu_2}E_{1;k,q}(v)\wtl D_{2;q,l}(v)\big]\\
&=\delta_{kj}(u+v)\sum_{p=1}^{\mu_1}\sum_{q=1}^{\mu_2}D_{1;i,p}(u)\big(E_{1;p,q}(v)-E_{1;p,q}(u)\big)\wtl D_{2;q,l}(v)\\&\ -\delta_{ik}(u-v)\sum_{p=1}^{\mu_1}\sum_{q=1}^{\mu_2}\big(D_{1;p,j}(u)E_{1;p,q}(v)-F_{1;q,p}(u)D_{1;p,j}(u)\big)\wtl D_{2;q,l}(v)\\
&\ +\delta_{ik}\sum_{p=1}^{\mu_1}\sum_{q=1}^{\mu_2}D_{1;p,j}(u)(E_{1;p,q}(v)-E_{1;p,q}(u)\big)\wtl D_{2;q,l}(v).
\end{align*}
Therefore, using the fact that $[D_{1;i,j}(u),\tl D_{2;q,l}(v)]=0$, we deduce that 
\begin{align}
\big[D_{1;i,j}(u),E_{1;k,l}(v)\big]&=\frac{\delta_{kj}}{u-v}\sum_{p=1}^{\mu_1}D_{1;i,p}(u)\big(E_{1;p,l}(v)-E_{1;p,l}(u)\big)\notag\\
&-\frac{\delta_{ik}}{u+v}\sum_{p=1}^{\mu_1}\Big(\big(E_{1;p,l}(v)-F_{1;l,p}(u)\big)D_{1;p,j}(u)+[D_{1;p,j}(u),E_{1;p,l}(v)]\Big)\label{n2pf1}\\&+\frac{\delta_{ik}}{u^2-v^2}\sum_{p=1}^{\mu_1}D_{1;j,p}(u)\big(E_{1;p,l}(v)-E_{1;p,l}(u)\big).\notag
\end{align}
Thus
\begin{align*}
\sum_{p=1}^{\mu_1}\big[D_{1;p,j}(u),E_{1;p,l}(v)\big]&=\frac{1}{u-v}\sum_{p=1}^{\mu_1}D_{1;j,p}(u)\big(E_{1;p,l}(v)-E_{1;p,l}(u)\big)\\
&-\frac{\mu_1}{u+v}\sum_{p=1}^{\mu_1}\Big(\big(E_{1;p,l}(v)-F_{1;l,p}(u)\big)D_{1;p,j}(u)+[D_{1;p,j}(u),E_{1;p,l}(v)]\Big)\\&+\frac{\mu_1}{u^2-v^2}\sum_{p=1}^{\mu_1}D_{1;j,p}(u)\big(E_{1;p,l}(v)-E_{1;p,l}(u)\big).
\end{align*}
Solving for $\sum\limits_{p=1}^{\mu_1}\big[D_{1;p,j}(u),E_{1;p,l}(v)\big]$, we have
\begin{align*}
\sum_{p=1}^{\mu_1}\big[D_{1;p,j}(u),E_{1;p,l}(v)\big]&=\frac{1}{u-v}\sum_{p=1}^{\mu_1}D_{1;j,p}(u)\big(E_{1;p,l}(v)-E_{1;p,l}(u)\big)\\
&-\frac{\mu_1}{u+v+\mu_1}\sum_{p=1}^{\mu_1}\big(E_{1;p,l}(v)-F_{1;l,p}(u)\big)D_{1;p,j}(u).
\end{align*}
Plugging this into \eqref{n2pf1} and simplifying, one obtains \eqref{d1e n2}. The relation \eqref{d1f n2} follows from \eqref{d1e n2} by applying the anti-automorphism $\tau$.

\mybox{Equation \eqref{d2e n2}}. It follows from \eqref{d1f n2} that
\beq\label{n2pf2}
\begin{split}
\sum_{p=1}^{\mu_1}\big[D_{1;i,p}(u),F_{1;q,p}(v)\big]&=\frac{1}{u-v}\sum_{p=1}^{\mu_1}\big(F_{1;q,p}(u)-F_{1;q,p}(v)\big)D_{1;p,i}(u)\\&+\frac{\mu_1}{u+v+\mu_1}\sum_{p=1}^{\mu_1}D_{1;i,p}(u)\big(F_{1;q,p}(v)-E_{1;p,q}(u)\big).
\end{split}
\eeq

By \eqref{sts}, we have
\begin{align*}
(u^2-v^2)[s_{i,\mu_1+j}(u),&\,\tl s_{\mu_1+k,\mu_1+l}(v)]=\delta_{jk}(u+v)\sum_{a=1}^Ns_{ia}(u)\tl s_{a,\mu_1+l}(v)\\+&\,\delta_{jl}(u-v)\sum_{a=1}^N\tl s_{\mu_1+k,a}(v)s_{ia}(u)-\delta_{jl}\sum_{a=1}^N\tl s_{\mu_1+k,a}(v)s_{ai}(u).
\end{align*}
Rewriting it using parabolic generators, we get
\begin{align*}
(u^2-v^2)&\Big[\sum_{p=1}^{\mu_1}D_{1;i,p}(u)E_{1;p,j}(u),\wtl D_{2;k,l}(v)\Big]\\
=&\,\delta_{jk}(u+v)\sum_{p=1}^{\mu_1}\sum_{q=1}^{\mu_2}D_{1;i,p}(u)(-E_{1;p,q}(v)+E_{1;p,q}(u))\wtl D_{2;q,l}(v)\\
+&\, \delta_{jl}(u-v)\sum_{p=1}^{\mu_1}\sum_{q=1}^{\mu_2}\wtl D_{2;k,q}(v)(-\boxed{F_{1;q,p}(v)D_{1;i,p}(u)}+D_{1;i,p}(u)E_{1;p,q}(u))\\
+&\, \delta_{jl}\sum_{p=1}^{\mu_1}\sum_{q=1}^{\mu_2}\wtl D_{2;k,q}(v)(F_{1;q,p}(v)-F_{1;q,p}(u))D_{1;p,i}(u).
\end{align*}
Commuting $F_{1;q,p}(v)$ and $D_{1;i,p}(u)$ in the above box via \eqref{n2pf2} and using $[D_{1;i,j}(u),\tl D_{2;q,l}(v)]=0$, we find that
\beq\label{n2pf3}
\begin{split}
\big[E_{1;i,j}(u),\wtl D_{2;k,l}(v)\big]
&=\frac{\delta_{jk}}{u-v}\sum_{q=1}^{\mu_2}(E_{1;i,q}(u)-E_{1;i,q}(v))\wtl D_{2;q,l}(v)\\
&+\frac{\delta_{jl}}{u+v+\mu_1}\sum_{q=1}^{\mu_2}\wtl D_{2;k,q}(v)(E_{1;i,q}(u)-F_{1;q,i}(v)).
\end{split}
\eeq

Then one easily deduces \eqref{d2e n2} from \eqref{n2pf3}.

\mybox{Equation \eqref{ef n2}}. We start with the relation from \eqref{qua},
\begin{align*}
(u-v)[s_{i,\mu_1+j}(u),\tl s_{\mu_1+k,l}(v)]=\delta_{jk}\sum_{a=1}^N s_{ia}(u)\tl s_{al}(v)-\delta_{il}\sum_{a=1}^N \tl s_{\mu_1+k,a}(v)s_{a,\mu_1+j}(u),
\end{align*}
which gives
\begin{align*}
    &\,(u-v)\Big[\sum_{p=1}^{\mu_1}D_{1;i,p}(u)E_{1;p,j}(u),-\sum_{q=1}^{\mu_2} \wtl D_{2;k,q}(v)F_{1;q,l}(v)\Big]\\
    =&\,\delta_{jk}\Big(\sum_{p=1}^{\mu_1}D_{1;i,p}(u)\wtl D_{1;p,l}(v)+\sum_{p=1}^{\mu_1}\sum_{q,t=1}^{\mu_2}D_{1;i,p}(u)(E_{1;p,t}(v)-E_{1;p,t}(u))\wtl D_{2;t,q}(v)F_{1;q,l}(v)\Big)\\
    -&\, \delta_{il}\Big(\sum_{q=1}^{\mu_2}\wtl D_{2;k,q}(v) D_{2;q,j}(u)+\sum_{p,r=1}^{\mu_1}\sum_{q=1}^{\mu_2}\wtl D_{2;k,q}(v)(F_{1;q,r}(u)-F_{1;q,r}(v))D_{1;r,p}(u)E_{1;p,j}(u)\Big).
\end{align*}
We further rewrite it as
\begin{align*}
&\,\sum_{p=1}^{\mu_1}\sum_{q=1}^{\mu_2}\wtl D_{2;k,q}(v)\Big(\boxed{(u-v)F_{1;q,l}(v)D_{1;i,p}(u)-\delta_{il}\sum_{r=1}^{\mu_1}\big(F_{1;q,r}(v)-F_{1;q,r}(u)\big)D_{1;r,p}(u)}\Big)E_{1;p,j}(u)\\
=&\,\sum_{p=1}^{\mu_1}\sum_{q=1}^{\mu_2}D_{1;i,p}(u)\Big(\boxed{(u-v)E_{1;p,j}(u)\wtl D_{2;k,q}(v)+\delta_{jk}\sum_{t=1}^{\mu_2}\big(E_{1;p,t}(v)-E_{1;p,t}(u)\big)\wtl D_{2;t,q}(v)}\Big) F_{1;q,l}(v)\\
&\hskip6.8cm +\delta_{jk}\sum_{p=1}^{\mu_1}D_{1;i,p}(u)\wtl D_{1;p,l}(v)-\delta_{il}\sum_{q=1}^{\mu_2}\wtl D_{2;k,q}(v)D_{2;q,j}(u).
\end{align*}
Using \eqref{d1f n2} and \eqref{n2pf3} to rewrite the formulas in boxes, we have
\begin{align*}
&(u-v)\sum_{p=1}^{\mu_1}\sum_{q=1}^{\mu_2}\wtl D_{2;k,q}(v)D_{1;i,p}(u)\Big(F_{1;q,l}(v)E_{1;p,j}(u)-\frac{1}{u+v+\mu_1}(F_{1;q,p}(v)-E_{1;p,q}(u))E_{1;l,j}(u)\Big)\\
&=(u-v)\sum_{p=1}^{\mu_1}\sum_{q=1}^{\mu_2}D_{1;i,p}(u)\wtl D_{2;k,q}(v)\Big(E_{1;p,j}(u)F_{1;q,l}(v)+\frac{1}{u+v+\mu_1}(E_{1;p,q}(u)-F_{1;q,p}(v))F_{1;j,l}(v)\Big)\\
&\hskip7.2cm +\delta_{jk}\sum_{p=1}^{\mu_1}D_{1;i,p}(u)\wtl D_{1;p,l}(v)-\delta_{il}\sum_{q=1}^{\mu_2}\wtl D_{2;k,q}(v)D_{2;q,j}(u).
\end{align*}
Multiplying it from left by $D_{2;\alpha,k}(v)\wtl D_{1;\beta,i}(u)$ and summing over $i=1,\ldots,\mu_1$ and $k=1,\dots,\mu_2$, one finds that
\begin{align*}
& (u-v)F_{1;\alpha,l}(v)E_{1;\beta,j}(u)-\frac{u-v}{u+v+\mu_1}(F_{1;\alpha,\beta}(v)-E_{1;\beta,\alpha}(u))E_{1;l,j}(u)\\
=&\, (u-v)E_{1;\beta,j}(u) F_{1;\alpha,l}(v)-\frac{u-v}{u+v+\mu_1}(F_{1;\alpha,\beta}(v)-E_{1;\beta,\alpha}(u))F_{1;j,l}(v)\\
&\hskip 6cm +\wtl D_{1;\beta,l}(v)D_{2;\alpha,j}(v)-\wtl D_{1;\beta,l}(u)D_{2;\alpha,j}(u),
\end{align*}
completing the proof of \eqref{ef n2}.

\mybox{Equations \eqref{ee n2} and  \eqref{e=f n2}}. Multiplying \eqref{d1e n2} by $u+v+\mu_1$ and setting $v=-u-\mu_1$, one finds that 
\[
\sum_{p=1}^{\mu_1}\big(F_{1;l,p}(u)-E_{1;p,l}(-u-\mu_1)\big)D_{1;p,j}(u)=0,
\]
which further implies \eqref{e=f n2}. Using \eqref{ef n2} and \eqref{e=f n2}, one obtains \eqref{ee n2}.
\end{proof}

The relations \eqref{dd n2}--\eqref{d2e n2} and \eqref{e=f n2} appeared in \cite{Br16} for the special case $N=3$ and $\mu=(2,1)$.

\begin{rem}
Interchanging $u$ and $v$ in \eqref{ee n2} and relabeling the subscripts, we have
\beq\label{n2pf5}
\begin{split}
[E_{1;i,j}(u),E_{1;k,l}(v)]&=\frac{1}{u-v}\big(E_{1;k,j}(u)-E_{1;k,j}(v)\big)\big(E_{1;i,l}(u)-E_{1;i,l}(v)\big)\\
&+\frac{1}{u+v+\mu_1}\big(\wtl D_{1;k,i}(-u-\mu_1) D_{2;j,l}(-u-\mu_1)-\wtl D_{1;k,i}(v) D_{2;j,l}(v)\big).
\end{split}
\eeq
Setting $u=v$ in \eqref{ee n2} and \eqref{n2pf5}, we have
\beq\label{n2pf6}
\begin{split}
\wtl D_{1;i,k}(u) D_{2;l,j}(u)-&\,\wtl D_{1;i,k}(-u-\mu_1) D_{2;l,j}(-u-\mu_1)\\=&\,\wtl D_{1;k,i}(-u-\mu_1) D_{2;j,l}(-u-\mu_1)-\wtl D_{1;k,i}(u) D_{2;j,l}(u),
\end{split}
\eeq
and
\beq\label{n2pf7}
[E_{1;i,j}(u),E_{1;k,l}(u)]=\frac{1}{2u+\mu_1}\big(\wtl D_{1;k,i}(-u-\mu_1) D_{2;j,l}(-u-\mu_1)-\wtl D_{1;k,i}(u) D_{2;j,l}(u)\big).
\eeq

On the other hand, it follows from \eqref{ee n2} and \eqref{n2pf5}, respectively,
\beq\label{n2pf8}
\begin{split}
[E_{1;i,j}^{(1)},E_{1;k,l}(v)]&=E_{1;i,l}(v)E_{1;k,j}(v)-\wtl D_{1;i,k}(-v-\mu_1)D_{2;l,j}(-v-\mu_1)+\delta_{ik}\delta_{jl},\\    %\label{n2pf8}\\
[E_{1;i,j}^{(1)},E_{1;k,l}(v)]&=E_{1;k,j}(v)E_{1;i,l}(v)-\wtl D_{1;k,i}(v)D_{2;j,l}(v)+\delta_{ik}\delta_{jl}.   %\label{n2pf9}
\end{split}
\eeq
Thus, we conclude that
\beq\label{n2pf10}
[E_{1;i,j}(u),E_{1;k,l}(u)]= \wtl D_{1;i,k}(-u-\mu_1) D_{2;j,l}(-u-\mu_1)-\wtl D_{1;k,i}(u) D_{2;l,j}(u).
\eeq
\end{rem}

\subsection{Relations among $H, B$ generators}

The following results are direct consequences of Proposition \ref{prop2block} together with (\ref{Ha})--(\ref{tlHH}).
\begin{prop}\label{prop2blockhb}
Let $\mu=(\mu_1,\mu_2)$ be a strict composition of $N$. Then the following relations hold in $\Y_\mu^+$:
\begin{align}
% [H_{a;i,j}(u),H_{b;k,l}(v)]=&\,\frac{1}{u-v}\big(H_{1;k,j}(u)H_{1;i,l}(v)-H_{1;k,j}(v)H_{1;i,l}(u)\big)\notag\\&-\frac{1}{u+v-\mu_1}\big(H_{1;i,k}(u)H_{1;j,l}(v)-H_{1;k,i}(v)H_{1;l,j}(u)\big)\label{h1h1 n2}\\&+\frac{1}{(u-v)(u+v-\mu_1)}\big(H_{1;k,i}(u)H_{1;j,l}(v)-H_{1;k,i}(v)H_{1;j,l}(u)\big),\notag
% \\
[H_{a;i,j}(u),H_{b;k,l}(v)]=&\,\delta_{ab}\Big(\frac{1}{u-v}\big(H_{a;k,j}(u)H_{a;i,l}(v)-H_{a;k,j}(v)H_{a;i,l}(u)\big)\notag\\&-\frac{1}{u+v}\big(H_{a;i,k}(u)H_{a;j,l}(v)-H_{a;k,i}(v)H_{a;l,j}(u)\big)\label{hahb n2}\\&+\frac{1}{u^2-v^2}\big(H_{a;k,i}(u)H_{a;j,l}(v)-H_{a;k,i}(v)H_{a;j,l}(u)\big)\Big),\notag
\\
[H_{1;i,j}(u),B_{1;k,l}(v)]=&\,\frac{\delta_{il}}{u-v+\tfrac{\mu_1}{2}}\sum_{p=1}^{\mu_1}\big(B_{1;k,p}(u+\tfrac{\mu_1}{2})-B_{1;k,p}(v)\big)H_{1;p,j}(u)\label{h1b n2}\\&\, +\frac{\delta_{jl}}{u+v+\tfrac{\mu_1}{2}}\sum_{p=1}^{\mu_1}H_{1;i,p}(u)\big(B_{1;k,p}(v)-B_{1;k,p}(-u-\tfrac{\mu_1}{2})\big),\notag
\\
% [H_{1;i,j}(u),C_{1;k,l}(v)]=&\,\frac{\delta_{jk}}{u-v}\sum_{p=1}^{\mu_1}H_{1;i,p}(u)\big(C_{1;p,l}(v)-C_{1;p,l}(u)\big)\label{h1c n2}\\&\, +\frac{\delta_{ik}}{u+v}\sum_{p=1}^{\mu_1}\big(B_{1;l,p}(u)-C_{1;p,l}(v)\big)H_{1;p,j}(u),\notag
% \\
[H_{2;i,j}(u),B_{1;k,l}(v)]=&\,\frac{1}{u+v}H_{2;i,k}(u)\big(B_{1;j,l}(-u)-B_{1;j,l}(v)\big)\label{h2b n2}\\&\qquad +\frac{1}{u-v} \big(B_{1;i,l}(v)-B_{1;i,l}(u)\big)H_{2;k,j}(u),\notag
\\
% [B_{1;i,j}(u),C_{1;k,l}(v)]=&\, \frac{1}{u-v}\big(\wtl H_{1;k,j}(v)H_{2;i,l}(v)-\wtl H_{1;k,j}(u)H_{2;i,l}(u)\big)\label{bc n2}\\
% &\qquad -\frac{1}{u+v}\big(C_{1;k,i}(v)-B_{1;i,k}(u)\big)\big(C_{1;j,l}(v)-B_{1;l,j}(u)\big),\notag
% \\
[B_{1;i,j}(u),B_{1;k,l}(v)]=&\,\frac{1}{u-v}\big(B_{1;k,j}(v)-B_{1;k,j}(u)\big)\big(B_{1;i,l}(u)-B_{1;i,l}(v)\big)\label{bb n2}\\
&+\frac{1}{u+v}\big(Z_{1;j,l,k,i}(v)-Z_{1;j,l,k,i}(-u)\big),\notag
\\
\label{b=c n2}
B_{1;i,j}(-u)=&\, C_{1;j,i}(u),\hskip 7cm
\end{align}
where $Z_{1;j,l,k,i}(u)=\wtl H_{1;j,l}(u-\tfrac{\mu_1}{2})H_{2;k,i}(u)$ is defined in \eqref{tlHH}.
\end{prop}

\begin{rem}
Using the same approach toward proving (\ref{n2pf5}), one can deduce the equivalent form of (\ref{bb n2}):
\begin{multline}\label{nbb n2}
[B_{1;i,j}(u),B_{1;k,l}(v)]=\frac{1}{u-v}\big(B_{1;i,l}(u)-B_{1;i,l}(v)\big)\big(B_{1;k,j}(v)-B_{1;k,j}(u)\big)\\
+\frac{1}{u+v}\big(Z_{1;l,j,i,k}(-v)-Z_{1;l,j,i,k}(u)\big).
\end{multline}
\end{rem}

\begin{prop}
The following relations hold in $\Y_\mu^+$:
\begin{align}
\label{hahbcoef}[H_{a;i,j}^{(r)},H_{b;k,l}^{(s)}] &= \delta_{ab}\Big(\sum_{t=0}^{r-1}\big(H_{a;k,j}^{(r-1-t)}H_{a;i,l}^{(s+t)}-H_{a;k,j}^{(s+t)}H_{a;j,l}^{(r-1-t)}\big)\\
&\quad -\sum_{t=0}^{r-1}(-1)^t\big( H_{a;i,k}^{(r-1-t)}H_{a;j,l}^{(s+t)}-H_{a;k,i}^{(s+t)}H_{a;l,j}^{(r-1-t)} \big)\notag\\
&\quad +\sum_{t=0}^{\lfloor r/2\rfloor-1}\big(H_{a;k,i}^{(r-2-2t)}H_{a;j,l}^{(s+2t)}-H_{a;k,i}^{(s+2t)}H_{a;j,l}^{(r-2-2t)}\big)\Big),\notag\\
\label{hbcoef}
[H_{1;i,j}^{(r)}, B_{1;k,l}^{(s)}] & = \sum_{p=1}^{\mu_1}\sum_{m=0}^{r-1}\sum_{t=0}^m (-1)^m\delta_{jl}{m \choose t}\Big(\frac{\mu_1}{2}\Big)^{m-t} H_{1;i,p}^{(r-m-1)}B_{1;k,p}^{(s+t)},\\
&\quad -\sum_{p=1}^{\mu_1}\sum_{m=0}^{r-1}\sum_{t=0}^m (-1)^{m-t}\delta_{il}{m \choose t}\Big(\frac{\mu_1}{2}\Big)^{m-t}B_{1;k,p}^{(s+t)}H_{1;p,j}^{(r-m-1)},\notag\\
\label{h2bcoef}
[H_{2;i,j}^{(r)}, B_{1;k,l}^{(s)}] & = \sum_{m=0}^{r-1} (-1)^{m+1}H_{2;i,k}^{(r-1-m)}B_{1;j,l}^{(s+m)} + \sum_{m=0}^{r-1} B_{1;i,l}^{(s+m)}H_{2;k,j}^{(r-1-m)},\\
 [B_{1;i,j}^{(r)},B_{1;k,l}^{(s)}]
 &=\sum_{t=1}^{r-1} B_{1;k,j}^{(r+s-1-t)} B_{1;i,l}^{(t)}  -  \sum_{t=1}^{s-1} B_{1;k,j}^{(r+s-1-t)} B_{1;i,l}^{(t)}  
-(-1)^{r} Z_{1;j,l,k,i}^{(r+s-1)} \label{bbcoef1} \\  \label{bbcoef2}
&=\sum_{t=1}^{s-1} B_{1;i,l}^{(t)} B_{1;k,j}^{(r+s-1-t)}  -  \sum_{t=1}^{r-1} B_{1;i,l}^{(t)} B_{1;k,j}^{(r+s-1-t)}  
 +  (-1)^{s} Z_{1;l,j;i,k}^{(r+s-1)}.
\end{align}
\end{prop}
\begin{proof}
The relations (\ref{hahbcoef}), (\ref{hbcoef}) and (\ref{h2bcoef}) follow from (\ref{hahb n2}), (\ref{h1b n2}) and (\ref{h2b n2}), respectively.
To show (\ref{bbcoef1}), we collect and compare the coefficient of $u^{-r}v^{-s}$ on both sides of (\ref{bb n2}). 
Use the following expansion for the right-hand-side,
\begin{align}\label{u-v}
\frac{1}{u-v} =  \sum_{a\gge 0}u^{-1-a}v^{a},\qquad 
\frac{1}{u+v}=  \sum_{a\gge 0}(-1)^au^{-1-a}v^{a}.
\end{align}
Since the terms 
$B_{1;k,j}(u) B_{1;i,l}(u)$ and $Z_{1;j,l,k,i}(-u)$ contribute nothing, we may ignore them. Therefore, we have
\begin{align*}
 [B_{1;i,j}^{(r)},B_{1;k,l}^{(s)}] &=\sum_{a=0}^{r-2} B_{1;k,j}^{(s+a)} B_{1;i,l}^{(r-1-a)} + \sum_{a=0}^{r-2} B_{1;k,j}^{(r-1-a)} B_{1;i,l}^{(s+a)}    \\
& \qquad - \sum_{x,y\gge 1}^{x+y=s+r-1}B_{1;k,j}^{(x)} B_{1;i,l}^{(y)}  + (-1)^{r-1} Z_{1;j,l;k,i}^{(r+s-1)}\\
& = \sum_{t=1}^{r-1} B_{1;k,j}^{(r+s-1-t)} B_{1;i,l}^{(t)} + \sum_{t=1}^{r-1} B_{1;k,j}^{(t)} B_{1;i,l}^{(r+s-1-t)}    \\
& \qquad - \sum_{t=1}^{r+s-2}B_{1;k,j}^{(t)} B_{1;i,l}^{(r+s-1-t)}  -  (-1)^{r} Z_{1;j,l;k,i}^{(r+s-1)}\\
&= \boxed{\sum_{t=1}^{r-1} B_{1;k,j}^{(r+s-1-t)} B_{1;i,l}^{(t)}  -  \sum_{t=1}^{s-1}B_{1;k,j}^{(r+s-1-t)} B_{1;i,l}^{(t)}} -  (-1)^{r} Z_{1;j,l;k,i}^{(r+s-1)},
\end{align*}
which is (\ref{bbcoef1}). Starting from (\ref{nbb n2}) with exactly the same argument, one deduces (\ref{bbcoef2}).
\end{proof}

Note that the boxed part above can be rewritten as
\[
\sum_{t=\min\{r,s\}}^{r-1} B_{1;k,j}^{(r+s-1-t)} B_{1;i,l}^{(t)}  -  \sum_{t=\min\{r,s\}}^{s-1}B_{1;k,j}^{(r+s-1-t)} B_{1;i,l}^{(t)}
\]
and hence does not contain the $B_{1;k,j}^{(t)}$ for small $t$. This observation will be used to define the shifted twisted Yangian in \S \ref{paraAI}.

\section{The $3$-block parabolic presentations}
  \label{sec:lower3}

In this section, we establish the $3$-block relations among parabolic generators for twisted Yangians of type AI. This completes the proof of the parabolic presentation of $\Y_N^+$ formulated in Theorem~\ref{mainthm}, following the convention \eqref{a1thetadef}.
In this section, we set $n=3$ and $\mu=(\mu_1,\mu_2,\mu_3)$.

\subsection{The case for $n=3$}

Let $\mu=(\mu_1,\mu_2,\mu_3)$ be a strict composition of $N$.
The parabolic generators associated to $\mu$ are given as follows,
\begin{align*}
    &D_{a;i,j}^{(r)},\wtl D_{a;i,j}^{(r)},\qquad\ \   a=1,2,3,\ 1\lle i,j\lle \mu_a,\ r\in\bN,\\
    &E_{a;i,j}^{(r)}, F_{a;j,i}^{(r)},\qquad\ \ \,  a=1,2, \ \ 1\lle i\lle \mu_a,\ 1\lle j\lle \mu_{a+1},\ r>0.
    \end{align*}
Comparing both sides of the matrix equation (\ref{S=FDE}) and its inverse, we have
\begin{align*}
&s_{ij}(u)=D_{1;i,j}(u), &1\lle i,j\lle \mu_1,\\
&s_{i,\mu_1+j}(u)=\sum_{p=1}^{\mu_1}D_{1;i,p}(u)E_{1;p,j}(u),&1\lle i\lle \mu_1, 1\lle j\lle \mu_2,\\
&s_{i,\mu_1+\mu_2+j}(u)=\sum_{p=1}^{\mu_1}D_{1;i,p}(u)E_{1,3;p,j}(u),&1\lle i\lle \mu_1, 1\lle j\lle \mu_3,\\
&\tl s_{i,\mu_1+j}(u)=-\sum_{q=1}^{\mu_2} E_{1;i,q}(u) \wtl D_{2;q,j}(u) - \sum_{r,t=1}^{\mu_3}\wtl E_{1,3;i,r}(u) \wtl D_{3;r,t}(u)F_{2;t,j}(u)  ,&1\lle i\lle \mu_1, 1\lle j\lle \mu_2,\\
&\tl s_{i,\mu_1+\mu_2+j}(u)=\sum_{r=1}^{\mu_3}\wtl E_{1,3;i,r}(u)\wtl D_{3;r,j}(u),&1\lle i\lle \mu_1, 1\lle j\lle \mu_3,\\
&\tl s_{\mu_1+i,\mu_1+j}(u)=\wtl D_{2;i,j}(u) + \sum_{r,t=1}^{\mu_3} E_{2;i,r}(u)\wtl D_{3;r,t}(u) F_{2;t,j}(u),&1\lle i,j\lle \mu_2,\\
&\tl s_{\mu_1+i,\mu_1+\mu_2+j}(u)=-\sum_{r=1}^{\mu_3}E_{2;i,r}(u)\wtl D_{3;r,j}(u),&1\lle i\lle \mu_2, 1\lle j\lle \mu_3,\\
&\tl s_{\mu_1+\mu_2+i,\mu_1+j}(u)=-\sum_{r=1}^{\mu_3}\wtl D_{3;i,r}(u)F_{2;r,j}(u),&1\lle i\lle \mu_3, 1\lle j\lle \mu_2,
\end{align*}
where $\wtl E_{1,3;i,r}(u)=\sum\limits_{q=1}^{\mu_2}E_{1;i,q}(u)E_{2;q,r}(u)-E_{1,3;i,r}(u)$.

\begin{lem}\label{e=f gen}
We have $E_{a,b;i,j}(u)=F_{b,a;j,i}(-u-\mu_{(a)})$ for $1\lle a<b\lle 3$. In particular, we have $B_{b,a;j,i}(u)=C_{a,b;i,j}(-u)$.
\end{lem}
\begin{proof}
The case when $a=1$ and $b=2$ follows from the previous Section \ref{sec:lower2}. For $a=2$ and $b=3$, this is obtained from the case $a=1$ and $b=2$ by taking the homomorphism $\psi_{\mu_1}$. As for the case $a=1$ and $b=3$, this follows from  a further application of \eqref{efgen}.
\end{proof}

\begin{lem}
We have
\begin{align}
(u-v)\big[E_{1;i,j}(u),E_{2;k,l}(v)\big]&=\delta_{jk}\Big(\sum_{q=1}^{\mu_2}\big(E_{1;i,q}(u)-E_{1;i,q}(v)\big)E_{2;q,l}(v)\nonumber\\&\qquad \qquad\qquad\qquad +E_{1,3;i,l}(v)-E_{1,3;i,l}(u)\Big),\label{e1e2 n3}\\
[E_{2;i,j}^{(1)},\wtl D_{3;k,l}(v)]&=-\sum_{r=1}^{\mu_3}\big(\delta_{jk}E_{2;i,r}(v)\wtl D_{3;r,l}(v)+\delta_{jl}\wtl D_{3;k,r}(v)F_{2;r,i}(v)\big).\label{e20d3}
\end{align}
\end{lem}
\begin{proof}
By \eqref{sts}, we have
\[
(u-v)[s_{i,\mu_1+j}(u),\tl s_{\mu_1+k,\mu_1+\mu_2+j}(v)]=\delta_{jk}\sum_{a=1}^N s_{ia}(u)\tl s_{a,\mu_1+\mu_2+l}(v).
\]
Using parabolic generators, we have
\begin{align}
&(u-v)\Big[\sum_{p=1}^{\mu_1}D_{1;i,p}(u)E_{1;p,j}(u),-\sum_{r=1}^{\mu_3}E_{2;k,r}(v)\wtl D_{3;r,l}(v)\Big]\label{n3pf1}\\
&=\delta_{jk}\sum_{p=1}^{\mu_1}\sum_{r=1}^{\mu_3}D_{1;i,p}(u)\Big(\sum_{q=1}^{\mu_2}\big(E_{1;p,q}(v)-E_{1;p,q}(u)\big)E_{2;q,r}(v)-E_{1,3;p,r}(v)+E_{1,3;p,r}(u)\Big)\wtl D_{3;r,l}(v).\nonumber
\end{align}
Note that by Corollary \ref{comcor}, we have
\[
[E_{1;p,j}(u),\wtl D_{3;r,l}(v)]=[D_{1;i,p}(u),E_{2;k,r}(v)]=0.
\]
Therefore, we have
\begin{align*}
(u-v)\Big[\sum_{p=1}^{\mu_1}D_{1;i,p}(u)E_{1;p,j}(u),&-\sum_{r=1}^{\mu_3}E_{2;k,r}(v)\wtl D_{3;r,l}(v)\Big]\\
&=(u-v)\sum_{p=1}^{\mu_1}\sum_{r=1}^{\mu_3}D_{1;i,p}(u)\big[E_{1;p,j}(u),-E_{2;k,r}(v)\big]\wtl D_{3;r,l}(v).
\end{align*}
Comparing  it with \eqref{n3pf1} and canceling $D_1(u)$ from left and $D_{3}(v)$ from right, one immediately obtains \eqref{e1e2 n3}.

Expanding \eqref{n2pf3} in the region $\{(u,v)~|~ u\gg v\}$ and taking the coefficients of $u^{-1}$, we find that
\[
[E_{1;i,j}^{(1)},\wtl D_{2;k,l}(v)]=-\sum_{q=1}^{\mu_2}\big(\delta_{jk}E_{1;i,q}(v)\wtl D_{2;q,l}(v)+\delta_{jl}\wtl D_{2;k,q}(v)F_{1;q,i}(v)\big).
\]
Applying the homomorphism $\psi_{\mu_1}$ to the above formula in the algebra $\Y_{(\mu_2,\mu_3)}^+$, the relation \eqref{e20d3} follows from Lemma \ref{psi lem}.
\end{proof}

Thus we have the following component-wise commutator relations for $B_{1;i,j}^{(r)}$ and $B_{2;k,l}^{(s)}$.
\begin{cor}
We have
\[
[B_{1;i,j}^{(r+1)},B_{2;k,l}^{(s)}]- [B_{1;i,j}^{(r)},B_{2;k,l}^{(s+1)}] - \frac{\mu_2}{2}[B_{1;i,j}^{(r)},B_{2;k,l}^{(s)}]=-\delta_{il}\sum_{q=1}^{\mu_2} B_{1;q,j}^{(r)}B_{2;k,q}^{(s)}.
\]
\end{cor}

\begin{lem}
We have
\begin{multline}
\big[E_{1,3;i,j}(u),E_{2;k,l}(v)\big]=\\ 
E_{2;k,j}(v)[E_{1;i,g}(u), E_{2;g,l}(v)]+\frac{1}{u+v+\mu_1}\sum_{q=1}^{\mu_2}\wtl D_{2;k,q}(v)(F_{1;q,i}(v)-E_{1;i,q}(u))D_{3;j,l}(v).
\label{e13e2 n3}
\end{multline}
\end{lem}
\begin{proof}
Start with (\ref{sts})
\begin{multline*}
(u-v)[s_{i,\mu_1+\mu_2+j}(u), \tl s_{\mu_1+\mu_2+k,\mu_1+l}(v)]= \\
\delta_{jk}
\Big( \sum_{a=1}^{\mu_1} s_{i,a}(u) \tl s_{a,\mu_1+l}(v)+\sum_{b=\mu_{(1)}+1}^{\mu_{(2)}} s_{i,b}(u) \tl s_{b,\mu_1+l}(v)
+\sum_{c=\mu_{(2)}+1}^{\mu_{(3)}} s_{i,c}(u) \tl s_{c,\mu_1+l}(v)\Big).
\end{multline*}
Replacing everything by parabolic generators, we obtain
\begin{multline*}
(u-v)\Big[ \sum_{p=1}^{\mu_1} D_{1;i,p}(u)E_{1,3;p,j}(u), -\sum_{r=1}^{\mu_3} \wtl D_{3;k,r}(v) F_{2;r,l}(v) \Big]=
\delta_{jk} \times \bigg\{\\
\sum_{p=1}^{\mu_1}\sum_{q=1}^{\mu_2} D_{1;i,p}(u) \Big( -  E_{1;p,q}(v)\wtl D_{2;q,l}(v)+ \sum_{r,t=1}^{\mu_3}(E_{1,3;p,r}(v)-E_{1;p,q}(v)E_{2;q,r}(v))\wtl D_{3;r,t}(v)F_{2;t,l}(v) \Big) \\
+\sum_{q=1}^{\mu_2} \Big(\sum_{p=1}^{\mu_1} D_{1;i,p}(u)E_{1;p,q}(u) \Big) \Big( \wtl D_{2;q,l}(v) + \sum_{r,t=1}^{\mu_3} E_{2;q,r}(v)\wtl D_{3;r,t}(v)F_{2;t,l}(v)\Big)  \\
+\sum_{r=1}^{\mu_3} \Big( \sum_{a=1}^{\mu_1} D_{1;i,p}(u)E_{1,3;p,r}(u) \Big) \Big(-\sum_{t=1}^{\mu_3}\wtl D_{3;r,t}(v)F_{2;t,l}(v)  \Big)\bigg\}.
\end{multline*}
Since $D_{1;i,j}(u)=s_{ij}(u)$ and $\tl s_{\mu_1+\mu_2+k,\mu_1+l}(v)$ commute, we may cancel $D_{1}(u)$ from the left and deduce that 
\begin{multline}
(u-v)\Big[ E_{1,3;i,j}(u), \sum_{r=1}^{\mu_3} \wtl D_{3;k,r}(v) F_{2;r,l}(v) \Big]=
\delta_{jk} \times \bigg\{ \sum_{q=1}^{\mu_2} \big(E_{1;i,q}(v)-E_{1;i,q}(u)\big)\wtl D_{2;q,l}(v) \\
+ \sum_{r,t=1}^{\mu_3}\sum_{q=1}^{\mu_2}\Big(E_{1,3;i,r}(u)-E_{1,3;i,r}(v)+\big(E_{1;i,q}(v)-E_{1;i,q}(u)\big)E_{2;q,r}(v)\Big)\wtl D_{3;r,t}(v)F_{2;t,l}(v) \bigg\}.\label{n3pf2}
\end{multline}
Using \eqref{e1e2 n3}, the RHS of \eqref{n3pf2} can be written as
\beq\label{n3pf3}
\delta_{jk}\sum_{q=1}^{\mu_2}\Big(\big(E_{1;i,q}(v)-E_{1;i,q}(u)\big)\wtl D_{2;q,l}(v)-(u-v)\sum_{r,t=1}^{\mu_3}[E_{1;i,j}(u),E_{2;k,r}(v)]\wtl D_{3;r,t}(v)F_{2;t,l}(v)\Big).
\eeq
For the LHS of \eqref{n3pf2}, we have
\begin{multline}
(u-v)\Big[ E_{1,3;i,j}(u), \sum_{r=1}^{\mu_3} \wtl D_{3;k,r}(v) F_{2;r,l}(v) \Big]\\
=(u-v)\sum_{r=1}^{\mu_3}\Big(\big[ E_{1,3;i,j}(u),  \wtl D_{3;k,r}(v)  \big]F_{2;r,l}(v)+\wtl D_{3;k,r}(v)\big[ E_{1,3;i,j}(u),   F_{2;r,l}(v) \big]\Big).\label{n3pf4}
\end{multline}
We shall work on the first term in the RHS of \eqref{n3pf4},
\begin{align}
\sum_{r=1}^{\mu_3}&\big[ E_{1,3;i,j}(u),  \wtl D_{3;k,r}(v)  \big]F_{2;r,l}(v)\notag\\
\stackrel{\eqref{efgen}}{=}&\sum_{r=1}^{\mu_3}\big[ [E_{1;i,g}(u),E_{2;g,l}^{(1)}], \wtl D_{3;k,r}(v)  \big]F_{2;r,l}(v)\notag\\
\stackrel{\eqref{defcom}}{=}&\sum_{r=1}^{\mu_3}\big[ E_{1;i,g}(u),[E_{2;g,l}^{(1)}, \wtl D_{3;k,r}(v)]  \big]F_{2;r,l}(v)\notag\\
\stackrel{\eqref{e20d3}}{=}&-\sum_{r,t=1}^{\mu_3}\big[ E_{1;i,g}(u),\delta_{jk}E_{2;g,t}(v)\wtl D_{3;t,r}(v)+\delta_{jr}\wtl D_{3;k,t}(v)F_{2;t,g}(v)  \big]F_{2;r,l}(v)\notag\\
\stackrel{\eqref{defcom}}{=}&-\delta_{jk}\sum_{r,t=1}^{\mu_3}\big[E_{1;i,g}(u), E_{2;g,t}(v)\big]\wtl D_{3;t,r}(v)F_{2;r,l}(v)\notag\\
&\hskip5cm-\sum_{r=1}^{\mu_3}\wtl D_{3;k,r}(v)\big[E_{1;i,g}(u),F_{2;r,g}(v)\big]F_{2;j,l}(v).\label{n3pf5}
\end{align}
Plugging \eqref{n3pf5} into \eqref{n3pf4} and comparing the resulting \eqref{n3pf4} with \eqref{n3pf3}, we conclude that
\begin{multline*}
    \delta_{jk}\sum_{q=1}^{\mu_2} \big(E_{1;i,q}(v)-E_{1;i,q}(u)\big)\wtl D_{2;q,l}(v)\\=-(u-v)\sum_{r=1}^{\mu_3}\Big(\wtl D_{3;k,r}(v)\big[E_{1;i,g}(u),F_{2;r,g}(v)\big]F_{2;j,l}(v)-\wtl D_{3;k,r}(v)\big[ E_{1,3;i,j}(u),   F_{2;r,l}(v) \big]\Big).
\end{multline*}
Multiplying $D_3(v)$ from left, we have
\beq\label{n3pf6}
\begin{split}
[E_{1,3;i,j}(u),F_{2;k,l}(v)]&=[E_{1;i,g}(u),F_{2;k,g}(v)]F_{2;j,l}(v)\\&+\frac{1}{u-v}\sum_{q=1}^{\mu_2}\big(E_{1;i,q}(v)-E_{1;i,q}(u)\big)\wtl D_{2;q,l}(v)D_{3;k,j}(v).
\end{split}
\eeq
Finally, applying the anti-automorphism $\tau$ to \eqref{n3pf6} with the substitution $u\to -u-\mu_1$, the relation \eqref{e13e2 n3} follows from Lemmas \ref{tauimg} and \ref{e=f gen}.
\end{proof}

\subsection{Preparation for Serre relations}
We prepare several identities that will be used to establish the Serre relations. By \eqref{e1e2 n3}, we have
\begin{align}
[E_{1;i,j}(u),E_{2;k,l}^{(1)}]&=\delta_{jk}E_{13;i,l}(u),\label{ser1}\\
[E_{1;i,j}^{(1)},E_{2;k,l}(u)]&=\delta_{jk}\Big(E_{13;i,l}(u)-\sum_{q=1}^{\mu_2}E_{1;i,q}(u)E_{2;q,l}(u)\Big).\label{ser2}
\end{align}
It follows from \eqref{e13e2 n3} that
\beq\label{ser3}
[E_{1,3;i,j}(u),E_{2;k,l}^{(1)}]=-\delta_{jl}E_{1;i,k}(u).
\eeq
Using Lemma \ref{e=f gen} and \eqref{n3pf6}, we obtain
\beq\label{ser4}
\begin{split}
&[E_{1,3;i,j}^{(1)},E_{2;lk}(u)]\\
&=[E_{1;i,g}^{(1)},E_{2;g,k}(u)]E_{2;l,j}(u)+\sum_{q=1}^{\mu_2}E_{1;i,q}(-u-\mu_{(2)})\wtl D_{2;q,l}(-u-\mu_{(2)}) D_{3;k,j}(-u-\mu_{(2)}).
\end{split}
\eeq
Applying the homomorphism $\psi_{\mu_1}$ to \eqref{n2pf8} and using Lemma \ref{psi lem}, we get
\beq\label{ser5}
[E_{2;i,j}(u),E_{2;k,l}^{(1)}]=\wtl D_{2;i,k}(u)D_{3;l,j}(u)-\delta_{ik}\delta_{lj}-E_{2;i,l}(u)E_{2;k,j}(u).
\eeq
\begin{lem} We have
\beq\label{ser6}
\begin{split}
&\big[[E_{1;i,j}^{(1)},E_{2;j,k}(u)],E_{2;l,g}^{(1)}\big]+\big[[E_{1;i,j}^{(1)},E_{2;j,k}^{(1)}],E_{2;l,g}(u)\big]\\
=&\sum_{q=1}^{\mu_2}\Big(E_{1;i,q}(-u-\mu_{(2)})\wtl D_{2;q,l}(-u-\mu_{(2)}) D_{3;g,k}(-u-\mu_{(2)})-E_{1;i,q}(u)\wtl D_{2;q,l}(u) D_{3;g,k}(u)\Big).
\end{split}
\eeq
\end{lem}
\begin{proof}
We have
\begin{align*}
&\big[[E_{1;i,j}^{(1)},E_{2;j,k}(u)],E_{2;l,g}^{(1)}\big]+\big[[E_{1;i,j}^{(1)},E_{2;j,k}^{(1)}],E_{2;l,g}(u)\big]
\\
\overset{\eqref{ser1}}{\underset{\eqref{ser2}}{=}}
&\Big[E_{1,3;i,k}(u)-\sum_{q=1}^{\mu_2}E_{1;i,q}(u)E_{2;q,k}(u),E_{2;l,g}^{(1)}\Big]+[E_{1,3;i,k}^{(1)},E_{2;l,g}(u)]\\
\overset{\eqref{ser3}}{\underset{\eqref{ser4}}{=}}&-\delta_{kg}E_{1;i,l}(u)-\sum_{q=1}^{\mu_2}[E_{1;i,q}(u),E_{2;l,g}^{(1)}]E_{2;q,k}(u)-\sum_{q=1}^{\mu_2}E_{1;i,q}(u)[E_{2;q,k}(u),E_{2;l,g}^{(1)}]\\
&+[E_{1;i,p}^{(1)},E_{2;p,g}(u)]E_{2;l,k}(u)+\sum_{q=1}^{\mu_2}E_{1;i,q}(-u-\mu_{(2)})\wtl D_{2;q,l}(-u-\mu_{(2)}) D_{3;g,k}(-u-\mu_{(2)})\\
\overset{\eqref{ser1}}{\underset{\eqref{ser5}}{=}}&-\delta_{kg}E_{1;i,l}(u)-E_{13;i,g}(u)E_{2;l,k}(u)\\
&-\sum_{q=1}^{\mu_2}E_{1;i,q}(u)\Big(\wtl D_{2;q,l}(u)D_{3;g,k}(u)-\delta_{kg}\delta_{ql}-E_{2;q,g}(u)E_{2;l,k}(u)\Big)\\
&+E_{13;i,g}(u)E_{2;l,k}(u)-\sum_{q=1}^{\mu_2}E_{1;i,q}(u)E_{2;q,g}(u)E_{2;l,k}(u)\\
&+\sum_{q=1}^{\mu_2}E_{1;i,q}(-u-\mu_{(2)})\wtl D_{2;q,l}(-u-\mu_{(2)}) D_{3;g,k}(-u-\mu_{(2)})\\
=&-\sum_{q=1}^{\mu_2}E_{1;i,q}(u)\wtl D_{2;q,l}(u) D_{3;g,k}(u)\\
&\qquad +\sum_{q=1}^{\mu_2}E_{1;i,q}(-u-\mu_{(2)})\wtl D_{2;q,l}(-u-\mu_{(2)}) D_{3;g,k}(-u-\mu_{(2)}),
\end{align*}
completing the proof.
\end{proof}
% Set 
% \begin{align}
% B_{a;i,j}(u)&=\sum_{r\gge 1}B_{a;i,j}^{(r)}u^{-r}:=F_{a;j,i}(u-\mu_{(a)}),\\
% H_{a;i,j}(u)&=\delta_{ij}+\sum_{r\gge 1}H_{a;i,j}^{(r)}u^{-r}:=D_{a;i,j}(u-\mu_{(a)}),\\
% \wtl H_{a;i,j}(u)&=\delta_{ij}+\sum_{r\gge 1}\wtl H_{a;i,j}^{(r)}u^{-r}:=\wtl D_{a;i,j}(u-\mu_{(a)}).
% \end{align}
Recall that by Lemma \ref{e=f gen}, we have
\beq
B_{a;j,i}(u)=E_{a;i,j}(-u-\tfrac{\mu_{(a)}}{2}).
\eeq
\begin{lem}\label{+deg}
We have
\begin{align*}
\Big[\wtl D_{1;i,i}^{(2)},E_{1;i,j}(u)\Big]&=-(2u+\mu_1)E_{1;i,j}(u)+2E_{1;i,j}^{(1)},\\
\Big[D_{2;j,j}^{(2)},E_{1;i,j}(u)\Big]&=-(2u+\mu_1)E_{1;i,j}(u)+2E_{1;i,j}^{(1)}.
\end{align*}
In particular,
\begin{align*}
\Big[\wtl D_{1;i,i}^{(2)},B_{1;j,i}^{(r)}\Big] =2B_{1;j,i}^{(r+1)},\qquad 
\Big[D_{2;j,j}^{(2)},B_{1;j,i}^{(r)}\Big] =2B_{1;j,i}^{(r+1)}.
\end{align*}
\end{lem}
\begin{proof}
We only prove the first one as the second one is similar.

It follows from \eqref{d1e n2} that
\beq\label{n5pf1}
\begin{split}
&(u-v)(u+v+\mu_1)[\wtl D_{1;i,j}(u),E_{1;k,l}(v)]\\&=(u+v+\mu_1)(E_{1;i,l}(u)-E_{1;i,l}(v))\wtl D_{1;k,j}(u)+(u-v)\wtl D_{1;i,k}(u)(E_{1;j,l}(v)-F_{1;l,j}(u)).    
\end{split}
\eeq
Taking the coefficients of $u$ in \eqref{n5pf1}, we find that
\beq\label{n5pf2}
[\wtl D_{1;i,j}^{(1)},E_{1;k,l}(v)]=-\delta_{jk}E_{1;i,l}(v)+\delta_{ik}E_{1;j,l}(v).
\eeq
In particular, we have 
\begin{align}
[\wtl D_{1;i,i}^{(1)},E_{1;k,l}(v)]&=0.\label{n5pf3}
% \\
% [\wtl D_{1;i,k}^{(1)},E_{1;k,l}(v)]&=-[\wtl D_{1;k,i}^{(1)},E_{1;k,l}(v)]=-E_{1;i,l}(v),\qquad \text{if }i\ne k.\label{n5pf4}
\end{align}
Taking the coefficients of $u^0$ and applying \eqref{n5pf3}, we obtain
\beq\label{n5pf5}
[\wtl D_{1;i,i}^{(2)},E_{1;k,l}(v)]=-(2v+\mu_1)\delta_{ik}E_{1;i,l}(v)+2\delta_{ik}E_{1;i,l}^{(1)}+\wtl D_{1;i,k}^{(1)}E_{1;i,l}(v)-E_{1;i,l}(v)\wtl D_{1;k,i}^{(1)}.
\eeq
Thus the first identity follows from \eqref{n5pf5} and \eqref{n5pf3}.
\end{proof}

\subsection{Serre relations}\label{ssec:serre}
In this subsection, we shall deduce the Serre relations involving
\[
\big[[E_{1;i,j}^{(s)},E_{2;j,k}^{(r_1)}],E_{2;l,g}^{(r_2)}\big]+\big[[E_{1;i,j}^{(s)},E_{2;j,k}^{(r_2)}],E_{2;l,g}^{(r_1)}\big].
\]

\subsubsection{The case $\mu_2>1$} 
Using \eqref{ee n2}, we find that
\beq\label{ss}
\begin{split}
&[E_{2;j,k}(v),E_{2;l,g}(w)]+[E_{2;j,k}(w),E_{2;l,g}(v)]\\
=&\frac{1}{v+w+\mu_{(2)}}\Big(\wtl D_{2;j,l}(v)D_{3;g,k}(v)-\wtl D_{2;j,l}(-w-\mu_{(2)})D_{3;g,k}(-w-\mu_{(2)})\\
&\hskip3cm +\wtl D_{2;j,l}(w)D_{3;g,k}(w)-\wtl D_{2;j,l}(-v-\mu_{(2)})D_{3;g,k}(-v-\mu_{(2)})\Big).
\end{split}
\eeq

In this case, it follows from \eqref{e1e2 n3} that $[E_{1;i,j}(u),E_{2;j,k}(v)]$ is independent of $j$ and hence we can assume further that $j\ne l$. By the Jacobi identity and \eqref{e13e2 n3}, we have
\begin{eqnarray*}
& & \big[[E_{1;i,j}(u),E_{2;j,k}(v)],E_{2;l,g}(w)\big]+\big[[E_{1;i,j}(u),E_{2;j,k}(w)],E_{2;l,g}(v)\big]\\
&=& \big[E_{1;i,j}(u),[E_{2;j,k}(v),E_{2;l,g}(w)]\big]+\big[E_{1;i,j}(u),[E_{2;j,k}(w),E_{2;l,g}(v)]\big]\\
& =&\big[E_{1;i,j}(u),[E_{2;j,k}(v),E_{2;l,g}(w)]+[E_{2;j,k}(w),E_{2;l,g}(v)]\big]\\
& =&\frac{1}{v+w+\mu_{(2)}}\Big[E_{1;i,j}(u),\wtl D_{2;j,l}(v)D_{3;g,k}(v)-\wtl D_{2;j,l}(-w-\mu_{(2)})D_{3;g,k}(-w-\mu_{(2)})\\
& &\hskip3.5cm +\wtl D_{2;j,l}(w)D_{3;g,k}(w)-\wtl D_{2;j,l}(-v-\mu_{(2)})D_{3;g,k}(-v-\mu_{(2)})\Big].
\end{eqnarray*}

Therefore, we have
\begin{align*}
\big[[B_{1;j,i}(u),&B_{2;k,j}(v)],B_{2;g,l}(w)\big]+\big[[B_{1;j,i}(u),B_{2;k,j}(w)],B_{2;g,l}(v)\big]\\
=&\frac{1}{v+w}\big[B_{1;j,i}(u),Z_{2;j,l,g,k}(v)+Z_{2;j,l,g,k}(w)-Z_{2;j,l,g,k}(-v)-Z_{2;j,l,g,k}(-w)\big].
\end{align*}
This gives rise to the following component-wise formula,
\beq\label{ss>1}
\big[[B_{1;j,i}^{(q)},B_{2;k,j}^{(r_1)}],B_{2;g,l}^{(r_2)}\big]+\big[[B_{1;j,i}^{(q)},B_{2;k,j}^{(r_2)}],B_{2;g,l}^{(r_1)}\big]=((-1)^{r_1}+(-1)^{r_2})[Z_{2;j,l,g,k}^{(r_1+r_2-1)},B_{1;j,i}^{(q)}].
\eeq

One can obtain a different version of Serre relations by using \eqref{n2pf5} instead of \eqref{ee n2},
\beq\label{serrepr}
\begin{split}
\big[[B_{1;j,i}(u),&B_{2;k,j}(v)],B_{2;g,l}(w)\big]+\big[[B_{1;j,i}(u),B_{2;k,j}(w)],B_{2;g,l}(v)\big]\\
=&\frac{1}{v+w}\big[B_{1;j,i}(u),Z_{2;l,j,k,g}(-v)+Z_{2;l,j,k,g}(-w)-Z_{2;l,j,k,g}(v)-Z_{2;l,j,k,g}(w)\big].
\end{split}
\eeq
Applying the automorphism $\zeta_N$ and using Proposition \ref{zeta-pro}, one obtains the other Serre relation
\[
\big[[B_{2;i,j}^{(q)},B_{1;k,l}^{(r_1)}],B_{1;f,g}^{(r_2)}\big]+\big[[B_{2;i,j}^{(q)},B_{1;k,l}^{(r_2)}],B_{1;f,g}^{(r_1)}\big]=\delta_{jk}((-1)^{r_1}+(-1)^{r_2})[Z_{1;l,g,f,k}^{(r_1+r_2-1)},B_{2;i,j}^{(q)}].
\]

\subsubsection{The case $\mu_2=1$}
In this case, \eqref{ser6} simplifies as
\beq\label{BBB}
\begin{split}
\big[[B_{1;j,i}^{(1)}&,B_{2;k,j}(u)],B_{2;g,j}^{(1)}\big]+\big[[B_{1;j,i}^{(1)},B_{2;k,j}^{(1)}],B_{2;g,j}(u)\big]\\
&=B_{1;j,i}(-u+\tfrac{1}{2})Z_{2;j,j,g,k}(u)-B_{1;j,i}(u+\tfrac{1}{2})Z_{2;j,j,g,k}(-u)
\end{split}
\eeq
Fix $i,j,g,k$. We shall use the following shorthand notation:
\beq\label{SymSerr}
\mathbb S(q;r_1,r_2):=\big[[B_{1;j,i}^{(q)},B_{2;k,j}^{(r_1)}],B_{2;g,j}^{(r_2)}\big]+\big[[B_{1;j,i}^{(q)},B_{2;k,j}^{(r_2)}],B_{2;g,j}^{(r_1)}\big],
\eeq
\beq\label{Serre-rhs}
\Omega(q;r)=\sum_{p_1=1}^{r-1}\sum_{p_2=0}^{r-1-p_1}(-1)^{p_1-1} 2^{-p_2}{p_1+p_2-1\choose p_2} B_{1;j,i}^{(p_1+q-1)}Z_{2;j,j,g,k}^{(r-1-p_1-p_2)}.
\eeq
Since $\mathbb S(q;r_1,r_2)$ is symmetric in $r_1$ and $r_2$, we shall assume further that $r_1\gge r_2$. 

\begin{lem}\label{serre-lem1}
Under the above assumption, we have 
\beq\label{serr-form}
\mathbb S(q;r_1,r_2)=\big((-1)^{r_1}+(-1)^{r_2}\big)\Omega(q;r_1+r_2).
\eeq
\end{lem}
\begin{proof}
Our proof follows the strategy of \cite[\S4.3]{LWZ23}.

By \eqref{BBB} we know that \eqref{serr-form} holds for the case $q=r_2=1$. Now we recursively apply the map $X \mapsto [\wtl D_{1;i,i}^{(2)}, X]$ to the equality $\mathbb S(1;r_1,1)=\big((-1)^{r_1}-1\big)\Omega(1;r_1+1)$. 
%Take the commutator recursively with $\wtl D_{1;i,i}^{(2)}$ to the equality $\mathbb S(1;r_1,1)=\big((-1)^{r_1}-1\big)\Omega(1;r_1+1)$.
Note that by Corollary \ref{comcor}, we have 
$$
[\wtl D_{1;i,i}^{(2)},B_{2;l,j}(u)]=[\wtl D_{1;i,i}^{(2)},Z_{2;j,j,g,k}(u)]=0. 
$$
Thus it follows from Lemma \ref{+deg} that the equality \eqref{serr-form} also holds for the case $r_2=1$ (namely $q$ can be arbitrary). 

Note that $[s_{11}(u),s_{11}(v)]=0$ by \eqref{qua}, and hence $[D_{2;j,j}^{(2)},D_{2,j,j}(u)]=0$ by Lemma \ref{psi lem}. Further, we have $[D_{2;j,j}^{(2)},Z_{2;j,j,g,k}(u)]=0$.

Finally, we prove \eqref{serr-form} by induction on $r_1+r_2$. The base case has been established above. Suppose we have $\mathbb S(q;r_1,r_2)=\big((-1)^{r_1}+(-1)^{r_2}\big)\Omega(q;r_1+r_2)$. We apply the commutator by $D_{2;j,j}^{(2)}$ to obtain
\[
2\mathbb S(q+1;r_1,r_2)-2\mathbb S(q;r_1+1,r_2)-2\mathbb S(q;r_1,r_2+1)=2\big((-1)^{r_1}+(-1)^{r_2}\big)\Omega(q+1;r_1+r_2).
\]
By induction hypothesis, we conclude that $\mathbb S(q;r_1,r_2+1)=- \mathbb S(q;r_1+1,r_2)$. Using this equality multiple times, we obtain
\begin{align*}
\mathbb S(q;r_1,r_2+1)&=(-1)^{r_2}\mathbb S(q;r_1+r_2,1)\\
&=(-1)^{r_2}((-1)^{r_1+r_2}-1)\Omega(q;r_1+r_2+1)\\
&=((-1)^{r_1}+(-1)^{r_2+1})\Omega(q;r_1+r_2+1).
\end{align*}
This completes the proof.
\end{proof}

Applying the automorphism $\zeta_N$ and using Proposition~ \ref{zeta-pro}, one obtains the other Serre relation:
\begin{align*}
&\big[[B_{2;i,j}^{(q)},B_{1;j,k}^{(r_1)}],B_{1;j,g}^{(r_2)}\big]+\big[[B_{2;i,j}^{(q)},B_{1;j,k}^{(r_2)}],B_{1;j,g}^{(r_1)}\big]\\
=& 
((-1)^{r_1}+(-1)^{r_2})\sum_{p_1=1}^{r_1+r_2-1}\sum_{p_2=0}^{r_1+r_2-1-p_1}(-1)^{p_2} 2^{-p_2}{p_1+p_2-1\choose p_2} B_{2;i,j}^{(p_1+q-1)}Z_{1;g,k,j,j}^{(r_1+r_2-1-p_1-p_2)}. 
\end{align*}

\section{Parabolic presentations for type AII: the proof} \label{sec:pfAII}

The proof of Theorem \ref{mainthm} for AII type is very similar to AI type, but in some places different treatments are required. In this section we provide the proof by explaining the differences in detail.
We assume $\mu=(\mu_1,\ldots,\mu_n)$ is an arbitrary even composition of $N$ and follow the convention \eqref{thetadef} throughout this section.

\subsection{PBW bases for twisted Yangians of type AII} 
The assumption that $\mu$ is even ensures that the Gauss decomposition \eqref{S=FDE} works equally well with respect to $\mu$. 
As a consequence, we obtain the associated parabolic generators $D_{a;i,j}^{(r)}$, $E_{a;i,j}^{(r)}$, $E_{a,b;i,j}^{(r)}$, $F_{a;j,i}^{(r)}$, $F_{b,a;j,i}^{(r)}$, $H_{a;i,j}^{(r)}$, $B_{a,b;i,j}^{(r)}$ in $\Y_N^-$ by following the same routine.
The next result is the analogue of Theorem \ref{PBWgauss} for type AII, which can be deduced from Proposition \ref{prop:PBW}(2) by exactly the same way.

\begin{prop}\label{PBWgauss-new}
The monomials in the elements
\begin{align*}
\big\{H_{a;2i-1,2i-1}^{(r)},\ H_{a;2i-1,2i}^{(2r-1)}, &\ H_{a;2i,2i-1}^{(2r-1)}\big\}_{1\lle a\lle  n, 1\lle i\lle \tfrac{\mu_a}{2},r\gge 1},
\\
\big\{H_{a;i,j}^{(r)}\big\}_{1\lle a\lle  n, 1\lle j<i\lle \mu_a,\lfloor \tfrac{i+1}{2}\rfloor\ne \lfloor \tfrac{j+1}{2}\rfloor,r\gge 1},
&\quad
\big\{B_{b,a;i,j}^{(r)} \big\}_{1\lle a<b\lle n, 1\lle i\lle \mu_b,1\lle j\lle \mu_a,r\gge 1},
\end{align*}
taken in any fixed linear order form a PBW basis of $\Y_N^-$.
\end{prop}

\subsection{Relations in 2-block case}
The relations in the cases of 2-block and 3-block are essential for the proof of Theorem \ref{mainthm}.
It turns out that the relations in the 3-block case are exactly the same as the AI type, and hence we list only the relations in the 2-block case explicitly. 
\begin{prop}\label{prop2blocka2}
Let $\mu=(\mu_1,\mu_2)$ be an even strict composition of $N$. Then the following relations hold in $\Y^-_\mu$:
\begin{align}
[H_{a;i,j}(u),H_{b;k,l}(v)]=&\,\delta_{ab}\Big(\frac{1}{u-v}\big(H_{a;k,j}(u)H_{a;i,l}(v)-H_{a;k,j}(v)H_{a;i,l}(u)\big)\notag\\&-\frac{1}{u+v}\big(\theta_k\theta_{j'}H_{a;i,k'}(u)H_{a;j',l}(v)-\theta_i\theta_{l'}H_{a;k,i'}(v)H_{a;l',j}(u)\big)\label{a2hahb n2}\\&+\frac{\theta_i\theta_{j'}}{u^2-v^2}\big(H_{a;k,i'}(u)H_{a;j',l}(v)-H_{a;k,i'}(v)H_{a;j',l}(u)\big)\Big),\notag
\\
[H_{1;i,j}(u),B_{1;k,l}(v)]=&\,\frac{\delta_{il}}{u-v+\tfrac{\mu_1}{2}}\sum_{p=1}^{\mu_1}\big(B_{1;k,p}(u+\tfrac{\mu_1}{2})-B_{1;k,p}(v)\big)H_{1;p,j}(u)\label{a2h1b n2}\\&\, +\frac{\delta_{jl'}\theta_{l'}}{u+v+\tfrac{\mu_1}{2}}\sum_{p=1}^{\mu_1}\theta_p H_{1;i,p}(u)\big(B_{1;k,p'}(v)-B_{1;k,p'}(-u-\tfrac{\mu_1}{2})\big),\notag
\\
[H_{2;i,j}(u),B_{1;k,l}(v)]=&\,\frac{\theta_{j'}\theta_k}{u+v}H_{2;i,k'}(u)\big(B_{1;j',l}(-u)-B_{1;j',l}(v)\big)\label{a2h2b n2}\\&\qquad +\frac{1}{u-v} \big(B_{1;i,l}(v)-B_{1;i,l}(u)\big)H_{2;k,j}(u),\notag
\\
[B_{1;i,j}(u),B_{1;k,l}(v)]=&\,\frac{1}{u-v}\big(B_{1;k,j}(v)-B_{1;k,j}(u)\big)\big(B_{1;i,l}(u)-B_{1;i,l}(v)\big)\label{a2bb n2}\\
&+\frac{\theta_i\theta_j}{u+v}\big(Z_{1;j',l,k,i'}(v)-Z_{1;j',l,k,i'}(-u)\big),\notag
\\
\label{a2b=c n2}
B_{1;i,j}(-u)=&\, \theta_i\theta_jC_{1;j',i'}(u),\hskip 7cm
\end{align}
where $Z_{1;j,l,k,i}(u)=\wtl H_{1;j,l}(u-\tfrac{\mu_1}{2})H_{2;k,i}(u)$.
\end{prop}

\begin{proof}
The proof is similar to that of Proposition \ref{prop2block}.
\end{proof}

\subsection{Some symmetry relations}

The following derived relations will be used in Definition \ref{def:shiftedII} of shifted twisted Yangians of type AII.
\begin{lem}
\label{lem:Za2}
The following $Z$-relations hold in $\Y_\mu^\pm$, for $1\lle a\lle n-1$: 
\beq\label{Za2}
\theta_{i}\theta_jZ_{a;j',l,k,i'}^{(2r-1)}+\theta_k\theta_l Z_{a;l',j,i,k'}^{(2r-1)}=0.
\eeq
\end{lem}
\begin{proof}
We carry out the proof in the case $a=1$, but the argument works well for any $a$.
To begin with, one can deduce the equivalent form of (\ref{a2bb n2}):
\beq\label{a2nbb n2}
\begin{split}
[B_{1;i,j}(u),B_{1;k,l}(v)]=\frac{1}{u-v}\big(B_{1;i,l}(u)-B_{1;i,l}(v)\big)\big(B_{1;k,j}(v)-B_{1;k,j}(u)\big)\\
+\frac{\theta_k\theta_l}{u+v}\big(Z_{1;l',j,i,k'}(-v)-Z_{1;l',j,i,k'}(u)\big).
\end{split}
\eeq
Setting $u=v$ in \eqref{a2bb n2}, we have
\[
[B_{1;i,j}(u),B_{1;k,l}(u)]=\frac{\theta_i\theta_j}{2u}(Z_{1;j',l,k,i'}(u)-Z_{1;j',l,k,i'}(-u)).
\]
Similarly, setting $u=v$ in \eqref{a2nbb n2}, we have
\[
[B_{1;i,j}(u),B_{1;k,l}(u)]=\frac{\theta_k\theta_l}{2u}(Z_{1;l',j,i,k'}(-u)-Z_{1;l',j,i,k'}(u)).
\]
Thus the desired relation \eqref{Za2} follows by taking the coefficients of $u^{-2r}$ on the RHS. 
\end{proof}
We point out the following analogue of Proposition \ref{prop:B=C} for AII type, whose proof is similar to Lemma~\ref{e=f gen} by using \eqref{a2b=c n2}.

\begin{prop}\label{a2prop:B=C}
For $1\lle a<b\lle n$, $1\lle j\lle \mu_a$, $1\lle i\lle \mu_b$, we have
\[
B_{b,a;i,j}(u)=\theta_i\theta_jC_{a,b;j',i'}(-u).
\]
\end{prop}

\subsection{Proof of Theorem \ref{mainthm} for type AII}\label{sec:pf-AII-parabolic}
The proof of this result is similar to type AI, and hence we will only indicate the steps of the argument which need to be treated differently.
Let $\bY_\mu^-$ denote the algebra generated by the generators and relations as in the statement of the theorem. 
By induction on $n$, one checks that the parabolic generators satisfy the listed relations, where the initial step is reduced to the case when $n=2,3$.
Then we have a surjective homomorphism $\bY_\mu^-\twoheadrightarrow \Y_N^-$ and it suffices to show this map is injective.

For $1\lle b<a\lle n$, $1\lle i\lle \mu_a$, and $1\lle j\lle \mu_b$, define the elements $B_{a,b;i,j}^{(r)}$ of $\bY_\mu^-$ by exactly the same rule \eqref{pfN0}, which do not depend on the choice of $k$ as well.
Let $\gr\bY_\mu^-$ be the associated graded algebra with respect to the loop filtration, which is defined similarly by setting $\deg H_{a;i,j}^{(r)}=\deg B_{a,b;k,l}^{(r)}=r-1$.
Let $\sfB_{a,a;i,j}^{(r)}$ and $\sfB_{a,b;k,l}^{(r)}$ be the images of $H_{a;i,j}^{(r+1)}$ and $B_{a,b;,k,l}^{(r+1)}$, respectively, in the $r$-th component of the graded algebra $\gr\bY_\mu^-$. 
In addition, for $1\lle b<a\lle n$, $1\lle i\lle \mu_a$, and $1\lle j\lle \mu_b$, define $\sfB_{b,a;j,i}^{(r)}:=(-1)^{r+1}\theta_i\theta_j\sfB_{a,b;i^\prime,j^\prime}^{(r)}$. 
A similar argument shows that for $b-a>1$ and $r,s\in\bN$, we have
\be
\sfB_{b,a;i,j}^{(r+s)}=[\sfB_{b,b-1;i,k}^{(r)},\sfB_{b-1,a;k,j}^{(s)}]
\ee
with any $1\lle k\lle \mu_{a-1}$. 

We claim the following identity holds for all $1\lle a,b\lle n$, $1\lle i \lle \mu_a$, $1\lle j \lle \mu_b$, $r\gge 0$:
\beq\label{new-pfn0}
\sfB_{a,b;i,j}^{(r)}=(-1)^{r+1}\theta_i\theta_j\sfB_{b,a;j^\prime,i^\prime}^{(r)}.
\eeq
If $a\neq b$, it follows from the definition. Assume from now on $a=b$. 
By \eqref{pr3}, we have
\begin{align}
[\sfB_{a,a;i,j}^{(0)}, \sfB_{a,a;k,l}^{(r)}] &= \delta_{kj}\sfB_{a,a;i,l}^{(r)} - \delta_{il} \sfB_{a,a;k,j}^{(r)}
-\big(\theta_k\theta_{j^\prime}\delta_{ik^\prime}\sfB_{a,a;j^\prime,l}^{(r)} - \theta_i\theta_{l^\prime}\delta_{l^\prime j}\sfB_{a,a;k,i^\prime}^{(r)}     \big), \label{new-pfn1} \\
[\sfB_{a,a;k,l}^{(r)}, \sfB_{a,a;i,j}^{(0)}] &= \delta_{il}\sfB_{a,a;k,j}^{(r)} - \delta_{kj} \sfB_{a,a;i,l}^{(r)}
-(-1)^{r}\big(\theta_i\theta_{l^\prime}\delta_{k^\prime i}\sfB_{a,a;l^\prime,j}^{(r)} - \theta_k\theta_{j^\prime}\delta_{j^\prime l}\sfB_{a,a;i,k^\prime}^{(r)}     \big).  \label{new-pfn2}
\end{align}
Equations \eqref{new-pfn1} and \eqref{new-pfn2} differ by $-1$ and hence they sum to zero, thus 
\beq\label{new-pfn3}
\delta_{ik^\prime}\big( \theta_k\theta_{j^\prime}\sfB_{a,a;j^\prime,l}^{(r)} + (-1)^r \theta_i\theta_{l^\prime}\sfB_{a,a;l^\prime,j}^{(r)} \big) 
- \delta_{jl^\prime}\big( \theta_i\theta_{l^\prime} \sfB_{a,a;k,i^\prime}^{(r)} + (-1)^r \theta_k\theta_{j^\prime} \sfB_{a,a;i,k^\prime}^{(r)} \big)=0.
\eeq
Setting $k=i^\prime$ and $l=j$ in \eqref{new-pfn3}, we have 
\be 
\theta_{i^\prime}\theta_{j^\prime}\sfB_{a,a;j^\prime,j}^{(r)} + (-1)^r \theta_i\theta_{j^\prime}\sfB_{a,a;j^\prime,j}^{(r)}=
\theta_i\theta_j\big(1+(-1)^{r+1}\big) \sfB_{a,a;j^\prime,j}^{(r)}
=0.
\ee 
Therefore $\sfB_{a,a;j^\prime,j}^{(r)}=0$ for all odd $r$. Note that \eqref{new-pfn0} holds trivially for even $r$ if $i=j^\prime$.
Setting $k=i^\prime$ and $l=j^\prime$ in \eqref{new-pfn3}, we have 
\beq\label{new-pfn5}
\big( \theta_{i^\prime}\theta_{j^\prime}\sfB_{a,a;j^\prime,j^\prime}^{(r)} + (-1)^r \theta_i\theta_{j}\sfB_{a,a;j,j}^{(r)} \big) 
-\big( \theta_i\theta_j \sfB_{a,a;i^\prime,i^\prime}^{(r)} + (-1)^r \theta_{i^\prime}\theta_{j^\prime} \sfB_{a,a;i,i}^{(r)} \big)=0
\eeq
Setting $i=j^\prime$ in \eqref{new-pfn5}, we have $\sfB_{a,a;i,i}^{(r)}=\sfB_{a,a;i^\prime,i^\prime}^{(r)}$ for all odd $r$. Next we show by induction on $a$ that $\sfB_{a,a;i,i}^{(r)}=-\sfB_{a,a;i^\prime,i^\prime}^{(r)}$ for all even $r$. The initial case $a=1$ follows from \eqref{pr1}, \eqref{pr3} and Remark~\ref{varsym}.
For $a\gge1$, setting $l=j'$ and $k=i'$ in \eqref{Za2}, we obtain $Z_{a;j',j',i,i}^{(2r+1)}= - Z_{a;j,j,i',i'}^{(2r+1)}$. By \eqref{Zdef}, their image in $\gr\bY_\mu^-$ gives the following equality
\[
-\sfB_{a,a;j',j'}^{(2r)} + \sfB_{a+1,a+1;i,i}^{(2r)} =  \sfB_{a,a;j,j}^{(2r)} - \sfB_{a+1,a+1;i',i'}^{(2r)}.
\]
Now $\sfB_{a,a;j,j}^{(2r)}=-\sfB_{a,a;j',j'}^{(2r)}$ due to the inductive hypothesis. This completes the induction.

Suppose now $i\neq j$. 
Setting $l=i^\prime=k$ in \eqref{new-pfn3}, we have
\be
\theta_{i^\prime}\theta_{j^\prime}\sfB_{a,a;j^\prime,i^\prime}^{(r)} + (-1)^r \theta_i\theta_{i}\sfB_{a,a;i,j}^{(r)} =0,
\ee
which completes the proof of \eqref{new-pfn0}.

We claim the following analogue of \eqref{todo} holds for AII type:
\beq\label{new-todo}
\begin{split}
[\sfB_{a,b;i,j}^{(r)},\sfB_{c,d;k,l}^{(s)}]=
\delta_{bc}\delta_{jk}\sfB_{a,d;i,l}^{(r+s)}-\delta_{ad}\delta_{il}\sfB_{c,b;k,j}^{(r+s)}
-(-1)^r\theta_i\theta_j \big( 
\delta_{ac}\delta_{i^\prime k}\sfB_{b,d;j^\prime,l}^{(r+s)}
-\delta_{bd}\delta_{jl^\prime}\sfB_{c,a;k,i^\prime}^{(r+s)} \big).
\end{split}
\eeq
By \eqref{new-pfn0} one may assume $a\gge b$ and $c\gge d$ in \eqref{new-todo}, which then can be established by case-by-case discussions as performed in AI and hence we omit the detail.
The injectivity of $\bY_\mu^-\twoheadrightarrow \Y_N^-$ follows from \eqref{new-todo} and Proposition \ref{PBWgauss-new}.

%%%%%%%\part{Shifted twisted Yangians}
\part{Truncated shifted twisted Yangians and finite $W$-algebras}

\section{Shifted twisted Yangians and parabolic presentations}
\label{sec:shiftedI}

In this section we define a shifted twisted Yangian $\Y_N^+(\sigma)$ of type AI, associated to every symmetric matrix with non-negative integral coefficients, via a parabolic presentation for each composition $\mu$ of $N$ admissible for $\sigma$. We  show that it can be naturally viewed as a subalgebra of $\Y_N^+$. We also show that the presentation of the shifted subalgebra is independent of $\mu$. Along the way we provide PBW bases for $\Y_N^+(\sigma)$ for each shape $\mu$ which is admissible for $\sigma$.

Parallel results for type AII are obtained under the additional assumption that every part of $\mu$ is even (from henceforth we say that such a $\mu$ is even).
We also highlight a different presentation for the shifted twisted Yangian of type AI using the Drinfeld presentation from \cite{LWZ23} and show that the algebras given by the parabolic presentations and by the Drinfeld presentations are isomorphic. 

\subsection{Shifted twisted Yangians of type AI}\label{paraAI}

Fix $1 \lle n \lle N$ and a symmetric matrix $\sigma=(\fks_{i,j})_{1\lle i,j\lle N}$ of non-negative integers (shifts) such that 
\beq\label{shift}
\fks_{i,j}+\fks_{j,k}=\fks_{i,k}
\eeq
provided $|i-j|+|j-k|=|i-k|$. Note that $\fks_{1,1}=\cdots=\fks_{n,n}=0$ and $\sigma$ is completely determined by $\fks_{a+1,a}$, $1\lle a<n$. We call $\sigma$ a \textit{shift matrix}.

Fix a shift matrix $\sigma=(\fks_{i,j})_{1\lle i,j\lle N}$. We say that a composition $\mu$ of $N$ of length $n$ is \textit{admissible} for $\sigma$ if $\fks_{i,j}=0$ for all $\mu_{(a-1)}<i,j\lle \mu_{(a)}$ for $1\lle a\lle n$. We also set
\beq\label{sab}
\fks_{a,b}(\mu):=\fks_{\mu_{(a)},\mu_{(b)}}.
\eeq
Note that the original shift matrix $\sigma$ can be recovered from $(\fks_{a,b}(\mu))_{1\lle a,b\lle n}$ if $\mu$ is admissible for $\sigma$. Below, we fix an admissible $\mu$ for $\sigma$.
\begin{dfn}
\label{defn:shiftedtwisted}
The (\textit{dominantly}) \textit{shifted twisted Yangian} associated to the matrix $\sigma$ and the composition $\mu$
is the algebra $\Y_\mu^+(\sigma)$ generated by
\[
\{H_{a;i,j}^{(r)},\wtl H_{a;i,j}^{(r)}\}_{1\lle a\lle n,1\lle i,j\lle \mu_a,r\gge 1},\quad \{B_{a;i,j}^{(r)}\}_{1\lle a< n,1\lle i\lle \mu_{a+1}, 1\lle j\lle \mu_a,r>\fks_{a+1,a}(\mu)}
\]
subject to the relations \eqref{pr1}--\eqref{pr-2} 
with all admissible indices together with the following relations (cf. Remark \ref{rem-Z})
\beq\label{Zshifted}
Z_{a;i,i,j,j}^{(2r-1)}=0,\quad 1\lle a<n,~1\lle i\lle \mu_a,~1\lle j\lle \mu_{a+1},~1\lle r\lle \fks_{a+1,a}(\mu).
\eeq
\end{dfn}
By Theorem \ref{mainthm}, if $\sigma$ is the zero matrix, then $\Y_\mu^+(\sigma)$ is exactly $\Y_N^+$. Clearly, we have the canonical homomorphism $\Y_\mu^+(\sigma)\to \Y_N^+=\Y^+_\mu$ sending the generators $H_{a;i,j}^{(r)}$ and $B_{a;i,j}^{(r)}$ of $\Y_\mu^+(\sigma)$ to elements of $\Y_N^+$ with same names. 

In the rest of this subsection, we shall prove that this homomorphism is injective and  its image is independent of the particular choice of the admissible shape $\mu$.

For $1\lle a<b\lle n$, $1\lle i\lle \mu_b$, $1\lle j\lle \mu_a$ and $r>\fks_{b,a}(\mu)$, define elements $B_{b,a;i,j}^{(r)}$ inductively by
\beq\label{x2}
B_{a+1,a;i,j}^{(r)}:=B_{a;i,j}^{(r)},\quad B_{b,a;i,j}^{(r)}:=[B_{b-1;i,k}^{(\fks_{b,b-1}(\mu)+1)},B_{b-1,a;k,j}^{(r-\fks_{b,b-1}(\mu))}],
\eeq
cf. Lemma \ref{genlem}. Again, the definition is  independent of the choice of $k$.
Recall the twisted current algebra $\gl_N[z]^\theta$ which has a basis 
$$
\big\{e_{ij}\otimes z^r-(-1)^re_{ji}\otimes z^r\big\}_{1\lle j\lle i\lle N, r\in \bN}\setminus\{0\}.
$$
Under the assumption \eqref{shift}, the vectors
$$
\big\{e_{ij}\otimes z^r-(-1)^re_{ji}\otimes z^r\big\}_{1\lle j<i\lle N, r\gge\fks_{i,j}}\setminus\{0\}
$$
span a subalgebra of $\gl_N[z]^\theta$, and call it the {\it shifted twisted current algebra}. We denote this subalgebra by $\gl_N[z]^\theta(\sigma)$. The universal enveloping algebra $\mathrm{U}(\gl_N[z]^\theta(\sigma))$ has a natural grading induced from the grading on $\gl_N[z]$, which places $xt^r$ in degree $r$.
\begin{thm}\label{pbwstw}
The monomials in 
\[
\{H_{a;i,i}^{(2r)}\}_{1\lle a\lle n,1\lle i\lle \mu_a,r>0},\quad \{H_{a;i,j}^{(r)}\}_{1\lle a\lle n,1\lle j<i\lle \mu_a,r>0},\quad 
\{B_{b,a;i,j}^{(r)}\}_{1\lle a< b< n,1\lle i\lle \mu_{b},1\lle j\lle \mu_a,r>\fks_{b,a}(\mu)},
\]
taken in some (any) fixed linear order form a PBW basis of $\Y_\mu^+(\sigma)$. In particular, the canonical map $\Y_\mu^+(\sigma)\to \Y_N^+$ is injective.
\end{thm}
\begin{proof}
Define a filtration $\mathcal F_0\Y_\mu^+(\sigma)\subset \mathcal F_1\Y_\mu^+(\sigma)\subset \cdots$ of $\Y_\mu^+(\sigma)$ by setting 
\beq\label{loopfiltw}
\deg H_{a;i,i}^{(2r)}=2r-1,\quad \deg H_{a;i,j}^{(r)}=r-1,\quad \deg B_{a;i,j}^{(r)}=r-1.
\eeq
Note that the elements $H_{a;i,i}^{(2r-1)}$ can be expressed in terms $H_{b;j,j}^{(2s)}$ by the relations \eqref{pr1}--\eqref{pr2} and \eqref{Z's}. Define elements $\{\sfe_{i,j;r}\}_{1\lle j\lle i\lle N,r\gge \fks_{i,j}}$ of the associated graded algebra $\gr\Y_\mu^+(\sigma)$ by the equations
\begin{align*}
\gr H_{a;i,j}^{(r+1)}=\sfe_{\mu_{(a-1)}+i,\mu_{(a-1)}+j;r},\\
\gr B_{b,a;i,j}^{(r+1)}=\sfe_{\mu_{(b-1)}+i,\mu_{(a-1)}+j;r}.
\end{align*}

We claim that $\sfe_{i,i;2r}=0$ for $r\in\bN$. It follows from \eqref{pr6} that 
\beq\label{996-01}
\big[\sfe_{\mu_{(a)}+i,\mu_{(a-1)}+j;r},\sfe_{\mu_{(a)}+k,\mu_{(a-1)}+l;s}\big]=(-1)^r \big(\sfe_{\mu_{(a)}+k,\mu_{(a)}+i;r+s}-\sfe_{\mu_{(a-1)}+j,\mu_{(a-1)}+l;r+s}\big).
\eeq
Interchanging $i\leftrightarrow k$, $j\leftrightarrow l$, and $r\leftrightarrow s$, we obtain
\beq\label{996-02}
\big[\sfe_{\mu_{(a)}+k,\mu_{(a-1)}+l;s},\sfe_{\mu_{(a)}+i,\mu_{(a-1)}+j;r}\big]=(-1)^s \big(\sfe_{\mu_{(a)}+i,\mu_{(a)}+k;r+s}-\sfe_{\mu_{(a-1)}+l,\mu_{(a-1)}+j;r+s}\big).
\eeq
Setting $s=r$, $k=i$, and $l=j$, we find that
\beq\label{991}
\sfe_{\mu_{(a)}+i,\mu_{(a)}+i;2r}-\sfe_{\mu_{(a-1)}+j,\mu_{(a-1)}+j;2r}=0,
\eeq
where $r\gge \fks_{a+1,a}(\mu)$. Note that \eqref{pr1} implies that $\sfe_{1,1;2s}=0$ for $s\in\bN$.  In addition, for $r<\fks_{a+1,a}(\mu)$, $\sfe_{i,i;2r}=0$ follows immediately from \eqref{pr1} and \eqref{Zshifted}. Now an induction on $a$ with the help of \eqref{991} implies that $\sfe_{i,i;2r}=0$ for all $r\in\bN$.

Repeating the same strategy as in the proof of Theorem \ref{mainthm} (the paragraph includes \eqref{pfn1} and \eqref{pfn2}), we find that for $r\in\bN$
$$
\sfe_{\mu_{(a)}+i,\mu_{(a)}+j;r}=(-1)^{r+1}\sfe_{\mu_{(a)}+j,\mu_{(a)}+i;r}.
$$

Extend the set of elements $\sfe_{i,j;r}$ by the rule: $\sfe_{j,i;r}=(-1)^{r+1}\sfe_{i,j;r}$, for $1\lle j\lle i\lle N$ and $r\gge \fks_{i,j}$. Following the proof of Theorem \ref{mainthm}, one verifies that these elements satisfy the relations \beq\label{com-eq:pf}
[\sfe_{i,j;r},\sfe_{k,l;s}]=\delta_{jk}\sfe_{i,l;r+s}-\delta_{li}\sfe_{k,j;r+s}-(-1)^r\delta_{ki}\sfe_{j,l;r+s}+(-1)^r\delta_{jl}\sfe_{k,i;r+s}.
\eeq
Thus there exists a well-defined epimorphism 
$\varpi:\mathrm{U}(\gl_N[z]^\theta(\sigma))\to \gr\Y_\mu^+(\sigma)$. 
\beq\label{mapp}
\begin{split}
&e_{i,j}\otimes z^r-(-1)^re_{j,i}\otimes z^r\mapsto \sfe_{i,j;r},\hskip 2cm \text{ for }1\lle j\lle i\lle N,r\gge \fks_{i,j}.
\end{split}
\eeq
Recall from the proof of theorem \ref{mainthm}  that for $\sigma=0$, the map $\varpi$ is an isomorphism. Moreover, the ordered monomials in the elements 
\[
\big\{\sfe_{i,i;2r+1}\big\}_{1\lle i\lle N,r\in\bN}\bigcup \big\{\sfe_{i,j;r}\big\}_{1\lle j<i\lle N,r\in\bN}
\]
are linearly independent in $\gr\Y_\mu^+$.

Consider the canonical map $\Y_\mu^+(\sigma)\to \Y_\mu^+$ which sends generators to generators of the same name. It is a homomorphism of filtered algebras and hence induces an algebra homomorphism $\gr\Y_\mu^+(\sigma)\to \gr\Y_\mu^+$ which sends $\sfe_{i,j;r}\in \gr\Y_\mu^+(\sigma)$ to $\sfe_{i,j;r}\in \gr\Y_\mu^+$. Therefore the ordered monomials in the elements 
\[
\big\{\sfe_{i,i;2r+1}\big\}_{1\lle i\lle N,r\in\bN}\bigcup \big\{\sfe_{i,j;r}\big\}_{1\lle j<i\lle N,r\gge \fks_{ij}}
\]
are linearly independent in $\gr\Y_\mu^+(\sigma)$. Thus $\varpi$ is an isomorphism.
\end{proof}
Thus, we identify $\Y_\mu^+(\sigma)$ as a subalgebra of $\Y_N^+$. Let $\nu$ be another admissible shape for $\sigma$, then we have another subalgebra $\Y_\nu^+(\sigma)$ of $\Y_N^+$.
\begin{prop}\label{ind}
The subalgebras $\Y_\mu^+(\sigma)$ and $\Y_\nu^+(\sigma)$ of $\Y_N^+$ coincide.
\end{prop}
\begin{proof}
The proof is completely parallel to that of $\mathrm{Y}(\gl_N)$; see \cite[\S 3]{BK06}. We sketch the argument. Suppose $\mu=(\mu_1,\ldots,\mu_n)$ is a composition of $N$ and $\nu$ is a refinement of $\mu$, such that
\[
\nu=(\mu_1,\dots,\mu_{a-1},\alpha,\beta,\mu_{a+1},\dots,\mu_n),
\]
where $\alpha+\beta=\mu_a$. It follows from \cite[Lem. 3.1]{BK06} that $\Y_\mu^+(\sigma)\subset \Y_\nu^+(\sigma)$ as subalgebras of $\Y_N^+$. On the other hand, one can show that their associated graded algebras coincide in $\gr\Y_N^+$. Therefore, we must have $\Y_\mu^+(\sigma)= \Y_\nu^+(\sigma)$. Since the composition $(1^N)$ is a refinement of any composition $\mu$, the statement follows.
\end{proof}
Due to independence of $\Y^{+}_\mu(\sigma)$ from $\mu$  from now on we will write $\Y_N^+(\sigma)$ again unless we want to put particular emphasis on a choice of $\mu$.
\subsection{The canonical filtration} \label{canof}
Recall that the canonical filtration $\mathrm F_0\Y_N^\pm\subset \mathrm F_1\Y_N^\pm\subset \mathrm F_2\Y_N^\pm\subset \cdots$ of $\Y_N^\pm$ is defined by setting 
\begin{align}  \label{cfilter:tY}
\deg_2 s_{ij}^{(r)}=r, \quad \forall r\gge 1
\end{align}
and letting $\mathrm F_d\Y_N^\pm$ denote the the span of monomials in $s_{ij}^{(r)}$ of total degree $\lle d$.
The associated graded algebra is denoted by $\gr'\Y_N^\pm$, and it is commutative, thanks to \eqref{qua}.

Suppose we are given a composition $\mu=(\mu_1,\ldots,\mu_n)$ of $N$.
Consider the parabolic generators $H_{a;i,j}^{(r)}, B_{b,a;i,j}^{(r)}$ of $\Y_N^+$ in Theorem \ref{PBWgauss}.
By \eqref{quasid} -- \eqref{quasif}, \eqref{Ha}, \eqref{Bb} and Lemma \ref{e=f gen}, we may express the parabolic generators $H_{a;i,j}^{(r)}$, $B_{b,a;k,l}^{(r)}$ as linear combinations of monomials in $s_{ij}^{(t)}$ of total degree $r$.  
Conversely, if we define $H_{a;i,j}^{(r)}$ and $B_{b,a;i,j}^{(r)}$ to be of degree $r$, then by \eqref{S=FDE} and \eqref{e=f gen}, each $s_{ij}^{(t)}$ is a linear combination of monomials in $H_{a;i,j}^{(r)}$ and $B_{b,a;i,j}^{(r)}$ of total degree $t$.
As a result, we may also describe $\mathrm F_d\Y_N^+$ as the span of all monomials in the elements $H_{a;i,j}^{(r)}$ , $B_{b,a;i,j}^{(r)}$ of total degree $\lle d$.

For $1\lle a\lle b\lle n$, $1\lle i\lle \mu_b$, $1\lle j\lle \mu_a$ and $r>0$, denote the images of the parabolic generators in $\gr'\Y_N^+$ by
\beq\label{grp-img}
\begin{split}
\sff_{b,a;i,j}^{(r)}:=\begin{cases}
\gr' H_{a;i,j}^{(r)}, & \text{ if } a=b,\\
\gr' B_{b,a;i,j}^{(r)}, & \text{ if } a<b. \\
\end{cases}
\end{split}
\eeq

\begin{prop}\label{polytY}
For any shape $\mu$, $\gr'\Y_N^+$ is the free commutative algebra on generators 
\begin{eqnarray*}
& & \big\{\,\sff_{a,a;i,i}^{(2r)} \, | \, 1\lle a\lle n, 1\lle i\lle \mu_a, r>0\big\}\\
& & \hspace{20pt}\cup \big\{\,\sff_{a,a;i,j}^{(r)} \, | \, 1\lle a\lle n, 1\lle j<i\lle \mu_a, r>0  \big\}\\
& & \hspace{40pt}\cup \big\{\,\sff_{b,a;i,j}^{(r)}\, | \, 1\lle a<b\lle n, 1\lle i\lle \mu_b, 1\lle j\lle \mu_a, r>0 \big\}.
\end{eqnarray*}
\end{prop}

\begin{proof}
This follows from Theorem \ref{PBWgauss} and the fact that $\gr'\Y_N^+$ is commutative.
\end{proof}

Now we consider the shifted twisted Yangian $\Y_N^+(\sigma)$ associated to a shift matrix $\sigma$.
Treating $\Y_N^+(\sigma)$ as a subalgebra of $\Y_N^+$, we introduce the canonical filtration on $\Y_N^+(\sigma)$
\begin{align}  \label{cFiltershiftedI}
\mathrm F_0\Y_N^+(\sigma)\subset \mathrm F_1\Y_N^+(\sigma)\subset \mathrm F_2\Y_N^+(\sigma)\subset \cdots,\qquad\qquad \Y_N^+(\sigma)=\bigcup_{s\gge 1}\mathrm F_s\Y_N^+(\sigma),
\end{align}
by defining $\mathrm F_s\Y_N^+(\sigma):= \Y_N^+(\sigma) \cap \mathrm F_s\Y_N^+$.
Then the natural embedding $\Y_N^+(\sigma) \hookrightarrow  \Y_N^+$ is a filtered map and its induced map $\gr'\Y_N^+(\sigma) \rightarrow  \gr'\Y_N^+$ is injective, allowing us to identify $\gr'\Y_N^+(\sigma)$ as a subalgebra of the commutative algebra $\gr'\Y_N^+$.
\begin{prop} \label{polystY}
For any admissible shape $\mu$, $\gr'\Y_N^+(\sigma)$ is the subalgebra of $\gr'\Y_N^+$ generated by
\begin{eqnarray*}
& & \big\{\,\sff_{a,a;i,i}^{(2r)} \, | \, 1\lle a\lle n, 1\lle i\lle \mu_a, r>0\big\}\\
& & \hspace{20pt}\cup \big\{\,\sff_{a,a;i,j}^{(r)} \, | \, 1\lle a\lle n, 1\lle j<i\lle \mu_a, r>0 \big\}\\
& & \hspace{40pt} \cup \big\{\,\sff_{b,a;i,j}^{(r)}\, | \, 1\lle a<b\lle n, 1\lle i\lle \mu_b, 1\lle j\lle \mu_a, r> \fks_{a,b}(\mu) \big\}.
\end{eqnarray*}
\end{prop}

\begin{proof}
Due to \eqref{pr7} all of these elements in $\gr'\Y_N^+(\sigma)$ are identified with the elements in $\gr'\Y_N^+$ sharing the same notation by the induced map $\gr'\Y_N^+(\sigma) \rightarrow  \gr'\Y_N^+$.
Then the statement follows from Theorem \ref{pbwstw} and Proposition \ref{polytY}.
\end{proof}

\begin{rem}\label{canofalt}
A consequence of Proposition \ref{polystY} is that one can intrinsically define the canonical filtration on $\Y_N^+(\sigma)$ by setting the elements $H_{a;i,j}^{(r)}$, $B_{b,a;i,j}^{(r)}$ of $\Y_N^+(\sigma)$ to be of degree $r$, and letting $\mathrm F_d\Y_N^+(\sigma)$ be the span of monomials in these elements of total degree $\lle d$. This definition is independent of the choice of admissible shape $\mu$. A similar statement for $\Y_N^-(\sigma)$ holds as well; see Proposition~ \ref{new-polystY}. 
\end{rem}

\subsection{An alternative presentation}
There is another natural way to define the shifted twisted Yangian of type AI, using the Drinfeld presentation of the twisted Yangian from \cite{LWZ23}. This definition has already appeared in \cite[\textsection 3]{TT24}.
We will recall it here for the convenience of the reader. To avoid confusion we will call it the (\textit{dominantly}) \textit{Drinfeld shifted twisted Yangian}. In this subsection, we will see that these two shifted twisted Yangians are isomorphic, as algebras.
Let $A=(c_{ij})_{1\lle i,j\lle N-1}$ be the Cartan matrix of type $A_{N-1}$, and set $c_{0j}=-\delta_{j,1}$.
\begin{dfn}
The (\textit{dominantly}) \textit{shifted Drinfeld twisted Yangian} associated to the matrix $\sigma$ is the algebra $^{\textnormal{Dr}}\Y_N^+(\sigma)$ over $\bC$ 
generated by $\{\ch_{i,r}\}_{r\gge 0}$, $\{\cb_{j,r}\}_{r\gge \fks_{j+1,j}}$, for $0\lle i<N$ and $1\lle j<N$, subject to the following relations, for $r,s \in \bN$:
\begin{align}
[ \ch_{i,r}, \ch_{j,s}]&=0, \qquad { \ch_{i,2r}=0, } \label{drs1}\\	
[ \ch_{i,r+1}, \cb_{j,s}]-[ \ch_{i,r-1}, \cb_{j,s+2}]&=c_{ij}\{ \ch_{i,r-1}, \cb_{j,s+1}\}+\frac{1}{4}c_{ij}^2[ \ch_{i,r-1}, \cb_{j,s}],\label{drs2}\\
[ \cb_{i,r+1}, \cb_{j,s}]-[ \cb_{i,r}, \cb_{j,s+1}]&=\frac{c_{ij}}{2}\{ \cb_{i,r}, \cb_{j,s}\}-2\delta_{ij}(-1)^r \ch_{i,r+s+1},\label{drs3}\\
[\cb_{i,r},\cb_{j,s}]&=0,\qquad  \text{ for }|i-j|>1,\label{drs4}\\
\mathrm{Sym}_{k_1,k_2}\big[\cb_{i,k_1},[\cb_{i,k_2},\cb_{j,r}] \big]& \label{drs5} \\
=
(-1)^{k_1}\sum_{p\gge 0}2^{-2p}  \big(&[\ch_{i,k_1+k_2-2p-1},\cb_{j,r+1}]-\{\ch_{i,k_1+k_2-2p-1},\cb_{j,r}\}\big),
\qquad \text{ if } c_{i,j}=-1. \notag
\end{align}
for all admissible indices $i,j,r,s$. Here by convention, $\ch_{i,-1}=1$.
\end{dfn}

If $\sigma$ is the zero matrix, then one recovers the Drinfeld presentation of twisted Yangians; see \cite[Thm. 5.1]{LWZ23}. The definition naturally extends to the symmetric Cartan matrices (i.e. simply laced types), which correspond to shifted twisted Yangians of split {\sf ADE} type; cf. \cite{LWZ25}. In the context of the current paper, the Drinfeld shifted twisted Yangians are associated to $\gl_N$. To obtain the $\mathfrak{sl}_N$ version, one only needs to exclude the generators $\ch_{0,r}$.

\begin{rem}
More general shifted twisted Yangians, which also include the antidominantly shifted case, are introduced and studied in \cite{LWW25} in connection with affine Grassmanian slices. Contrary to the dominantly shifted case, the antidominantly shifted twisted Yangians are usually not subalgebras of twisted Yangians.  
\end{rem}

\begin{prop} \label{shapeindprop}
The algebras $^{\textnormal{Dr}}\Y_N^+(\sigma)$ and $\Y_\mu^+(\sigma)$ are isomorphic.
\end{prop}
\begin{proof}
By Proposition \ref{ind}, it suffices to prove the statement for the case when $\mu=(1^N)$. Again, we only sketch the proof.

First, we show that the map
\beq\label{themap}
\sfb_{a,r}\mapsto B_{a;a+1,a}^{(r+1)},\quad \sfh_{a,r}\mapsto Z_{a;a,a,a+1,a+1}^{(r+1)},\quad \sfh_{0,r}\mapsto H_{1;1,1}^{(r+1)},\quad 1\lle a<N,r\in\bN,
\eeq
defines an algebra homomorphism $^{\textnormal{Dr}}\Y_N^+(\sigma)\to \Y_\mu^+(\sigma)$. We need to check the relations \eqref{drs1}--\eqref{drs5} are satisfied by the images of $\sfh_{a,r}$ and $\sfb_{a,r}$ in $\Y_\mu^+(\sigma)$ under the map \eqref{themap}. The relations \eqref{drs1}--\eqref{drs4} can be verified as in \cite[\S4--\S5]{LWZ23}. To verify \eqref{drs5}, we check that assuming the relations \eqref{drs1}--\eqref{drs4} then the relation \eqref{serrepr} (which is used to obtain the Serre relations for the parabolic presentations when the particular block has size one) implies the relation \eqref{drs5}; see also Remark \ref{remext}. Thus the map \eqref{themap} defines an algebra homomorphism. To see this is an isomorphism, one uses the PBW theorems (\cite[Thm. 3.2(2)]{TT24} and Theorem \ref{pbwstw}) to see that a basis maps to a basis.
\end{proof}
By Proposition \ref{shapeindprop} we can drop the prefix ${\textnormal{Dr}}$ from $^{\textnormal{Dr}}\Y^{+}_N(\sigma)$, and we will do so unless we want to emphasize that we are working with this particular presentation.

\subsection{Shifted twisted Yangians of type AII}

The notion of parabolic presentation allows one to define the shifted twisted Yangian in a very similar way as the AI type, except that we need to assume that all parts of the composition $\mu$ are even, and we will refer to such compositions as {\it even}. 
Another major difference is that the $Z$-relations \eqref{Zshifteda2} for type AII are different from \eqref{Zshifted} for type AI.

In this subsection, we follow the convention \eqref{thetadef}.
Let $\sigma$ be a symmetric shift matrix of even size $N$.
Recall a composition $\mu=(\mu_1,\ldots,\mu_n)$ of $N$ is admissible to $\sigma$ if 
%every $\mu_i$ is even and
the condition in \S\ref{paraAI} holds.

\begin{dfn} \label{def:shiftedII}
Let $\mu=(\mu_1,\ldots,\mu_n)$ be an even composition of $N$, which is admissible to $\sigma$, which means $\fks_{i,i+1} = 0$ for $i$ odd. The {\it (dominantly) shifted twisted Yangian of type AII associated to $\sigma$}, denoted by $\Y_\mu^-(\sigma)$, is the associative unital algebra over $\bC$ generated by
\[
\{H_{a;i,j}^{(r)},\ \wtl H_{a;i,j}^{(r)}\}_{1\lle a\lle n,1\lle i,j\lle \mu_a,r\gge 1},\quad \{B_{a;i,j}^{(r)}\}_{1\lle a< n,1\lle i\lle \mu_{a+1}, 1\lle j\lle \mu_a,r>\fks_{a+1,a}(\mu)}
\]
subject to the relations given in Theorem \ref{mainthm} with all admissible indices together with the following relation (cf. \eqref{Za2})
\beq\label{Zshifteda2}
Z_{a;i,i,j,j}^{(2r-1)}+Z_{a;i',i',j',j'}^{(2r-1)}=0,\quad 1\lle a<n,~1\lle i\lle \mu_a,~1\lle j\lle \mu_{a+1},~1\lle r\lle \fks_{a+1,a}(\mu).
\eeq
\end{dfn}

By Theorem \ref{mainthm}, if $\sigma$ is zero matrix, $\Y_\mu^-(\sigma)$ is exactly $\Y_N^-$. Clearly, we have the canonical homomorphism $\Y_\mu^-(\sigma)\to \Y_\mu^-= \Y_N^-$ sending the generators $H_{a;i,j}^{(r)}$ and $B_{a;i,j}^{(r)}$ of $\Y_\mu^-(\sigma)$ to the elements of $\Y_N^-$ with same names. 

Exactly the same as AI type, for $1\lle a<b\lle n$, $1\lle i\lle \mu_b$, $1\lle j\lle \mu_a$ and $r>\fks_{b,a}(\mu)$, define elements $B_{b,a;i,j}^{(r)}\in \Y_\mu^-(\sigma)$ inductively by
\begin{equation*}
B_{a+1,a;i,j}^{(r)}:=B_{a;i,j}^{(r)},\quad B_{b,a;i,j}^{(r)}:=[B_{b-1;i,k}^{(\fks_{b,b-1}(\mu)+1)},B_{b-1,a;k,j}^{(r-\fks_{b,b-1}(\mu))}].
\end{equation*}
By the same argument, one deduces that the definition is independent of the choice of $1\lle k\lle \mu_{b-1}$.

The following results are parallel to Theorem \ref{pbwstw} and Proposition \ref{ind}, and the proofs are almost identical.

\begin{thm}\label{a2shPBWgauss}
The monomials in 
\begin{align*}
&\big\{H_{a;2i-1,2i-1}^{(r)},H_{a;2i-1,2i}^{(2r-1)},H_{a;2i,2i-1}^{(2r-1)}\big\}_{1\lle a\lle n,1\lle i\lle \tfrac{\mu_a}{2},r>0},\\
&\big\{H_{a;i,j}^{(r)}\big\}_{1\lle a\lle n,1\lle j<i\lle \mu_a,\lfloor \tfrac{i+1}{2}\rfloor\ne \lfloor \tfrac{j+1}{2}\rfloor,r>0},\qquad
\big\{B_{b,a;i,j}^{(r)}\big\}_{1\lle a< b< n,1\lle i\lle \mu_{b},1\lle j\lle \mu_a,r>\fks_{b,a}(\mu)},
\end{align*}
taken in some fixed linear order form a basis of $\Y_\mu^-(\sigma)$. In particular, the canonical map $\Y_\mu^-(\sigma)\to \Y_N^-$ is injective.
\end{thm}

\begin{cor}\label{a2shapeindcor}
Up to isomorphism, the definition of $\Y_\mu^-(\sigma)$ is independent of the choice of even admissible $\mu$. 
\end{cor}

The canonical filtration for shifted twisted Yangian of AII type
\beq\label{canofa2}
\mathrm F_0\Y_{N}^-(\sigma)\subset \mathrm F_1\Y_{N}^-(\sigma)\subset \mathrm F_2\Y_{N}^-(\sigma)\subset \cdots,\qquad\qquad 
\Y_{N}^-(\sigma)=\bigcup_{s\gge 1}\mathrm F_s\Y_{N}^-(\sigma),
\eeq
can be given exactly the same way as AI type by defining $\mathrm F_s\Y_N^-(\sigma):= \Y_N^-(\sigma) \cap \mathrm F_s\Y_N^-$. 
When $\sigma=0$, it follows from \eqref{qua} that the associated graded $\gr'\Y_N^-$ is commutative and, similarly, that $\gr'\Y_N^-(\sigma)$ can naturally be viewed as a subalgebra of $\gr'\Y_N^-$.

Consider the elements of $\Y_N^-(\sigma)$ in Theorem \ref{a2shPBWgauss}.
Denote their images in $\gr'\Y_N^-(\sigma)$ by
\beq\label{grp-img-new}
\begin{split}
\sff_{b,a;i,j}^{(r)}:=\begin{cases}
\gr' H_{a;i,i}^{(r)} & \text{ if } a=b, \text{ and } i=j \text{ is odd },\\[1mm]
\gr' H_{a;i,i-1}^{(r)}, & \text{ if } a=b, \, r \text{ is odd and } i=j+1 \text{ is even } ,\\[1mm]
\gr' H_{a;i,i+1}^{(r)}, & \text{ if } a=b, \, r \text{ is odd and } i=j-1 \text{ is odd } ,\\[1mm]
\gr' H_{a;i,j}^{(r)}, & \text{ if } a=b, \,\, 1\lle j<i \lle\mu_a \text{ and } \lfloor \tfrac{i+1}{2}\rfloor\ne \lfloor \tfrac{j+1}{2}\rfloor,\\[1mm]
\gr' B_{b,a;i,j}^{(r)}, & \text{ if } a<b. \\
\end{cases}
\end{split}
\eeq
The following result is an analogue of Proposition \ref{polystY}.
\begin{prop} \label{new-polystY}
For any even admissible shape $\mu$, $\gr'\Y_N^-(\sigma)$ is the subalgebra of $\gr'\Y_N^-$ generated by
\begin{eqnarray*}
& & \{\sff_{a,a;2i-1,2i-1}^{(r)}, \sff_{a,a;2i-1,2i}^{(2r-1)},  \sff_{a,a;2i,2i-1}^{(2r-1)} \, | \, 1\lle a\lle n, 1\lle i\lle \tfrac{\mu_a}{2}, r\gge1\}\\
& & \hspace{20pt}\cup \{\sff_{a,a;i,j}^{(r)} \, | \, 1\lle a\lle n, 1\lle j<i\lle \mu_a, \lfloor \tfrac{i+1}{2}\rfloor\ne \lfloor \tfrac{j+1}{2}\rfloor, r\gge1\, \}\\
& & \hspace{40pt}\cup \{\sff_{a,b;i,j}^{(r)}\, | \, 1\lle a<b \lle n, 1\lle i\lle \mu_b, 1\lle j\lle \mu_a, r>\fks_{b,a}(\mu) \}.
\end{eqnarray*}
\end{prop}

\begin{rem}\label{Zsty}
It follows from Proposition~\ref{sdetdecomp} that the center of the twisted Yangian $Z(\Y_N^{\pm})$ is contained in the center of the shifted twisted Yangian $Z(\Y_N^{\pm}(\sigma))$. In fact, for any $\sigma$ we have 
\beq\label{eqcent}
Z(\Y_N^{\pm})=Z(\Y_N^{\pm}(\sigma)).
\eeq
Let $\fkc=\gl_N[z]^\theta$ denote the twisted current algebra defined in \S\ref{ss:TwistedYangians}, $\fkc_\sigma=\gl_N[z]^\theta(\sigma)$ denote its subalgebra defined in \S\ref{paraAI}, and $\fkc_r$ denote the subalgebra of $\fkc$ spanned by elements $xz^s$ with $s > r$. 
To establish \eqref{eqcent} one must check the inclusion $Z(\Y_N^{\pm}(\sigma))\subseteq Z(\Y_N^{\pm})$ and, for this, it suffices to show $S(\fkc_\sigma)^{\fkc_\sigma}\subseteq S(\fkc)^\fkc$.
By a slight modification of the argument for \cite[Lem 2.8.1]{Mol07} one can prove that $S(\fkc)^\fkc = S(\fkc)^{\fkc_r}$ for any $r\gge 0$. Picking $r$ large enough so that $\fkc_r \subseteq \fkc_\sigma$, we have $S(\fkc_\sigma)^{\fkc_\sigma}\subseteq S(\fkc)^{\fkc_\sigma} \subseteq S(\fkc)^{\fkc_r}=S(\fkc)^{\fkc}$.
\end{rem}

\section{Baby comultiplication}\label{sec:baby}
In this section, we establish a baby comultiplication, which is an algebra homomorphism from $\Y_\mu^\pm(\sigma)$ to $\Y_\mu^\pm(\dot\sigma)\otimes \rU(\gl_t)$. This provides a powerful technical tool for studying both structure and representation theory (cf. \cite{BK08}). 

\subsection{The formulation}
Assume $\sigma\neq 0$ and suppose $\mu=(\mu_1,\ldots, \mu_n)$ is the minimal admissible shape of $\sigma$.
Let $\Y_\mu^\pm(\sigma)$ be the shifted twisted Yangian associated to $\sigma$ and $\mu$, where we assume $\mu$ to be even for $\Y_\mu^-(\sigma)$. 
Define $\dot\sigma:=(\dot\fks_{i,j})_{1\lle i,j\lle N}$ according to the following rule
\beq\label{dotsij}
\dot\fks_{i,j}:=\begin{cases}
\fks_{i,j}-1, & \text{ if }i\lle N-\mu_n<j  \text{ or } j\lle N-\mu_n<i,\\
\fks_{i,j}, &\text{ otherwise. }\\
\end{cases}
\eeq
By definition, $\dot\sigma$ is a symmetric shift matrix and $\mu$ is admissible for $\dot\sigma$. Therefore the definition of $\Y_\mu^\pm(\dot\sigma)$ makes sense. 
Due to Theorem~\ref{pbwstw} we have the embedding $\Y_\mu^\pm(\sigma) \hookrightarrow \Y_\mu^\pm(\dot\sigma)$, sending $H_{a;i,j}^{(r)} \mapsto \dot{H}_{a;i,j}^{(r)}$ and $B_{a;f,g}^{(s)}\mapsto \dot{B}_{a;f,g}^{(s)}$. Here and after we use $\dot{H}_{a;i,j}^{(r)}$, $\dot{B}_{a;f,g}^{(s)}$ to denote elements in $\Y_\mu^\pm(\dot\sigma)$ to avoid possible confusion.

\begin{prop}
    \label{delR0}
The homomorphism  
\beq\label{Ka}
\Delta_R=(1\otimes \pi_N)\circ \Delta:\Y_N^\pm\longrightarrow \Y_N^\pm\otimes \rU(\gl_N)
\eeq
sends
\begin{equation}\label{Kacomp}
 s_{ij}^{(r)} \mapsto s_{ij}^{(r)}\otimes 1
+\sum_{p=1}^N s_{ip}^{(r-1)}\otimes e_{pj}
-\sum_{q=1}^N s_{qj}^{(r-1)}\otimes e_{q'i'}\theta_q\theta_i
-\sum_{p,q=1}^N s_{qp}^{(r-2)}\otimes e_{q'i'}e_{pj}\theta_q\theta_i.
\end{equation}
\end{prop}

\begin{proof}
    Follows from the definitions.
\end{proof}

The following is an analogue of \cite[Thm.~4.1]{BK05} for shifted twisted Yangians and will provide a powerful tool in later sections. Its proof is much more technically challenging than {\em loc. cit.} and will be presented in the two subsections below.

\begin{thm}\label{delR}
Suppose $\sigma \neq 0$. Let $\mu=(\mu_1,\ldots, \mu_n)$ be the minimal admissible shape associated to $\sigma$ and let $t=\mu_n$. The map $\Delta_R: \Y_\mu^\pm(\sigma)\longrightarrow \Y_\mu^\pm(\dot\sigma)\otimes \rU(\gl_t)$ such that 
\begin{align}
\Delta_R(H_{a;i,j}^{(r)})&=\dot{H}_{a;i,j}^{(r)}\otimes 1, \,\,\qquad \forall 1\lle a\lle n-1,
\label{Hdel1}\\
\Delta_R(B_{a;f,g}^{(s)})&=\dot{B}_{a;f,g}^{(s)}\otimes 1- \delta_{a,n-1}  \sum_{p=1}^{t}\dot{B}_{n-1;p,g}^{(s-1)}\otimes e_{p'f'}\theta_p\theta_f, \quad \forall 1\lle a\lle n-1,
\label{Bdel1}\\
\Delta_R(H_{n;l,k}^{(r)})&=\dot{H}_{n;l,k}^{(r)}\otimes 1+ \sum_{p=1}^{t} \dot{H}_{n;l,p}^{(r-1)}\otimes e_{pk} - \sum_{q=1}^{t}  \dot{H}_{n;q,k}^{(r-1)}\otimes e_{q'l'}\theta_q\theta_l
-\sum_{p,q=1}^t    \dot{H}_{n;q,p}^{(r-2)}\otimes e_{q'l'}e_{p k}\theta_q\theta_l, 
\label{bp:Hm}
\end{align}
is a well-defined algebra homomorphism; it is understood that the last term with superscript $(r-2)$ in \eqref{bp:Hm} drops if $r=1$.
\end{thm}
We shall refer to $\Delta_R$ as the baby comultiplication. 

\subsection{Proof of Theorem \ref{delR} for AI type}\label{subsec:babypfa1}

\begin{proof}
We assume the convention \eqref{a1thetadef} in this subsection.
Let $\sigma\neq0$. Recall from Definition~ \ref{defn:shiftedtwisted} that 
the defining relations for $\Y_\mu^+(\sigma)$ are precisely the relations \eqref{pr1}--\eqref{pr-2} with all admissible indices together with \eqref{Zshifted}. Checking the relations is trivial unless they involve $B_{n-1;f,g}^{(s)}$ and $H_{n;l,k}^{(r)}$, and so it suffices to check the relations involving only the very last two or three blocks. Thus we may simply assume $\mu=(\mu_1,\mu_2)$ or $\mu=(\mu_1,\mu_2,\mu_3)$.

Assume $\mu=(\mu_1,\mu_2)$. We check that the relation (\ref{pr6}), which is (\ref{bbcoef1}), is preserved by $\Delta_R$. Recall from (\ref{tlHH}) that
\[
Z_{i,k,l,j}(v) := \sum_{r\gge 0} Z_{1;i,k,l,j}^{(r)}v^{-r}= \wtl H_{1;i,k}(v-\tfrac{\mu_1}{2})H_{2;l,j}(v).
\]
Its image under $\Delta_R$ is given by
\begin{equation}\label{delZ}
\Delta_R(Z_{i,k,l,j}^{(r)})=\dot Z_{i,k,l,j}^{(r)}\otimes 1 + \sum_{p=1}^{\mu_2} \dot Z_{i,k,l,p}^{(r-1)} \otimes e_{pj} - \sum_{q=1}^{\mu_2}\dot Z_{i,k,q,j}^{(r-1)} \otimes e_{ql} - \sum_{p,q=1}^{\mu_2}\dot Z_{i,k,q,p}^{(r-2)} \otimes e_{ql}e_{pj}.
\end{equation}
Using (\ref{b=c n2}), we can rewrite (\ref{bbcoef1}) in the following equivalent form
\begin{equation}\label{CC=CZ}
[C_{1;i,j}^{(r)},C_{1;k,l}^{(s)}]
 =\sum_{a=1}^{s-1} C_{1;i,l}^{(a)} C_{1;k,j}^{(r+s-1-a)}  -  \sum_{a=1}^{r-1} C_{1;i,l}^{(a)} C_{1;k,j}^{(r+s-1-a)}  
-(-1)^{s} Z_{1;i,k,l,j}^{(r+s-1)}.
\end{equation}
Similarly, the image of $\Delta_R$ on $C_{1;f,g}^{(s)}$ is explicitly given by
\begin{equation}\label{DelonC}
\Delta_R(C_{1;f,g}^{(s)})=\dot{C}_{1;f,g}^{(s)}\otimes 1+ \sum_{p=1}^{t}\dot{C}_{1;f,p}^{(s-1)}\otimes e_{pg},
\end{equation}
thus we may compute the bracket as  
\begin{align*}
[&\Delta_R(C_{1;i,j}^{(r)}), \Delta_R(C_{1;k,l}^{(s)})]=
[\dot C_{1;i,j}^{(r)}\otimes 1 + \sum_{p=1}^{\mu_2} \dot C_{1;i,p}^{(r-1)}\otimes e_{pj}, \,\, \dot C_{1;k,l}^{(s)}\otimes 1 + \sum_{q=1}^{\mu_2} \dot C_{1;k,q}^{(s-1)}\otimes e_{ql}]\\
&= [\dot C_{1;i,j}^{(r)}, \dot C_{1;k,l}^{(s)}] \otimes 1 
+ \sum_{p=1}^{\mu_2} [\dot C_{1;i,p}^{(r-1)}, \dot C_{1;k,l}^{(s)}] \otimes e_{pj}
+ \sum_{q=1}^{\mu_2} [\dot C_{1;i,j}^{(r)}, \dot C_{1;k,q}^{(s-1)}] \otimes e_{ql} \\
& + \sum_{p,q=1}^{\mu_2} \big( \dot C_{1;i,p}^{(r-1)}\dot C_{1;k,q}^{(s-1)} \otimes e_{pj}e_{ql} 
- \dot C_{1;k,q}^{(s-1)} \dot C_{1;i,p}^{(r-1)} \otimes e_{ql}e_{pj}
\big)\\
&= [\dot C_{1;i,j}^{(r)}, \dot C_{1;k,l}^{(s)}] \otimes 1 
+ \sum_{p=1}^{\mu_2} [\dot C_{1;i,p}^{(r-1)}, \dot C_{1;k,l}^{(s)}] \otimes e_{pj}
+ \sum_{q=1}^{\mu_2} [\dot C_{1;i,j}^{(r)}, \dot C_{1;k,q}^{(s-1)}] \otimes e_{ql} \\
& + \sum_{p,q=1}^{\mu_2} [ \dot C_{1;i,p}^{(r-1)}, \dot C_{1;k,q}^{(s-1)}]\otimes e_{ql}e_{pj}  + \sum_{p,q=1}^{\mu_2} \dot C_{1;i,p}^{(r-1)}\dot C_{1;k,q}^{(s-1)} \otimes (\delta_{qj}e_{pl}-\delta_{lp}e_{qj})\\
& =\big( \sum_{a=1}^{s-1} \dot C_{1;i,l}^{(a)}\dot C_{1;k,j}^{(r+s-1-a)} 
    - \sum_{a=1}^{r-1} \dot C_{1;i,l}^{(a)}\dot C_{1;k,j}^{(r+s-1-a)} 
    - (-1)^s \dot Z_{i,k,l,j}^{(r+s-1)}\big) \otimes 1\\
& + \sum_{p=1}^{\mu_2} \big( \sum_{a=1}^{s-1} \dot C_{1;i,l}^{(a)}\dot C_{1;k,p}^{(r+s-2-a)} 
    - \sum_{a=1}^{r-2} \dot C_{1;i,l}^{(a)}\dot C_{1;k,p}^{(r+s-2-a)} 
    - (-1)^s \dot Z_{i,k,l,p}^{(r+s-2)}\big) \otimes e_{pj}\\
& + \sum_{q=1}^{\mu_2} \big( \sum_{a=1}^{s-2} \dot C_{1;i,q}^{(a)}\dot C_{1;k,j}^{(r+s-2-a)} 
    - \sum_{a=1}^{r-1} \dot C_{1;i,q}^{(a)}\dot C_{1;k,j}^{(r+s-2-a)} 
    - (-1)^{s-1}\dot  Z_{i,k,q,j}^{(r+s-2)}\big) \otimes e_{ql}\\
& + \sum_{p,q=1}^{\mu_2} \big( \sum_{a=1}^{s-2} \dot C_{1;i,q}^{(a)}\dot C_{1;k,p}^{(r+s-3-a)} 
    - \sum_{a=1}^{r-2} \dot C_{1;i,q}^{(a)}\dot C_{1;k,p}^{(r+s-3-a)} 
    -(-1)^{s-1}\dot  Z_{i,k,q,p}^{(r+s-3)}\big) \otimes e_{ql}e_{pj}\\
& - \sum_{q=1}^{\mu_2} \dot C_{1;i,l}^{(r-1)}\dot C_{1;k,q}^{(s-1)} \otimes e_{qj}
  + \sum_{p=1}^{\mu_2} \dot C_{1;i,p}^{(r-1)}\dot C_{1;k,j}^{(s-1)} \otimes e_{pl}.
\end{align*}
Note that the last two terms can be collected in the summation in $(\cdot)\otimes e_{pj}$ and $(\cdot)\otimes e_{ql}$, respectively. 

On the other hand, we may compute the bracket first by (\ref{CC=CZ}) and then apply $\Delta_R$ to the result using (\ref{DelonC}), therefore 
\begin{align*}
&\Delta_R([ C_{1;i,j}^{(r)}, C_{1;k,l}^{(s)}])= \\
&\sum_{a=1}^{s-1} \big( (\dot C_{1;i,l}^{(a)} \otimes 1 + \sum_{p=1}^{\mu_2} \dot C_{1;i,p}^{(a-1)} \otimes e_{pl})
(\dot C_{1;k,j}^{(r+s-1-a)} \otimes 1 + \sum_{q=1}^{\mu_2} \dot C_{1;k,q}^{(r+s-2-a)} \otimes e_{qj} ) \big)\\
-&\sum_{a=1}^{r-1} \big( (\dot C_{1;i,l}^{(a)} \otimes 1 + \sum_{p=1}^{\mu_2} \dot C_{1;i,p}^{(a-1)} \otimes e_{pl})
(\dot C_{1;k,j}^{(r+s-1-a)} \otimes 1 + \sum_{q=1}^{\mu_2} \dot C_{1;k,q}^{(r+s-2-a)} \otimes e_{qj} ) \big)\\
- & (-1)^s \Delta_R(Z_{i,k,l,j}^{(r+s-1)})\\
=& \sum_{a=1}^{s-1} (\dot C_{1;i,l}^{(a)}\dot C_{1;k,j}^{(r+s-1-a)} \otimes 1) + \sum_{a=2}^{s-1} \sum_{p=1}^{\mu_2} ( \dot C_{1;i,p}^{(a-1)}\dot C_{1;k,j}^{(r+s-1-a)} \otimes e_{pl} ) 
\\
+ & \sum_{a=1}^{s-1} \sum_{q=1}^{\mu_2} (\dot C_{1;i,l}^{(a)}\dot C_{1;k,q}^{(r+s-2-a)}\otimes e_{qj}) + \sum_{a=2}^{s-1}  \sum_{p,q=1}^{\mu_2} (\dot C_{1;i,p}^{(a-1)} \dot C_{1;k,q}^{(r+s-2-a)} \otimes e_{pl}e_{qj}  )\\
- & \sum_{a=1}^{r-1} (\dot C_{1;i,l}^{(a)}\dot C_{1;k,j}^{(r+s-1-a)} \otimes 1 ) - \sum_{a=2}^{r-1}\sum_{p=1}^{\mu_2}( \dot C_{1;i,p}^{(a-1)}\dot C_{1;k,j}^{(r+s-1-a)} \otimes e_{pl})\\
- &  \sum_{a=1}^{r-1}\sum_{q=1}^{\mu_2} ( \dot C_{1;i,l}^{(a)}\dot C_{1;k,q}^{(r+s-2-a)}\otimes e_{qj} )  -\sum_{a=2}^{r-1}\sum_{p,q=1}^{\mu_2} ( \dot C_{1;i,p}^{(a-1)} \dot C_{1;k,q}^{(r+s-2-a)} \otimes e_{pl}e_{qj}) \\
-& (-1)^s \big( \dot Z_{i,k,l,j}^{(r+s-1)}\otimes 1 + \sum_{p=1}^{\mu_2}\dot  Z_{i,k,l,p}^{(r+s-2)}\otimes e_{pj} - \sum_{q=1}^{\mu_2}\dot Z_{i,k,q,j}^{(r+s-2)}\otimes e_{ql} - \sum_{p,q=1}^{\mu_2} \dot Z_{i,k,q,p}^{(r+s-3)}\otimes e_{ql}e_{pj}\big).
\end{align*}
Comparing the coefficients of $(\cdot)\otimes 1$, $(\cdot)\otimes e_{pj}$, $(\cdot)\otimes e_{ql}$ and $(\cdot)\otimes e_{ql}e_{pj}$ on both sides, we find that (\ref{CC=CZ}) is preserved by $\Delta_R$ and hence (\ref{bbcoef1}) as well.

We need some preparation to show that \eqref{Zshifted} is preserved by $\Delta_R$.
Fix any $1\lle k \lle \mu_1$. By \eqref{pr3}, the series $\{ Z_{k,k,i,j}(u) \,|\, 1\lle i,j\lle \mu_2 \}$ satisfies the quaternary relation \eqref{qua} where $Z_{k,k,i,j}(u)$ replaces the role of $s_{ij}(u)$. Together with \eqref{Z's} and Lemma \ref{lem:s11}, the following holds for all $1\lle i,j\lle \mu_2$:
\beq\label{Zsym}
Z_{k,k,j,i}(-u) = Z_{k,k, i,j}(u)+ \frac{Z_{k,k,i,j}(u)-Z_{k,k, i,j(-u)}}{2u}.
\eeq
Let $r$ be an odd integer. Then
\[
\Delta_R(Z_{i,i,j,j}^{(r)}) = \sum_{a,b\gge 0}^{a+b=r} \Delta_R \big( \wtl H_{1;i,i}^{[a]}(u-\tfrac{\mu_1}{2}) \big) \Delta_R \big(H_{2;j,j}^{(b)}(u) \big), 
\]
where $X^{[a]}(u-c)$ denotes the coefficient of $u^{-a}$ in the expansion of the series $X(u-c)$. By definition of $\Delta_R$, the above equals
\beq\label{delZ1}
\begin{split}
\sum_{a,b\gge 0}^{a+b=r} &\big(  \wtl H_{1;i,i}^{[a]}(u-\tfrac{\mu_1}{2}) H_{2;j,j}^{(b)}\otimes 1  
+ \sum_{p=1}^{t} \wtl H_{1;i,i}^{[a]}(u-\tfrac{\mu_1}{2}) H_{2;j,p}^{(b-1)}\otimes e_{pj}  \\
& - \sum_{p=1}^{t} \wtl H_{1;i,i}^{[a]}(u-\tfrac{\mu_1}{2}) H_{2;p,j}^{(b-1)}\otimes e_{pj} 
- \sum_{p,q=1}^{t} \wtl H_{1;i,i}^{[a]}(u-\tfrac{\mu_1}{2}) H_{2;q,p}^{(b-2)}\otimes e_{qj}e_{pj} \big).
\end{split}
\eeq
The very first term is precisely $Z_{i,i,j,j}^{(r)}\otimes 1$, which is zero due to \eqref{Z's}.
The second and third terms can be collected as
\beq\label{delZ2}
\sum_{p=1}^t \big( Z_{i,i,j,p}^{(r-1)} - Z_{i,i,p,j}^{(r-1)} \big) \otimes e_{pj} = \sum_{p=1}^t Z_{i,i,p,j}^{(r-2)}\otimes e_{pj}
= \sum_{p<j} Z_{i,i,p,j}^{(r-2)}\otimes e_{pj} + \sum_{p>j} Z_{i,i,p,j}^{(r-2)}\otimes e_{pj},
\eeq
where the first equality follows from \eqref{Zsym} and the fact that $r-1$ is even.
We separate the last term of \eqref{delZ1} into three parts: $p>q$, $p<q$, and $p=q$ respectively:
\beq\label{delZ3}
-\sum_{p>q} Z_{i,i,q,p}^{(r-2)} \otimes e_{qj}e_{pj}  
-\sum_{p<q} Z_{i,i,q,p}^{(r-2)}\otimes e_{qj}e_{pj} 
-\sum_{p=q} Z_{i,i,q,p}^{(r-2)}\otimes e_{qj}e_{pj}.
\eeq
Those terms $p=q$ of \eqref{delZ3} are all zero by \eqref{Z's}. We use the commutator relation in $\rU(\gl_t)$ to rewrite those terms $p<q$ of \eqref{delZ3} as
\beq\label{delZ4}
\begin{split}
& -\sum_{p<q} Z_{i,i,q,p}^{(r-2)}\otimes (e_{pj}e_{qj} +\delta_{jp}e_{qj}-\delta_{jq}e_{pj})
\\
&= 
- \sum_{p<q} Z_{i,i,q,p}^{(r-2)}\otimes e_{pj}e_{qj} 
- \sum_{j<q} Z_{i,i,q,j}^{(r-2)}\otimes e_{qj} 
+ \sum_{p<j} Z_{i,i,j,p}^{(r-2)}\otimes  e_{pj}.
\end{split}
\eeq
In \eqref{delZ4}, the terms $j<q$ cancel the terms $p>j$ of \eqref{delZ2}, while the terms $p<j$ can be collected with those $p<j$ in \eqref{delZ2} into
\[
\sum_{p<j}  \big( Z_{i,i,p,j}^{(r-2)}+Z_{i,i,j,p}^{(r-2)} \big)\otimes e_{pj},
\]
which is zero due to \eqref{Zsym} since $r-2$ is odd.
Finally, interchanging the indices $p$ and $q$ in the remaining term $p<q$ of \eqref{delZ4} and it can be collected with the first term of \eqref{delZ3} as
\[
-\sum_{p>q}  \big( Z_{i,i,q,p}^{(r-2)}+Z_{i,i,p,q}^{(r-2)} \big)\otimes e_{qj}e_{pj},
\]
which is zero due to \eqref{Zsym}. This shows that \eqref{Z's}, and in particular \eqref{Zshifted}, is preserved by $\Delta_R$.

We now check (\ref{pr4}), which can be reduced to (\ref{h2bcoef}), is preserved by $\Delta_R$. In the following proof, the indices $p,q,x$ are summed from 1 to $\mu_2$. 
We first rewrite (\ref{h2bcoef}) into the following equivalent form
\begin{equation}\label{h2ccoef}
[H_{2;i,j}^{(r)}, C_{1;k,l}^{(s)}]  = -\sum_{a=0}^{r-1} H_{2;i,l}^{(r-1-a)}C_{1;k,j}^{(s+a)} + \sum_{a=0}^{r-1}(-1)^{a} C_{1;k,i}^{(s+a)}H_{2;l,j}^{(r-1-a)}.  
\end{equation}

We start with $[\Delta_R(H_{2;i,j}^{(r)}),\Delta_R(C_{1;k,l}^{(s)})]$, which equals to
\begin{equation*}
[A,X]+[A,Y]+[B,X]+[B,Y]-[C,X]-[C,Y]-[D,X]-[D,Y]
%[H_{2;i,j}^{(r)}\otimes 1 + H_{2;i,p}^{(r-1)}\otimes e_{pj} - H_{2;q,j}^{(r-1)}\otimes e_{qi} - H_{2;p,q}^{(r-2)}\otimes e_{pi}e_{qj},  \, \, \dot B_{1;k,l}^{(s)}\otimes 1 + \dot B_{1;k,x}^{(s-1)} \otimes e_{xl}]
\end{equation*}
where $A=\dot H_{2;i,j}^{(r)}\otimes 1$, $B=\dot H_{2;i,p}^{(r-1)}\otimes e_{pj}$, $C=\dot H_{2;q,j}^{(r-1)}\otimes e_{qi}$, $D=\dot H_{2;p,q}^{(r-2)}\otimes e_{pi}e_{qj}$, $X=\dot C_{1;k,l}^{(s)}\otimes 1$ and  $Y=\dot C_{1;k,x}^{(s-1)}~\otimes~e_{xl}$. 
Those terms without $Y$ and those with $A$ can be explicitly written as
$[A,X]=[\dot H_{2;i,j}^{(r)},\dot C_{1;k,l}^{(s)}]\otimes 1$, $[B,X]= [\dot H_{2;i,p}^{(r-1)},\dot C_{1;k,l}^{(s)}]\otimes~e_{pj}$, $[C,X]= [\dot H_{2;q,j}^{(r-1)},\dot C_{1;k,l}^{(s)}]\otimes e_{qi}$, $[D,X]=[\dot H_{2;p,q}^{(r-2)}, \dot C_{1;k,l}^{(s)}]\otimes e_{pi}e_{qj}$, $[A,Y]=[\dot H_{2;i,j}^{(r)}, \dot C_{1;k,x}^{(s-1)}] \otimes e_{xl}$. 

We rewrite the remaining 3 terms by the commutator relations on the $\rU(\gl_t)$ factor. 
For example,
\begin{align*}
[B,Y]&=\dot H_{2;i,p}^{(r-1)} \dot C_{1;k,x}^{(s-1)}\otimes e_{pj}e_{xl} - \dot C_{1;k,x}^{(s-1)}\dot H_{2;i,p}^{(r-1)}  \otimes e_{xl}e_{pj}\\
       &=\dot H_{2;i,p}^{(r-1)} \dot C_{1;k,x}^{(s-1)}\otimes (e_{xl}e_{pj}+\delta_{xj}e_{pl}-\delta_{pl}e_{xj}) - \dot C_{1;k,x}^{(s-1)}\dot H_{2;i,p}^{(r-1)}  \otimes e_{xl}e_{pj}\\
       &=[\dot H_{2;i,p}^{(r-1)}, \dot C_{1;k,x}^{(s-1)}]\otimes e_{xl}e_{pj} + \dot H_{2;i,p}^{(r-1)}\dot C_{1;k,j}^{(s-1)}\otimes e_{pl} - \dot H_{2;i,l}^{(r-1)}\dot C_{1;k,x}^{(s-1)}\otimes e_{xj}.
\end{align*}
Similarly, we have
\begin{align*}
[C,Y]=&[\dot H_{2;q,j}^{(r-1)}, \dot C_{1;k,x}^{(s-1)}]\otimes e_{qi}e_{xl} - \dot C_{1;k,x}^{(s-1)} \dot H_{2;l,j}^{(r-1)} \otimes e_{xi} + \dot C_{1;k,i}^{(s-1)} \dot H_{2;q,j}^{(r-1)}\otimes e_{ql}\\
[D,Y]=&[\dot H_{2;p,q}^{(r-2)}, \dot C_{1;k,x}^{(s-1)}]\otimes e_{pi}e_{xl}e_{qj} 
            +  \dot H_{2;p,q}^{(r-2)} \dot C_{1;k,j}^{(s-1)} \otimes e_{pi}e_{ql}
            -  \dot H_{2;p,l}^{(r-2)} \dot C_{1;k,x}^{(s-1)} \otimes e_{pi}e_{xj}\\
         &   + \dot C_{1;k,i}^{(s-1)}\dot H_{2;p,q}^{(r-2)}\otimes e_{pl}e_{qj}
               - \dot C_{1;k,x}^{(s-1)}\dot H_{2;l,q}^{(r-2)}\otimes e_{xi}e_{qj}.
\end{align*}
All of the brackets appearing above can be computed using (\ref{h2ccoef}), while those non-bracket terms can either be collected, or cancel certain terms in the results according to their corresponding factor in $\rU(\gl_t)$. 
For example, the term $\dot H_{2;i,p}^{(r-1)}\dot C_{1;k,j}^{(s-1)}\otimes e_{pl}$ in $[B,Y]$ can be collected, and it will cancel the term $a=0$, in the first summation of the result of $[A,Y]$,
while the term $\dot C_{1;k,x}^{(s-1)} \dot H_{2;l,j}^{(r-1)} \otimes e_{xi}$ in $[C,Y]$ will be the term $a=-1$ in the second summation of the result of $[C,X]$. After suitable changing of indices in the summations, we list the results as follows:
\begin{align}\notag
&[A,X]+[A,Y]+[B,X]+[B,Y]-[C,X]-[C,Y]-[D,X]-[D,Y]\\
\notag& =\Big( -\sum_{a=0}^{r-1} \dot H_{2;i,l}^{(r-1-a)}\dot C_{1;k,j}^{(s+a)} + \sum_{a=0}^{r-1}(-1)^a \dot C_{1;k,i}^{(s+a)}\dot H_{2;l,j}^{(r-1-a)} \Big) \otimes 1\\
+\notag& \Big( -\sum_{a=0}^{r-2} \dot H_{2;i,x}^{(r-2-a)}\dot C_{1;k,j}^{(s+a)} + \sum_{a=0}^{r-2}(-1)^{a+1} \dot C_{1;k,i}^{(s+a)}\dot H_{2;x,j}^{(r-2-a)} \Big) \otimes e_{xl}\\
+\notag& \Big( -\sum_{a=0}^{r-1}\dot  H_{2;i,l}^{(r-1-a)}\dot C_{1;k,p}^{(s-1+a)} + \sum_{a=0}^{r-1}(-1)^{a+1} \dot C_{1;k,i}^{(s-1+a)}\dot H_{2;l,p}^{(r-1-a)} \Big) \otimes e_{pj}
             + \dot C_{1;k,i}^{(s-1)}\dot H_{2;l,p}^{(r-1)}\otimes e_{pj}  \\
+\notag& \Big( -\sum_{a=0}^{r-3} \dot H_{2;i,x}^{(r-3-a)}\dot C_{1;k,p}^{(s+a)} + \sum_{a=0}^{r-3}(-1)^{a+1} \dot C_{1;k,i}^{(s+a)}\dot H_{2;x,p}^{(r-3-a)} \Big) \otimes e_{xl}e_{pj}
              - \dot  H_{2;i,x}^{(r-2)}  \dot C_{1;k,p}^{(s-1)} \otimes e_{xl}e_{pj}   \\
+\notag& \Big( \sum_{a=0}^{r-1} \dot H_{2;q,l}^{(r-1-a)}\dot C_{1;k,j}^{(s-1+a)} - \sum_{a=0}^{r-1}(-1)^{a+1} \dot C_{1;k,q}^{(s-1+a)}\dot H_{2;l,j}^{(r-1-a)} \Big) \otimes e_{qi}
             - \dot  H_{2;q,l}^{(r-1)}  \dot C_{1;k,j}^{(s-1)} \otimes e_{qi}  \\
+\notag& \Big( \sum_{a=0}^{r-3} \dot H_{2;q,x}^{(r-3-a)}\dot C_{1;k,j}^{(s+a)} - \sum_{a=0}^{r-3}(-1)^{a+1} \dot C_{1;k,q}^{(s+a)}\dot H_{2;x,j}^{(r-3-a)} \Big) \otimes e_{qi}e_{xl}
              - \dot C_{1;k,q}^{(s-1)}\dot H_{2;x,j}^{(r-2)}\otimes e_{qi}e_{xl}  \\
+\notag& \Big( \sum_{a=0}^{r-2}\dot  H_{2;p,l}^{(r-2-a)}\dot C_{1;k,q}^{(s-1+a)} - \sum_{a=0}^{r-2}(-1)^{a+1} \dot C_{1;k,p}^{(s-1+a)}\dot H_{2;l,q}^{(r-2-a)} \Big) \otimes e_{pi}e_{qj}\\              
+& \Big( \sum_{a=0}^{r-3}\dot  H_{2;p,x}^{(r-3-a)}\dot C_{1;k,q}^{(s-1+a)} - \sum_{a=0}^{r-3}(-1)^a \dot C_{1;k,p}^{(s-1+a)}\dot H_{2;x,q}^{(r-3-a)} \Big) \otimes e_{pi}e_{xl}e_{qj}.    \label{dhblhs}        \end{align}

On the other hand, using (\ref{h2ccoef}), we have
\begin{align*}
\Delta_R & ([H_{2;i,j}^{(r)}, C_{1;k,l}^{(s)}])=-\sum_{a\gge 0} \Delta_R(H_{2;i,l}^{(r-1-a)})\Delta_R(C_{1;k,j}^{(s+a)}) 
 + \sum_{a\gge 0}(-1)^a \Delta_R(C_{1;k,i}^{(s+a)})\Delta_R(H_{2;l,j}^{(r-1-a)})  \\
&= -\sum_{a\gge 0}\big( \dot H_{2;i,l}^{(r-1-a)}\otimes 1+ \dot  H_{2;i,p}^{(r-2-a)}\otimes e_{pl} - \dot H_{2;q,l}^{(r-2-a)}\otimes e_{qi} - \dot H_{2;p,q}^{(r-3-a)}\otimes e_{pi}e_{ql}\big) \\
 &\hspace{9cm} \times \big( \dot C_{1;k,j}^{(s+a)}\otimes 1 + \dot C_{1;k,x}^{(s-1+a)}\otimes e_{xj}  \big)\\
 &+ \sum_{a\gge 0} (-1)^a  \big(   \dot C_{1;k,i}^{(s+a)}\otimes 1 + \dot C_{1;k,x}^{(s-1+a)}\otimes e_{xi} \big) \\ 
& \hspace{2cm} \times \big( \dot H_{2;l,j}^{(r-1-a)}\otimes 1+  \dot H_{2;l,p}^{(r-2-a)}\otimes e_{pj} - \dot H_{2;q,j}^{(r-2-a)}\otimes e_{ql} - \dot H_{2;p,q}^{(r-3-a)}\otimes e_{pl}e_{qj}\big) \\
&=-\sum_{a\gge 0} \Big(\dot  H_{2;i,l}^{(r-1-a)}\dot C_{1;k,j}^{(s+a)}\otimes 1  +  \dot  H_{2;i,l}^{(r-1-a)}\dot C_{1;k,x}^{(s-1+a)}\otimes e_{xj} + \dot H_{2;i,p}^{(r-2-a)} \dot C_{1;k,j}^{(s+a)}\otimes e_{pl} \\
&\hspace{1cm} + \dot H_{2;i,p}^{(r-2-a)}\dot C_{1;k,x}^{(s-1+a)}\otimes e_{pl}e_{xj} 
 -\dot  H_{2;q,l}^{(r-2-a)}\dot C_{1;k,j}^{(s+a)}\otimes e_{qi} - \dot   H_{2;q,l}^{(r-2-a)}\dot C_{1;k,x}^{(s-1+a)}\otimes e_{qi}e_{xj} \\
&\hspace{4cm}  -\dot H_{2;p,q}^{(r-3-a)}\dot C_{1;k,j}^{(s+a)}\otimes e_{pi}e_{ql} 
   - \dot H_{2;p,q}^{(r-3-a)}\dot C_{1;k,x}^{(s-1+a)}\otimes e_{pi}e_{ql} e_{xj}  \Big) \\
     &+\sum_{a\gge 0}(-1)^a\Big( \dot C_{1;k,i}^{(s+a)}\dot H_{2;l,j}^{(r-1-a)}\otimes 1 +  \dot C_{1;k,x}^{(s-1+a)} \dot H_{2;l,j}^{(r-1-a)}\otimes e_{xi} +  \dot C_{1;k,i}^{(s+a)} \dot H_{2;l,p}^{(r-2-a)}\otimes e_{pj} \\
  &\hspace{1cm} + \dot C_{1;k,x}^{(s-1+a)} \dot H_{2;l,p}^{(r-2-a)}\otimes e_{xi}e_{pj} 
     -\dot C_{1;k,i}^{(s+a)}\dot H_{2;q,j}^{(r-2-a)}\otimes e_{ql} -  \dot C_{1;k,x}^{(s-1+a)}\dot H_{2;q,j}^{(r-2-a)}\otimes e_{xi}e_{ql} \\
 &\hspace{5cm} - \dot C_{1;k,i}^{(s+a)}\dot H_{2;p,q}^{(r-3-a)}\otimes e_{pl}e_{qj} - \dot C_{1;k,x}^{(s-1+a)}\dot H_{2;p,q}^{(r-3-a)}\otimes e_{xi}e_{pl}e_{qj}\Big).
\end{align*}
Collecting the terms according to the factor in $\rU(\gl_t)$ and changing the indices when necessary, the above can be written as
\begin{align*}
& \Big( -\sum_{a=0}^{r-1}\dot  H_{2;i,l}^{(r-1-a)}\dot C_{1;k,j}^{(s+a)} + \sum_{a=0}^{r-1}(-1)^a \dot C_{1;k,i}^{(s+a)}\dot H_{2;l,j}^{(r-1-a)} \Big) \otimes 1 \\
+& \Big( -\sum_{a=0}^{r-2} \dot H_{2;i,x}^{(r-2-a)}\dot C_{1;k,j}^{(s+a)} + \sum_{a=0}^{r-2}(-1)^{a+1} \dot C_{1;k,i}^{(s+a)}\dot H_{2;x,j}^{(r-2-a)} \Big) 
\otimes e_{xl}   \\
+& \Big( -\sum_{a=0}^{r-1} \dot  H_{2;i,l}^{(r-1-a)}\dot C_{1;k,p}^{(s-1+a)} + \sum_{a=0}^{r-2}(-1)^a   \dot C_{1;k,i}^{(s+a)}\dot  H_{2;l,p}^{(r-2-a)}\Big) \otimes e_{pj}      \\
+& \Big( \sum_{a=0}^{r-2} \dot  H_{2;q,l}^{(r-2-a)}\dot C_{1;k,j}^{(s+a)} + \sum_{a=0}^{r-1}(-1)^a   \dot C_{1;k,q}^{(s-1+a)}\dot  H_{2;l,j}^{(r-1-a)}\Big) \otimes e_{qi}      \\
+& \Big( \sum_{a=0}^{r-2}\dot  H_{2;p,l}^{(r-2-a)}\dot C_{1;k,q}^{(s-1+a)} + \sum_{a=0}^{r-2}(-1)^a \dot C_{1;k,p}^{(s-1+a)}\dot H_{2;l,q}^{(r-2-a)} \Big) 
\otimes e_{pi}e_{qj}        \\      
+& \Big( -\sum_{a=0}^{r-2}\dot  H_{2;i,x}^{(r-2-a)}\dot C_{1;k,p}^{(s-1+a)} + \sum_{a=0}^{r-3}(-1)^{a+1} \dot C_{1;k,i}^{(s+a)}\dot H_{2;x,p}^{(r-3-a)} \Big) \otimes e_{xl}e_{pj}    \\
+& \Big( \sum_{a=0}^{r-3}\dot  H_{2;q,x}^{(r-3-a)}\dot C_{1;k,j}^{(s+a)} - \sum_{a=0}^{r-2}(-1)^a \dot C_{1;k,q}^{(s-1+a)}\dot H_{2;x,j}^{(r-2-a)} \Big) \otimes e_{qi}e_{xl}    \\          
+& \Big( \sum_{a=0}^{r-3}\dot  H_{2;p,x}^{(r-3-a)}\dot C_{1;k,q}^{(s-1+a)} - \sum_{a=0}^{r-3}(-1)^a \dot C_{1;k,p}^{(s-1+a)}\dot H_{2;x,q}^{(r-3-a)} \Big) \otimes e_{pi}e_{xl}e_{qj}.          
\end{align*}
Note that some summations of certain terms have different upper bounds. After making changes of the indices by $a\mapsto a+1$ or $a\mapsto a-1$, the result will be precisely (\ref{dhblhs}), which shows that (\ref{h2ccoef}) and hence (\ref{h2bcoef}) is preserved by $\Delta_R$. This completes the verification of relations when $\mu=(\mu_1,\mu_2)$. 

Assume now $\mu=(\mu_1,\mu_2,\mu_3)$. Note that \eqref{pr-2} is obviously preserved by $\Delta_R$ since there is only one term with non-central image in the factor of $\rU(\gl_t)$ and it remains to check \eqref{pr-1}.
We first consider the following shift matrix $\eta=(\hat\fks_{i,j})_{1\lle i,j\lle N}$, where the entries are given by \beq\label{hatsij}
\hat\fks_{i,j}:=\begin{cases}
0, & \text{ if either }1\lle i,j\lle \mu_{(2)}  \text{ or } \mu_{(2)}+1\lle i,j\lle N,\\
\fks_{2,3}(\mu), &\text{ otherwise. }\\
\end{cases}
\eeq
By construction, the minimal admissible shape with respect to $\eta$ is $\nu=(\nu_1,\nu_2)=(\mu_{(2)},\mu_3)$. 
It follows from Proposition \ref{ind} that we have the following chain of subalgebras of $\Y_N^+$:
\[
\Y_\mu^+(\sigma) \subseteq  \Y_\mu^+(\eta)=\Y_\nu^+(\eta) \subseteq \Y_N^+
\]
and all generators of $\Y_\mu^+(\sigma)$ belong to $\Y_{\mu}^+(\eta)$ as well.
By our discussion above, we have proved that $^\nu\Delta_R: \Y_\nu^+(\eta)\rightarrow \Y_\nu^+(\dot{\eta})\otimes \rU(\gl_t)$ is a homomorphism since $\nu$ is of length 2.

Take inverse on both sides of the matrix equation \eqref{S=FDE} with respect to $\mu$ and $\nu$. Since the corresponding entries in the southeastern $\mu_3\times\mu_3$ corner of $S^{-1}(u)$ are the same, we have $^\mu\wtl D_{3;i,j}(u)$= $^\nu\wtl D_{2;i,j}(u)$, for $1\lle i,j\lle \mu_3=\nu_2$, which further implies that 
\beq%\label{mutonu}
{^\mu} F_{2;f,g}(u)= {^\nu} F_{1;f,\mu_1+g}(u),\qquad  \forall 1\lle f\lle \mu_3=\nu_2, \quad 1\lle g\lle \mu_2.
\eeq
Therefore 
\begin{align*}
{^\nu}\Delta_R({^\mu} F_{2;f,g}^{(r)}) ={^\nu}\Delta_R ({^\nu} F_{1;f,\mu_1+g}^{(r)})&={^\nu} F_{1;f,\mu_1+g}^{(r)}\otimes 1 - \sum_{p=1}^{\nu_2} {^\nu} F_{1;p,\mu_1+g}^{(r-1)} \otimes e_{pf}\\
&= {^\mu} F_{2;f,g}^{(r)}\otimes 1 - \sum_{p=1}^{\mu_3} {^\mu} F_{2;p,g}^{(r-1)} \otimes e_{pf} = {^\mu}\Delta_R( {^\mu}F_{2;f,g}^{(r)}).
\end{align*}
This means that $^\mu\Delta_R$ and the restriction of $^\nu\Delta_R$ to $\Y_\mu^+(\sigma)$ agree on  ${^\mu}F_{2;f,g}^{(r)}$.

By \eqref{quasid}, \eqref{blockS} and \eqref{quasif}, we have the following matrix identities:
\begin{align}\label{stod}
{^\nu}D_1(u)&={^\nu}S_{1,1}(u) = \left(
\begin{array}{ccc}
{^\mu}S_{1,1}(u)&{^\mu}S_{1,2}(u)\\
{^\mu}S_{2,1}(u)&{^\mu}S_{2,2}(u)
\end{array}
\right),
\\
\label{ftos}{^\mu}S_{2,1}(u)&={^\mu}F_{1}(u) {^\mu}D_{1}(u).
\end{align}
The formulas \eqref{stod} and \eqref{ftos} together imply that 
$$
{^\mu}F_{1;i,j}(u)=\sum_{p=1}^{\mu_1} {^\nu}D_{1;\mu_1+i,p}(u){^\mu} \wtl D_{1;p,j}(u), \quad \forall 1\lle i\lle \mu_2, 1\lle j\lle \mu_1.
$$
Note that for all $1\lle i,j\lle \mu_1$, the elements $^\mu D_{1;i,j}^{(r)}$ are mapped to $^\mu D_{1;i,j}^{(r)}\otimes 1$ by both of ${^\mu}\Delta_R$ and ${^\nu}\Delta_R$ and similarly for ${^\mu}\wtl D_{1;i,j}^{(r)}$. As a consequence,
\begin{align*}
{^\nu}\Delta_R( {^\mu}F_{1;i,j}^{(r)}) & =  {^\nu}\Delta_R(  \sum_{p=1}^{\mu_1} \sum_{x,y\gge 0}^{x+y=r} {^\nu}D_{1;\mu_1+i,p}^{(x)}  {^\mu} \wtl D_{1;p,j}^{(y)}   ) \\
& = (\sum_{p=1}^{\mu_1} \sum_{x,y\gge 0}^{x+y=r} {^\nu}D_{1;\mu_1+i,p}^{(x)}  {^\mu} \wtl D_{1;p,j}^{(y)})\otimes 1
= {^\mu}F_{1;i,j}^{(r)} \otimes 1 = {^\mu}\Delta_R ( {^\mu}F_{1;i,j}^{(r)} ) .
\end{align*}
This means that $^\mu\Delta_R$ and the restriction of $^\nu\Delta_R$ to $\Y_\mu^+(\sigma)$ agree on  ${^\mu}F_{1;i,j}^{(r)}$.

Recall that $X^{[r]}(u-c)$ denotes the coefficient of $u^{-r}$ in the expansion of the series $X(u-c)$.
Applying $^\mu\Delta_R$ to the LHS of \eqref{pr-1}, we have
\begin{align*}
&{^\mu}\Delta_R\Big(  \big[ [{^\mu}B_{1;i,j}^{(q)}, {^\mu}B_{2;k,l}^{(r_1)}], {^\mu}B_{2;f,g}^{(r_2)}\big]+\big[ [ {^\mu}B_{1;i,j}^{(q)}, {^\mu}B_{2;k,l}^{(r_2)}], {^\mu}B_{2;f,g}^{(r_1)}\big] \Big) \\ 
&={^\mu}\Delta_R\Big(  \big[[{^\mu}F_{1;i,j}^{[q]}(u-\frac{\mu_1}{2}),{^\mu}F_{2;k,l}^{[r_1]}(v-\frac{\mu_{(2)}}{2}) ,{^\mu}F_{2;f,g}^{[r_2]}(w-\frac{\mu_{(2)}}{2}) \big]  + \\
&\hspace{5cm} \big[[{^\mu}F_{1;i,j}^{[q]}(u-\frac{\mu_1}{2}) ,{^\mu}F_{2;k,l}^{[r_2]}(w-\frac{\mu_{(2)}}{2}) ],{^\mu}F_{2;f,g}^{[r_1]}(v-\frac{\mu_{(2)}}{2}) \big] \Big).
\end{align*}
The argument above shows that we may replace $^\mu\Delta_R$ by $^\nu\Delta_R$. Since $^\nu\Delta_R$ is a homomorphism, it commutes with all of the brackets. Replacing $^\nu\Delta_R$ by $^\mu\Delta_R$ again, we have proved that \eqref{pr-1} is preserved by $^\mu\Delta_R$.
\end{proof}

\subsection{Proof of Theorem \ref{delR} for AII type}\label{subsec:babypfa2}
\begin{proof}
We assume the convention \eqref{thetadef} in this subsection.
Let $\sigma\neq0$.
One needs to check that $\Delta_R$ preserves the defining relations in Theorem \ref{mainthm} when the minimal admissible shape $\mu$ is of length 2 and 3. We omit the details here, since they are similar to the proof provided in the previous subsection.

It remains to check that the relations \eqref{Zshifteda2} are preserved by $\Delta_R$ when the minimal admissible shape $\mu=(\mu_1,\mu_2)$ is of length 2. This implies that $\fks_{1,2}(\mu)=\fks_{2,1}(\mu)=k$ for some $k\in\bZ_{>0}$. Let $\sigma_k$ denote the corresponding shift matrix and let $\Delta_R^{(k)}$ denote the baby coproduct with domain $\Y_{\mu}^-(\sigma_k)$.

When $k=1$, the relations \eqref{Zshifteda2} are trivially preserved by $\Delta_R^{(1)}$. Hence we have a homomorphism 
\[
\Delta_R^{(1)}:\Y_\mu^-(\sigma_1)\rightarrow \Y_\mu^-(\dot\sigma_1) \otimes \rU(\gl_{\mu_2})=\Y_N^- \otimes \rU(\gl_{\mu_2}).
\] 
For a general $k>1$, note that the following elements 
\[\{ H_{1;i,j}^{(r)},\, H_{2;p,q}^{(r)},\, Z_{1;i,j,p,q}^{(r)}  \, | \, 1\lle i,j\lle \mu_1, \, 1\lle p,q\lle \mu_2, \, r\gge 1\}\] 
of $\Y_\mu^-(\sigma_k)$ can be identified as the elements sharing the same notations in $\Y_\mu^-(\sigma_1)$.
By definition, $\Delta_R^{(k)}$ and $\Delta_R^{(1)}$ agree on these elements, which implies that \eqref{Zshifteda2} are preserved by $\Delta_R^{(k)}$. It then follows that $\Delta_R^{(k)}$ is a homomorphism.
\end{proof}

\subsection{Some consequences}
\begin{rem}
It can be observed from the proof that the baby coproduct $\Delta_R$ is defined for any admissible shape $\nu=(\nu_1,\ldots,\nu_m)$ such that $\nu_m=\mu_n$.
\end{rem}

\begin{lem}[cf. {\cite[Lem. 4.2]{BK05}}] \label{drbba}
Retain the same assumption as Theorem \ref{delR}.
For all $1\lle a<b-1\lle n-2$ and admissible $i,j,r$, we have
\begin{align*}
\Delta_R(B_{b,a;i,j}^{(r)})= \dot B_{b,a;i,j}^{(r)} \otimes 1,
\end{align*}
and
\[
\Delta_R(B_{n,a;i,j}^{(r)})=  [\dot B_{n-1;i,h}^{(\fks_{n,n-1}(\mu)+1)},\dot B_{n-1,a;h,j}^{(r-\fks_{n,n-1}(\mu))}]\otimes 1 
- \sum_{p=1}^t \dot B_{n,a;p,j}^{(r-1)} \otimes e_{p'i'}\theta_p\theta_i,
\]
for any $1\lle h\lle \mu_{n-1}$.
\end{lem}
\begin{proof}
Recall from (\ref{x2}) that $B_{n,a;i,j}^{(r)}:=[B_{n-1;i,h}^{(\fks_{n,n-1}(\mu)+1)},B_{n-1,a;h,j}^{(r-\fks_{n,n-1}(\mu))}]$ for any $1\lle h\lle \mu_{n-1}$. By definition, $\Delta_R(B_{n-1,a;h,j}^{(r-\fks_{n,n-1}(\mu))})=\dot B_{n-1,a;h,j}^{(r-\fks_{n,n-1}(\mu))}\otimes 1$. Then
\begin{align*}
\Delta_R(B_{n,a;i,j}^{(r)})&= [\dot B_{n-1;i,h}^{(\fks_{n,n-1}(\mu)+1)}\otimes 1 
- \sum_{p=1}^t \dot B_{n-1;p,h}^{(\fks_{n,n-1}(\mu))}\otimes e_{p'i'}\theta_p\theta_i \, , \, \dot B_{n-1,a;h,j}^{(r-\fks_{n,n-1}(\mu))} \otimes 1]\\
&= [\dot B_{n-1;i,h}^{(\fks_{n,n-1}(\mu)+1)},\dot B_{n-1,a;h,j}^{(r-\fks_{n,n-1}(\mu))}]\otimes 1 
- \sum_{p=1}^t [ \dot B_{n-1;p,h}^{(\fks_{n,n-1}(\mu))} , \dot B_{n-1,a;h,j}^{(r-\fks_{n,n-1}(\mu))} ]  \otimes e_{p'i'}\theta_p\theta_i\\
&=  [\dot B_{n-1;i,h}^{(\fks_{n,n-1}(\mu)+1)},\dot B_{n-1,a;h,j}^{(r-\fks_{n,n-1}(\mu))}]\otimes 1 
- \sum_{p=1}^t \dot B_{n,a;p,j}^{(r-1)} \otimes e_{p'i'}\theta_p\theta_i.\qedhere
\end{align*}
\end{proof}

\begin{prop}\label{delRinj}
The baby comultiplication $\Delta_R:\Y_\mu^\pm(\sigma)\longrightarrow \Y_\mu^\pm(\dot\sigma)\otimes \rU(\gl_t)$ is injective.
\end{prop}
\begin{proof}
Define $\varepsilon:\rU(\gl_t) \to \bC$ by $\varepsilon(e_{i,j})=0$ for all $1\lle i,j\lle t$. Then the composition 
\[
(1\bar\otimes \varepsilon)\circ \Delta_R:\Y_\mu^\pm(\sigma)\to \Y_{\mu}^\pm(\dot\sigma)
\] coincides with the natural embedding $\Y_\mu^\pm(\sigma)\hookrightarrow \Y_{\mu}^\pm(\dot\sigma)$, showing that $\Delta_R$ is injective.
\end{proof}

The canonical filtration \eqref{cfilter:tY} on $\Y_\mu^\pm(\sigma)$ can be extended to a filtration on $\Y_\mu^\pm(\sigma)\otimes \rU(\gl_t)$ by setting the degree of $e_{ij}\in \rU(\gl_t)$ to be 1.
As a consequence of Lemma \ref{drbba}, the baby comultiplication $\Delta_R$ defined in Theorem \ref{delR} is a filtered map. By exactly the same argument in Proposition \ref{delRinj} one deduces the following result.
\begin{prop}\label{grdelRinj}
The associated graded map $\mathrm{gr}' \Delta_R: \gr'\Y_\mu^\pm(\sigma) \rightarrow \gr'\Y_\mu^\pm(\dot\sigma)\otimes \rU(\gl_t)$ is injective.
\end{prop}

\section{The associated graded of  $\Y_\mu^{-}(\sigma)$ via Dirac reduction} 
\label{sec:DiracII}
This section is a type AII counterpart of \cite[\S3]{TT24}. The latter treated the type AI. We show that the associated graded algebra of a shifted twisted Yangian with respect to the canonical filtration can be identified as Poisson algebras with a Dirac reduction of its untwisted counterpart. Our proofs shall be brief whenever we can refer to {\em loc. cit.} for parallel arguments. 

\subsection{A Poisson presentation of $\gr' \Y_\mu^{-}(\sigma)$}

Recall the definition of the canonical filtration from \eqref{canofa2}. As shown after \eqref{canofa2} the associated graded algebra $\gr' \Y_\mu^{-}(\sigma)$ is commutative.

We take a brief digression to discuss presentations of Poisson algebras. The free Poisson algebra generated by a set $X$ is the symmetric algebra $S(L_X)$ of the free Lie algebra $L_X$ generated by the set $X$. A {\it Poisson presentation} of a Poisson algebra $A$ is an exact sequence of the form
$$I \longrightarrow S(L_X) \longrightarrow A \longrightarrow 0,$$
where $I \subseteq S(L_X)$ is a Poisson ideal. The generators of the Poisson ideal will be be referred to as {\it Poisson relations}. In the sequel it will be important to distinguish between generators of $A$ viewed as a commutative algebra and as a Poisson algebra, and so we emphasize this terminology whenever any ambiguity is possible. We refer the reader to \cite[\S3.1]{To23} for slightly more detail. As $\gr' \Y_{N,\ell}^{\pm}(\sigma)$ is commutative it carries a natural Poisson structure.

\begin{lem}\label{grpresAII}
The associated graded of the shifted twisted Yangian, $\gr' \Y_\mu^{-}(\sigma)$, is Poisson generated by the elements 
$$
\{h_{a;i,j}^{(r)},\wtl h_{a;i,j}^{(r)}\}_{1\lle a\lle n,1\lle i,j,\lle \mu_a,r\gge 0},
\qquad 
\{b_{a;i,j}^{(r)}\}_{1\lle a<n,1\lle i\lle \mu_{a+1},1\lle j\lle \mu_a,r>\mathfrak{s}_{b,a}(\mu)}$$ 
subject to the relations: 
\begin{align}
h_{a;i,j}^{(0)}=\delta_{ij},\quad &\sum_{p=1}^{\mu_a}\sum_{t=0}^r h_{a;i,p}^{(t)}\wtl h_{a;p,j}^{(r-t)} =\delta_{r0}\delta_{ij},\quad h_{1;1,1}^{(2r-1)}+h_{1;1',1'}^{(2r-1)}=0, \label{a2screl1}\\
z_{a;i,i,j,j}^{(2r-1)}&+z_{a;i',i',j',j'}^{(2r-1)}=0, \qquad 1\lle r\lle \fks_{a+1,a}(\mu),\label{a2screl2} \\
\{h_{a;i,j}^{(r)},h_{b;k,l}^{(s)}\} &= \delta_{ab}\Big(\sum_{t=0}^{r-1}\big(h_{a;k,j}^{(r-1-t)}h_{a;i,l}^{(s+t)}-h_{a;k,j}^{(s+t)}h_{a;i,l}^{(r-1-t)}\big) \label{a2screl3}\\
&\quad -\sum_{t=0}^{r-1}(-1)^t\big(\theta_{k}\theta_{j'} h_{a;i,k'}^{(r-1-t)}h_{a;j',l}^{(s+t)}-\theta_i\theta_{l'}h_{a;k,i'}^{(s+t)}h_{a;l',j}^{(r-1-t)} \big),\notag\\
\{h_{a;i,j}^{(r)}, b_{b;k,l}^{(s)} \} & = \delta_{ab}\Big(\sum_{p=1}^{\mu_a}\sum_{t=0}^{r-1} (-1)^t\delta_{jl'}\theta_{l'}\theta_p h_{a;i,p}^{(r-t-1)}b_{a;k,p'}^{(s+t)}  \label{a2screl4}\\
&\qquad ~~ -\sum_{p=1}^{\mu_a}\sum_{t=0}^{r-1}\delta_{il}b_{a;k,p}^{(s+t)}h_{a;p,j}^{(r-t-1)}\Big)\notag\\
&\quad  +\delta_{a,b+1}\Big(\sum_{m=0}^{r-1} (-1)^{m+1}\theta_{j'}\theta_k H_{a;i,k'}^{(r-1-m)}B_{b;j',l}^{(s+m)} + \sum_{t=0}^{r-1} b_{b;i,l}^{(s+t)}h_{a;k,j}^{(r-1-t)}\Big),\notag\\
 \{b_{a;i,j}^{(r)},b_{b;k,l}^{(s)}\}&=0 ,\qquad \text{ if } |a-b|>1 \text{ or if }b=a+1, \text{ and } i\ne l, \label{a2screl5}\\
 \{b_{a;i,j}^{(r)},b_{a;k,l}^{(s)} \}
 &=\sum_{t=1}^{r-1} b_{a;k,j}^{(r+s-1-t)} b_{a;i,l}^{(t)}  -  \sum_{t=1}^{s-1} b_{a;k,j}^{(r+s-1-t)} b_{a;i,l}^{(t)}  
-(-1)^{r}\theta_{i}\theta_j z_{a;j',l,k,i'}^{(r+s-1)}, \label{a2screl6}\\
\{b_{a;i,j}^{(r+1)},b_{a+1;k,l}^{(s)}\}&= \{b_{a;i,j}^{(r)},b_{a+1;k,l}^{(s+1)}\} -\delta_{il}\sum_{q=1}^{\mu_{a+1}} b_{a;q,j}^{(r)}b_{a+1;k,q}^{(s)}, \label{a2screl7}
\end{align}
and the Serre relations
\begin{align}
\big\{\{b_{a;i,j}^{(q)},b_{a+1;k,l}^{(r_1)}\},b_{a+1;f,g}^{(r_2)}\big\}&+\big\{\{b_{a;i,j}^{(q)},b_{a+1;k,l}^{(r_2)}\},b_{a+1;f,g}^{(r_1)}\big \} \label{a2screl8}\\
&=
\delta_{il}((-1)^{r_1}+(-1)^{r_2})\theta_{k}\theta_l\{z_{a+1;l',g,f,k'}^{(r_1+r_2-1)},b_{a;i,j}^{(q)}\},\notag \\
\big\{\{b_{a+1;i,j}^{(q)},b_{a;k,l}^{(r_1)}\},b_{a;f,g}^{(r_2)}\big\}&+\big\{\{b_{a+1;i,j}^{(q)},b_{a;k,l}^{(r_2)}\},b_{a;f,g}^{(r_1)}\big\} \label{a2screl9}\\\
&=
\delta_{jk}((-1)^{r_1}+(-1)^{r_2})\theta_{k}\theta_l\{z_{a;l',g,f,k'}^{(r_1+r_2-1)},b_{a+1;i,j}^{(q)}\}.\notag
\end{align} 
Here, we define $z_{a;i,j,k,l}^{(r)} :=\sum_{t=0}^r \tilde{h}_{b;i,j}^{(t)} h_{b+1;,k,l}^{(r-t)}$.
\end{lem}
\begin{proof}
Let $\widehat{\bf y}^{-}_\mu(\sigma)$ be the algebra with generators and relations as in the statement of the lemma. From Theorem \ref{mainthm} and Lemma \ref{Za2} it follows that there is a canonical surjection $\widehat{\bf y}^{-}_\mu(\sigma) \to \gr' \Y_\mu^{-}(\sigma)$. We inductively define the following elements in  $\widehat{\bf y}^{-}_\mu(\sigma)$: 
\begin{align}
\begin{split}
{}^\sigma b_{a+1,a,i,j}^{(r)} &:= b_{a,i,j}^{(r)},  \  \text{ for } \ 1< i \lle n, \ r > s_{i+1,i}, \\
{}^\sigma b_{b, a,i,j}^{(r)} &:= \{{}^\sigma b_{b, b-1,i,k}^{(\mathfrak{s}_{b,b-1}(\mu)+1)}, {}^\sigma b_{b-1,a,k,j}^{(r-(\mathfrak{s}_{b,b-1}(\mu))}\}, \  \text{ for } \   1\lle j < i \lle n, \ r > s_{i+1,j},\\
{}^{\sigma}b_{a,a,i,j}^{(r)} &:=h_{a,i,j}^{(r)}, \ \text{ for }  \ 1\lle i \lle n, \ r > 0,
\end{split}
\end{align}
and let ${}^{\sigma}b_{a,b,i,j}^{(r)}:=(-1)^{(r+i+j)}{}^{\sigma}b_{b,a,j',i'}^{(r)}$ for $a < b$.
    From here on the proof is exactly the same as the proof of \cite[Prop. 3.4]{TT24}
    using the proof of Theorem \ref{mainthm} instead of the proof of \cite[Thm. 5.11]{LWZ23} and Theorem \ref{a2shPBWgauss} instead of \cite[Lem. 3.3]{TT24}.
\end{proof}

\subsection{Relation to Dirac reduction}

In \cite{TT24} a certain Poisson algebra $R(^{\mu}y_N(\sigma),\tau)$, known as the Dirac reduction of the shifted Yangian, was introduced. This algebra depends on a parameter $\beta=\pm 1$ and an involution $\tau$ of $\gl_N$.

 In \cite[Thm. 5.13]{TT24} a Poisson presentation for $R(^{\mu}y_N(\sigma),\tau)$ was given in terms of generators \begin{equation}\label{diryanggens}
\begin{split}
&\{\eta_{a;i,j}^{(r)} | {1 \lle a \lle m, 1 \lle i,j \lle \nu_a,
r > 0}\},\\
&\{\theta_{a;i,j}^{(r)} | {1 \lle a < m, 1 \lle i 
\lle \nu_a, 1 \lle j \lle \nu_{a+1},
r > s_{a,a+1}(\nu)}\}.
\end{split}
\end{equation}
We refer the reader to {\it loc. cit.} for a full list of relations.

We will need two more relations which were not recorded in \cite{TT24}.
\begin{lem}
    Let $\widetilde{\eta}_{a,i,j}^{(0)}=\delta_{i,j}$ and $\widetilde{\eta}_{a,i,j}^{(r)}=-\sum_{g=1}^{\nu_a} \sum_{t=1}^{r} \eta_{a,i,g}^{(t)} \widetilde{\eta}_{a,g,j}^{(r-t)}$.
    We then have 
    \begin{align}
    \{\widetilde{\eta}_{a,i,j}^{(r)}, \theta_{b,h,k}^{(s)}\}
    &=
  \frac{1}{2} \Big( \delta_{a,b+1}\delta_{i,k}\sum_{t=0}^{r-1}\sum_{g=1}^{\nu_a} 
\theta_{b;h,g}^{(r+s-1-t)}\widetilde{\eta}_{a;g,j}^{(t)}-\delta_{a,b} 
 \sum_{t=0}^{r-1}
 \theta_{a;i,k}^{(r+s-1-t)}\widetilde{\eta}_{a;h,j}^{(t)} \Big) 
 \label{etatildetheta} \\
\notag 
   & \quad + \frac{\beta^{i+j}}{2}(-1)^{r}\Big( \delta_{a,b+1}\delta_{j',k}\sum_{t=0}^{r-1}\sum_{g=1}^{\nu_a} 
\theta_{b;h,g}^{(r+s-1-t)}\widetilde{\eta}_{a;g,i'}^{(t)}-\delta_{a,b} 
 \sum_{t=0}^{r-1}
 \theta_{a;j',k}^{(r+s-1-t)}\widetilde{\eta}_{a;h,i'}^{(t)}\Big),
    \\
    \label{tildesym}
\widetilde{\eta}_{a;i,j}^{(r)} &=\beta^{i+j}(-1)^{r}\widetilde{\eta}_{a;j',i'}^{(r)}. 
\end{align}    
\end{lem}
\begin{proof}
    The first relation is proven exactly the same way as \cite[(5.78)]{TT24} using the relations \cite[(5.61)-(5.62)]{TT24}. For the second relation \eqref{tildesym} we induct on $r$. The base case for $r=0$ holds since $\widetilde{\eta}_{a,i,j}^{(0)} = \delta_{i,j}$ by the definition (cf. \cite[Theorem 5.13]{TT24}). Now assume the relation holds for $l<r$. Using \cite[5.76]{TT24} and the induction hypothesis this gives
    \begin{align*}
    \widetilde{\eta}_{a,i,j}^{(r)}=-\sum_{g=1}^{\nu_a} \sum_{t=1}^{r} (-1)^{r}\beta^{i+j}\eta_{a,g',i'}^{(t)} \widetilde{\eta}_{a,j',g'}^{(r-t)}=(-1)^r\beta^{i+j}\widetilde{\eta}_{a,j',i'}^{(r)},
    \end{align*}
    where the last equality can be seen as follows. Define the matrix $\eta_{a}(u) \in R(^{\mu}y_N(\sigma),\tau)[[u^{-1}]]$ where $(\eta_a(u))_{i,j}=\sum_{r=0}^{\infty}\eta_{a,i,j}^{(r)}u^{-r}$. It follows from a short calculation that $\eta_{a}(u)\widetilde{\eta}_a(u)=1$. Similarly letting $\hat{\eta}_{a,i,j}:=-\sum_{g=1}^{\nu_a} \sum_{t=1}^{r} \eta_{a,g,j}^{(t)} \widetilde{\eta}_{a,i,g}^{(r-t)}$ one sees that $\hat{\eta}_a(u)\eta_a(u)=1$. This gives $\hat{\eta}_{a,i,j}^{(r)}=\widetilde{\eta}_{a,i,j}^{(r)}$, as desired.
\end{proof}

As mentioned before the algebra $R(^{\mu}y_N(\sigma),\tau)$ depends on a parameter $\beta=\pm 1$ and in \cite[Thm. 3.8]{TT24} it was proven that the Dirac reduction with parameter $\beta=1$ is Poisson isomorphic to $\gr' \Y_\mu^{+}(\sigma)$; it was conjectured that the Dirac reduction with parameter $\beta=-1$ is Poisson isomorphic to $\gr' \Y_\mu^{-}(\sigma)$, although the precise description of the latter algebra did not exist back then. 

\begin{thm}\label{dirredisoAII}
We have $\gr' \Y_\mu^{-}(\sigma) \cong R(^{\mu}y_N(\sigma),\tau)$ for $\beta=-1$ as graded Poisson algebras.
\end{thm}

\begin{proof}
We prove that by setting $\beta=-1$ in \cite[(5.76)--(5.83)]{TT24}, we obtain the stated isomorphism.

To begin with, we prove the statement for $\mu=(2^{\frac{N}{2}})$. The general statement then follows from \cite[Prop. 5.9]{TT24} and Corollary \ref{a2shapeindcor}, by composing isomorphisms.

We define a map 
 $ \gr' \Y_\mu^{-}(\sigma) \to R(^\mu y_N(\sigma),\tau)$ by  
\begin{equation}
\label{egogg}
\begin{split}
h_{a,i,j}^{(r)} & \mapsto 2^{r}\eta_{a,i,j}^{(r)} \ \textnormal{ for } r>0,   \\
\widetilde{h}_{a,i,j}^{(r)} & \mapsto 2^{r}\widetilde{\eta}_{a,i,j}^{(r)} \ \textnormal{ for } r>0,      \\
b_{a,i,j}^{(r)}& \mapsto (-1)^{\frac{(\mathfrak{s}_{a,a+1}(\mu)+1)}{2}}(-2)^{r}(-1)^{(i+j)}\theta_{a,j',i'}^{(r)} \textnormal{ for } r>\mathfrak{s}_{a,a+1}(\mu).
\end{split}
\end{equation}
It follows from \cite[(5.76)-(5.83)]{TT24} that the map is a homomorphism of Poisson algebras. Recall~\eqref{thetadef}.
Then the first relation in \eqref{a2screl1} follows by definition,  the second relation in \eqref{a2screl1} can be checked using the definition of $\widetilde{\eta}_{a,i,j}^{(r)}$ in \cite[Thm. 5.13]{TT24}, whereas the third relation in \eqref{a2screl1}  can be checked using \cite[(5.76)]{TT24}. 

The relation \eqref{a2screl2} now follows by definition of $z_{a,i,i',j,j'}^{(2r-1)}$, \cite[(5.76)]{TT24} and \eqref{tildesym}. One can check the relations \eqref{a2screl3} and \eqref{a2screl4} by using (5.76)--(5.77) and (5.76)--(5.78) {\em loc. cit.}, respectively. Additionally, one can check the relations \eqref{a2screl5}, \eqref{a2screl6} and \eqref{a2screl7} by using (5.81), (5.79), and (5.80) {\em loc. cit.}, respectively.

Finally we check the Serre relations, \eqref{a2screl8} and \eqref{a2screl9}. When proving the first of these, one can apply the same argument appearing after \eqref{pfN0} to reduce to the case $i \neq g$, and for the latter one can similarly assume $j \neq f$. One can then check the Serre relations, \eqref{a2screl8} and \eqref{a2screl9}, using \eqref{etatildetheta} together with \cite[(5.78), (5.82), (5.83)]{TT24}.

Using Proposition \ref{new-polystY} together with \cite[Cor. 5.16]{TT24} we  conclude that the map \eqref{egogg} sends a basis to a basis. Hence it is an isomorphism.
\end{proof}

%\part{Truncated shifted twisted Yangians and finite $W$-algebras}

\section{Pyramids and truncated shifted twisted Yangians}
\label{sec:truncationI}

In this section, we first recall the notion of pyramids in the special case we need for both type AI and AII. We then formulate the truncated shifted twisted Yangians of both types, and establish their PBW bases.

\subsection{Pyramids} \label{subsec:pyramid}

We recall the notion of {\em pyramid} in a rather special case which we need, which is simply a symmetrized variant of Young diagrams. 
%from \cite{EK05, BK06}. 
This combinatorial object can be used to both label the truncated shifted twisted Yangians, and also to label even nilpotent orbits in classical Lie algebra. This bridges the gap to finite $W$-algebras, which will be discussed later. 

Suppose first that a pair $(\sigma, \ell)$ is given, where $\sigma=(\fks_{i,j})$ is a symmetric shift matrix of size $N$ and $\ell>2\fks_{1,N}$ is a positive integer. Start with a rectangular Young diagram $\Xi$ consisting of $N\times \ell$ boxes, where $N$ is the height and $\ell$ is the width of $\Xi$. We label the rows of $\Xi$ from top to bottom by $1,\ldots, N$ and for each $1\lle i\lle N$ we remove $\fks_{i,N}$ boxes from both of the left-hand-side and the right-hand-side of the $i^{th}$ row of $\Xi$. The resulting diagram is the pyramid corresponding to $(\sigma,\ell)$. Note that any pyramid obtained in this fashion satisfies the following properties:
\begin{itemize}
\item Height $N$ and width $\ell$;
\item The number of boxes in each row are either all even or all odd;
\item The diagram is symmetric with respect to its vertical axis.
\end{itemize}
Conversely, a pyramid satisfying the properties above uniquely determines a pair $(\sigma, \ell)$ as follows. Let $N$ be the height and let $\ell$ be the width of the pyramid, and for each $1\lle i\lle N$ we let $\fks_{i,N}$ denote the number $\frac{\ell-p_i}{2}$, where $p_i$ means the number of boxes in the $i^{th}$ row of the pyramid labelled from top to bottom. The two assignments given above are inverses of each other. Hence we have proved the following bijection.

\begin{prop}\label{pyrabi}
The set of pairs $\{(\sigma, \ell)\}$, where $\sigma$ is a symmetric shift matrix of size $N$ and $\ell$ is a positive integer strictly larger than $2\fks_{1,N}$, is in bijection with the set of symmetric pyramids of height $N$ and width $\ell$ such that all $p_i$ have the same parity.
\end{prop}

For example,
\[
\sigma = \left(\begin{array}{llll}
0&0&2&3\\
0&0&2&3\\
2&2&0&1\\
3&3&1&0
\end{array}\right), \, \ell=8 \qquad \longleftrightarrow \qquad 
\ytableausetup{boxsize=1em}
\ytableausetup{centertableaux}
\begin{ytableau}
\none & \none & \none & ~ & ~ & \none & \none & \none \\
\none & \none & \none & ~ & ~ & \none & \none & \none	 \\
\none & ~ & ~ & ~ & ~ & ~ & ~ & \none	\\
~ & ~ & ~ & ~ & ~ & ~ & ~ & ~	
\end{ytableau}
\]
% $$
% \sigma = \left(\begin{array}{llll}
% 0&0&2&3\\
% 0&0&2&3\\
% 2&2&0&1\\
% 3&3&1&0
% \end{array}\right), \, \ell=8
% \begin{picture}(200, 65)%
% \put(15,15){$\longleftrightarrow$}
% \put(60,-25){\line(1,0){160}}
% \put(60,-5){\line(1,0){160}}
% \put(80,15){\line(1,0){120}}
% \put(120,35){\line(1,0){40}}
% \put(120,55){\line(1,0){40}}
% \put(60,-25){\line(0,1){20}}
% \put(80,-25){\line(0,1){40}}
% \put(100,-25){\line(0,1){40}}
% \put(120,-25){\line(0,1){80}}
% \put(140,-25){\line(0,1){80}}
% \put(160,-25){\line(0,1){80}}
% \put(180,-25){\line(0,1){40}}
% \put(200,-25){\line(0,1){40}}
% \put(220,-25){\line(0,1){20}}
% \end{picture}
% $$
\begin{rem}
Here we follow the conventions in \cite{BK06} regarding the definition of the level. 
%We could have equally well defined a different $\ell$ to be the size of the smallest part in the partition instead (or the size of any other part) and obtained equivalent data. 
We warn the reader that \cite{TT24} uses the convention that the level is the smallest part of the partition instead.
\end{rem}

\begin{dfn}
A partition $\la$ is called orthogonal (resp. symplectic) if each even (resp. odd) part appears with even multiplicity.
\end{dfn}

Via shapes of Jordan blocks nilpotent orbits in orthogonal (resp. symplectic) Lie algebras are parametrized by orthogonal (resp. symplectic) partitions; cf. \cite{Ja04}.
In this article, we focus on orthogonal or symplectic partitions $\la=(\la_1,\ldots,\la_N)$ satisfying the following additional {\em parity assumption}:
\beq\label{pasump}
\la_i \equiv \la_j \pmod{2}, \quad \forall 1\lle i,j\lle N. 
\eeq
The condition \eqref{pasump} implies that the Dynkin grading associated to the $\fksl_2$-triple is always even, see \S\ref{ss:slicesandWalgebras}.

Let $\la=(\la_1,\ldots,\la_N)$ be an orthogonal or symplectic partition satisfying \eqref{pasump}. We assign a unique pyramid to $\la$ as follows. We first draw the corresponding Young diagram associated with $\la$ in French style; that is, the diagram is left-justified and the longest row is placed at the bottom. Then we move the rows of the Young diagram horizontally to obtain a pyramid $\pi(\la)$ which is symmetric with respect to its vertical axis in the middle. 

With $\pi(\la)$, one obtains the corresponding pair $(\sigma,\ell)$ through the bijection in Proposition~ \ref{pyrabi}. 
The pair $(\sigma,\ell)$ will determine a {\it truncated shifted twisted Yangian} which will be defined in the following subsections, where the corresponding type depends not only on the type of partition $\la$ but also on the (parity of) size of its largest Jordan block, which is exactly $\ell$. 
The general principle of the correspondence is that the types are interchanged if and only if $\ell$ is even.
To be explicit, 
\begin{itemize}
\item[(i)] If $\la$ is orthogonal and $\ell$ is odd, then the corresponding twisted Yangian is of type AI;
\item[(ii)] If $\la$ is orthogonal and $\ell$ is even, then the corresponding twisted Yangian is of type AII;
\item[(iii)] If $\la$ is symplectic and $\ell$ is odd, then the corresponding twisted Yangian is of type AII;
\item[(iv)] If $\la$ is symplectic and $\ell$ is even, then the corresponding twisted Yangian is of type AI.
\end{itemize}

This correspondence explains the reason why the even assumption on the composition is natural when defining the twisted Yangian of type AII by parabolic presentation. 
In particular, in the rectangular case, the above correspondence matches that in \cite[Thm.~1.1]{Br09}.

In the remainder of this section, we shall deal separately with the type AI in \S\ref{sec:trshtwa1}--\ref{sec:trpbwa1} and type AII in \S\ref{sec:trshtwa2}--\ref{sec:trpbwa2}.

\subsection{Truncated shifted twisted Yangians of type AI}\label{sec:trshtwa1}

For a symmetric shift matrix $\sigma=(\fks_{i,j})_{1\lle i,j\lle N}$,
take a positive integer $\ell>2\fks_{1,N}$, which will be called the {\em level}.
For each $1\lle i\lle N$, define
\beq\label{rowpi}
p_i:=\ell-2\fks_{i,N}. 
\eeq
If $\mu=(\mu_1,\ldots,\mu_n)$ is an admissible shape of $\sigma$, we set the following shorthand notation for each $1\lle a\lle n$:
\beq\label{pmua}
p_a(\mu):=p_{\mu_{(a)}}.
\eeq
Note that when $\sigma$ is fixed, $\ell$ and $p_1$ determine each other, and all $p_i$ have the same parity as $\ell$.

\begin{dfn}
Let $\mu=(\mu_1,\ldots,\mu_n)$ be the minimal admissible shape of $\sigma$. Furthermore if $\mu_1=1$ introduce the following power series for $1 \lle i \lle \mu_2$
\begin{equation}\label{shiftedbseries}
\mathring{B}_{1;i,1}(u):=\sum_{r > 0}B_{1;i,1}^{(r+\mathfrak{s}_{1,2})}u^{-r}
\end{equation}
and let $\wtl{B}_{1;i,1}^{(\mathfrak{s}_{1,2};p_1+1)}$ be the coefficient of 
$u^{-(p_1+1)}$ in the series 
$\mathring{B}_{1;i,1}(u+\frac{1}{2})H_{1;1,1}(u)$.
More explicitly,
\begin{equation} \label{extragens}
\wtl{B}_{1;i,1}^{(\mathfrak{s}_{1,2};p_1+1)}= \sum_{t=0}^{p_1} \sum_{k=0}^{p_1-t} \binom{p_1-t}{k} (-2)^{-p_1+t+k}B_{1;i,1}^{(k+1+\mathfrak{s}_{1,2})} H_{1;1,1}^{(t)}.   
\end{equation}
Define $I_\ell$ to be the 2-sided ideal of $\Y_N^+(\sigma)$ generated by 
\begin{align}\label{trdef}
\begin{split}
\begin{dcases}
 \{ H_{1;i,j}^{(r)}+2^{-1}H_{1;i,j}^{(r-1)} \mid 1\lle i,j\lle \mu_1, r>p_1 \}, &  \text{ if } \ell \text{ is odd},
 \\
 \{ H_{1;i,j}^{(r)} \mid 1\lle i,j\lle \mu_1, r>p_1 \}, &  \text{ if } \ell \text{ is even and } \mu_1 \gge 2,
 \\
 \{ H_{1;1,1}^{(r)} \mid r>p_1 \} \cup \{\wtl{B}_{1;i,1}^{(\mathfrak{s}_{1,2};p_1+1)} \mid 1 \lle i \lle \mu_2 \}, &  \text{ if } \ell \text{ is even and } \mu_1=1,
\end{dcases}
\end{split}
\end{align}
and define the truncated shifted twisted Yangian to be the quotient 
\beq\label{trtwy}
\Y^+_{N,\ell}(\sigma):= \Y_N^+(\sigma)/I_\ell.
\eeq
\end{dfn}
Abusing notation, we will denote elements of $\Y_N^+(\sigma)$ and their images in the quotient $\Y^+_{N,\ell}(\sigma)$ by the same symbols when there is no possible confusion from the context.
Note that when $\sigma=0$, the quotient $\Y_N^+/I_\ell$ is the twisted Yangian of level $\ell$ considered in \cite{Br09}.

The following result implies that one may choose any admissible shape to define $I_\ell$ and $\Y^+_{N,\ell}(\sigma)$,
and they are independent of the choice of the shape, thus our notation is unambiguous.
\begin{prop}\label{2ideal} 
Let $\nu=(\nu_1,\ldots, \nu_m)$ be an admissible shape with respect to $\sigma$. Then the ideal $I_\ell$ equals the 2-sided ideal of $\Y_N^+(\sigma)$ generated by 
\begin{align}
\begin{dcases}
 \{ H_{1;i,j}^{(r)}+2^{-1}H_{1;i,j}^{(r-1)} \mid 1\lle i,j\lle \nu_1, r>p_1 \}, &  \text{ if } \ell \text{ is odd},
 \\
 \{ H_{1;i,j}^{(r)} \mid 1\lle i,j\lle \nu_1, r>p_1 \}, &  \text{ if } \ell \text{ is even and } \nu_1 \gge 2,
 \\
 \{ H_{1;1,1}^{(r)} \mid r>p_1 \} \cup \{\wtl{B}_{1;i,1}^{(\mathfrak{s}_{1,2};p_1+1)} \mid 1 \lle i \lle \nu_2 \}, &  \text{ if } \ell \text{ is even and } \nu_1=1. 
\end{dcases}
 \label{ideo}
\end{align}
\end{prop}
\begin{proof}
First deal with the case when $\ell$ is even and $\nu_1\gge 2$.
Let $K$ be the 2-sided ideal generated by 
\[\{s_{ij}^{(r)}\, | \, 1\lle i,j\lle 2, r>p_1\}.\]
Note that $K \subseteq I_\ell$.
By \eqref{qua0}, we have
$[s_{12}^{(t)}, s_{2k}^{(1)}]=s_{1k}^{(t)} \in K$ for all $k>2$ and $t>p_1$.
Using \eqref{qua0} again, we have 
$[s_{1k}^{(t)},s_{j1}^{(1)}]=\delta_{jk}s_{11}^{(t)}-s_{jk}^{(t)}\in K$ for all $j>1$ and $t>p_1$, and it follows that $I_\ell \subseteq K$.

Next we consider the case when $\ell$ is odd.
Let $J$ be the ideal generated by $\{\sfh_{0,r} \, | \, r\gge p_1\}$, using the Drinfeld presentation.
It suffices to show $I_\ell=J$, where the general case follows immediately since the generating set of $J$ is contained in the generating set of the 2-sided ideal for any admissible shape.

Clearly $J \subseteq I_\ell$ since $\sfh_{0,r}= s_{11}^{(r+1)}= H_{1;1,1}^{(r+1)}$.
We deduce some general facts regardless of the parity of $\ell$.
Since $H_{1;i,j}^{(t)}=s_{ij}^{(t)}$ for all $1\lle i,j\lle \mu_1$, by \eqref{qua0} we have
\beq\label{s11t}
[s_{11}^{(t)}, s_{kl}^{(1)}]=\delta_{1k}(s_{1l}^{(t)}+s_{l1}^{(t)})-\delta_{1l}(s_{1k}^{(t)}+s_{k1}^{(t)}).
\eeq
It then follows that
\beq\label{s1k}
s_{1k}^{(t)}+s_{k1}^{(t)}\in J, \qquad 2\lle k\lle \mu_1, \, t>p_1.
\eeq
Moreover, the symmetry relation \eqref{sym} implies that 
\beq\label{sijsym}
s_{ij}^{(2t+1)}=-s_{ji}^{(2t+1)}, \qquad    s_{ji}^{(2t+1)}=s_{ij}^{(2t+2)}-s_{ji}^{(2t+2)}.
\eeq
By \eqref{pr3} and the facts that $s_{ij}^{(0)}=\delta_{ij}$, $s_{11}^{(1)}=0$, we have 
\[
[s_{11}^{(2)},s_{k1}^{(t)}+s_{1k}^{(t)}]
=2(s_{1k}^{(t+1)}-s_{k1}^{(t+1)}) + (s_{1k}^{(t)}-s_{k1}^{(t)}) + s_{11}^{(t)}(s_{k1}^{(1)}-s_{1k}^{(1)})+(s_{k1}^{(1)}-s_{1k}^{(1)})s_{11}^{(t)},
\]
and hence $2(s_{1k}^{(t+1)}-s_{k1}^{(t+1)})+(s_{1k}^{(t)}-s_{k1}^{(t)}) \in J$ for all $t>p_1$.
Together with \eqref{s1k}, we have
\[
s_{k1}^{(t+1)}+\frac{1}{2}s_{k1}^{(t)} \in J  \quad \text{and} \quad s_{1k}^{(t+1)}+\frac{1}{2}s_{1k}^{(t)} \in J, \quad \forall t>p_1.
\]

Suppose now $\ell$ is odd.
We need to show that
$H_{1;i,j}^{(t)}+\frac{1}{2}H_{1;i,j}^{(t-1)}\in J$
for all $1 \lle i,j\lle \mu_1$ and $t>p_1$.
Since $p_1$ is also odd, we have
\[s_{1k}^{(p_1+1)}+\frac{1}{2}s_{1k}^{(p_1)}
\stackrel{\eqref{sijsym}}{=}
s_{1k}^{(p_1+1)}+\frac{1}{2}(s_{k1}^{(p_1+1)}-s_{1k}^{(p_1+1)})=\frac{1}{2}(s_{1k}^{(p_1+1)}+s_{k1}^{(p_1+1)})
\stackrel{\eqref{s1k}}{\in}
J.\]
We have proved that $s_{1k}^{(t)}+\frac{1}{2}s_{1k}^{(t-1)}\in J$ for all $t>p_1$.
By \eqref{pr3} again, for $j>1$ we have
\[
[s_{j1}^{(1)},s_{1k}^{(t)}+\frac{1}{2}s_{1k}^{(t-1)}]=(s_{jk}^{(t)}+\frac{1}{2}s_{jk}^{(t-1)})-\delta_{jk}(s_{11}^{(t)}+\frac{1}{2}s_{11}^{(t-1)}) \in J.
\]
If $j=k$, note that $s_{kk}^{(p_1+1)}+\frac{1}{2}s_{kk}^{(p_1)}=2(s_{kk}^{(p_1+2)}+\frac{1}{2}s_{kk}^{(p_1+1)})\in J$ since $s_{kk}^{(p_1)}=s_{kk}^{(p_1+2)}=0$.
As a consequence, we have proved that $s_{jk}^{(t)}+\frac{1}{2}s_{jk}^{(t-1)}\in J$ for all $1\lle j, k \lle \mu_1$ with $t>p_1$, completing the proof when $\ell$ is odd.

Now assume $\ell$ is even and $\mu_1=\nu_1=1$, and let
$M$ be the $2$-sided ideal generated by 
\beq\label{genM}
\{H_{1;1,1}^{(r)} \, | \, r>p_1 \} \cup \{\wtl{B}_{1;1,1}^{(\mathfrak{s}_{1,2}+1+p_1)}\}.
\eeq
Again we have $M \subseteq I_{\ell}$ since its generating set is contained in $I_\ell$.
Note that $M$ and $I_\ell$ are exactly the same if $\mu_2=\nu_2$, thus we only need to show $I_\ell\subseteq M$, assuming $\nu_2< \mu_2$. 
Now use the explicit description \eqref{extragens} together with \eqref{pr3} and \eqref{pr4} the elements
\beq\label{tildesequal}
[H_{2;i,1}^{(1)},\wtl{B}_{1;1,1}^{(\mathfrak{s}_{1,2}+1+p_1)}]=\wtl{B}_{1;i,1}^{(\mathfrak{s}_{1,2};p_1+1)}
\eeq
are in $M$ for all $2\lle i\lle \mu_2$, proving that $I_\ell\subseteq M$.

Finally we show that when $\ell$ is even and $\mu_1\gge 2$, the ideals defined by \eqref{trdef} and \eqref{ideo} coincide. 
Let $\mathcal I$ be the ideal generated by \eqref{ideo} with respect to some admissible shape $\nu$ with $\nu_1=1$ and let $I_\ell$ be the ideal defined by \eqref{trdef}.
Note that $\mu_1\gge 2$ implies that $\mathfrak{s}_{1,2}=0$, and the admissible condition forces that $\nu_1+\nu_2\lle \mu_1$. 
Similar to the previous case, it is enough to take \eqref{genM} as the generating set of $\mathcal I$.
Moreover, by \eqref{sparabolic} we have
\[
\wtl{B}_{1;1,1}^{(\mathfrak{s}_{1,2}+1+p_1)}=\wtl{B}_{1;1,1}^{(1+p_1)}=s_{21}^{(p_1+1)}={}^\mu H_{1;2,1}^{(p_1+1)},
\]
which immediately implies $\mathcal I\subseteq I_\ell$.
On the other hand, with $\wtl{B}_{1;1,1}^{(\mathfrak{s}_{1,2}+1+p_1)}=s_{21}^{(p_1+1)}\in \mathcal I$, by taking the same brackets as in previous discussions, we have $s_{21}^{(r)}= H_{1;2,1}^{(r)}\in \mathcal I$ for all $r>p_1$. This suffices to imply $I_\ell\subseteq \mathcal I$ since we have shown that $I_\ell$ is generated by $\{ H_{1;ij}^{(r)} \, | \, 1\lle i,j\lle 2, r>p_1\}$ in the previous case.
\end{proof}

\begin{cor}\label{trtwywd}
The ideal $I_\ell$ and the truncation $\Y^+_{N,\ell}(\sigma)$ are independent of the choice of the admissible shape.
\end{cor}
Even though $\Y^{+}_{N,\ell}(\sigma)$ is independent of the choice of shape we will write $\Y^{+}_{\mu,\ell}(\sigma)$ for $\Y^{+}_{\mu}(\sigma)/I_{\ell}$ if we want to emphasize the choice of a particular shape $\mu$.

\subsection{PBW bases for $\Y^+_{N,\ell}(\sigma)$}\label{sec:trpbwa1}

The natural quotient map $\Y_N^+(\sigma)\twoheadrightarrow \Y^+_{N,\ell}(\sigma)$ carries the canonical filtration \eqref{cFiltershiftedI} on $\Y_N^+(\sigma)$ to the truncated shifted twisted Yangian 
\[
\mathrm F_0(\Y^+_{N,\ell}(\sigma))\subset \mathrm F_1(\Y^+_{N,\ell}(\sigma))\subset \mathrm F_2(\Y^+_{N,\ell}(\sigma))\subset \cdots,\qquad\qquad \Y^+_{N,\ell}(\sigma)=\bigcup_{s\gge 1}\mathrm F_s(\Y^+_{N,\ell}(\sigma)).
\]
Similar to the discussion given in \textsection \ref{canof}, one can declare the elements $H_{a;i,j}^{(t)}$ and $B_{b,a;i,j}^{(t)}$ in $\Y^+_{N,\ell}(\sigma)$ to be of degree $t$ and then $\mathrm F_s(\Y^+_{N,\ell}(\sigma))$ is the span of monomials in these elements of total degree $\lle s$. Due to Proposition~\ref{2ideal}, this is independent of the choice of the admissible shape $\mu$.

For $1\lle a \lle b\lle n$, $1\lle i\lle \mu_b$, $1\lle j\lle \mu_a$ and $ r> \fks_{a,b}(\mu)$, by abusing notation, define elements $\sff_{b,a;i,j}^{(r)}\in\gr'\Y^+_{N,\ell}(\sigma)$
according to the same rule \eqref{grp-img}.
Being a quotient of $\gr'\Y^+_{N}(\sigma)$, the commutative algebra $\gr'\Y^+_{N,\ell}(\sigma)$ is generated by $\{\sff_{b,a;i,j}^{(r)}\}$ as described in Proposition \ref{polystY}.
\begin{lem}\label{pretrpbw}
For any admissible shape $\mu=(\mu_1,\ldots,\mu_n)$, $\gr'\Y^+_{N,\ell}(\sigma)$ is generated by the elements 
\begin{eqnarray}
    \label{grptrPBWgens}
\begin{array}{l}
\{\sff_{a,a;i,i}^{(2r)} \, | \, 1\lle a\lle n, 1\lle i\lle \mu_a, 0<2r\lle p_a(\mu)  \}\\
\hspace{20pt}\cup \{\sff_{a,a;i,j}^{(r)} \, | 1\lle a\lle n, 1\lle j<i\lle \mu_a, 0<r\lle p_a(\mu) \}\\
\hspace{40pt} \cup \{\sff_{b,a;i,j}^{(r)}\,|\, 1\lle a<b\lle n, 1\lle i\lle \mu_b, 1\lle j\lle \mu_a, \fks_{a,b}(\mu)<r \lle \fks_{a,b}(\mu)+p_a(\mu)\}.
\end{array}
\end{eqnarray}
\end{lem}
\begin{proof}
We first prove the statement when $\mu=(1^N)$. 
Define $\g$ by the following combinatorial recipe
\begin{equation}\label{g0}
 \g=\begin{cases}
 \mathfrak{so}_M, & \textnormal{ for } \ell \textnormal{ odd}, \\
  \mathfrak{sp}_M, & \textnormal{ for } \ell \textnormal{ even},
 \end{cases}
\end{equation}
where $M$ is the number of boxes in the pyramid $\pi$ corresponding to the pair $(\sigma,\ell)$.
It then follows from \cite[Thms. 3.8 and 6.2]{TT24} together with \cite[Prop. 4.7, (4.28)]{To23} that there is a graded surjection 
\begin{equation}\label{slicetotryang}
   \mathbb{C}[\cS_\chi]\twoheadrightarrow \gr'\Y^+_{N,\ell}(\sigma),
\end{equation}  
where the grading on $\gr' \Y^{+}_{N,\ell}(\sigma)$ is given by doubling the canonical grading. The $\cS_\chi$ here denotes the Slodowy slice $\kappa(e+\g^f)$ associated to a nilpotent element with partition $\lambda$ obtained from $\ell,\sigma$, where $\kappa: \g \stackrel{\cong}{\rightarrow} \g^*$ is induced from the Killing form. Also, $M$ is the number of boxes in the pyramid associated to $\lambda$ and $\g$ is $\mathfrak{so}_M$ for $\ell$ odd and $\mathfrak{sp}_{M}$ for $\ell$ even. ($ \mathbb{C}[\cS_\chi]$ here was denoted $S(\g_0,e)$ in \cite{TT24}.)
Furthermore, using \cite[Thm. 3.8 and (6.5)]{TT24} it follows that a generating set for $ \mathbb{C}[\cS_\chi]$ gets mapped to the elements in the statement of the lemma. In particular, the elements in the statement generate $\gr'\Y^+_{N,\ell}(\sigma)$ as a commutative algebra and hence the ordered monomials in the following elements span $\Y^+_{N,\ell}(\sigma)$:
\begin{eqnarray}
\label{trPBWgens}
\begin{array}{l}
\{H_{a;i,i}^{(2r)} \, | \, 1\lle a\lle n, 1\lle i\lle \mu_a, 0<2r\lle p_a(\mu)  \}\\
\hspace{20pt}\cup \{H_{a;i,j}^{(r)} \, | 1\lle a\lle n, 1\lle j<i\lle \mu_a, 0<r\lle p_a(\mu) \}\\
\hspace{40pt} \cup \{B_{b,a;i,j}^{(r)}\,|\, 1\lle a<b\lle n, 1\lle i\lle \mu_b, 1\lle j\lle \mu_a, \fks_{a,b}(\mu)<r \lle \fks_{a,b}(\mu)+p_a(\mu)\}.
\end{array}
\end{eqnarray}

 Now let $\nu$ be any  shape which is admissible for $\sigma$. It follows from Corollary \ref{trtwywd} that there is an isomorphism $\varphi:\Y^+_{\mu,\ell}(\sigma) \to \Y^+_{\nu,\ell}(\sigma)$.
Consider the isomorphism 
$\gr \varphi:\gr \Y^+_{\mu,\ell}(\sigma) \to \gr \Y^+_{\nu,\ell}(\sigma)$ (now with respect to the loop filtration).
By the proof of Theorem \ref{pbwstw} and the result for $\mu=(1^N)$ it follows that a spanning set for $\gr \Y^+_{\mu,\ell}(\sigma)$ gets sent to the top filtered components of ordered monomials in \eqref{trPBWgens} for $\gr \Y^+_{\nu,\ell}(\sigma)$. In particular, the ordered monomials in \eqref{trPBWgens} span $\Y^+_{\nu,\ell}(\sigma)$.
\end{proof}

We claim that the baby coproduct $\Delta_R$ defined in Theorem \ref{delR} factors through the quotient and induces a map 
\beq\label{trdelR}
\Delta_R:\Y^+_{N,\ell}(\sigma)\longrightarrow \Y^+_{N,\ell-2}(\dot\sigma)\otimes \rU(\gl_t).
\eeq
This is non-trivial only when $\ell$ is even and $\mu=(1,N-1)$ is the minimal admissible shape. 
In this case, compute the bracket $[H_{1;1,1}^{(2)},\wtl{B}_{1;i,1}^{(p_1+1+\dot{\mathfrak{s}}_{1,2})}]$ using $\eqref{pr4}$ and the fact that $\dot{\mathfrak{s}}_{1,2}={\mathfrak{s}}_{1,2}-1$. From the result one sees the image of $\wtl{B}_{1;i,1}^{(p_1+1+\mathfrak{s}_{1,2})}$ under $\Delta_R$ is contained in $\dot{I}_{\ell-2}\otimes 1$, where $\dot{I}_{\ell-2}$ is the 2-sided ideal defined with respect to $\Y_N^+(\dot{\sigma})$ and level $\ell-2$, and our claim follows.

By the same approach explained in the paragraph prior to Proposition \ref{grdelRinj}, $\Delta_R$ respects the canonical filtration and induces the following graded algebra homomorphism
\beq\label{grtrdelR}
\gr'\Delta_R:\gr'\Y^+_{N,\ell}(\sigma)\longrightarrow \gr' \Y^+_{N,\ell-2}(\dot\sigma)\otimes \rU(\gl_t).
\eeq

\begin{thm}\label{trpbw}
For any admissible shape $\mu=(\mu_1,\ldots,\mu_n)$, $\gr'\Y^+_{N,\ell}(\sigma)$ is the free commutative algebra on generators
\begin{eqnarray*}
& &\{\sff_{a,a;i,i}^{(2r)} \, | \, 1\lle a\lle n, 1\lle i\lle \mu_a, 0<2r\lle p_a(\mu)  \}\\
& & \hspace{20pt}\cup \{\sff_{a,a;i,j}^{(r)} \, |\, 1\lle a\lle n, 1\lle j<i\lle \mu_a, 0<r\lle p_a(\mu) \}\\
& & \hspace{40pt} \cup \{\sff_{b,a;i,j}^{(r)}\,|\, 1\lle a<b\lle n, 1\lle i\lle \mu_b, 1\lle j\lle \mu_a, \fks_{a,b}(\mu)<r \lle \fks_{a,b}(\mu)+p_a(\mu)\}.
\end{eqnarray*}
Moreover, the maps $\Delta_R$ and $\gr'\Delta_R$ defined in \eqref{trdelR} and \eqref{grtrdelR} are injective.
\end{thm}
\begin{proof}
We proceed by induction on $\ell$.
The initial cases $\ell=1$ and $2$ follow from \cite[Thm.~2.3]{Br09} since $I_\ell$ defined in \eqref{trdef} is precisely the kernel of $\kappa_\ell$ described in \cite[(1.12), (1.13)]{Br09}. 
Assume now $\ell\gge 3$ and the statement has been proved for all positive integers smaller than $\ell$.
By Lemma~\ref{pretrpbw}, it suffices to establish the induction step assuming $\mu=(\mu_1,\ldots,\mu_n)$ is the minimal admissible shape.

Set $t=\mu_n$ and the following notation for elements in $\gr'\Y_{N,\ell-2}^+(\dot\sigma)\otimes \rU(\gl_t)$
\beq\label{dotgrp}
\begin{split}
\dot\sff_{b,a;i,j}^{(r)}:=\begin{cases}
\mathrm{gr}^\prime_r \dot{H}_{a;i,j}^{(r)}\otimes 1 & \text{ if } a=b, \\
\mathrm{gr}^\prime_r \dot{B}_{b,a;i,j}^{(r)}\otimes 1 & \text{ if } a<b, \\
\end{cases}
\end{split}
\eeq
and $x_{i,j}:=\mathrm{gr}^\prime_1 1\otimes e_{ij}$ for $1\lle i,j\lle t$.
By Theorem \ref{delR}, Lemma \ref{drbba} and Lemma \ref{pretrpbw}, there exist polynomials $g_{a;i,j}^{(r)}$ in all variables $\dot\sff_{b,a;i,j}^{(r)}$ such that $\gr'\Delta_R$ maps
\beq\label{grfab}
\sff_{b,a;i,j}^{(r)} \mapsto \dot\sff_{b,a;i,j}^{(r)}
\eeq
for $1\lle a\lle b<n$, $1\lle i\lle \mu_b$, $1\lle j\lle \mu_a$ with $\fks_{a,b}(\mu)<r \lle \fks_{a,b}(\mu)+p_a(\mu)$, and
\beq\label{grfan}
\sff_{n,a;i,j}^{(r)} \mapsto \sum_{p=1}^t \dot\sff_{n,a;p,j}^{(r-1)}x_{p,i}+g_{a;i,j}^{(r)}
\eeq
for $1\lle a< n$, $1\lle i\lle t$, $1\lle j\lle \mu_a$ with $\fks_{a,n}(\mu)<r \lle \fks_{a,n}(\mu)+p_a(\mu)$, and
\beq\label{grfnn}
\sff_{n,n;i,j}^{(r)} \mapsto \sum_{p=1}^t \dot\sff_{n,n;i,p}^{(r-1)}x_{p,j}-\sum_{q=1}^t \dot\sff_{n,n;q,j}^{(r-1)}x_{q,i}-\sum_{p,q=1}^t \dot\sff_{n,n;q,p}^{(r-2)}x_{q,i}x_{p,j}+g_{n;i,j}^{(r)}
\eeq
for $1\lle i,j\lle t$, $1\lle r \lle \ell$, where $\dot\sff_{n,n;i,j}^{(0)}:=\delta_{i,j}$ and $\dot\sff_{n,n;i,j}^{(-1)}:=0$.

By the induction hypothesis and the PBW theorem for $\rU(\gl_t)$, the following elements are algebraically independent in $\gr' \Y_{N,\ell-2}^+(\dot\sigma)\otimes \rU(\gl_t)$:
\begin{align*}
&\{ x_{i,j} \, | \, 1\lle i,j\lle t\}, \\
&\{\dot\sff_{a,a;i,i}^{(2r)}\,|\, 1\lle a < n, 1\lle i\lle \mu_a, 0<2r\lle p_a(\mu) \}, \\
&\{\dot\sff_{a,a;i,j}^{(r)} \, | \, 1\lle a < n, 1\lle j<i\lle \mu_a, 0<r \lle p_a(\mu) \}, \\
&\{\dot\sff_{b,a;i,j}^{(r)} \, | \, 1\lle a< b < n, 1\lle i\lle \mu_b, 1\lle j\lle \mu_a, \fks_{a,b}(\mu) <r \lle \fks_{a,b}(\mu)+p_a(\mu) \}, \\
&\{\dot\sff_{n,a;i,j}^{(r)} \, | \, 1\lle a \lle n, 1\lle i\lle t, 1\lle j\lle \mu_a, \fks_{a,n}(\mu)-1<r \lle \fks_{a,n}(\mu)+p_a(\mu)-1 \}, \\
&\{\dot\sff_{n,n;i,i}^{(2r)} \, | \, 1\lle i\lle t, 0< 2r \lle \ell-2 \}. 
\end{align*}
With this fact in hand, together with \eqref{grfab}, \eqref{grfan} and \eqref{grfnn}, one sees that the images of the generators of $\gr^\prime\Y^+_{N,\ell}(\sigma)$ given in Lemma \ref{pretrpbw} under the map $\gr'\Delta_R$ defined in \eqref{grtrdelR} are algebraically independent in $\gr' \Y_{N,\ell-2}^+(\dot\sigma)\otimes \rU(\gl_t)$. Therefore, $\gr'\Delta_R$ is injective and those generators are themselves algebraically independent in $\gr^\prime \Y^+_{N,\ell}(\sigma)$, completing the induction step. 
\end{proof}

\begin{cor}\label{trtwypbw}
For any admissible shape $\mu=(\mu_1,\ldots,\mu_n)$, the monomials in the elements
\begin{eqnarray*}
& & \{H_{a;i,i}^{(2r)} \, | \, 1\lle a\lle n, 1\lle i\lle \mu_a, 0<2r\lle p_a(\mu)  \}\\
& & \hspace{20pt} \cup \{H_{a;i,j}^{(r)} \, | \, 1\lle a\lle n, 1\lle j<i\lle \mu_a, 0<r\lle p_a(\mu) \}\\
& & \hspace{40pt} \cup \{B_{b,a;i,j}^{(r)}\,|\, 1\lle a<b\lle n, 1\lle i\lle \mu_b, 1\lle j\lle \mu_a, \fks_{a,b}(\mu)<r \lle \fks_{a,b}(\mu)+p_a(\mu)\}
\end{eqnarray*}
taken in any fixed linear order form a PBW basis for $\Y^+_{N,\ell}(\sigma)$.
\end{cor}
\begin{cor}\label{loopgradedcentralizerAI}
    We have $\gr \Y^+_{N,\ell}(\sigma) \cong \rU(\g^{e})$. Here $e$ is a nilpotent element with partition $\lambda$ obtained from $(\sigma,\ell)$ as described in \S\ref{subsec:pyramid} and
    $\g$ is defined in \eqref{g0}.
\end{cor}
\begin{proof}
Define $\gl_{N,\ell}[z]^{\theta}(\sigma):=\gl_{N}[z]^{\theta}(\sigma)/J_{\ell}$ where $J_{\ell}$ is the Lie algebra ideal generated by
 \begin{equation}
     \begin{cases}
      \{e_{1,1}\otimes z^r +\theta( e_{1,1} \otimes z^r  ) \mid r \gge p_1 \} \ &\textnormal{ for } \ell \textnormal { odd,}  \\
\{e_{1,1}\otimes z^r +\theta( e_{1,1} \otimes z^r  ) \mid r \gge p_1\} \ \cup \{e_{1,2} \otimes z^{\mathfrak{s}_{1,2}+p_1}+\theta(e_{1,2} \otimes z^{\mathfrak{s}_{1,2}+p_1}) \} &\textnormal{ for } \ell \textnormal { even.}
     \end{cases}
 \end{equation}
 Furthermore, define $\gl_{N,\ell}[z]^{\tau}(\sigma)$ similarly by replacing $\theta$ with $\tau$, where $\tau:\gl_N[z] \to \gl_N[z]$ is given by $\tau(e_{i,j}\otimes t^r)=(-1)^{r-1+\mathfrak{s}_{i,j}}e_{j',i'} \otimes t^r$. Note that both $\theta$ and $\tau$ can be understood as the composition of an involution on $\gl_N$ given by the Cartan involution followed by conjugation by a diagonal matrix with $\pm 1$ on the diagonal and sending $z^r \to (-z)^r$. In particular there exists an isomorphism between $\gl_{N}[z]^{\tau}(\sigma) \to \gl_{N}[z]^{\theta}(\sigma)$ 
 induced by conjugation by a diagonal matrix on 
 $\gl_N$ and it follows that this descends to an isomorphism $\gl_{N,\ell}[z]^{\tau}(\sigma) \to \gl_{N,\ell}[z]^{\theta}(\sigma)$.
 
 It is known from \cite[Lem. 5.5.(2)]{TT24} that there exists an isomorphism $\gl_{N,\ell}[z]^{\tau}(\sigma) \overset{\sim}{\longrightarrow}\g^{e}$. Now it is also clear that there exists a canonical surjection $\rU(\gl_{N,\ell}[z]^{\theta}(\sigma)) \twoheadrightarrow \gr \Y^{+}_{N,\ell}(\sigma) $ in particular there exists a surjection $\rU(\g^{e}) \to \gr \Y^{+}_{N,\ell}(\sigma) $. Examining the PBW basis for $\rU(\g^{e})$ (see for example \cite[Lems. 4.1 and 4.2(iii)]{TT24}), and using theorem \ref{trtwypbw}, we see that this map sends a basis to a basis which concludes the proof. 
\end{proof}
Recall the Slodowy slice $\cS_\chi$. It comes equipped with a natural Poisson structure from Hamiltonian reduction and a natural grading called the Kazhdan grading \cite{GG02}; see \S\ref{ss:slicesandWalgebras} for more details.
\begin{cor}\label{gryangsliceAI}
We have $\gr' \Y^+_{N,\ell}(\sigma) \cong \mathbb{C}[\cS_\chi]$ as  graded Poisson algebras. Here $e$ is a nilpotent element with partition $\lambda$ obtained from $(\sigma,\ell)$ as described in \S\ref{subsec:pyramid}, the grading on $\gr' \Y^{+}_{N,\ell}(\sigma)$ is given by doubling the canonical grading, 
 $\g$ is defined in \eqref{g0}, and $\cS_\chi$ is the Slodowy slice associated to $(\g,e)$.
\end{cor}
\begin{proof}
Without loss of generality fix the shape to be $\mu=(1^N)$. As explained in the proof of Lemma~\ref{pretrpbw} (see \cite[Thms.~ 3.8, 6.2]{TT24}) there is a surjective map of graded Poisson algebras $ \mathbb{C}[\cS_\chi] \twoheadrightarrow \gr'\Y^+_{N,\ell}(\sigma)$. It follows from the proof of \cite[Thm. 6.2]{TT24} that the map sends the basis in Theorem \ref{trpbw} to a basis. Therefore it is an isomorphism.
\end{proof}

\subsection{Truncated shifted twisted Yangians of type AII}
 \label{sec:trshtwa2}
In the following two subsections we establish parallel results to \S\ref{sec:trshtwa1} and \S\ref{sec:trpbwa1} for shifted twisted Yangians of type AII. 
Let $\Y_N^-(\sigma)$ denote the shifted twisted Yangian of type AII associated to the shift matrix $\sigma$ defined in Definition \ref{def:shiftedII} with respect to the following matrix $G$ as given in \eqref{thetadef}:
\[
G = \left(
\begin{array}{cccc}
J_{2} & 0 & \cdots &0\\
0 & J_{2} & \cdots & 0\\
\vdots & \vdots & \ddots & \vdots\\
0 & 0 & \cdots & J_{2}
\end{array}
\right), \quad \text{where}\quad J_2=\begin{pmatrix}
    0 &-1\\
    1 & 0
\end{pmatrix}.
\]

\begin{dfn}\label{idAII}
Let $\mu=(\mu_1,\ldots,\mu_n)$ be even and minimal admissible to $\sigma$, and fix a positive integer $\ell>2s_{1,N}$.
Define $I_\ell$ to be the 2-sided ideal of $\Y_N^-(\sigma)$ generated by 
\begin{align} \label{trdefA2}
\begin{dcases} 
 \big\{ H_{1;i,j}^{(r)} \mid 1\lle i,j\lle \mu_1, r>p_1 \big\}, &  \text{ if } \ell \text{ is even},
  \\
\big\{ H_{1;i,j}^{(r)}-2^{-1}H_{1;i,j}^{(r-1)} \mid 1\lle i,j\lle \mu_1, r>p_1 \big\}, & \text{ if } \ell \text{ is odd},
\end{dcases}
\end{align}
and define the truncated shifted twisted Yangian to be the quotient 
\beq\label{trtwyA2}
\Y_{N,\ell}^-(\sigma):= \Y_N^-(\sigma)/I_\ell.
\eeq
\end{dfn}

We will denote elements of $\Y_N^-(\sigma)$ and their images in the quotient $\Y_{N,\ell}^-(\sigma)$ by the same symbols when there is no possible confusion from the context.
The following result is the analogue of Proposition \ref{2ideal} for type AII.
\begin{prop}\label{2ideala2} 
Let $\nu=(\nu_1,\ldots, \nu_m)$ be an even admissible shape with respect to $\sigma$.
Then $I_\ell$ equals the 2-sided ideal of $\Y_{N}^-(\sigma)$ generated by 
\begin{align*}
\begin{dcases}
\big\{ H_{1;i,j}^{(r)} \mid 1\lle i,j\lle \nu_1, r>p_1 \big\}, &  \text{ if } \ell \text{ is even},
  \\
\big\{ H_{1;i,j}^{(r)}-2^{-1}H_{1;i,j}^{(r-1)} \mid 1\lle i,j\lle \nu_1, r>p_1 \big\}, & \text{ if } \ell \text{ is odd}.
\end{dcases}
\end{align*}
\end{prop}

\begin{proof}
Let $K$ denote the ideal as in the statement of the proposition. Similar to the type AI case, it suffices to show $I_\ell \subseteq K$ when $\nu_1=2<\mu_1$, and the general cases follow. 

Suppose $\ell$ is even. By the parity assumption \eqref{pasump}, every $p_i$ is even and hence they all have even multiplicity since the associated pyramid is obtained from an orthogonal partition. This implies that $\mu_1$ is even and hence $\mu_1\gge 4$.
By definition, $^\mu H_{1;i,j}(u)=s_{i,j}(u)$ for all $1\lle i,j\lle \mu_1$. It now suffices to show $s_{i,j}^{(r)}\in K$ for all $1\lle i,j\lle \mu_1$ and $r>p_1$. 

%From \eqref{sym}, one deduces that  %% This one is not needed
%\beq\label{a2sy1}
%s_{j^\prime i^\prime}^{(2t+1)}=-\theta_i\theta_js_{ji}^{(2t+1)}, \qquad s_{j^\prime i^\prime}^{(2t)}=\theta_i\theta_j( s_{ij}^{(2t)}-s_{ij}^{(2t-1)}).
%\eeq
By \eqref{qua0}, we have
\beq\label{a2qua0}
[s_{ij}^{(r)},s_{kl}^{(1)}]=\delta_{kj}s_{il}^{(r)}-\delta_{il}s_{kj}^{(r)}+\theta_i\theta_{l^\prime}\delta_{i^\prime k}s_{l^\prime j}^{(r)}-\theta_k\theta_{j^\prime}\delta_{j^\prime l}s_{ik^\prime}^{(r)}.
\eeq
Taking $(i,j,l)=(2,2,1)$ in \eqref{a2qua0} and letting $k$ run over $3\lle k \lle \mu_1$, we have $s_{2k}^{(r)}\in K$ for all $1\lle k\lle \mu_1$.
Similarly, taking $(i,j,l)=(1,1,1)$ in \eqref{a2qua0}, we have $s_{k1}^{(r)}\in K$ for all $1\lle k\lle \mu_1$;
taking $(i,j,k)=(1,1,1)$ and letting $l$ vary, we have $s_{1l}^{(r)}\in K$ for all $1\lle l\lle \mu_1$;
taking $(i,j,k)=(2,2,1)$, we have $s_{l1}^{(r)}\in K$ for all $1\lle l\lle \mu_1$.
Finally, taking $j=k=1$ and letting $3\lle i,l \lle \mu_1$, we have $[s_{i1}^{(r)},s_{1l}^{(1)}]=s_{il}^{(r)}-\delta_{i,l}s_{11}^{(r)}\in K$ and hence $s_{il}^{(r)}\in K$ for all $3\lle i,l\lle \mu_1$. This completes the proof when $\ell$ is even.

Suppose now $\ell$ is odd. By \eqref{pasump}, every $p_i$ is odd and hence they all have even multiplicity since the associated pyramid is obtained from a symplectic partition. This implies that $\mu_1$ is even and hence $\mu_1\gge 4$.
By definition, it suffices to check that $^\mu H_{1;ij}^{(r)}-\frac{1}{2} ^\mu H_{1;ij}^{(r-1)}=s_{ij}^{(r)}-\frac{1}{2}s_{ij}^{(r-1)}\in K$ for all $1\lle i,j\lle \mu_1$ and $r>p_1$.
This can be easily deduced from \eqref{a2qua0} similar to the case when $\ell$ is even.
\end{proof}

\begin{cor}\label{trtwywda2}
The ideal $I_\ell$ and the truncation $\Y_{N,\ell}^-(\sigma)$ are independent of the choice of the admissible shape.
\end{cor}

\subsection{PBW bases for $\Y_{N,\ell}^-(\sigma)$}
\label{sec:trpbwa2}

The natural quotient map $\Y_N^-(\sigma)\twoheadrightarrow \Y_{N,\ell}^-(\sigma)$ carries the canonical filtration to the truncated shifted twisted Yangian 
\[
\mathrm F_0(\Y_{N,\ell}^-(\sigma))\subset \mathrm F_1(\Y_{N,\ell}^-(\sigma))\subset \mathrm F_2(\Y_{N,\ell}^-(\sigma))\subset \cdots,\qquad\qquad 
\Y_{N,\ell}^-(\sigma)=\bigcup_{s\gge 1}\mathrm F_s(\Y_{N,\ell}^-(\sigma)).
\]
Similar to previous discussions, we declare the elements $H_{a;i,j}^{(t)}$ and $B_{b,a;i,j}^{(t)}$ in $\Y_{N,\ell}^-(\sigma)$ to be of degree $t$ and then define $\mathrm F_s(\Y_{N,\ell}^-(\sigma))$ to be the span of monomials in these elements of total degree $\lle s$. Due to Corollary~\ref{a2shapeindcor} and Corollary~\ref{trtwywda2}, this is independent of the choice of the admissible shape $\mu$.

Being a quotient of $\gr'\Y_{N}^-(\sigma)$, the commutative algebra $\gr'\Y_{N,\ell}^-(\sigma)$ is generated by the elements $\{\sff_{b,a;i,j}^{(r)}\}$ (actually, their image in the truncation) described in Proposition~ \ref{new-polystY}. 

Observe that the baby coproduct $\Delta_R$ defined in Theorem \ref{delR} factors through the quotient and induces a map on the truncation
\beq\label{trdelR-new}
\Delta_R:\Y_{N,\ell}^-(\sigma)\rightarrow \Y_{N,{\ell-2}}^-(\dot\sigma)\otimes \rU(\gl_t).
\eeq
Similar to AI type, the map \eqref{trdelR-new} respects the canonical filtration, inducing a graded algebra homomorphism 
\beq\label{grtrdelR-new}
\gr^\prime\Delta_R:\gr^\prime\Y_{N,\ell}^-(\sigma)\rightarrow \gr^\prime(\Y_{N,\ell-2}^-(\dot\sigma)\otimes \rU(\gl_t)).
\eeq

\begin{lem}\label{pretrpbw-new}
For any even admissible shape $\mu=(\mu_1,\ldots,\mu_n)$ , $\gr'\Y^-_{N,\ell}(\sigma)$ is generated only by the following elements 
\beq\label{grptrPBWgens-new}
\begin{split}
&\{\sff_{a,a;2i-1,2i-1}^{(r)} \, | \, 1\lle a\lle n, 1\lle i\lle \tfrac{\mu_a}{2}, 1\lle r \lle p_a(\mu)\}\\
\cup &\{\sff_{a,a;2i-1,2i}^{(2r-1)},  \sff_{a,a;2i,2i-1}^{(2r-1)} \, | \, 1\lle a\lle n, 1\lle i\lle \tfrac{\mu_a}{2}, 1\lle 2r-1 \lle p_a(\mu)\}\\
\cup &\{\sff_{a,a;i,j}^{(r)} \, | \, 1\lle a\lle n, 1\lle j<i\lle \mu_a, \lfloor \tfrac{i+1}{2}\rfloor\ne \lfloor \tfrac{j+1}{2}\rfloor, 1\lle r\lle p_a(\mu)\, \}\\
\cup &\{\sff_{b,a;i,j}^{(r)}\, | \, 1\lle a<b\lle n, 1\lle i\lle \mu_b, 1\lle j\lle \mu_a, \fks_{a,b}(\mu)<r \lle \fks_{a,b}(\mu)+p_a(\mu) \}.
\end{split}
\eeq
\end{lem}
\begin{proof}
This can be proved by the same argument as for AI in Lemma \ref{pretrpbw}, except that we use the results for AII established in Section \ref{sec:DiracII} in place of \cite{TT24}.
\end{proof}

\begin{thm}\label{trpbw-new}
For any even admissible shape $\mu=(\mu_1,\ldots,\mu_n)$, $\gr'\Y_{N,\ell}^-(\sigma)$ is the free commutative algebra on generators
\begin{align*}
&\{\sff_{a,a;2i-1,2i-1}^{(r)} \, | \, 1\lle a\lle n, 1\lle i\lle \tfrac{\mu_a}{2}, 1\lle r \lle p_a(\mu)\}\\
\cup &\{\sff_{a,a;2i-1,2i}^{(2r-1)},  \sff_{a,a;2i,2i-1}^{(2r-1)} \, | \, 1\lle a\lle n, 1\lle i\lle \tfrac{\mu_a}{2}, 1\lle 2r-1 \lle p_a(\mu)\}\\
\cup &\{\sff_{a,a;i,j}^{(r)} \, | \, 1\lle a\lle n, 1\lle j<i\lle \mu_a, \lfloor \tfrac{i+1}{2}\rfloor\ne \lfloor \tfrac{j+1}{2}\rfloor, 1\lle r\lle p_a(\mu)\, \}\\
\cup &\{\sff_{b,a;i,j}^{(r)}\, | \, 1\lle a<b\lle n, 1\lle i\lle \mu_b, 1\lle j\lle \mu_a, \fks_{a,b}(\mu)<r \lle \fks_{a,b}(\mu)+p_a(\mu) \}.
\end{align*}
Moreover, the maps $\Delta_R$ and $\gr^\prime\Delta_R$ defined in \eqref{trdelR-new} and \eqref{grtrdelR-new} are injective.
\end{thm}
\begin{proof}
Similar to the proof of Theorem \ref{trpbw}.
\end{proof}

\begin{cor}\label{trtwypbw-new}
For any even admissible shape $\mu=(\mu_1,\ldots,\mu_n)$, the monomials in the elements
\begin{align*}
&\{H_{a;2i-1,2i-1}^{(r)} \, | \, 1\lle a\lle n, 1\lle i\lle \tfrac{\mu_a}{2}, 1\lle r \lle p_a(\mu)\}\\
\cup &\{H_{a;2i-1,2i}^{(2r-1)},  \,\, H_{a;2i,2i-1}^{(2r-1)} \, | \, 1\lle a\lle n, 1\lle i\lle \tfrac{\mu_a}{2}, 1\lle 2r-1 \lle p_a(\mu)\}\\
\cup &\{H_{a;i,j}^{(r)} \, | \, 1\lle a\lle n, 1\lle j<i\lle \mu_a, \lfloor \tfrac{i+1}{2}\rfloor\ne \lfloor \tfrac{j+1}{2}\rfloor, 1\lle r\lle p_a(\mu)\, \}\\
\cup &\{B_{b,a;i,j}^{(r)}\, | \, 1\lle a<b\lle n, 1\lle i\lle \mu_b, 1\lle j\lle \mu_a, \fks_{a,b}(\mu)<r \lle \fks_{a,b}(\mu)+p_a(\mu) \}
\end{align*}
taken in any fixed linear order form a basis for $\Y_{N,\ell}^-(\sigma)$.
\end{cor}
\begin{cor}\label{loopgradedcentralizerAII}
    We have $\gr \Y^-_{N,\ell}(\sigma) \cong \rU(\g^{e})$. Here $e$ is a nilpotent element with partition $\lambda$ obtained from $(\sigma,\ell)$ as described in \S\ref{subsec:pyramid}; $M$ is the number of boxes in the pyramid associated to $e,\ell$, and
    $\g$ is defined according to the following combinatorial recipe:
\begin{equation}\label{g0AII}
 \g=\begin{cases}
 \mathfrak{so}_M, & \textnormal{ for } \ell \textnormal{ even}, \\
  \mathfrak{sp}_M, & \textnormal{ for } \ell \textnormal{ odd}.
 \end{cases}
\end{equation}
\end{cor}
\begin{proof}
Fix the shape $\mu=(2^{\frac{N}{2}})$.  
The proof is completely parallel to the proof of Corollary~ \ref{loopgradedcentralizerAI} with the following differences: $J_\ell$ in this case is given by
 \begin{equation}
      \{e_{i,j}\otimes z^r +\theta( e_{i,j} \otimes z^r  ) \mid r \gge p_1 \ , 1 \lle i,j \lle 2\}, 
 \end{equation}
and $\tau(e_{i,j}\otimes z^r)=(-1)^{i+j+r-1+\mathfrak{s}_{i,j}} e_{j',i'}\otimes z^r$ in this case. The involutions on $\gl_N$ are no longer represented by conjugation by a diagonal matrix but by conjugation by a block diagonal matrix consisting of $2 \times 2$ antisymmetric matrices instead. We can still find an isomorphism $\gl_{N}[z]^{\tau}(\sigma) \to \gl_{N}[z]^{\theta}(\sigma)$ this time  induced by conjugation by a block diagonal matrix with all blocks being $2 \times 2$. Well definedness follows because we assumed $\sigma$ was admissible for $(2^{\frac{N}{2}})$. Everything else in the proof remains the same.
\end{proof}
\begin{cor} \label{gryangsliceAII}
We have $\gr' \Y^{-}_{N,\ell}(\sigma) \cong \mathbb{C}[\cS_\chi]$ as graded Poisson algebras, where the grading on $\gr' \Y^{-}_{N,\ell}(\sigma)$ is given by doubling the canonical grading. Here $e$ is a nilpotent element with partition $\lambda$ obtained from $(\sigma, \ell)$ as described in \S\ref{subsec:pyramid}, $\g$ is as in \eqref{g0AII}
and $\cS_\chi$ is the Slodowy slice associated to $(\g,e)$.
\end{cor}
\begin{proof}
Without loss of generality fix the shape to be $\mu=(2^{\frac{N}{2}})$. The proof is then exactly the same as the one for Corollary \ref{gryangsliceAI} replacing the use of \cite[Thm. 3.6]{TT24} with Theorem \ref{dirredisoAII} and the definition of the truncation ideal with the corresponding Definition \ref{trdefA2} in type AII. Existence of the map follows just as in the type AI case.
\end{proof}

\section{Centers of the truncated shifted twisted Yangians} \label{sec:center}

In this section, under some mild assumptions on $N$ and $\ell$ \eqref{nonpfcases}, we prove that the centers of truncated shifted twisted Yangians $\Y^\pm_{N,\ell}(\sigma)$ are polyonomial algebras in distinguished generators, and conjecture that this statement on the centers holds without these assumptions. The conjecture can be rephrased as the conjectural existence of a Pfaffian generator in the centers of $\Y^\pm_{N,\ell}(\sigma)$ in the cases excluded from \eqref{nonpfcases}, and we construct a Pfaffian generator in a special case.

\subsection{Baby coproduct and sdet}
We take a digression to prepare statements that will be used. 
In general the shift homomorphism $\psi_m$ from Lemma \ref{psilem} is only an algebra homomorphism for extended twisted Yangians rather than for twisted Yangians. Thus for given  $1\lle a\lle n$, the subalgebra generated by coefficients of $H_{a;i,j}(u)$ for $1\lle i,j\lle \mu_a$ is a quotient of the extended twisted Yangian $\X_{\mu_a}^\pm$. Hence one can define the Sklyanin determinant for $H_a(u)$. We would like to calculate the image of $\sdet\, H_a(u)$ under the baby coproduct from Theorem \ref{delR}.

By slightly modifying  $H_a(u)$ in each diagonal block, we can obtain a smaller twisted Yangian. We consider this type by type.

\begin{prop}\label{block=tYI}
Consider $\Y^+_N$ of type AI. 
Given $1\lle a<n$ and $1\lle i\lle \mu_a$, the subalgebra of $\Y_N^+$ generated by $Z_{a;i,i,j,k}^{(r)}$ for $1\lle j,k\lle \mu_{a+1}$ and $r\gge 1$ is isomorphic to $\Y_{\mu_{a+1}}^+$.
\end{prop}
\begin{proof}
Let $\mathscr{A}$ be the subalgebra of $\Y_N^+$ generated by $Z_{a;i,i,j,k}^{(r)}$ for $1\lle j,k\lle \mu_{a+1}$ and $r\gge 1$.

The relation \eqref{pr3} implies that $[H_{a;i,i}(v),H_{a+1;j,k}(u)]=0$ and $H_{a+1;j,k}(u)$ satisfy the quaternary relations \eqref{qua}. Thus by \eqref{tlHH}, $Z_{a;i,i,j,k}(u)$ also satisfy the quaternary relations \eqref{qua}. It follows from Lemma \ref{lem:s11} and \eqref{Z's} that $Z_{a;i,i,j,k}(u)$ also satisfy the symmetry relations \eqref{sym}. Therefore, we have a surjective algebra homomorphism $\Y_{\mu_{a+1}}^+ \twoheadrightarrow \mathscr A$. Therefore, it suffices to show the homomorphism is injective. This follows from the proof of Theorem~ \ref{mainthm} as their associated graded algebras (with respect to the loop filtration) have the same size. 
\end{proof}

Recall $\gamma^-_c(u)$ from \eqref{gammau}.

\begin{prop}\label{block=tYII}
Consider $\Y^-_N$ of type AII. Given $0\lle a< n$, the map
\[
s_{ij}(u)\mapsto \gamma_{-\mu_{(a)}}^-(u) \, \mathcal S_{12\cdots \mu_{(a)},\mu_{(a)+j}}^{1 2 \cdots \mu_{(a)},\mu_{(a)+i}}(u+\tfrac{\mu_{(a)}}{2}),\qquad 1\lle i,j\lle \mu_{a+1}
\]
defines an algebra homomorphism $\Y^-_{\mu_{a+1}}\to \Y^-_N$. Moreover, the subalgebra of $\Y^-_N$ generated by the coefficients of $\mathcal S_{12\cdots \mu_{(a)},\mu_{(a)+j}}^{1 2 \cdots \mu_{(a)},\mu_{(a)+i}}(u)$ for $1\lle i,j\lle \mu_{a+1}$ is isomorphic to $\Y^-_{\mu_{a+1}}$.
\end{prop}
\begin{proof}
The first statement follows from \cite[Prop. 4.1.10]{Mol07}. Thus the subalgebra of $\Y^-_N$ generated by the coefficients of $\mathcal S_{12\cdots \mu_{(a)},\mu_{(a)+j}}^{1 2 \cdots \mu_{(a)},\mu_{(a)+i}}(u)$ for $1\lle i,j\lle \mu_{a+1}$ is a quotient of $\Y^-_{\mu_{a+1}}$. In order to show that this is an isomorphism, it suffices to show that their associated graded algebras under the loop filtration are isomorphic. 

Note that it follows from Lemma \ref{psi lem} and Lemma \ref{lem:shiftt} that
\beq\label{SinH}
\begin{split}
\gamma_{-\mu_{(a)}}^-(u) \, \mathcal S_{12\cdots \mu_{(a)},\mu_{(a)+j}}^{1 2 \cdots \mu_{(a)},\mu_{(a)+i}}(u+\tfrac{\mu_{(a)}}{2})&=\gamma^-_{-\mu_{(a)}}(u) \, \mathcal S_{12\cdots \mu_{(a)}}^{1 2 \cdots \mu_{(a)}}(u+\tfrac{\mu_{(a)}}{2})\psi_{\mu_{(a)}}(s_{ij}(u))\\&=\gamma^-_{-\mu_{(a)}}(u) \, \mathcal S_{12\cdots \mu_{(a)}}^{1 2 \cdots \mu_{(a)}}(u+\tfrac{\mu_{(a)}}{2})H_{a;i,j}(u).
\end{split}
\eeq
Similar to the proof of \cite[Thm 2.8.2]{Mol07}, the graded image of the coefficient of $u^{-r-1}$ in the series $\gamma^-_{-\mu_{(a)}}(u) \, \mathcal S_{12\cdots \mu_{(a)}}^{1 2 \cdots \mu_{(a)}}(u+\tfrac{\mu_{(a)}}{2})H_{a;i,j}(u)$ is given by
\[
(e_{\mu_{(a)}+i,\mu_{(a)}+j}-(-1)^r\theta_i\theta_je_{\mu_{(a)}+j',\mu_{(a)}+i'})\otimes z^r+\delta_{ij}(e_{11}+\cdots+e_{\mu_{(a)},\mu_{(a)}})(1-(-1)^r)\otimes z^r.
\]
On the other hand, the graded image of the coefficient of $u^{-r-1}$ in the series $s_{ij}(u)$ is given by $(e_{i,j}-(-1)^r\theta_i\theta_je_{j',i'})\otimes z^r$. Thus the claim follows.
\end{proof}

Recall $\Delta_R: \Y_\mu^\pm(\sigma)\longrightarrow \Y_\mu^\pm(\dot\sigma)\otimes \rU(\gl_t)$ from Theorem \ref{delR}.
\begin{prop}\label{sdetcomul}
We have
\[
\Delta_R\big( \sdet\,H_n(u) \big)=\sdet\,\dot H_n(u)\otimes \pi_{t}\big(\qdet \,T(u)\,\qdet \,T(-u+t-1)\big).
\]
\end{prop}
\begin{proof}
We first prove it for type AI with the help of Proposition \ref{block=tYI}. Denote by $s^{*}_{ij}(u):=Z_{n-1;1,1,i,j}(u)$, then it follows from Proposition \ref{block=tYI} that the series $s^*_{ij}(u)$ for $1\lle i,j\lle t$ satisfy the defining relations of twisted Yangian $\Y_{t}^+$. 

Recall from \eqref{tlHH} that $s^{*}_{ij}(u)=\wtl H_{n-1;1,1}(u-\tfrac{\mu_{n-1}}2)H_{n;i,j}(u)$. By \eqref{Hdel1} we have
\[
\Delta_R\big(\wtl H_{n-1;1,1}(u-\tfrac{\mu_{n-1}}2)\big)=\wtl H_{n-1;1,1}(u-\tfrac{\mu_{n-1}}2)\otimes 1.
\]
By \eqref{bp:Hm}, we have 
\[
\Delta_R\big(H_{n;i,j}(u)\big)=\sum_{p,q=1}^t H_{n;q,p}(u)\otimes (\delta_{qi}-e_{qi}u^{-1})(\delta_{pj}+e_{p,j}u^{-1}).
\]
Thus
\[
\Delta_R\big(s^*_{i,j}(u)\big)=\sum_{p,q=1}^t s^*_{q,p}(u)\otimes (\delta_{qi}-e_{qi}u^{-1})(\delta_{pj}+e_{p,j}u^{-1}).
\]
This implies that the action of $\Delta_R$ on $s^*_{ij}(u)$ is equivalent to the composition of the coproduct for the twisted Yangian (generated by $s^*_{ij}(u)$) and $\mathrm{id}\otimes \pi_t$, where $\pi_t$ is the evaluation homomorphism for the Yangian $\rY(\gl_t)$. Thus if we set $S^*(u)=(s^*_{ij}(u))$, then by \eqref{sdetcp} we find that
\[
\Delta_{R}\big(\sdet\,S^*(u)\big)=\sdet\,S^*(u)\otimes \pi_{t}\big(\qdet\,T(u)\,\qdet\,T(-u+t-1)\big).
\]
By \eqref{Hdel1} and the fact that $[\wtl H_{n-1;1,1}(v),H_{n;i,j}(u)]=0$, we can move all $\wtl H_{n-1;1,1}(u)$ (possibly with $u$ shifted) in $\sdet S^*(u)$ to the left and cancel. Since $\Delta_R\big(\wtl H_{n-1;1,1}(v)\big)=\wtl H_{n-1;1,1}(v)\otimes 1$, it only affects the first factor of the tensor product. Thus we obtain the desired formula.

For type AII, the argument is almost the same by using Proposition \ref{block=tYII} and the fact that 
$$
[\mathcal S_{12\cdots \mu_{(a)}}^{1 2 \cdots \mu_{(a)}}(v),H_{a;i,j}(u)]=0,
$$ 
which follows from the first equality of Proposition \ref{sdetdecomp}  and \eqref{pr3}. Note that with \eqref{SinH}, one needs to replace $\wtl H_{n-1;1,1}(v)$ in type AI by $\mathcal S_{12\cdots \mu_{(a)}}^{1 2 \cdots \mu_{(a)}}(v)$ in type AII.
\end{proof}

\subsection{Miura transform and sdet}
Throughout this subsection we work with the minimal admissible shape $\mu=(\mu_1, \dots, \mu_n)$ for $\sigma$. We will write $\Y_N^{\pm}(\sigma)$ and, as before, whenever $\pm$ or $\mp$ appears, the upper sign will correspond to type AI and the lower sign to type AII. For type AI we will use the presentation from Theorem \ref{mainthm} following \eqref{a1thetadef}, while for type AII following \eqref{thetadef}.
Recall the respective evaluation homomorphisms $\pi_N: \rY(\gl_N) \to \rU(\gl_N)$ and $ \xi^{\pm}_N:\Y^{\pm}_N \to \rU(\g^{\pm}_N)$ from \eqref{eva} and \eqref{tYev}, respectively. By \cite{Mol07,Br09} we can identify $\Y^{\pm}_{N,1}$ with $\xi^{\pm}_N(\Y_N^\pm)$. Let $p=(p_1, \dots ,p_N)$ be the row lengths of the pyramid $\pi=\pi(\la)$ associated with the pair $(\sigma,\ell)$ starting from the top and define
\begin{equation}
\tilde{p}_i:= \begin{dcases}  \frac{p_i-1}{2}, &\textnormal { for } \ell \textnormal{ odd} \\
\frac{p_i}{2}, & \textnormal { for } \ell \textnormal{ even}.
\end{dcases}
\end{equation}
In particular, let $s=\tp_N$.
Let $q=(q_1,q_2, \dots, q_\ell)$ denote the column heights of the pyramid $\pi$ from left to right. As $\pi$ is symmetric we only record the heights of the columns to the right about the symmetry, thus we define the following shorthand notation for $0\lle i\lle s$: 
%{(note that $p_N=\ell$)}%
\begin{equation}
\tilde{q}_i:= \begin{cases} q_{\frac{\ell+1}{2}+i} \ , & \textnormal{for } 0 \lle i \lle s\textnormal{ and } \ell \textnormal{ odd } \\
 q_{\frac{\ell}{2}+i} \ , & \textnormal{for } 1 \lle i \lle s \textnormal{ and } \ell \textnormal{ even } \\ 
 0,\, & \textnormal{for } i=0 \textnormal{ and } \ell \textnormal{ even}.
\end{cases}
\end{equation}
 Analogously to \cite{BK06} and \cite{Br09} we define the \textit{Miura transform} 
 \begin{equation}
 \chi^{\pm}:\Y_{N,\ell}^{\pm}(\sigma) \to 
 \begin{cases}
 \rU(\g_N^{\pm}) \otimes \rU(\gl_{\tq_1}) \otimes \dots \otimes \rU(\gl_{\tq_s}), &\textnormal{ for } \ell \textnormal{ odd }\\
 \rU(\gl_{N}) \otimes \rU(\gl_{\tq_2}) \otimes \dots \otimes \rU(\gl_{\tq_s}), &\textnormal{ for } \ell \textnormal{ even},
 \end{cases}
 \end{equation}
as the baby comultiplication applied $s$  times. Note that in the first case we have used the fact that $\Y^{\pm}_{N,1}$ can be canonically identified with $\rU(\g^{\pm}_N)$ as the image of the evaluation map for the twisted Yangian by \cite{Br09}. Furthermore, note that $\chi^\pm$ is injective by Theorems \ref{trpbw} and \ref{trpbw-new}. \\

By Theorem \ref{freegen-center}, the even coefficients $c_2,c_4, \dots$ of $\mathcal C_N(u)=\sdet \,S_N(u)$ are algebraically independent generators for the center of $\Y^{\pm}_N$. By Proposition \ref{sdetdecomp} these coefficients also lie in
 $\Y_N^{\pm}(\sigma)$. 
 Furthermore we will write $\mathcal B_N(u):=\textnormal{qdet}\,T_N(u) \in \rY(\gl_N)$.
 From now on by slight abuse of notation we will also write $\mathcal C_N(u)$ for the image of $\mathcal C_N(u)$ in $\Y_{N,\ell}^{\pm}(\sigma)$ and $\mathcal B_N(u)$ for the image of $\mathcal B_N(u)$ under the evaluation map in $\rU(\gl_N)$.
\begin{lem}\label{sdetmiura}
We have 
 \begin{equation}
 \begin{split}
 \chi^\pm(\mathcal{C}_N(u)) =
 \mathcal C_{\tq_0}(u) &\otimes \mathcal B_{\tq_1}(u-\tfrac{N-\tilde{q}_1}{2})\mathcal B_{\tq_1}(-u+\tfrac{N+\tq_1}{2}-1) 
 \\
 \otimes \cdots & \otimes \mathcal B_{\tq_s}(u-\tfrac{N-\tq_s}{2})\mathcal B_{\tq_s}(-u+\tfrac{N+\tq_s}{2}-1).
 \end{split}
 \end{equation}
 \end{lem}
 \begin{proof}
Recall the decomposition 
\begin{align*}
\sdet \, S(u)=\sdet \, D_1(u-\mu_{(0)}) \,\sdet \,D_2(u-\mu_{(1)}) \cdots \sdet \,D_n(u-\mu_{(n-1)})
\end{align*}
from Proposition \ref{sdetdecomp}. Using \eqref{Ha} this can be rewritten as 
\begin{align*}
    \sdet \,S(u)=\sdet \,H_1(u-\tfrac{\mu_{(0)}}{2}) \,\sdet \,H_2(u-\tfrac{\mu_{(1)}}{2}) \cdots \sdet \,H_n(u-\tfrac{\mu_{(n-1)}}{2}).
\end{align*}
It thus follows by Proposition \ref{sdetcomul} and the definition of $\Delta_R$ that we have
\begin{equation}
    \begin{split}
        \Delta_R\big(\mathcal C_N(u)\big)&=\Delta_R\Big(\sdet \,H_1(u-\tfrac{\mu_{(0)}}{2}) \,\sdet \,H_2(u-\tfrac{\mu_{(1)}}{2}) \cdots \sdet \,H_n(u-\tfrac{\mu_{(n-1)}}{2})\Big) \\
        &=\Big(\big(\sdet \,H_1(u-\tfrac{\mu_{(0)}}{2}) \,\sdet \,H_2(u-\tfrac{\mu_{(1)}}{2}) \cdots\sdet\, H_{n-1}(u-\tfrac{\mu_{(n-2)}}{2})\big)\otimes 1\Big) \\
        &\quad\times \Big( \sdet \,H_n(u-\tfrac{\mu_{(n-1)}}{2}) \otimes \mathcal B_{\tilde{q}_s}(u-\tfrac{N-\tilde{q}_s}{2}) \mathcal B_{\tilde{q}_s}(-u+\tfrac{\tilde{q}_s+N}{2}-1)\Big) \\
        &=\mathcal C_N(u) \otimes \mathcal B_{\tilde{q}_s}(u-\tfrac{N-\tilde{q}_s}{2})\mathcal B_{\tilde{q}_s}(-u+\tfrac{\tilde{q}_s+N}{2}-1),
    \end{split}
\end{equation}
where we have used the facts that $\mu_n=\tilde{q}_s$ and $\mu_{(n-1)}=N-\tilde{q}_s$. The statement of the lemma then follows by applying this result repeatedly, along with the definition of the Miura transform.
 \end{proof}
 
 \subsection{Description of central elements}

In this subsection we provide an explicit description of algebraically independent elements of the center $Z(\Y^{\pm}_{N,\ell}(\sigma))$ of $\Y^{\pm}_{N,\ell}(\sigma)$. 

\begin{prop}\label{coeffalgind}
Let $M=\sum_{i=1} ^{N} \lambda_{i}$ and let $m$ be such that $M=2m$ or $M=2m+1$.
The series 
\begin{equation}
\begin{split}
Z_M(u):=
\begin{dcases}
\prod_{i=1}^{s}\prod_{j=1}^{\tq_i}\Big((\tfrac{\tq_i+1}{2}-j)^2-u^2\Big) &
\\
\quad \quad \quad \times\prod_{i=-k}^{k}(u+\rho^{\pm}_i) \Big(\gamma_{N}^{\pm}(u+\tfrac{N-1}{2})\Big)^{-1} \mc 
 C_N(u+\tfrac{N-1}{2}), & \textnormal{ for } \ell \textnormal { odd} 
\\ 
\prod_{i=1}^{s}\prod_{j=1}^{\tq_i}\Big((\tfrac{\tq_i+1}{2}-j)^2-u^2\Big) \  \mc C_N(u+\tfrac{N-1}{2}), & \textnormal{ for } \ell \textnormal { even}
\end{dcases}
\end{split}
\end{equation}
is a polynomial $Z_M(u)=u^M+z_1u^{M-1}+\ldots+z_M$ of degree $M$ in $u$ with coefficients in $\Y^{\pm}_{N,\ell}(\sigma)$. Moreover, the coefficients $z_2,z_4, \ldots,z_{2m}$ are algebraically independent elements of $Z(\Y^{\pm}_{N,\ell}(\sigma))$, of canonical degree $2,4, \ldots, 2m$, respectively. 
\end{prop}
\begin{proof}
It is known from \cite[Thm. 7.1.1]{Mol07} that the series \begin{equation*}
\tilde{\mc B}_N(u):=u(u-1) \dots (u-N+1) \pi_N(\mc B_N(u))
\end{equation*}
is a polynomial $\tilde{\mc B}_N(u)=u^N+\tilde{b}_1u^{N-1}+\ldots+\tilde{b}_N$ of degree $N$ in $u$.
Moreover, the coefficients $\tilde{b}_1, \dots , \tilde{b}_N$ are algebraically independent generators of $Z(\rU(\gl_N))$. The analogous result for $\Y_N^{\pm}$ \cite[Thm. 7.1.6]{Mol07} states the following:
\begin{equation*}
\tilde{\mc C}_N(u):=\prod_{i=-k}^{k}(u+\rho^{\pm}_i) 
\, \Big(\gamma_{N}^{\pm}\big(u+\tfrac{N-1}{2}\big)\Big)^{-1} \, \xi_N^{\pm}\Big(\mc C_N\big(u+\tfrac{N-1}{2}\big)\Big)
\end{equation*}
is a polynomial $\tilde{\mc C}_N(u)=u^N+\tilde{c}_1u^{N-1}+\ldots+\tilde{c}_N$ of degree $N$ in $u$.
Moreover, denoting $N=2n+1$ or $N=2n$, the coefficients $\tilde{c}_2, \ldots , \tilde{c}_{2n}$ are algebraically independent elements  of $Z(\rU(\g^{\pm}_N))$, and each $c_{i}$ lies in $\mathbb C[\tilde{c}_2, \ldots , \tilde{c}_{2n}]$. We define, for $i=1, \ldots, n$,
\begin{equation*}
\rho^+_{-i}=-\rho^+_{i} = 
\begin{cases} i-1, & \textnormal{ for } N=2n, \\
i-\frac{1}{2}, & \textnormal{ for } N=2n+1,
\end{cases}
\qquad
\rho^{-}_{-i}=-\rho^{-}_{i}=i.
\end{equation*}
Further set $\rho_0=\frac{1}{2}$ if $N=2n+1$.
By Lemma \ref{sdetmiura}, we have
\begin{align*}
&\chi^\pm(Z_M(u))= 
\\
&\begin{dcases} \tilde{\mc C}_{\tq_0}(u) \otimes \tilde{\mc B}_{\tq_1}(u+\tfrac{\tq_1-1}{2})\tilde{\mc B}_{\tq_1}(-u+\tfrac{\tq_1-1}{2}) \otimes \dots \otimes \tilde{\mc B}_{\tq_s}(u+\tfrac{\tq_s-1}{2})\tilde{\mc B}_{\tq_s}(-u+\tfrac{\tq_s-1}{2}),
& \textnormal{for } \ell \textnormal{ odd} \\
  \tilde{\mc B}_{\tq_1}(u+\tfrac{\tq_1-1}{2})\tilde{\mc B}_{\tq_1}(-u+\tfrac{\tq_1-1}{2}) \otimes \dots \otimes \tilde{\mc B}_{\tq_s}(u+\tfrac{\tq_s-1}{2})\tilde{\mc B}_{\tq_s}(-u+\tfrac{\tq_s-1}{2}),
& \textnormal{for } \ell \textnormal{ even}.
\end{dcases}
\end{align*}
As $\chi^{\pm}$ is injective the algebraic independence then follows from the previously mentioned \cite[Thms. 7.1.1 \& 7.1.6]{Mol07}.
The statement about the canonical degrees follows from definition of $\mc C_N(u)$.
\end{proof}

The following is the main result of this section.

\begin{thm}\label{truncatedcenter}
Assume we are in one of the following cases:
\begin{equation}\label{nonpfcases}
    \begin{dcases} 
    \Y_{N,\ell}^{+}(\sigma) \  \text{ and } \ N  \text{ is odd, } \\
    \Y_{N,\ell}^{+}(\sigma) \ \text{ and \ both of } N,   \ell \text{ are even, } \\
\Y_{N,\ell}^{-}(\sigma) \  \text{ and } \ \ell \text{ is odd.}
\end{dcases}    
\end{equation}
Then $Z(\Y_{N,\ell}^{\pm}(\sigma))$ 
is freely generated by the coefficients $z_2, \dots, z_{2m}$. Furthermore, these elements lift a complete collection of homogeneous generators of the Poisson center of the Slodowy slice, via the isomorphisms of Corollaries~\ref{gryangsliceAI} and \ref{gryangsliceAII}.
\end{thm}

Before proving this theorem we need to recall some classical invariant theory from \cite[§7]{Ja04}.
Let $x \in \gl_M$ and consider the characteristic polynomial 
\begin{align*}    
q(x,t)=\det( t \ \id-x)=t^M+\sum_{i=1}^{M}q_i(x)t^{M-i} \in \mathbb{C}[\gl_M]^{\GL_M}.
\end{align*}
In fact by \cite[Prop. 7.9]{Ja04} the coefficients $q_i(x)$ are homogeneous polynomials of degree $i$ in $\gl_M$ and are algebraically independent generators of $\mathbb{C}[\gl_M]^{\GL_M}$. Now let $\g$ be one of $\mathfrak{so}_{2m+1}$ or $\mathfrak{sp}_{2m}$. Then it is known \cite[Prop. 7.9]{Ja04} that the restriction maps 
$\mathbb{C}[\gl_{2m+1}]^{\GL_{2m+1}} \to \mathbb{C}[\mathfrak{so}_{2m+1}]^{\SO_{2m+1}}$ and $\mathbb{C}[\gl_{2m}]^{\GL_{2m}} \to \mathbb{C}[\mathfrak{sp}_{2m}]^{\Sp_{2m}}$ are surjective and that the kernel is generated by the $q_{2i+1}$. Consequently in these cases $\mathbb{C}[\g]^{G}$ is generated by ${q_{2i}}_{\vert {\g}}$.

If we fix a choice of $G$-equivariant isomorphism $\g \cong \g^*$ then we can use this to transfer the natural Poisson structure on $\bC[\g^*]$ to a Poisson structure on $\bC[\g]$. Also note that the Hamiltonian derivations of $\mathbb{C}[\g]$ are generated by $\ad^*(\g)$ as a $\mathbb{C}[\g]$-module. Since $G$ is connected, it follows that $\mathbb{C}[\g]^{G}$ coincides with the Poisson center of $\mathbb{C}[\g]$. We will denote the Poisson center of a Poisson algebra $A$ by $Z(A)$.

\begin{proof}[Proof of Theorem \ref{truncatedcenter}.]
First note from Proposition \ref{coeffalgind} that $z_2, \dots, z_{2m}$ are algebraically independent.
So it remains to prove that $z_2, \dots, z_{2m}$ generate $Z(\Y_{N,\ell}^{\pm}(\sigma))$.
 By  the footnote to \cite[Question 5.1]{PrJI} (see also \cite[Theorem~11.1]{AM21} for a more detailed proof), the Kazhdan graded map $\mathbb{C}[\g]^{G} \to \mathbb{C}[\cS_\chi]$ restricts to an isomorphism 
 $\mathbb{C}[\g]^{G} \xrightarrow{\sim} Z(\mathbb{C}[\cS_\chi])$.
 
 By the definition of the Kazhdan grading (see \S\ref{ss:slicesandWalgebras}, for instance) we see that for Poisson central elements, the PBW degree coincides with twice the Kazhdan degree. In particular in view of the previously discussed invariant theory $Z(\mathbb{C}[\cS_\chi])$ for $\g$ one of $\mathfrak{so}_{2m+1},\mathfrak{sp}_{2m}$ is freely generated by $m$ elements of Kazhdan degree $4,8,12, \dots,4m$.
 
 We know from Corollaries \ref{gryangsliceAI} and \ref{gryangsliceAII} that we have an isomorphism $\gr' \Y_{N,\ell}^{\pm}(\sigma) \xrightarrow{\sim}  \mathbb{C}[\cS_\chi]$ where in the cases we are considering $\g$ is one of $\mathfrak{so}_{2m+1},\mathfrak{sp}_{2m}$. By \cite[Prop. 4.7]{To23} this isomorphism is graded if we equip $\gr' \Y_{N,\ell}^{\pm}(\sigma)$ with the doubled canonical grading and $\mathbb{C}[\cS_\chi]$ with the Kazhdan grading. In particular $Z(\gr' \Y_{N,\ell}^{\pm}(\sigma))$ is freely generated by $m$ elements in canonical degrees $2,4, \dots,2m$.
 
 Now by Proposition \ref{coeffalgind} the top graded components of the elements $z_2,z_4, \ldots, z_{2m}$ are algebraically independent Poisson central elements in $\gr' \Y_{N,\ell}^{\pm}(\sigma)$ of the correct canonical degrees so that they must generate the whole Poisson center. In particular, these elements generate the center of $\Y_{N,\ell}^{\pm}(\sigma)$.
\end{proof}

\subsection{Pfaffian generators}
 \label{ss:pfaffian}

Throughout this subsection let $N$ be even, and let $M=2m$ be the number of boxes in the pyramid.
As we have seen in the above subsection, the elements $z_2, \dots z_{2m}$ generate the center of $\Y_{N,\ell}^{\pm}(\sigma)$ if we are in one of the cases \eqref{nonpfcases}. In fact, one can easily check that the cases \eqref{nonpfcases} correspond precisely to the truncated shifted twisted Yangians such that $\gr' \Y^\pm_{N,\ell}(\sigma)$ is identified with a Slodowy slice in a classical Lie algebra of type {\sf B} or {\sf C}.

If $\g = \fkso_{2m}$ has type {\sf D} then one still knows that the map $\bC[\g] \to \mathbb{C}[\cS_\chi]$ gives an isomorphism of Poisson centers. In this case, it is well known from classical invariant theory \cite[§7.7]{Ja04} that the restrictions of the coefficients of the characteristic polynomial from $\gl_{2m}$ to $\mathfrak{so}_{2m}$ do not generate the whole center. Instead the restriction of the coefficient $q_{2m}$ to $\mathfrak{so}_{2m}$ admits a square root $\pf$ called the \textit{Pfaffian}. It is not hard to see that the Pfaffian $\pf|_{\cS_\chi}$ can never appear as the top graded term of an element of $\Y_N^\pm(\sigma)$ constructed from the Sklyanin determinant.

We introduce the following definition.
\begin{dfn}\label{Pflike}
An element $\pf\in \Y^{\pm}_{N,\ell}(\sigma)$ is called a \textit{Pfaffian generator} if it admits the following three properties:
\begin{itemize}
    \item $\pf$ is central,
    \item $\pf$ has canonical degree $m$,
    \item The elements $z_2, \dots z_{2m-2},\pf$ are algebraically independent.
\end{itemize}
\end{dfn}

\begin{conj}\label{pfconj}
A Pfaffian generator exists if (and only if) one of the following cases holds:
\begin{enumerate}
\item $\Y_{N,\ell}^+(\sigma)$ with $N$ even and $\ell$ odd,
\item $\Y_{N,\ell}^-(\sigma)$ with $\ell$ even.
\end{enumerate}
\end{conj}
Note that by Theorem \ref{truncatedcenter} a Pfaffian generator can only exist in the cases not mentioned in \eqref{nonpfcases}. Thus the ``only if" part of the conjecture holds.

The work of Brown \cite{Br09} supports the existence of a Pfaffian generator in all the unshifted cases in Conjecture~\ref{pfconj}. As new supporting evidence for the conjecture, we construct a Pfaffian generator in the case $\Y^{+}_{2,\ell}(\sigma)$ for odd $\ell$. 
The remainder of the section is devoted to proving the following.

\begin{thm}\label{pfr2thm}
    The element 
    $\pf:=\wtl{B}_{1;i,1}^{(\mathfrak{s}_{1,2};p_1)}$ as defined in \eqref{extragens} is a Pfaffian generator in $\Y^{+}_{2,\ell}(\sigma)$, for $\ell$ odd.
\end{thm}
Before proving this theorem we state a corollary, and then prove some preliminary results.
\begin{cor}
\label{pfcases}
    For $\ell$ odd, the center $Z(\Y_{2,\ell}^+(\sigma))$ is freely generated by $z_2, \dots, z_{2m-2}, \pf$. Furthermore, these elements lift a complete collection of homogeneous generators of the Poisson center of the Slodowy slice, via the isomorphism of Corollary~\ref{gryangsliceAI}.
\end{cor}
\begin{proof}
    The argument of Theorem~\ref{truncatedcenter} works almost verbatim, after taking into account the well-known fact that the (total) degrees of homogeneous generators of $\bC[\fkso_{2m}]^{\SO_{2m}}$ are $2, 4, 6, ..., 2m-2$, and $m$; cf. \cite[\textsection 7.7]{Ja04}.
\end{proof}
\begin{rem}
    If a Pfaffian generator $\pf$ exists in $\Y_{N,\ell}^\pm(\sigma)$, then the center is freely generated by the elements $z_2, \dots z_{2m-2},\pf$, by the same argument given above.
\end{rem}

In light of the above remark, the following simpler variant of Conjecture~\ref{pfconj} is slightly weaker but suffices for our application to finite $W$-algebras. As the formulation is less technical, it is conceivable that different techniques can be applicable to attack Conjecture \ref{centerconj}.

\begin{conj}\label{centerconj}
Consider $\Y_{N,\ell}^+(\sigma)$ with $N$ even and $\ell$ odd or $\Y_{N,\ell}^-(\sigma)$ with $\ell$ even. The center of $\Y_{N,\ell}^\pm(\sigma)$ is isomorphic to a polynomial algebra with generators of canonical degree $2, 4, 6, ..., 2m-2$, and $m$. 
\end{conj}

\begin{lem}\label{pfcentralunshifted}
The element $s_{1,2}^{(\ell)} \in \Y^{+}_{2,\ell}(0)$ is central.    
\end{lem}
\begin{proof}
The proof is by direct computation and by a case-by-case analysis.
We will examine the commutators $[s_{i,j}^{(r)},s_{1,2}^{(\ell)}]$.
Consider first the case $i=j=1$ and note that by the symmetry relation \eqref{sym} this forces $r$ to be even. By the quaternary relation \eqref{qua} we have
\begin{align*}
[&s_{1,1}^{(r)},s_{1,2}^{(\ell)}]
\stackrel{(\text{i})}{=} \sum_{t=0}^{r-1}\big(s_{1,1}^{(r-1-t)}s_{1,2}^{(\ell+t)}-s_{1,1}^{(\ell+t)}s_{1,2}^{(r-1-t)}\big)
-\sum_{t=0}^{r-1}(-1)^t \big(s_{1,1}^{(r-1-t)}s_{1,2}^{(\ell+t)}-s_{1,1}^{(\ell+t)}s_{2,1}^{(r-1-t)}\big) \\
&+\sum_{t=0}^{\lfloor \frac{r}{2} \rfloor-1}\big(s_{1,1}^{(r-2-2t)}s_{1,2}^{(\ell+2t)}-s_{1,1}^{(\ell+2t)}s_{1,2}^{(r-2-2t)}\big)
\\
&\stackrel{(\text{ii})}{=} \sum_{t=0}^{\lfloor \frac{r}{2} \rfloor -1}\big(s_{1,1}^{(r-1-2t)}s_{1,2}^{(\ell+2t)}-s_{1,1}^{(\ell+2t)}s_{1,2}^{(r-1-2t)}\big)
+\sum_{t=0}^{\lfloor \frac{r}{2} \rfloor -1}\big(s_{1,1}^{(r-2-2t)}s_{1,2}^{(\ell+2t+1)}-s_{1,1}^{(\ell+2t+1)}s_{1,2}^{(r-2-2t)}\big) \\
&-\sum_{t=0}^{\lfloor \frac{r}{2} \rfloor-1} \big(s_{1,1}^{(r-1-2t)}s_{1,2}^{(\ell+2t)}-s_{1,1}^{(\ell+2t)}s_{2,1}^{(r-1-2t)} \big)
+\sum_{t=0}^{\lfloor \frac{r}{2} \rfloor-1} \big(s_{1,1}^{(r-2-2t)}s_{1,2}^{(\ell+2t+1)}-s_{1,1}^{(\ell+2t+1)}s_{2,1}^{(r-2-2t)}\big) \\
&+\sum_{t=0}^{\lfloor \frac{r}{2} \rfloor-1}\big(s_{1,1}^{(r-2-2t)}s_{1,2}^{(\ell+2t)}-s_{1,1}^{(\ell+2t)}s_{1,2}^{(r-2-2t)}\big)
\\
&\stackrel{(\text{iii})}{=} \sum_{t=0}^{\lfloor \frac{r}{2} \rfloor -1}\big(s_{1,1}^{(r-2-2t)}s_{1,2}^{(\ell+2t+1)}-s_{1,1}^{(\ell+2t+1)}s_{1,2}^{(r-2-2t)}\big) +\sum_{t=0}^{\lfloor \frac{r}{2} \rfloor-1} \big(s_{1,1}^{(r-2-2t)}s_{1,2}^{(\ell+2t+1)}-s_{1,1}^{(\ell+2t+1)}s_{2,1}^{(r-2-2t)}\big) \\
&+\sum_{t=0}^{\lfloor \frac{r}{2} \rfloor-1}\big(s_{1,1}^{(r-2-2t)}s_{1,2}^{(\ell+2t)}-s_{1,1}^{(\ell+2t)}s_{1,2}^{(r-2-2t)}\big)\\
&\stackrel{(\text{iv})}{=} 2 \sum_{t=0}^{\lfloor \frac{r}{2} \rfloor-1} \big(s_{1,1}^{(r-2-2t)}s_{1,2}^{(\ell+2t+1)}\big)+\sum_{t=0}^{\lfloor \frac{r}{2} \rfloor-1}\big(s_{1,1}^{(r-2-2t)}s_{1,2}^{(\ell+2t)}\big) \\
&=\sum_{t=0}^{\lfloor \frac{r}{2} \rfloor-1}s_{1,1}^{(r-2-2t)}(2s_{1,2}^{(\ell+2t+1)}+s_{1,2}^{(\ell+2t)})
\stackrel{(*)}{=}0.
\end{align*}
Let us break down where these equalities come from. The equality (i) is just given by the quaternary relation whereas the equality (ii) is given from splitting up the sum. For the equality (iii) note that due to the symmetry relation \eqref{sym} we have that $s_{i,i}^{(2s+1)}=0$ and by our assumptions we have that $r-1-2t$ and $\ell+2t$ are odd. This explains the vanishing of the first and third sums on the left-hand side of (iii). For the equality (iv) note that by \eqref{trdef} and the symmetry relation again we have that $s_{1,1}^{(\ell+2t+1)} \in I_{\ell}$ which gives the next two equalities. Finally for the last equality (*) note that $2s_{1,2}^{(\ell+2t+1)}+s_{1,2}^{(\ell+2t)} \in I_{\ell}$ by definition. The case where $i=j=2$ works exactly the same way so we will omit it here.

Consider now the case where $i \neq j$ and note that due to the symmetry relation we may assume $i=1, j=2$.
Then
\begin{align}
[s_{1,2}^{(r)},s_{1,2}^{(\ell)}]&= \sum_{t=0}^{r-1}\Big(s_{1,2}^{(r-1-t)}s_{1,2}^{(\ell+t)}-s_{1,2}^{(\ell+t)}s_{1,2}^{(r-1-t)}\Big)\notag\\
&-\sum_{t=0}^{r-1}(-1)^t \Big(s_{1,1}^{(r-1-t)}s_{2,2}^{(\ell+t)}-s_{1,1}^{(\ell+t)}s_{2,2}^{(r-1-t)}\Big)\label{1212pf}\\
&+\sum_{t=0}^{\lfloor \frac{r}{2} \rfloor-1}\Big(s_{1,1}^{(r-2-2t)}s_{2,2}^{(\ell+2t)}-s_{1,1}^{(\ell+2t)}s_{2,2}^{(r-2-2t)} \Big)=\sum_{t=0}^{r-1}[s_{1,2}^{(r-1-t)},s_{1,2}^{(\ell+t)}]=0.\notag
\end{align}
Again the first equality follows from the quaternary relation and the second equality follows as we have $s_{i,i}^{(2s+1)}=0$ by the symmetry relation as well as $s_{i,i}^{(\ell+2t+1)} \in I_{\ell}$.
For the final equality we use an induction.
First note that by the quaternary relation we have
$[s_{1,2}^{(1)},s_{1,2}^{(\ell)} ]=s_{1,1}^{(\ell)}-s_{2,2}^{(\ell)}=0$ by the symmetry relation. Now assume the statement $[s_{1,2}^{(r-1-t)},s_{1,2}^{(\ell+t)}]=0$ has been proven for all $t<s$. Note that by \eqref{trdef} we may always replace $s_{1,2}^{(\ell+t)}$ by a scalar multiple of $s_{1,2}^{(\ell)}$ and then we can just repeat the argument from \eqref{1212pf}.
\end{proof}
\begin{lem}\label{pf2algind}
The elements $z_2,\ldots,z_{2 \ell-2},s_{1,2}^{(\ell)}$ are algebraically independent.
\end{lem}
\begin{proof}
It is known from Proposition \ref{coeffalgind} that the elements $z_2, \ldots, z_{2 \ell-2},z_{2\ell}$ are algebraically independent. Using the expansion formula for the Sklyanin determinant from \cite[Thm. 2.7.2]{Mol07}
\be
\sdet(S(u))=\sum_{ \omega\in \fkS_{2}}\sgn(\omega)\sgn(\omega') s^t_{\omega(1),\omega'(1)}(u-1) s_{\omega(2),\omega'(2)}(-u),
\ee
where $\omega'$ is defined differently in \textit{op. cit.} from \eqref{a1thetadef},
we see that the top graded part of the coefficient of $u^{-2\ell}$ in the canonical filtration is given by the top graded part of 
\be
\sum_{\omega \in \fkS_{2}}\sgn(\omega)\sgn(\omega') s_{\omega(\omega')^{-1}(1),1}^{(\ell)} s_{\omega(\omega')^{-1}(2),2}^{(\ell)},
\ee
where we have used the fact that $s_{i,j}^{(\ell)}=-s_{j,i}^{(\ell)}$ in $\gr' \Y^+_{2,\ell}$ (by the symmetry relation \eqref{sym} and oddness of $\ell$) and the fact that $s_{i,j}^{(r)}=0 $ for $r> \ell$ in $\gr' \Y^+_{2,\ell}$. Now by  \cite[Lemma 2.7.5]{Mol07} the map $\fkS_{2} \to \fkS_{2}$ given by $\omega \mapsto \omega(\omega')^{-1}$ is bijective so that we get that the top graded part of the coefficient of $u^{-2 \ell}$ of $\sdet(S(u))$ is given by  the top graded part of $\sum_{\omega \in \fkS_{2}}\sgn(\omega)\sgn(\omega') s_{\omega(1),1}^{(\ell)} s_{\omega(2),2}^{(\ell)}$. However in $\gr' \Y^{+}_{2,\ell}$ this is simply given by $(s_{1,1}^{(\ell)}s_{2,2}^{(\ell)}-s_{1,2}^{(\ell)}s_{2,1}^{(\ell)})=(s_{1,2}^{(\ell)})^2$, see also \cite[Ex. 2.7.3]{Mol07}. For the last step, again we used the fact that in the associated graded we have $s_{i,j}^{(\ell)}=-s_{j,i}^{(\ell)}$. Furthermore it follows immediately by definition that the top graded part of $z_{2\ell}$ in $\gr' \Y^{+}_{2,\ell}$ is the same as the top graded part of the coefficient of $u^{-2\ell}$ of $\sdet(S(u))$. In particular this gives that the top graded part of $s_{1,2}^{(\ell)}$ is the square root of the top graded part of $z_{2\ell}$ and hence the top graded parts of $z_{2}, \ldots, z_{2 \ell-2},s_{1,2}^{(\ell)}$ are algebraically independent in $\gr' \Y^{+}_{2,\ell}$ and hence these elements are also algebraically independent in $ \Y^{+}_{2,\ell}$.
\end{proof}
We are now in a position to prove Theorem
\ref{pfr2thm}.
\begin{proof}[Proof of Theorem \ref{pfr2thm}]
First note that the statement about the canonical degree is immediate by definition of  $\wtl{B}_{1;i,1}^{(\mathfrak{s}_{1,2};p_1)}$.
The proof is by induction on $\mathfrak{s}_{1,2}$. Note that when $\mathfrak{s}_{1,2}=0$ we have  $\wtl{B}_{1;1,1}^{(\mathfrak{s}_{1,2};p_1)}=-s_{1,2}^{(\ell)}$ and because we have $s_{1,2}^{(\ell)}=-s_{2,1}^{(\ell)}$ by the symmetry relation so that in this case the result is proven by Lemmas \ref{pfcentralunshifted} and \ref{pf2algind}.

Now assume the result has been proven for $\mathfrak{s}_{1,2}<k$ and let $\mathfrak{s}_{1,2}=k$. 
Recall the baby comultiplication from Theorem~\ref{delR}. We have seen in \eqref{trdelR} and in Theorem \ref{trpbw} that this induces an injective map $\Y^{+}_{2,\ell}(\sigma) \hookrightarrow \Y^{+}_{2,\ell-2}(\dot{\sigma})\otimes \rU(\gl_1)$ where we have used that the minimal admissible shape is $(1,1)$ if $\mathfrak{s}_{1,2}>0$ and note that $\dot{\mathfrak{s}}_{1,2}=\mathfrak{s}_{1,2}-1$. We will write $x$ for the generator of $\rU(\gl_1) \cong \mathbb{C}[x]$.
Consider the image of $\wtl{B}_{1;1,1}^{(\mathfrak{s}_{1,2};p_1)}$ under the baby comultiplication. 
Using the definition of $\Delta_R$ we see that the image is given by $\wtl{B}_{1;1,1}^{(\mathfrak{s}_{1,2}-1;p_1)}\otimes x+\dot{\wtl{B}}_{1;1,1}^{(\mathfrak{s}_{1,2};p_1)}\otimes 1$. It is clear that $\wtl{B}_{1;1,1}^{(\mathfrak{s}_{1,2}-1;p_1)}\otimes x$ is central by the induction hypothesis.

Now consider the bracket $[H_{1;1,1}^{(2)},\wtl{B}_{1;1,1}^{(\mathfrak{s}_{1,2}-1;p_1)}]$. Using the explicit description \eqref{extragens} and relation \eqref{pr2} we see that 
$[H_{1;1,1}^{(2)},\wtl{B}_{1;1,1}^{(\mathfrak{s}_{1,2}-1;p_1)}]=-\dot{\widetilde{B}}_{1;1,1}^{(\mathfrak{s}_{1,2};p_1)}$ but by assumption $\widetilde{B}_{1;1,1}^{(\mathfrak{s}_{1,2}-1;p_1)}$ is central in $\Y^{+}_{2,\ell-2}(\dot{\sigma})$ so that we get $\dot{\widetilde{B}}_{1;1,1}^{(\mathfrak{s}_{1,2};p_1)}=0$. In particular, as $\Delta_R$ is injective, this gives that $\widetilde{B}_{1;1,1}^{(\mathfrak{s}_{1,2};p_1)}$ is central.

Furthermore, by the induction hypothesis, we know that $Z(\Y^{+}_{2,\ell-2}(\dot{\sigma})\otimes \rU(\gl_1))$ is a polynomial ring generated by $z_2 \otimes1, \ldots, z_{2m-4}\otimes1,\widetilde{B}_{1;1,1}^{(\mathfrak{s}_{1,2}-1;p_1)}\otimes1,1 \otimes x$.
For algebraic independence we once again pass to the associated graded with respect to the canonical filtration. Here $z_{2i}$ are just the top graded pieces of the coefficients of $u^{-2i}$ in $\sdet(S(u))$. Using Proposition \ref{sdetcomul}, the fact that the top graded pieces of $\tilde{B}_{\tq_1}(u+\frac{\tq_1-1}{2})\tilde{B}_{\tq_1}(-u+\frac{\tq_1-1}{2})$ are the same as the top graded pieces of $\tilde{B}_{\tq_1}(u)\tilde{B}_{\tq_1}(-u)$ and thus have no even coefficients. By the induction hypothesis one explicitly verifies that the images under $\gr' \Delta_R$ of the top graded pieces of the elements $z_2,\ldots, z_{2m-2},\widetilde{B}_{1;1,1}^{(\mathfrak{s}_{1,2};p_1)}$ are algebraically independent.
\end{proof}

\section{Equivariant quantizations and finite $W$-algebras}
\label{sec:singularities}

In this section, we recall finite $W$-algebras as quantization of Slodowy slices. We achieve our main goal by showing that a truncated shifted twisted Yangian is isomorphic to a certain finite $W$-algebra of type $\sf B\sf C\sf D$ under some suitable assumptions; see Theorem~\ref{t:finalisothm}. The proof uses geometric techniques related to conic symplectic singularities. In particular we use the existence of universal equivariant quantizations of nilpotent Slodowy varieties, which we also recall. 

\subsection{Slodowy slices and finite $W$-algebras}
\label{ss:slicesandWalgebras}

Throughout this section all graded vector spaces are $\bN$-graded. We will often use the well-known correspondence between connected $\bN$-graded commutative algebras and affine schemes with contracting $\bC^\times$-action.

Let $G$ be a connected complex reductive group, and $\g = \Lie(G)$. A chosen non-degenerate $G$-invariant bilinear form on $\g$ induces a $G$-module isomorphism $\kappa\colon \g \cong\g^*$. The nilpotent cone $\cN = \cN(\g^*)$ is the image under $\kappa$ of the set of nilpotent elements in $\g$. It can also be defined as the central fibre of the $\bC^\times$-equivariant coadjoint quotient map $\g^* \to \g^*/\!/G$, which is a reduced scheme (see \cite[\S 5]{Pr02} and the references therein).

Let $e\in \g$ be nilpotent, and write $\mathbb O$ for the adjoint orbit of $e$. Let $(e,h,f)$ be an $\fksl_2$-triple for $e$. We denote $\chi := \kappa(e) \in \g^*$. The Slodowy slice is the affine space $\cS_\chi := \chi + \kappa(\g^f) \subseteq \g^*$, which is transversal to the $G$-orbits (i.e. the symplectic leaves) of $\g^*$. It is equipped with a Poisson structure \cite{GG02}, which we now explain.

The semisimple operator $\ad(h)$ induces the Dynkin grading $\g = \bigoplus \g(i)$ where $\g(i) = \{x\in \g \mid [h,x] = i x\}.$ Then $S(\g) = \bC[\g^*]$ is equipped with the Kazhdan grading, which places $\g(i) \subseteq S(\g)$ in degree $i + 2$. We note that the Kazhdan graded pieces of $S(\g)$ are all infinite dimensional. The form $\kappa$ places $\g(-1)$ and $\g(1)$ in non-degenerate pairing, and so we may choose a Lagrangian subspace $\fkl \subseteq \g(-1)$ with respect to the form $\langle x, y \rangle := \chi([x,y])$. Define $\fkm := \fkl \oplus \bigoplus_{i< -1} \g(i)$ and let $M \subseteq G$ be the unipotent group with Lie algebra $\fkm$.

Considering the moment map $\g^* \to \fkm^*$ for the coadjoint action of $M$. The fiber over $\chi$ is $\chi + (\g/\fkm)^*$, with $M$ acting freely, and $\big(\chi + (\g/\fkm)^* \big)/\!/ M \cong \cS_\chi$. Thus, $\cS_\chi$ acquires a Poisson structure by Hamiltonian reduction.

The Kazhdan grading on $S(\g)$ descends to a non-negative connected grading on $\bC[\chi + (\g/\fkm)^*/\!/ M]$, which equips $\cS_\chi$ with a contracting $\bC^\times$-action. We note that the Poisson bracket on $\bC[\cS_\chi]$ has degree $-2$ with respect to the grading
\begin{eqnarray*}
\{\cdot\, , \, \cdot\}  :  \bC[\cS_\chi]_i \wedge \bC[\cS_\chi]_j \longrightarrow \bC[\cS_\chi]_{i+j-2}.
\end{eqnarray*}
The coadjoint quotient map $\cS_\chi \to \g^*/ G$ is $\bC^\times$-equivariant, faithfully flat and admits reduced, irreducible, normal fibres \cite[\S 5]{Pr02}. The central fibre of the morphism $\cS_\chi \to \g^*/\!/G$ is denoted $\cN_\chi$ and called the {\it nilpotent Slodowy variety}. It will play a central role in the rest of the paper. By our previous remarks, $\cN_\chi$ coincides with the reduced scheme $\cN \cap \cS_\chi$.

We note that the Poisson varieties $\cS_\chi$ and $\cN_\chi$ only depend on the orbit $\Ad(G)e$ up to isomorphism. In particular it is independent of the choice of Lagrangian $\fkl$. 

%The map $\varphi^* : \bC[\g^*]^G \to \bC[\cS_\chi]$ provides an isomorphism to the Poisson center \cite{}, which we denote $Z \bC[\cS_\chi]$. 

The finite $W$-algebra $\rU(\g, e)$ is constructed from the enveloping algebra $\rU(\g)$ by quantum Hamiltonian reduction. In more detail, the Kazhdan grading on $S(\g)$ lifts to a Kazhdan filtration on $\rU(\g)$ by placing $\g(i) \subseteq \rU(\g)$ in degree $i + 2$. The precise quantum analogue of $\bC[\chi + (\g/\fkm)^*]$ is
$$\rU(\g,e) := (\rU(\g) / \rU(\g) \{x - \chi(x) \mid x \in \fkm\})^{M}.$$
This algebra is equipped with a Kazhdan filtered algebra structure inherited from $\rU(\g)$, and it was proven in \cite[Thm.~4.1]{GG02} that
\begin{eqnarray}
    \label{e:Walgquantizes}
    \gr \rU(\g,e) = \bC[\cS_\chi].
\end{eqnarray}
Furthermore the footnote to \cite[Question 5.1]{PrJI} explains that
\begin{eqnarray}
    \label{e:Walgcenterquantizes}
    \gr Z\rU(\g,e) = Z\bC[\cS_\chi],
\end{eqnarray}
where $Z$ indicates the center of an algebra, or Poisson algebra.

\subsection{Equivariant Poisson deformations of symplectic singularities}

A symplectic singularity is a normal algebraic variety $X$ with a closed, nondegenerate 2-form $\omega$ on the smooth locus $X^{\reg}$, and a resolution of singularities $\rho : Y \to X$ such that the pullback $\rho^* \omega$ extends 
%over the exceptional locus
to a closed (possibly degenerate) regular 2-form \cite{Be00}. If the 2-form extends to a regular symplectic form on $Y$  then $\rho$ is known as a {\it symplectic resolution}. 

If a symplectic singularity $X$ is irreducible, affine and carries a contracting $\bC^\times$-action such that $\omega$ is homogeneous of weight $w > 0$ then we say that $X$ is a {\it conic symplectic singularity}.

The closed 2-form on $X^\reg$ equips $X^\reg$ with a Poisson bivector which, by normality, extends to a Poisson bracket on $\bC[X]$. %The theory of deformations and quantizations of these Poisson algebras is remarkably rigid, and this is the key feature which we leverage.
The conical structure equips $\bC[X]$ with a connected, non-negative $\bZ$-grading $\bC[X] = \bigoplus_{i \gge 0} \bC[X]_i$ satisfying $\bC[X]_0 = \bC$.

The homogeneity of the 2-form on $X^\reg$ ensures that the Poisson bracket is graded in degree $-w$:
\begin{eqnarray}
\label{e:bracketdegreew}
    \{ \cdot, \cdot\} : \bC[X]_i \wedge \bC[X]_j \longrightarrow \bC[X]_{i+j-w}.
\end{eqnarray}

\begin{example}
\label{eg:sliceCSS}
   It follows from \cite{Gi09} that $\cN_\chi$ is a conic symplectic singularity. A symplectic resolution of $\cN_\chi$ may be constructed by applying Hamiltonian reduction to the Springer resolution of $\cN$ (Prop.~2.1.2 of {\it op. cit}).
\end{example}

%\subsection{Equivariant Poisson deformations of symplectic singularities}

Let $X$ be a conic symplectic singularity and $\bC[X]$ together with its Poisson structure.  A flat graded Poisson deformation is a triple $(A, B, i)$ where:
\begin{itemize}
\setlength{\itemsep}{4pt}
    \item $B$ is a non-negatively graded commutative algebra;
    \item $A$ is a graded Poisson algebra admitting a flat homomorphism $B \to  A$ with image lying in the Poisson center;
    \item $i : A \otimes_B \bC \to \bC[X]$ is a graded Poisson isomorphism, where $\bC$ denotes the unique graded simple $B$-module.
\end{itemize}

The map $i$ should be referred to as the {\it tethering isomorphism} and the algebra $B$ will be referred to as the {\it base} of the deformation. One can show that, in fact, the image of $B$ in $A$ is equal to the Poisson center.

An isomorphisms of flat graded Poisson deformations over a base $B$ is a homomorphism of Poisson algebras which respects the tethering isomorphisms, see \cite[\S 2.3]{ACET23} for more detail.

Write $\textbf{GrAlg}$ for the category of graded commutative algebras. The functor of flat graded Poisson deformations of $X$
\begin{eqnarray}
    \rPD_X : \textbf{GrAlg} \longrightarrow \textbf{Set}
    \label{e:deformationfunctor}
\end{eqnarray}
sends a graded commutative algebra $B$ to the set of isomorphism classes of flat graded Poisson deformations with base $B$. To a homomorphism of graded commutative algebras $B_1 \to B_2$ we associate the morphism map of sets which sends the isomorphism class of $(A, B_1, i)$ to the isomorphism class of $(A \otimes_{B_1} B_2, B_2, i)$.

If $\Gamma$ is a group then a $\Gamma$-algebra is an algebra $A$ equipped with a homomorphism $\Gamma \to \Aut(A)$, and $\Gamma$-homomorphisms of such algebras are defined similarly.

Now let $\Gamma$ be a reductive group acting by graded Poisson automorphisms on $\bC[X]$. We say that $(A, B, i)$ is a Poisson $\Gamma$-deformation if it is equipped with a graded $\Gamma$-action such that the induced map $A \twoheadrightarrow \bC[X]$ is $\Gamma$-equivariant, and such that the action on the base is trivial. Isomorphisms of Poisson $\Gamma$-deformations are the isomorphisms respecting the $\Gamma$-action, and the functor of $\Gamma$-equivariant Poisson deformations
\begin{eqnarray}
    \rPD_{X,\Gamma}  :  \textbf{GrAlg} \longrightarrow \textbf{Set}
    \label{e:deformationfunctorequivariant}
\end{eqnarray}
associates to each graded algebra $B$ the set of isomorphism classes of flat graded Poisson $\Gamma$-deformations with base $B$.

\begin{example}
\label{eg:slicedeformation}
    The Poisson center $Z\bC[\g^*]$ coincides with the ring of invariants $\bC[\g^*]^G$, and a theorem of Kostant states that there exists a natural isomorphism $i_0 : \bC[\g^*] \otimes_{Z\bC[\g^*]} \bC \to \bC[\cN]$. In fact $(\bC[\g^*], Z\bC[\g^*], i_0)$ is a flat graded Poisson deformation of $\cN$.

    Similarly if $e \in \g$ is a nilpotent element then restriction $\bC[\g^*]^G \to \bC[\mc S_\chi]$ provides an isomorphism to the Poisson center of functions on the Slodowy slice (see \cite[Lem.~3.3]{ACET23} for references) and, by a result of Premet \cite[Thm.~5.2]{Pr02}, there is a natural Poisson isomorphism $i_\chi : \bC[\mc S_\chi] \otimes_{Z\bC[\g^*]} \bC\to \cN_\chi$. In this case $(\bC[\mc S_\chi], Z\bC[\g^*], i_\chi)$ is a flat graded Poisson deformation of the central fiber $\cN_\chi$.
\end{example}

%Returning to our running example, let $\g$ be a complex semisimple Lie algebra with coadjoint nullcone $\cN \subseteq \g^*$. The Poisson center of $\bC[\g^*]$ coincides with the ring of invariants $\bC[\g^*]^\g$, and a theorem of Kostant states that there exists a natural isomorphism $i_0 : \bC[\g^*] \otimes_{\bC[\g^*]^\g} \bC \to \bC[\cN]$. In fact $(\bC[\g^*], \bC[\g^*]^\g, i_0)$ is a flat graded Poisson deformation of $\mc N$.

%Similarly if $\chi \in \cN$ is a nilpotent linear form on $\g$, then restriction $\bC[\g^*]^\g \to \bC[\mc S_\chi]$ provides an isomorphism to the Poisson center of functions on the Slodowy slice, and there is a natural Poisson isomorphism $i_\chi : \bC[\mc S_\chi] \otimes_{\bC[\g^*]^\g} \bC\to \cN_\chi$, by a result of Premet \cite[Thm.~5.2]{Premet02}. Once again we have $(\bC[\mc S_\chi], \bC[\g^*]^\g, i_\chi)$ is a Poisson deformation of the central fibre $\cN_\chi$.

\subsection{The functor of equivariant quantizations} % of flat graded Poisson deformations}

Recall that the Poisson bracket on $\bC[X]$ lies in some fixed degree $w > 0$, \eqref{e:bracketdegreew}. 

A filtered quantization (of a flat graded Poisson deformation) of $\bC[X]$ is a triple $(\mc A, B, i)$  with:
\begin{itemize}
\setlength{\itemsep}{4pt}
    \item $\mc A$ a non-negatively filtered almost commutative algebra $\mc A = \bigcup_{j\gge 0} \mc A_j$, such that $[\mc A_{j_1} , \mc A_{j_2}] \subseteq \mc A_{j_1 + j_2-w}$. It follows that $\gr \mc A$ admits a Poisson bracket is graded in degree $-w$;
    \item $B$ is a split filtered commutative algebra, i.e. $\gr B \cong B$, admitting a map $B \to \mc A$ to the center of $\mc A$. Thus $B \to \gr \mc A$ maps to the Poisson center;
    \item $(\gr \mc A, B, i)$ is a flat graded Poisson deformation of $\bC[X]$.
\end{itemize}
We call $B$ the {\it base} of the quantization. It is not hard to see that the map $B \to \mc A$ gives an isomorphism to the center.

Write $\textbf{SFAlg}$ for the category of split filtered commutative algebras, i.e. filtered algebras satisfying $\gr B \cong B$. The functor of filtered quantizations of $X$
\begin{eqnarray}
    \rQ_X  :  \textbf{SFAlg} \longrightarrow \textbf{Set}
    \label{e:quantizationfunctor}
\end{eqnarray}
sends a split filtered graded commutative algebra $B$ to the set of isomorphism classes of filtered quantizations of $\bC[X]$. The effect on morphisms is identical to $\rPD_X$, and isomorphisms of filtered quantizations are defined similarly.% To a homomorphism $B_1 \to B_2$ is associated the map which sends the isomorphism class of $(\mc A, B_1, i)$ to the isomorphism class of $(\mc A \otimes_{B_1} B_2, B_2, i)$.

Let $\Gamma$ be a reductive group acting by graded Poisson automorphisms on $\bC[X]$. We say that $(\mc A, B, i)$ is a $\Gamma$-quantization if it is equipped with a $\Gamma$-action such that $(\mc A, B, i)$ is a Poisson $\Gamma$-deformation of $\bC[X]$, and the $\Gamma$-action on $B$ is trivial.

Isomorphisms of $\Gamma$-quantizations are the isomorphisms respecting the $\Gamma$-action, and the functor of $\Gamma$-equivariant filtered quantizations
\begin{eqnarray}
    \rQ_{X,\Gamma}  :  \textbf{SFAlg} \longrightarrow \textbf{Set}
    \label{e:quantizationfunctorequivariant}
\end{eqnarray}
associates to each split filtered graded algebra $B$ the set of isomorphism classes of Poisson $\Gamma$-deformations with base $B$.

\begin{example}
    We offer two extremal examples of elements of $\rQ_{\cN, \Gamma}(Z(\g))$ when $\Gamma$ is the trivial group. The enveloping algebra gives rise to an element $(\rU(\g), Z(\g), i)$. On the other hand, one can take $\rU(\g)/(\ker \chi_0) \otimes_{\bC} Z(\g)$ where $\chi_0 : Z(\g) \to \bC$ is the character of the trivial representation. The former is the largest quantization of $\cN$, in a sense which we make precise in the following section. The latter is a quantization of the trivial deformation of $\cN$.
\end{example}

\subsection{Universal elements of the deformation and quantization functors}

We recap the basics of universal elements of functors, see \cite[III.1]{Ma71} for more detail. Let ${\mathbf{C}}$ be a category. A functor $\cF : \mathbf{C} \to \mathbf{Set}$ is representable over a base $B$ if there exists a natural equivalence $\Hom_{\textbf{C}}(B, \cdot) \to \cF$, and a representation of $\cF$ is a choice of natural equivalence, and $B$ may be called the base of the representation.

A {\it universal element} of $\cF$ is a pair $(B, b)$ with $b\in \cF(B)$ such that for every other such pair $(C, c)$ there exists a unique $\phi \in \Hom_{\mathbf{C}}(B, C)$ such that $\cF(\phi)(b) = c$. By Yoneda's lemma, every representable functor admits a universal element: if $\cF$ is represented by $B$ then $b \in \cF(\operatorname{Id}_B)(B)$ supplies a universal element. If the base $B$ is clear, then we often just write $b$ for the universal element. Note that the choice of universal element is not unique. Every automorphism $\phi \in \Aut_{\mathbf{C}}(B)$ gives rise to a universal element, and so the set of universal elements of $\cF$ is an $\Aut_{\mathbf{C}}(B)$-torsor.

Namikawa has shown that for a conic symplectic singularity $X$ the functor of (formal) Poisson deformations is prorepresentable and unobstructed \cite{Na11}, and in the sequel \cite{Na10} he used the conic structure to show that the functor $\rPD_X$ is representable (see \cite{ACET23} for a algebraic survey of these results). If $B_X \in \textbf{GrAlg}$ represents $\rPD_X$ then an element of $\rPD_X(B_X)$ is called a {\it universal Poisson deformation of $\bC[X]$} (we drop the words flat graded for brevity), and the group of graded automorphisms of $B_X$ acts simply transitively on $\rPD_X(B_X)$ (this algebra was denoted $B_u$ in \cite{ACET23}).

Losev proved in \cite[Prop.~3.5]{Lo22} that Namikawa's universal flat graded Poisson deformation of $\bC[X]$ admits a filtered quantization which satisfies a certain universal property. In \cite[\S 2.7\&2.8]{ACET23} this universal property was reformulated in terms of the functor $\rQ_X$, and upgraded to the $\Gamma$-equivariant setting.

\begin{thm}
\label{T:universalforX}
    Let $X$ be a conic symplectic singularity and $\Gamma$ be a reductive group of Poisson $\bC^\times$-automorphisms of $X$.
    \begin{enumerate}
    \setlength{\itemsep}{4pt}
    \item If $\rPD_X$ is represented over $B_X \in \textnormal{\bf{GrAlg}}$ then $\rPD_{X, \Gamma}$ is represented by the coinvariant algebra
    $$B_{X, \Gamma} := B_X / (b - \gamma \cdot b \mid \gamma \in \Gamma, \ b \in B_X).$$
    The universal element of $\rPD_{X,\Gamma}$ is called the {\it universal $\Gamma$-deformation}.
    \item Viewing $B_X$ as a split filtered algebra, $\rQ_X$ is represented over the base $B_X$ and $\rQ_{X,\Gamma}$ is represented over $B_{X,\Gamma}$. The universal element of $\rQ_{X,\Gamma}$ is called the universal $\Gamma$-quantization.
    \item If we pick a representation
    $\zeta :   \Hom_\textnormal{{\textbf{GrAlg}}}(B_{X, \Gamma}, \cdot) \overset{\sim}{\rightarrow} \rPD_{X,\Gamma},$
        then we may choose a representation
        $\eta : \Hom_\textnormal{{\textbf{SFAlg}}}(B_{X, \Gamma}, \cdot)\overset{\sim}{\longrightarrow} \rQ_{X,\Gamma}$ so that the following diagram commutes:
        \begin{eqnarray}
        \label{e:optimalquantizationtheory}
\begin{tikzcd}
\Hom_{\textnormal{\textbf{SFAlg}}}(B_{X,\Gamma}, \cdot)   \arrow[d,,"\gr"']  \arrow[r, "\eta"] & \rQ_{X, \Gamma} (\cdot)\arrow[d,,"\gr"] \\
  \Hom_{\textnormal{\textbf{GrAlg}}}(\gr B_{X,\Gamma}, \gr (\cdot )) \arrow[r, "\zeta"] &  \rPD_{X,\Gamma} \gr(\cdot )
\end{tikzcd} 
        \end{eqnarray}
       %$\hfill\qed$
        %\[ (\mc A_{X,\Gamma} \otimes_{B_{X,\Gamma}} B, B, i_{X,\Gamma} \otimes 1 )\overset{\sim}{\longrightarrow} (\mc A, B, i)\]
        %as $\Gamma$-quantizations of $\bC[X]$, and there exists a non-unique homomorphism $\phi : B_{x, \Gamma} \to B$ inducing 
       % \[\gr \mc A , B_{X,\Gamma}, i_{X,\Gamma}) \overset{\sim}{\longrightarrow} (\gr \mc A, B, i \otimes 1)\]
       % as flat graded Poisson $\Gamma$-deformations of $\bC[X]$. The following diagram commutes
        \end{enumerate}
\end{thm}

\begin{cor}
\label{C:uniqueuniversalquantization}
Let $X$ is a conic symplectic singularity, for example, $X = \cN_\chi$. Let $\Gamma$ be a reductive group of Poisson $\bC^\times$-automorphisms of $\bC[X]$. Suppose that:
\begin{enumerate}
    \item $(A_{X, \Gamma}, B_{X,\Gamma}, i_{X,\Gamma})$ is a universal flat graded Poisson $\Gamma$-deformation of $X$;
    \item $(\mc A, B, i)$ is a $\Gamma$-quantization of $X$ such that the associated graded flat graded Poisson $\Gamma$-deformation is $(A_{X, \Gamma}, B_{X,\Gamma}, i_{X,\Gamma})$.
\end{enumerate}
    Then $(\mc A, B, i)$ is isomorphic to $(A_{X, \Gamma}, B_{X,\Gamma}, i_{X,\Gamma})$ as a quantization. In particular $\mc A$ is isomorphic to $\mc A_{X,\Gamma}$ as algebras.
\end{cor}
\begin{proof}
    Let $(\mc A_{X,\Gamma}, B_{X,\Gamma}, i_{X,\Gamma})$ be a choice of universal $\Gamma$-quantizaton. Then by the universal property in Theorem~\ref{T:universalforX}(2) there is a unique homomorphism $B_{X,\Gamma} \to B$ and a non-unique isomorphism $\mc A \otimes_{B_{X,\Gamma}} B \to \mc A$. By \eqref{e:optimalquantizationtheory} the associated graded of this map is just the map arising from the universal property of $(A_{X, \Gamma}, B_{X,\Gamma}, i_{X,\Gamma})$, and by assumption (1), the map $\gr B_{X,\Gamma} \to \gr B$ is an isomorphism. It follows that $B_{X,\Gamma} \to B$ is an isomorphism, and so $\mc A_{X,\Gamma} \otimes_{B_{X,\Gamma}} B \to \mc A$ is an isomorphism.  
\end{proof}

\subsection{Twisted Yangians, Kleinian singularities and subregular slices}
\label{ss:twistedKleinianSubregular}

Consider the admissible shape $\mu = (1,1,1)$ for the $N=3$ twisted Yangian of type AI. Fix $m > 0$ and define
\begin{eqnarray}
\label{e:typeBfinalshiftmatrix}
    \sigma = \left(\begin{array}{ccc} 0 & 0 & m \\ 0 & 0 & m \\ m & m & 0 
    \end{array}
    \right).
\end{eqnarray}
Combining \cite[Thm.~A]{TT24} with \cite[Thm.~3.7]{To23} and \eqref{e:tauautsubregB}, we see that $y_{\mu,2m-1}^+(\sigma) := \gr' \Y_{\mu,2m-1}^+(\sigma)$ is Poisson generated by elements $h_a^{(r)} := \gr' H_a^{(r)}$, for $a = 1,2,3$, and $b_a^{(r)} := \gr' B_a^{(r)}$, for $a = 1,2$, subject to certain relations. This is a type AI analogue of Lemma~\ref{grpresAII}.

Now let $X$ be the Kleinian singularity of type {\sf A$_{2m-1}$}. More explicitly, let $\Xi \subseteq \SL_2(\bC)$ be the subgroup generated by the complexification a rotation of $\mathbb{R}^2$ by $\frac{2\pi}{2m-1}$, so that $\Xi \cong \bZ / (2m-1) \bZ$. This preserves the standard symplectic structure on $\bC^2$ and the quotient $X = \bC^2 / \Xi$ is a conic symplectic singularity with $\bC[X] \cong \bC[z^{2m} - xy]$. The Poisson brackets on $\bC[X]$ are given explicitly by
\begin{eqnarray}
\label{e:Kleinianbrackets}
    \{x, y\} = -2m z^{2m-1}, \quad \quad \{z, x\} = x, \quad \quad \{z, y\} = -y.
\end{eqnarray}

Let $e\in \fkso_{2m+1}$ be a subregular nilpotent element, i.e., Jordan blocks of sizes $(2m-1, 1, 1)$. Let $\chi \in \fkso_{2m+1}^*$ be the element associated to $e$ via the Killing form. Also, pick a subregular nilpotent element in $\fksl_{2m}$, i.e. Jordan blocks of sizes $(2m-1, 1)$. Write $\eta \in \fksl_{2m}^*$ for the corresponding linear form under an identification $\fksl_{2m} \cong \fksl_{2m}^*$, and $\cN_\eta \subseteq \fksl_{2m}^*$ for the associated nilpotent Slodowy variety.

Let $Z \subseteq y_{\mu,2m-1}^+(\sigma)$ denote the Poisson center and $Z_+$ the unique maximal Kazhdan graded ideal.
\begin{lem}
\label{L:isomorphicPoissonAlgebras}
The following graded Poisson algebras are isomorphic:
\begin{enumerate}
    \item $\bC[X]$ with grading given by $\deg(z) = 2$ and $\deg(x) = \deg(y) = m$;
    \item $\bC[\cN_\chi]$ with its Kazhdan grading;
    \item $\bC[\cN_\eta]$ with its Kazhdan grading;
    \item $y_{\mu,2m-1}^+(\sigma) / (Z_+)$ with its doubled canonical grading.
\end{enumerate}
\end{lem}
\begin{proof}
    The isomorphism between (1) and (2), as graded commutative algebras, is due to Brieskorn. We refer the reader to \cite{Sl80} for a detailed proof and related results.  Thanks to the proof of \cite[Lem.~4.2]{ACET23} this isomorphism can be upgraded to an isomorphism of Poisson algebras.
    
    The isomorphism between (2) and (3) is a special case of {\it loc. cit.}

    Thanks to Corollary \ref{gryangsliceAI} we have a graded Poisson isomorphism $y_{\mu,2m-1}^+(\sigma) \cong \bC[\cS_\chi]$. As we noted in Example~\ref{eg:slicedeformation} we have $\bC[\cN_\chi] \cong \bC[\cS_\chi] / (Z_+)$, where $Z_+ \subseteq \bC[\cS_\chi]$ is the maximal Kazhdan graded ideal.
\end{proof}

We study properties of certain automorphisms of these Poisson varieties. We begin by stating the classification of graded automorphisms.
\begin{lem}
\label{L:classifyautsA2n-1}
    Write $\PAut_{\bC^\times}(\bC[X])$ for the group of graded Poisson automorphisms of $\bC[X]$. Then 
    $$\PAut_{\bC^\times}(\bC[X]) \cong \bC^\times \rtimes (\bZ / 2\bZ).$$ Furthermore:
    \begin{itemize}
        \item  The $\bC^\times$-action is given by $t\cdot x = tx, t\cdot y = t^{-1} y, t\cdot z = z$ for $t\in \bC^\times$.
        \item The nontrivial element $\theta \in \bZ / 2\bZ$ acts by $x \mapsto y, y \mapsto x, z \mapsto -z$.
    \end{itemize}
    An automorphism that is equivalent to $\theta$ modulo $\bC^\times$ is called \emph{outer}, whilst an automorphism that is trivial modulo $\bC^\times$ is called \emph{inner}.
\end{lem}
\begin{proof}
This can be proven by direct computation, since $\bC[X]$ is generated by its graded pieces of degree 2 and $m$, which have dimensions 1 and 2, respectively. The result can also be deduced from the main theorem of \cite{Ca25}.
\end{proof}

\begin{prop}
\label{P:identicalsubfunctors}
    If $\Theta \subseteq \PAut_{\bC^\times}(\bC[X])$ is inner, then the natural transformation $\rPD_{X, \Theta} \to \rPD_{X}$ is an equivalence. Consequently, $\rPD_{X, \tau}$ and $\rPD_{X, \tau'}$ are identical subfunctors of $\rPD_X$, for any outer involutions $\tau, \tau'$. The same holds for $\rQ_{X, \tau}$ and $\rQ_{X, \tau'}$ as subfunctors of $\rQ_X$.
\end{prop}
\begin{proof}
    To prove the first claim, it suffices to take $\Theta = \bC^\times$, the group of inner Poisson automorphisms of $\bC[X]$, and show that a universal Poisson $\Theta$-deformation is a universal Poisson deformation. By Theorem~\ref{T:universalforX} we only need to show that $\Theta$ acts trivially on the base $B_X$ of the universal Poisson deformation.
    
    By \cite[Thm.~3.5]{ACET23} the universal Poisson deformation is realized by the morphism $\cS_\eta \to \fksl_{2m}^* /\!/ \SL_{2m}$, where the central fiber $\cN_\eta$ is identified with $X$ by Lemma~\ref{L:isomorphicPoissonAlgebras}. If $C(\eta)$ denotes the reductive part of the centralizer of $\eta$ in $\SL_{2m}$, then $C(\eta) \cong \bC^\times$, by \cite[Lem.~7.5.4]{Sl80}. This group acts non-trivially on $\cS_\eta$ by graded Poisson automorphisms (see \cite[Lem.~4.1]{ACET23}, for example) and it follows that the map $C(\eta) \to \PAut_{\bC^\times} \bC[X]$ surjects onto $\Theta$. Since $C(\eta) \subseteq \SL_{2m}$, we deduce that acts trivially on $\fksl_{2m}^*/\!/ \SL_{2m}$. This completes the proof of the first claim.

    If $\tau, \tau'$ are both outer, then there exists an inner automorphism $g\in \Theta$ such that $\tau = g \tau'$. If $B_X$ is the base of the universal Poisson deformation then $g$ acts trivially on $B_X$, so the $\tau$-coinvariants and $\tau'$-coinvariants are equal. By Theorem~\ref{T:universalforX} these coinvariant algebras represent $\rPD_{X, \tau}$ and $\rPD_{X, \tau'}$ respectively, and the penultimate of the proposition follows. An identical argument applies to the functors of equivariant quantizations.
\end{proof}

We introduce an automorphism $\tau \in \Aut \Y_{\mu}^+$ defined by
\begin{eqnarray}
\label{e:tauautsubregB}
\begin{array}{rcccl}
    \tau(B_{1}^{(r)}) &=& -B_{1}^{(r)},& &\\
    \tau(B_2^{(r)}) &=& B_2^{(r)}, & &\\
    \tau(H_a^{(r)}) &=& H_a^{(r)}, & \text{ for all} & a = 1,2,3.
    \end{array}
\end{eqnarray}
It is easily verified that the relations of $\Y_{\mu}^+$ are preserved by $\tau$, and that $\tau$ stabilises the subalgebra $\Y_\mu^+(\sigma)$. By \eqref{trdef} it descends to an automorphism of the truncation $\Y_{\mu,2m-1}^+(\sigma)$.

Abusing notation, we also write $\tau$ for the associated graded Poisson automorphism of $y_{\mu,2m-1}^+(\sigma)$. Then $\tau$ acts on $y_{\mu,2m-1}^+(\sigma)$ by $\tau(b_1^{(r)}) = - b_1^{(r)}$, while fixing all other generators. Via Lemma~\ref{L:isomorphicPoissonAlgebras} we view $\tau$ as a Poisson automorphism of $\bC[X]$.

Now introduce another automorphism $\tau' \in \PAut_{\bC^\times} \bC[X]$. The automorphism group $\Aut(\fksl_{2m})$ is isomorphic to $\operatorname{PGL}_{2m} \rtimes (\bZ/2 \bZ)$, where the cyclic group of order two is generated by the Cartan involution $x\mapsto -x^T$. Therefore the group of outer automorphisms $\Aut(\fksl_{2m})/\Aut(\fksl_{2m})^\circ$ is isomorphic to $\bZ/2\bZ$. One can construct a splitting $s : \bZ / 2\bZ \to \Aut(\fksl_{2m})$ such that the image stabilizes $\cS_\chi$. We denote by $\tau'$ the nontrivial automorphism of $\cS_\chi$ coming from such a splitting.

\begin{lem}
\label{lemonslemma}
    The functors $\rPD_{X,\tau}$ and $\rPD_{X,\tau'}$ are identical as subfunctors of $\rPD_X$. Similarly, $\rQ_{X, \tau}$ and $\rQ_{X, \tau'}$ are identical as subfunctors of $\rQ_X$.
\end{lem}
\begin{proof}
    By Proposition~\ref{P:identicalsubfunctors}, it suffices to show that both $\tau$ and $\tau'$ induce outer automorphisms of $\bC[X]$.
    
    If $y_{\mu, 2m-1}(\sigma)/(Z_+) \to \bC[X]$ is a choice of isomorphism (Lemma~\ref{L:isomorphicPoissonAlgebras}) then $b_1^{(1)}$ maps to a scalar multiple of $z$, since $\bC z$ is the degree 2 piece of $\bC[X]$. Examining the classification of automorphisms in Lemma~\ref{L:classifyautsA2n-1}, and using only the fact that $\tau(b_1^{(1)}) = - b_1^{(1)}$, one sees that $\tau$ is outer.

    The fact that $\tau'$ is an outer automorphism can be deduced from Proposition~\ref{P:identicalsubfunctors}, since $\tau'$ acts non-trivially on the universal base of $\rPD_X$ \cite[Thm.~4.6]{ACET23}.
\end{proof}

\subsection{Truncated shifted twisted Yangians and finite $W$-algebras}\label{ss:YangWalg}
In this final subsection we prove that a truncated shifted Yangian is isomorphic to a certain finite $W$-algebra. The proof needs to be broken down into five cases, and the proof is slightly different in each case. When a choice of $\g$ and $e$ are fixed we assume the notation $\cS_\chi$ and $\cN_\chi$ from \S\ref{ss:slicesandWalgebras}.

If $(\mc A, B, i)$ is a flat graded Poisson deformation then we omit the base and tethering automorphism from the discussion whenever such a choice is clear, and refer to $\mc A$ as the deformation, and similar for quantizations.
\subsubsection{Regular finite $W$-algebras in types{ \sf BCD} }\label{regularBCD}
Let $e$ be a nilpotent element in types {\sf BCD} belonging to the regular Orbit. Equivalently $e$ has Jordan Blocks $(2m+1),(2m),(2m-1,1)$ in $\fkso_{2n+1},\fksp_{2n},\fkso_{2n}$ respectively. Let 
\beq
\label{e:regularfinalshiftmatrix}
 \sigma =
 \begin{pmatrix}
    0 & m-1 \\ 
    m-1 & 0 
\end{pmatrix}.
\eeq
We will prove that
\beq \label{e:regularfinaliso}
    \begin{array}{ccl}
    \Y_{1,2m+1}^+ &\overset{\sim}{\longrightarrow}& \rU(\fkso_{2m+1}, e), \\
    \Y_{1,2m}^+ &\overset{\sim}{\longrightarrow} &\rU(\fksp_{2m}, e), \\
    \Y_{2,2m-1}^+(\sigma) &\overset{\sim}{\longrightarrow} &\rU(\fkso_{2m}, e).
    \end{array}
\eeq
For $e$ in the regular nilpotent orbit by a result of Kostant \cite[Thm. 2.4.1]{Ko78} we have $\rU(\g,e) \cong Z(\rU(\g))$. The isomorphism with the corresponding truncated shifted twisted Yangian then follows from noting that the inclusion $Z(\Y^{\pm}_{N,\ell}(\sigma)) \hookrightarrow \Y^{\pm}_{N,\ell}(\sigma)$ is an isomorphism in these cases. To prove this is an isomorphism in these cases one uses an argument analogous to the proof of Theorem \ref{truncatedcenter} noting that the associated graded map $\gr'Z(\Y^{\pm}_{N,\ell}(\sigma)) \hookrightarrow \gr' \Y^{\pm}_{N,\ell}(\sigma) \cong \mathbb{C}[\cS_\chi] \cong Z(\mathbb{C}[\g^*])$ is an isomorphism by explicit description of $Z(\Y^{\pm}_{N,\ell}(\sigma))$ in Theorem \ref{truncatedcenter} and Corollary \ref{pfcases} as well as the invariant theory discussed in \S\ref{ss:pfaffian} and after Theorem \ref{truncatedcenter}, the isomorphism $\mathbb{C}[\cS_\chi] \cong Z(\mathbb{C}[\g^*])$ is  explained in the footnote to \cite[Question 5.1]{PrJI}. This gives that $\Y^{\pm}_{N,\ell}(\sigma)$ is commutative in these cases and using the PBW Theorems \ref{trtwypbw} and \ref{trtwypbw-new} as well as an explicit description of $Z(\rU(\g))$ in terms of the Harish-Chandra isomorphism (see e.g. \cite[\textsection 7.1]{Mol07}) one sees that these algebras are isomorphic.
\subsubsection{Subregular finite $W$-algebras in type {\sf B}}
\label{ss:subregularB}

Let $e\in \fkso_{2m+1}$ be a nilpotent element belonging to the subregular orbit. Equivalently, $e$ has Jordan blocks of sizes $(2m-1, 1, 1)$. Let $\sigma$ be the shift matrix described in \eqref{e:typeBfinalshiftmatrix}. We will prove that
\begin{eqnarray}
    \label{e:typeBfinaliso}
    \Y_{\mu,2m-1}^+(\sigma) \overset{\sim}{\longrightarrow} \rU(\fkso_{2m+1}, e).
\end{eqnarray}

Let $\chi \in \g^*$ be the form associated with $e$. Let $\tau$ and $\tau'$ be the involutions of $\cN_\chi$ introduced in \S\ref{ss:twistedKleinianSubregular}. By \cite[Thm.~4.6]{ACET23} we know that $\bC[\cS_\chi]$ is the universal Poisson $\tau'$-deformation and $U(\fkso_{2m+1}, e)$ is the universal $\tau'$-quantization of $\cN_\chi$.

By Corollary~\ref{gryangsliceAI} we see that $y^+_{\mu, 2m-1}(\sigma)$ is a universal Poisson $\tau'$-deformation of $\cN_\chi$. By Proposition~\ref{P:identicalsubfunctors} it is also a universal Poisson $\tau$-deformation.

By construction \eqref{e:tauautsubregB}, the automorphism $\tau$ lifts to $\Y^+_{\mu, 2m-1}(\sigma)$. By Theorem~\ref{truncatedcenter} we deduce that $\Y^+_{\mu, 2m-1}(\sigma)$ is a $\tau$-quantization of $\cN_\chi$. Applying \cite[Theorem~1.1(2)]{ACET23}, we conclude that $\Y^+_{\mu, 2m-1}(\sigma)$ is a universal $\tau$-quantization of $\cN_\chi$.

It follows from Lemma~\ref{lemonslemma} that the universal elements of $\rQ_{X,\tau}$ and $\rQ_{X, \tau'}$ are isomorphic. Combining these remarks, we have proven \eqref{e:typeBfinaliso}.

\subsubsection{Finite $W$-algebras in type {\sf C}: two even Jordan blocks}
\label{ss:twoevenblocksC}
    Let $m > 1$ and $e \in \fksp_{2m}$ belong to the nilpotent orbit with partition $\lambda = (\lambda_1, \lambda_2)$, where $\lambda_1 \lle \lambda_2$ are both even. Let $\sigma = (\fks_{i,j})$ be the $2\times 2$ shift matrix with $\fks_{1,2} = \fks_{2,1} = \frac{\lambda_2-\lambda_1}{2}$ and $\fks_{1,1} = \fks_{2,2} = 0$, and let $\mu = (1,1)$. 

    Our goal is to show that
\begin{eqnarray}
    \label{e:typeCevenblocksfinaliso}
    \Y_{\mu,\lambda_2}^+(\sigma) \overset{\sim}{\longrightarrow} \rU(\fksp_{2m}, e).
\end{eqnarray}
    
    According to \cite[Thm.~1.1]{AT25} the nilpotent Slodowy variety $\cN_\chi$ associated to the pair $(\g,e)$ is $\bC^\times$-Poisson isomorphic to the nilpotent Slodowy variety $\cN_\psi$ associated to a nilpotent element of $\fkso_{2m+2}$ with partition $(\lambda_1 + 1, \lambda_2+1)$.  Identify $\cN_\chi = \cN_{\psi}$ by a choice of $\bC^\times$-Poisson isomorphism.
    
    When $m \ne 3$ we have $\Aut(\fkso_{2m+2}) \twoheadrightarrow \Gamma := \Out(\fkso_{2m+2}) \cong \bZ /2\bZ$ and, under our assumptions on $\lambda$, the orbit of $\psi$ is stable under $\Aut(\fkso_{2m+2})$. When $m = 3$ we let $\Gamma \subseteq \Aut(\fkso_{8})$ be the cyclic group of order 2 generated by the graph automorphism exchanging the vertices labelled 3 and 4 in the Bourbaki ordering. Following \cite[Lemma~4.6]{AT25} we choose a splitting $s : \Gamma \to \Aut(\fkso_{2m+2})$ such that $s(\Gamma)$ preserves $\cS_{\psi}$.
    
    By Theorem~1.3 of {\it op. cit.} the finite $W$-algebra $\rU(\fksp_{2m}, e)$ is a universal $\Gamma$-quantization of $\cN_\chi$.

    Let $\gamma \in s(\Gamma)$ denote the non-trivial automorphism. The effect of $\gamma$ on generators of $\bC[\cS_\chi]$ was calculated in \cite[Prop.~4.9]{AT25}. To describe it, we resume the notation $h_i^{(r)}$ and $b_1^{(r)}$ from \S\ref{ss:twistedKleinianSubregular}. Then $y^+_{\mu, \lambda_2}(\sigma)$ is generated by these elements and, passing through the isomorphism \cite[Thm.~A]{TT24}, we see that $\gamma(b_1^{(r)}) = -1$, while $\gamma(h_i^{(r)}) = h_i^{(r)}$ for all $i,r$.

    Lifting this automorphism to the generators $\Y_{\mu,\lambda_2}^+(\sigma)$ in the obvious manner, one can check the relation of Theorem~\ref{mainthm} to conclude that this extends to a well-defined automorphism of $\Y_{\mu,\lambda_2}^+(\sigma)$. To see that $\Y_{\mu, \lambda_2}^+(\sigma)$ is a $\Gamma$-quantization of $\cN_\chi$ it suffices to show that the Poisson center of $y^+_{\mu, \lambda_2}(\sigma)$ lifts to $\Y^+_{\mu, \lambda_2}(\sigma)$, which follows from Theorem~\ref{nonpfcases}.

    We have shown that both $\Y^+_{\mu, \lambda_2}(\sigma)$ and $\rU(\fksp_{2m}, e)$ are elements of $\rQ_{\cN_\chi, \Gamma}$ with associated graded Poisson deformation $\bC[\cS_\chi]$. The isomorphism \eqref{e:typeCevenblocksfinaliso} follows from Corollary~\ref{C:uniqueuniversalquantization}.

\subsubsection{Finite $W$-algebras in type {\sf C}: two odd Jordan blocks}
\label{ss:twooddblocksC}

    In this subsection, let $m$ be odd and let $e \in \fksp_{2m}$ be a nilpotent element with two Jordan blocks, each of size $m$. The isomorphism
    \begin{eqnarray}
    \label{e:typeCoddblocksfinaliso}
    \Y_{2,m}^- \overset{\sim}{\longrightarrow} \rU(\fksp_{2m}, e)
\end{eqnarray}
is a special case of the main result of \cite{Br09}.

\subsubsection{Finite $W$-algebras for the remaining even orbits in types {\sf B} and {\sf C}}
\label{ss:alltheothers}
 Let $\g$ be a Lie algebra of type {\sf B} or {\sf C}, and assume that $\mathbb O = \Ad(G)e$ is an even nilpotent orbit, not appearing in \S\ref{ss:subregularB}--\ref{ss:twooddblocksC}. This nilpotent element gives rise to a partition $\lambda = (\lambda_1,...,\lambda_N)$, and therefore a symmetric pyramid and a symmetric shift matrix by Proposition~\ref{pyrabi}. We set the level to be $\ell := \lambda_N$.
 
 Let $\Y_{N,\ell}^\pm(\sigma)$ be the truncated shifted twisted Yangian, where the type depends on $(\lambda, \ell)$ in accordance with cases (i)--(iv) of \S\ref{subsec:pyramid}. We prove
 \begin{eqnarray}
     \label{e:othereventypesBCfinaliso}
     \Y^\pm_{N, \ell}(\sigma) \overset{\sim}{\longrightarrow} \rU(\g, e).
 \end{eqnarray}
 By \cite[Thm.~3.5]{ACET23} we see that $\bC[\cS_\chi]$ is the universal Poisson deformation of $\cN_\chi$, i.e. it is an initial element of $\rPD_X$. To prove the isomorphism \eqref{e:othereventypesBCfinaliso}, we apply Corollary~\ref{C:uniqueuniversalquantization} with $\Gamma$ trivial. It will be enough to show that both $\Y^\pm_{N, \ell}(\sigma)$ and $\rU(\g, e)$ are filtered quantizations of $\bC[\cS_\chi]$ and that the Poisson center lifts to both quantizations.

 These properties were noted in \eqref{e:Walgquantizes} and \eqref{e:Walgcenterquantizes}, for $\rU(\g,e)$. For $\Y^\pm_{N,\ell}(\sigma)$ use Corollary~\ref{gryangsliceAI}, Corollary~\ref{gryangsliceAII} and Theorem~\ref{truncatedcenter}.

 \subsubsection{Finite $W$-algebras in type {\sf D}: remaining two Jordan block cases}
 
 Let $e \in \fkso_{2m}$ have partition $\lambda = (\lambda_1, \lambda_2)$ with $\lambda_1 \lle \lambda_2$. Let $\sigma$ be the even shift matrix associated with $\lambda$ by Proposition~\ref{pyrabi}, i.e. $\fks_{1,2} = \fks_{2,1} = \frac{\lambda_2-\lambda_1}{2}$, and set $\ell = \lambda_2$. To prove 
 \begin{eqnarray}
 \label{e:typeDtwoclocksfinaliso}
     \Y^+_{2,\ell}(\sigma) \overset{\sim}{\longrightarrow} \rU(\fkso_{2m}, e)
 \end{eqnarray}
 we apply precisely the same argument as \S\ref{ss:alltheothers}, applying Corollary~\ref{pfcases} in place of Theorem~\ref{truncatedcenter}.
 In summary we have the following theorem.
 
 \begin{thm}\label{t:finalisothm}
     Assume we are in one of the cases in \eqref{nonpfcases}, we have $\sigma=0$ or we are in the case $\Y_{2,\ell}^{+}(\sigma)$ for $\ell$ odd. Then we have 
     \begin{equation}\label{e:YangiantoWAlgiso}
       \Y_{N,\ell}^{\pm}(\sigma) \overset{\sim}{\longrightarrow} \rU(\g_{M}^{\sgn(\mp (-1)^{\ell})}, e)  
     \end{equation} 
     where $e$ is a nilpotent element with partition $\lambda$ obtained from $\ell, \sigma$ as explained in \S\ref{subsec:pyramid}, $M$ is the number of boxes in the pyramid associated to $\lambda$ and $\g_M^{\pm}$ is as in \eqref{liealgtype}.
 \end{thm}
 \begin{proof}
 The cases \eqref{nonpfcases} and $\Y_{2,\ell}^{+}(\sigma)$ for $\ell$ odd follow from the previous results in \S \ref{ss:YangWalg}, whereas the case $\sigma=0$ follows from the main result of \cite{Br09}. 
 \end{proof}

\end{document}